\documentclass[10pt]{amsart}
\usepackage{amssymb}

\newtheorem{thm}{Theorem}[section]
\newtheorem{lemma}[thm]{Lemma}
\newtheorem{proposition}[thm]{Proposition}

\newtheorem{definition}[thm]{Definition}
\newtheorem{corollary}[thm]{Corollary}

\newtheorem{notation}[thm]{Notation}
\newtheorem{beginthm}{Theorem}

\newtheorem{assumption}[thm]{Assumption}

\newcommand{\p}{\mathbb{P}}
\newcommand{\q}{\mathbb{Q}}

\newcommand{\cf}{\mathrm{cf}}
\newcommand{\cof}{\mathrm{cof}}
\newcommand{\dom}{\mathrm{dom}}

\newcommand{\cl}{\mathrm{cl}}

\begin{document}

\title{The approachability ideal without a maximal set}

\author{John Krueger}

\address{John Krueger \\ Department of Mathematics \\ 
University of North Texas \\
1155 Union Circle \#311430 \\
Denton, TX 76203}
\email{jkrueger@unt.edu}

\date{July 2016; revised October 2018}



\thanks{This material is based upon work supported by the National 
Science Foundation under Grant No.\ DMS-1464859.}

\begin{abstract}
We develop a forcing poset with finite conditions 
which adds a partial square sequence on a given 
stationary set, with adequate sets of models as side conditions. 
We then develop a kind of side condition product forcing for simultaneously 
adding partial square sequences on multiple stationary sets. 
We show that certain quotients of such forcings have the 
$\omega_1$-approximation property. 
We apply these ideas to prove, assuming the consistency of a greatly Mahlo 
cardinal, that it is consistent that 
the approachability ideal $I[\omega_2]$ does not have a 
maximal set modulo clubs.
\end{abstract}

\maketitle

\tableofcontents

\newpage

\begin{center}
\textbf{Introduction}
\end{center}

\bigskip

Let $\kappa$ be an uncountable cardinal. 
For a given sequence $\vec a = \langle a_i : i < \kappa^+ \rangle$ 
of subsets of $\kappa^+$ with size less than $\kappa$, define $S_{\vec a}$ 
to be the set of limit ordinals $\beta < \kappa^+$ for which there exists 
a set $c \subseteq \beta$, which is cofinal in $\beta$, and has order type 
equal to $\cf(\beta)$, which is approximated by the sequence 
$\vec a \restriction \beta$ in the sense that for all $\gamma < \beta$, 
$c \cap \gamma \in \{ a_i : i < \beta \}$.

Define the approachability ideal 
$I[\kappa^+]$ as the collection of sets $S \subseteq \kappa^+$ such that 
for some sequence $\vec a$ as above, 
and for some club $C \subseteq \kappa^+$, 
$S \cap C \subseteq S_{\vec a}$. 
In other words, $I[\kappa^+]$ is the ideal on $\kappa^+$ which is 
generated over the nonstationary ideal on $\kappa^+$ 
by sets of the form $S_{\vec a}$. 
The collection $I[\kappa^+]$ is a normal ideal on $\kappa^+$.

The approachability ideal was introduced by Shelah in the 1970's (\cite{shelah1}), 
and since then it 
has played a role as an important tool in combinatorial set theory and forcing. 
A major result on the approachability ideal is that if 
$\kappa$ is a regular uncountable cardinal, then the 
set $\kappa^+ \cap \cof(<\! \kappa)$ 
is a member of $I[\kappa^+]$ (\cite[Section 4]{shelah2}).
Hence, when $\kappa$ is a regular uncountable cardinal, 
the structure of $I[\kappa^+]$ is completely determined by which 
stationary subsets of $\kappa^+ \cap \cof(\kappa)$ are in $I[\kappa^+]$.

Shelah \cite{shelah2} raised the question whether it is consistent that 
there are no stationary subsets of $\kappa^+ \cap \cof(\kappa)$ in 
$I[\kappa^+]$. 
This problem was solved by Mitchell \cite{mitchell}, who proved that 
it is consistent, relative to the consistency of a greatly Mahlo cardinal, 
that there is no stationary subset of $\omega_2 \cap \cof(\omega_1)$ 
in $I[\omega_2]$. 
In the process of solving this problem, Mitchell introduced a number of powerful 
new ideas in forcing, including strongly generic conditions, strong properness, 
and a method for using side conditions to add by forcing a club subset of 
$\omega_2$ with finite conditions (see Friedman \cite{friedman} for a similar 
method which was introduced independently).

Assuming that 
$(\kappa^+)^{< \kappa} = \kappa^+$, 
we can enumerate all subsets of $\kappa^+$ of size less than $\kappa$ 
in a single sequence $\vec b = \langle b_i : i < \kappa^+ \rangle$. 
It is not hard to show that if $\vec a = \langle a_i : i < \kappa^+ \rangle$ 
is any sequence of subsets of $\kappa^+$ of size less than $\kappa$, then 
there exists a club $C \subseteq \kappa^+$ such that 
$S_{\vec a} \cap C \subseteq S_{\vec b}$. 
It follows that $I[\kappa^+]$ is generated over the nonstationary 
ideal on $\kappa^+$ by the single set $S_{\vec b}$. 
Another way of describing this conclusion is that $I[\kappa^+]$ has a 
maximum set modulo clubs, which is easily seen to be equivalent to it having a 
\emph{maximal} set modulo clubs.

A natural question is whether the approachability ideal $I[\kappa^+]$ must 
\emph{always} have a maximal set modulo clubs, regardless of any 
cardinal arithmetic assumptions. 
That is, is it consistent that $I[\kappa^+]$ does not have a single generator 
over the nonstationary ideal. 
By the normality of $I[\kappa^+]$, 
this possibility is equivalent to having not fewer 
than $\kappa^{++}$ many generators. 
This question was first raised by Shelah in \cite{shelah2}, in the same place where 
he mentions the possibility of $I[\kappa^+]$ not containing any stationary subset 
of $\kappa^+ \cap \cof(\kappa)$. 
The problem also appears at the end of \cite{mitchell}, where Mitchell 
suggests that the methods introduced in his paper are likely to be useful 
for answering the question. 
In this paper we will solve this problem.

\begin{beginthm}
Assuming the consistency of a greatly Mahlo cardinal, it is consistent that 
$I[\omega_2]$ does not have a maximal set modulo clubs.
\end{beginthm}

Let us give an overview of the structure of the paper 
and the ideas which will be used in the proof of the theorem.  
In Part 1, we present the material on side conditions which will be the 
foundation for everything else in the paper. 
We will use the adequate set framework of side conditions. 
This framework was first introduced in \cite{jk21}. 
In this paper, 
we will follow the presentation which was given later in \cite{jk27}.

In Part II, we develop a forcing poset with finite conditions 
for adding a partial square sequence to a 
given stationary set 
$S \subseteq \omega_2 \cap \cof(\omega_1)$. 
Recall that a \emph{partial square sequence on $S$} is a sequence 
$\langle c_\alpha : \alpha \in S \rangle$ satisfying that for each $\alpha \in S$, 
$c_\alpha$ is a club subset of $\alpha$ with order type equal to $\omega_1$, 
and whenever $\gamma$ is a common limit point of $c_\alpha$ and 
$c_\beta$, then $c_\alpha \cap \gamma = c_\beta \cap \gamma$. 

If there exists a partial square sequence on $S$, then 
$S$ is in $I[\omega_2]$. 
Namely, define a sequence 
$\vec a = \langle a_\gamma : \gamma < \omega_2 \rangle$ 
as follows. 
For a given ordinal $\gamma$, if there exists some $\alpha \in S$ 
strictly greater than $\gamma$ 
such that $\gamma$ is a limit point of $c_\alpha$, then let 
$a_\gamma := c_\alpha \cap \gamma$. 
Define $a_\xi$ for all other ordinals $\xi$ in such a way as to include any 
initial segment of any set of the form 
$a_\gamma$, where $\gamma$ is an ordinal of the first type. 
One can easily check that for some club $C \subseteq \omega_2$, 
$S \cap C \subseteq S_{\vec a}$. 
Therefore, $S \in I[\omega_2]$.

Forcing a square sequence with finite conditions was first achieved by 
Dolinar and Dzamonja \cite{mirna}, using Mitchell's style of models as 
side conditions \cite{mitchell}. 
Later, Krueger \cite{jk23} developed a forcing poset 
for adding a square sequence with 
finite conditions using the framework of coherent adequate sets. 
And Neeman \cite{neeman} 
defined a forcing poset for adding 
a square sequence using his framework of two-type side conditions.

The forcing poset we present in Part II for adding a 
partial square sequence 
is similar to the forcings of \cite{jk23} and 
\cite{neeman} for adding a square sequence. 
However, we will need to develop the properties of our forcing poset in much 
greater detail than was done in those papers, so that we can use it to prove 
the consistency result. 
In particular, in Sections 8 and 9 we will derive some very technical information 
in order to prove that certain quotients of the 
forcing poset satisfy the $\omega_1$-approximation property.

In Part III, we develop a forcing poset $\q$ 
which simultaneously adds a partial square 
sequence on multiple sets. 
This forcing poset is similar to a product forcing, since the 
different posets which are incorporated in the forcing are 
independent of each other, 
except for the presence of a shared side condition. 
We believe that it is likely that this kind of side condition product will have 
other applications in the future.

A crucial property of the product forcing $\q$ presented in 
Section 10 for proving our consistency result on $I[\omega_2]$ is that 
certain quotients of it satisfy the $\omega_1$-approximation property. 
More specifically, in Section 12 we will show that 
for certain uncountable models $P$, 
$P \cap \q$ is a regular suborder of $\q$, and the quotient forcing 
$\q / \dot G_{P \cap \q}$ has the $\omega_1$-approximation property 
in $V^{P \cap \q}$.

A similar result about certain quotients having the 
$\omega_1$-approximation property 
was used by Mitchell \cite{mitchell} in his proof of the consistency 
that $I[\omega_2]$ does not contain a stationary subset of 
$\omega_2 \cap \cof(\omega_1)$. 
This result followed from the equation 
$$
(p \land q) \restriction P = 
(p \restriction P) \land (q \restriction P),
$$
where $\land$ denotes greatest lower bound, 
which holds below a strongly $P$-generic condition which is tidy 
(see \cite[Definition 2.20, Lemma 2.22]{mitchell}). 
Unfortunately, our forcing poset $\q$ does not satisfy this equation. 
First, our forcing poset $\q$ does not even have greatest lower bounds. 
Secondly, even if the definition of $\q$ is adjusted so that $\q$ has 
greatest lower bounds, which is possible, the above equation still fails, 
even on any dense set.

Nonetheless, we are able to make use of some of the ideas in 
Mitchell's original argument for 
the $\omega_1$-approximation property \cite[Lemma 2.22]{mitchell}, 
by replacing the above equation with something weaker, and more complicated, 
namely, 
$$
(q \oplus^N p) \restriction P = 
(q \restriction P) \oplus^{N \cap P} (p \restriction P). 
$$
In this equation, $a \oplus^M b$ denotes the amalgam of a condition $b$ 
with a condition $a$ which is in the model $M$ and is below the projection 
$b \restriction M$ (see Proposition 12.6). 
We believe that this equation will be useful in future applications 
for verifying the approximation property, in cases 
where Mitchell's original tidy property fails.

Finally, in Section 13 we complete the proof of the consistency that $I[\omega_2]$ 
does not have a maximal set modulo clubs. 
Assuming that $\kappa$ is a greatly Mahlo cardinal, we get a sequence 
$\langle B_i : i < \kappa^+ \rangle$ of Mahlo sets. 
We use the forcing poset $\q$ from Part 3 to simultaneously 
add partial square sequences on $B_i \setminus B_{i+1}$, for each 
$i < \kappa^+$, while collapsing $\kappa$ to become $\omega_2$. 
This will place each such set in the approachability ideal $I[\omega_2]$. 
We make use of the approximation property of certain quotients of $\q$ to 
show that $I[\omega_2]$ does not have a maximal set.

\bigskip

I would like to thank Thomas Gilton for carefully proofreading several drafts 
of this paper and making many useful suggestions.

\bigskip

\part{Background}

\bigskip

\addcontentsline{toc}{section}{1. Preliminaries}

\textbf{\S 1. Preliminaries}

\stepcounter{section}

\bigskip

The prerequisites for reading this paper are a 
background of one year of graduate level study in set theory, 
a working knowledge of forcing, and some basic familiarity with 
proper forcing and generalized stationarity.

For a regular uncountable cardinal $\lambda$ 
and a set $X$ with $\lambda \subseteq X$, we let 
$P_\lambda(X)$ denote the set $\{ a \subseteq X : |a| < \lambda \}$. 
A set $S \subseteq P_\lambda(X)$ is stationary iff 
for any function $F : X^{<\omega} \to X$, 
there exists $b \in S$ such that $b \cap \lambda \in \lambda$ and 
$b$ is closed under $F$. 

\bigskip

In this paper, a \emph{forcing poset} is a pair $(\p,\le_\p)$, where 
$\p$ is a nonempty set and $\le_\p$ is a reflexive and transitive relation on $\p$. 
To simplify notation, we usually refer to $\p$ itself as a forcing poset, with 
the relation $\le_\p$ being implicit. 
If $\q$ is a forcing poset, we will write $\dot G_\q$ for the 
canonical $\q$-name for a generic filter on $\q$.

Let $\p$ and $\q$ be forcing posets. 
Then $\p$ is a \emph{suborder} of $\q$ if $\p \subseteq \q$ and 
$\le_\p \ = \ \le_\q \cap \ (\p \times \p)$. 
Let $\p$ be a suborder of $\q$. 
We say that $\p$ is a \emph{regular suborder} of $\q$ if:
\begin{enumerate}
\item whenever $p$ and $q$ are in $\p$ and are incompatible in $\p$, then 
$p$ and $q$ are incompatible in $\q$;
\item if $A$ is a maximal antichain of $\p$, then $A$ is predense in $\q$.
\end{enumerate}

\begin{lemma}
Suppose that $\p_0$ is a suborder of $\p_1$, and $\p_1$ is a suborder 
of $\q$. 
Assume, moreover, that $\p_0$ and $\p_1$ are both regular 
suborders of $\q$. 
Then $\p_0$ is a regular suborder of $\p_1$.
\end{lemma}

\begin{proof}
Straightforward.
\end{proof}

Let $\p$ be a regular suborder of $\q$, and assume that $G$ is a generic 
filter on $\p$. 
In $V[G]$, define the forcing poset $\q / G$ to consist of conditions $q \in \q$ 
such that for all $s \in G$, $q$ and $s$ are compatible in $\q$, with the same 
ordering as $\q$. 
Then $\q$ is forcing equivalent to the two-step iteration 
$\p * (\q / \dot G_\p)$. 
Moreover:

\begin{lemma}
Let $\p$ be a regular suborder of $\q$.
\begin{enumerate}
\item Suppose that $H$ is a $V$-generic filter on $\q$. 
Then $H \cap \p$ is a $V$-generic filter on $\p$, and 
$H$ is a $V[H \cap \p]$-generic filter on $\q / (H \cap \p)$. 
\item Suppose that $G$ is a $V$-generic filter on $\p$ and $H$ is a 
$V[G]$-generic filter on $\q / G$. 
Then $H$ is a $V$-generic filter on $\q$, 
$G = H \cap \p$, and $V[G][H] = V[H]$.
\end{enumerate}
\end{lemma}

\begin{proof}
See \cite[Lemma 1.6]{jk26}.
\end{proof}

Let $\p$ and $\q$ be forcing posets with maximum conditions. 
A function $\pi : \q \to \p$ is said to be a \emph{projection mapping} if:
\begin{enumerate}
\item $\pi$ maps the maximum condition in $\q$ to the maximum condition 
in $\p$;
\item if $q \le p$ in $\q$, then $\pi(q) \le \pi(p)$ in $\p$;
\item if $v \le \pi(q)$ in $\p$, then there is 
$r \le q$ in $\q$ such that $\pi(r) \le v$ in $\p$.
\end{enumerate}

If $\pi : \q \to \p$ is a projection mapping, and $G$ is a generic filter on $\q$, 
then the set 
$\{ s \in \p : \exists p \in G \ \pi(p) \le s \}$ is a generic filter 
on $\p$.

\bigskip

Let $\q$ be a forcing poset. 
For a set $N$ and a condition $q \in \q$, we say that $q$ is 
\emph{strongly $N$-generic} if whenever $D$ is a dense subset of 
$N \cap \q$, then $D$ is predense below $q$ in $\q$.

We say that $\q$ is \emph{strongly proper on a stationary set} if 
for all sufficiently large cardinals $\lambda$ with 
$\q \subseteq H(\lambda)$, 
there are stationarily many $N$ in $P_{\omega_1}(H(\lambda))$ such that 
for all $p \in N \cap \q$, there is $q \le p$ which is strongly $N$-generic. 
If $\q$ is strongly proper on a stationary set, then $\q$ preserves $\omega_1$, 
because being a strongly $N$-generic condition implies being an $N$-generic 
condition in the sense of proper forcing.

Let $\lambda_\q$ denote the smallest cardinal such that 
$\q \subseteq H(\lambda_\q)$. 
Note that a condition $q$ is strongly $N$-generic iff 
$q$ is strongly $(N \cap H(\lambda_\q))$-generic. 
Using this fact, 
standard arguments show that $\q$ is strongly proper on a stationary set 
iff there are stationarily many $N$ in $P_{\omega_1}(H(\lambda_\q))$ 
such that for all $p \in N \cap \q$, there is $q \le p$ which is 
strongly $N$-generic.

\bigskip

Let $V \subseteq W$ be transitive class models of \textsf{ZFC}. 
A set $X \subseteq V$ is said to be \emph{countably approximated by $V$} 
if for any set $a \in V$ which is countable in $V$, $a \cap X \in V$. 
We say that the pair $(V,W)$ has the 
\emph{$\omega_1$-approximation property} if whenever 
$X$ is a subset of $V$ in $W$ which is countably approximated 
by $V$, then $X \in V$. 
A forcing poset $\q$ is said to have the 
\emph{$\omega_1$-approximation property} if $\q$ forces that 
the pair $(V,V^\q)$ has the 
$\omega_1$-approximation property.

Note that the $\omega_1$-approximation property is equivalent 
to the definition in the previous paragraph, 
except replacing the assumption that $X$ is a 
subset of $V$ with the assumption that $X$ is a set of ordinals. 
Namely, if $X \subseteq V$, then for some $\alpha$, $X \subseteq V_\alpha$. 
And in $V$ we can fix a bijection $g : V_\alpha \to \mu$ for some 
ordinal $\mu$. 
Then $X$ is countably approximated by $V$ iff $g[X]$ is countably 
approximated by $V$, and $X \in V$ iff $g[X] \in V$.

The next lemma shows that the $\omega_1$-approximation property defined above 
is equivalent to the 
version of the property used in \cite{mitchell}.

\begin{lemma}
A pair $(V,W)$ has the $\omega_1$-approximation property iff 
whenever $\mu$ is an ordinal, 
$k : \mu \to On$ is in $W$, and 
for any set $a$ in $V$ which is 
countable in $V$, $k \restriction a \in V$, then $k \in V$.
\end{lemma}

\begin{proof}
Assume that the pair $(V,W)$ 
has the $\omega_1$-approximation property, 
and let $k : \mu \to On$ be a function in $W$ satisfying that 
for any countable set $a$ in $V$, 
$k \restriction a \in V$. 
We will show that $k \in V$. 
It suffices to show that whenever $x \in V$ is countable in $V$, 
then $k \cap x \in V$.

Suppose that $x$ is a countable set in $V$, and we will show 
that $k \cap x \in V$. 
Define $x_0 := \{ \xi \in \mu : \exists z \ (\xi,z) \in x \}$. 
Then $x_0$ is a countable subset of $\mu$ in $V$, so 
$k \restriction x_0 \in V$. 
It is easy to check that 
$k \cap x = (k \restriction x_0) \cap x$, 
and hence $k \cap x \in V$.

Conversely, suppose that 
whenever $\mu$ is an ordinal, 
$k : \mu \to On$ is in $W$, and 
for any set $a$ in $V$ which is 
countable in $V$, $k \restriction a \in V$, then $k \in V$. 
We will prove that $(V,W)$ has the $\omega_1$-approximation property. 
Let $X$ be a subset of $V$ in $W$ which is countably 
approximated by $V$, and we will show that $X \in V$. 
By the comments preceding the lemma, 
we may assume that $X \subseteq \mu$ 
for some ordinal $\mu$.

Let $k : \mu \to 2$ be the characteristic function of $X$, 
so that $k(\alpha) = 1$ iff $\alpha \in X$. 
Then $k \in V$ iff $X \in V$, so it suffices to show that $k \in V$. 
To show that $k \in V$, it suffices to show that whenever $a$ is a 
countable set in $V$, then $k \restriction a \in V$. 
So let $a$ be a countable set in $V$. 
Then $X \cap a \in V$ by assumption. 
But $k \restriction a$ is equal to the function with domain 
$a \cap \mu$ such that 
for all $\alpha \in a \cap \mu$, 
$\alpha$ is mapped to $1$ iff $\alpha \in X \cap a$. 
Since $a \cap \mu$ and 
$X \cap a$ are in $V$, so is $k \restriction a$.
\end{proof}

\begin{lemma}
Let $\p$ be a regular suborder of $\q$, and suppose that 
$\q$ forces that the pair 
$(V[\dot G_{\q} \cap \p],V[\dot G_{\q}])$ has the 
$\omega_1$-approximation property. 
Then $\p$ forces that $\q / \dot G_{\p}$ has the 
$\omega_1$-approximation property.
\end{lemma}

\begin{proof}
Let $G$ be a $V$-generic filter on $\p$. 
Then by Lemma 1.2(2), 
whenever $H$ is a $V[G]$-generic filter on $\q / G$, 
then $H$ is a $V$-generic filter on $\q$, $H \cap \p = G$, and 
$V[H] = V[G][H]$. 
By assumption, 
the pair $(V[H \cap \p],V[H])$ has the 
$\omega_1$-approximation property. 
But $(V[H \cap \p],V[H]) = (V[G],V[G][H])$. 
Thus, for any $V[G]$-generic filter $H$ on $\q / G$, 
the pair $(V[G],V[G][H])$ has the $\omega_1$-approximation property. 
This means that $\q / G$ has the $\omega_1$-approximation property 
in $V[G]$.
\end{proof}

\bigskip

\addcontentsline{toc}{section}{2. Side conditions}

\textbf{\S 2. Side conditions}

\stepcounter{section}

\bigskip

In this section, we lay out the basic framework of side conditions which will 
serve as the foundation for almost everything in the paper. 
Our goal is to make this material as self-contained as possible. 
However, we do not want to prove all of the results from scratch, since 
that has already been done in other papers, and some of the proofs are tedious. 
So several of the assumptions and results in this section will be 
stated without proof; we will provide specific references here and in 
Section 13 so that an interested reader can easily find the complete details.

\bigskip

The basic objects we introduce are the cardinals $\kappa$ and $\lambda$, 
the set $\Lambda$, the structure $\mathcal A$ on $H(\lambda)$, 
the classes of countable models $\mathcal X$ and uncountable 
models $\mathcal Y$, 
and the comparison point $\beta_{M,N}$, for all 
$M$ and $N$ in $\mathcal X$.

\begin{notation}
For the remainder of the paper, $\kappa$ is a regular cardinal 
satisfying that $\omega_2 \le \kappa$, and $\lambda$ is a regular cardinal 
such that $\kappa \le \lambda$.
\end{notation}

In this paper, our interest will be in the cases where either 
$\kappa = \omega_2$, or $\kappa$ is an inaccessible cardinal which is intended 
to become $\omega_2$ in some generic extension. 
In the context of adding a single object by forcing, it is natural to let 
$\lambda = \kappa$. 
If multiple objects are being added by forcing, then $\lambda$ will be 
at least $\kappa^+$.

\begin{notation}
Fix a set $\Lambda$ such that for some club $C^* \subseteq \kappa$, 
$$
\Lambda = C^* \cap \cof(>\! \omega).
$$
\end{notation}

\begin{notation}
Fix a structure $\mathcal A$, whose underlying set is $H(\lambda)$, 
which has a well-ordering of $H(\lambda)$ 
as a predicate, 
and for which the sets $\kappa$ and $\Lambda$ are definable predicates.
\end{notation}

Since $\mathcal A$ has a well-ordering as a predicate, it has definable 
Skolem functions. 
For any set $a \subseteq H(\lambda)$, let $Sk(a)$ denote the Skolem 
hull of $a$ in $\mathcal A$ under some (any) complete set of definable 
Skolem functions.

\begin{notation}
Fix sets $\mathcal X$ and $\mathcal Y$ satisfying:
\begin{enumerate}
\item for all $M \in \mathcal X$, $M$ is a countable elementary 
substructure of $\mathcal A$;
\item for all $P \in \mathcal Y$, $P$ is an elementary substructure of 
$\mathcal A$, $|P| < \kappa$, $P \cap \kappa \in \kappa$, and 
$\cf(P \cap \kappa) > \omega$.
\end{enumerate}
\end{notation}

The next assumption describes some closure properties of $\mathcal X$ and 
$\mathcal Y$.

\begin{assumption}
\begin{enumerate}
\item If $P$ and $Q$ are in $\mathcal Y$, then $P \cap Q \in \mathcal Y$.
\item If $M \in \mathcal X$ and $P \in \mathcal Y$, then 
$M \cap P \in \mathcal X$.
\end{enumerate}
\end{assumption}

Following Friedman \cite{friedman}, we say that 
a stationary set $T^* \subseteq P_{\omega_1}(\kappa)$ is 
\emph{thin} 
if for all $\beta < \kappa$, 
$$
| \{ a \cap \beta : a \in T^* \} | < \kappa.
$$
Implicit in our listed assumptions is the 
existence of a thin stationary set $T^*$, which is used in the actual 
definitions of $\mathcal A$ and $\mathcal X$. 
For example, it will be the case that 
for all $M \in \mathcal X$, $M \cap \kappa \in T^*$. 
In particular, there are a limited number of sets of the form $M \cap \kappa$, 
where $M \in \mathcal X$. 
The next assumption follows as a consequence; see 
\cite[Proposition 1.11]{jk21} for the details.

\begin{assumption}
If $M \in \mathcal X$, $\alpha \in \Lambda \cup \{ \kappa \}$, 
and $Sk(\alpha) \cap \kappa = \alpha$ if 
$\alpha < \kappa$, then $M \cap \alpha \in Sk(\alpha)$.
\end{assumption}

Note that there are club many $\alpha < \kappa$ such that 
$Sk(\alpha) \cap \kappa = \alpha$.  
Also, $Sk(\kappa) \cap \kappa = \kappa$ is immediate.

We have enough information now to derive some useful properties.

\begin{lemma}
Let $M \in \mathcal X$ and 
$\alpha \in \Lambda \cup \{ \kappa \}$. 
If $\alpha < \kappa$, assume that 
$Sk(\alpha) \cap \kappa = \alpha$. Then:
\begin{enumerate}
\item $Sk(M \cap \alpha) = M \cap Sk(\alpha)$;
\item $Sk(M \cap \alpha) \cap \kappa = M \cap \alpha$.
\end{enumerate}
\end{lemma}

\begin{proof}
(1) The forward inclusion is immediate. 
For the reverse inclusion, suppose that $x \in M \cap Sk(\alpha)$. 
Then since $x \in Sk(\alpha)$, there are ordinals 
$\beta_0,\ldots,\beta_{n-1}$ in $\alpha$ and a definable Skolem function $f$ 
of $\mathcal A$ such that $x = f(\beta_0,\ldots,\beta_{n-1})$. 
Since $\kappa$ and $f$ are definable in $\mathcal A$, the 
lexicographically least tuple of ordinals 
$\gamma_0,\ldots,\gamma_{n-1}$ in $\kappa$ 
such that $x = f(\gamma_0,\ldots,\gamma_{n-1})$ is 
definable in $\mathcal A$ from $x$. 
Since $x \in M \cap Sk(\alpha)$, it follows that 
$\gamma_0,\ldots,\gamma_{n-1}$ are in 
$M \cap Sk(\alpha)$. 
But $Sk(\alpha) \cap \kappa = \alpha$, 
so $\gamma_0,\ldots,\gamma_{n-1} \in M \cap \alpha$. 
Therefore, $x = f(\gamma_0,\ldots,\gamma_{n-1}) \in Sk(M \cap \alpha)$.

(2) Using (1) and our assumption about $\alpha$, we have 
$$
Sk(M \cap \alpha) \cap \kappa = (M \cap Sk(\alpha)) \cap \kappa = 
M \cap (Sk(\alpha) \cap \kappa) = M \cap \alpha.
$$
\end{proof}

\begin{lemma}
Let $M$ and $N$ be in $\mathcal X$ 
and $\alpha \in \Lambda \cup \{ \kappa \}$. 
If $\alpha < \kappa$, assume that 
$Sk(\alpha) \cap \kappa = \alpha$. Then:
\begin{enumerate}
\item If $M \cap \alpha \in N$, then 
$M \cap \alpha \in Sk(N \cap \alpha)$;
\item If $M \cap \kappa \subseteq \alpha$ and $M \cap \kappa \in N$, then 
$M \cap \kappa \in Sk(N \cap \alpha)$.
\end{enumerate}
\end{lemma}

\begin{proof}
(1) By Assumption 2.6, $M \cap \alpha \in Sk(\alpha)$. 
By Lemma 2.7(1), $M \cap \alpha \in N \cap Sk(\alpha) = Sk(N \cap \alpha)$. 
(2) If $M \cap \kappa \subseteq \alpha$, then 
$M \cap \kappa = M \cap \alpha$. 
So (2) follows immediately from (1).
\end{proof}

We now introduce the \emph{comparison point} $\beta_{M,N}$, 
for all $M$ and $N$ in $\mathcal X$. 
The actual definition of $\beta_{M,N}$ is not important for us 
in this paper. 
The only properties of $\beta_{M,N}$ which we will need are stated 
in Lemma 2.10 and Proposition 2.11 below.

\begin{notation}
For $M$ and $N$ in $\mathcal X$, $\beta_{M,N}$ will denote the 
\emph{comparison point} of $M$ and $N$, as defined in 
\cite[Definition 1.14]{jk27}.
\end{notation}

\begin{lemma}
Let $M$ and $N$ be in $\mathcal X$. 
Then:
\begin{enumerate}
\item $\beta_{M,N} \in \Lambda$ and $\beta_{M,N} = \beta_{N,M}$;
\item if $\beta < \beta_{M,N}$ and $\beta \in \Lambda$, then 
$M \cap [\beta,\beta_{M,N}) \ne \emptyset$;
\item if $K \in \mathcal X$ and $K \subseteq M$, then 
$\beta_{K,N} \le \beta_{M,N}$.
\end{enumerate}
\end{lemma}

\begin{proof}
See \cite[Definition 1.14, Lemma 1.16(1,3)]{jk27}.
\end{proof}

For a set of ordinals $a$, let $\cl(a)$ denote the union of $a$ together with the 
set of limit points of $a$.

\begin{proposition}
Let $M$ and $N$ be in $\mathcal X$. 
Then 
$$
\cl(M \cap \kappa) \cap \cl(N \cap \kappa) \subseteq \beta_{M,N}.
$$
\end{proposition}

\begin{proof}
See \cite[Lemma 1.15]{jk27}.
\end{proof}

Since the property described  in Proposition 2.11 is extremely important for 
what follows, let us review it for emphasis. 
The property says that if $\gamma$ is an ordinal which is either in $M \cap \kappa$ 
or is a limit point of $M \cap \kappa$, and at the same time, is either in 
$N \cap \kappa$ or is a limit point of $N \cap \kappa$, then 
$\gamma < \beta_{M,N}$.

\begin{lemma}
Let $P \in \mathcal Y$ and $M \in P \cap \mathcal X$. 
Then for all $K \in \mathcal X$, $\beta_{K,M} < P \cap \kappa$.
\end{lemma}

\begin{proof}
Roughly speaking, the reason why this is true is because, given $M$, 
there are 
only countably many possibilities for the value of $\beta_{K,M}$, and hence they 
are all in $P$ by elementarity. 
See \cite[Lemma 1.34(1)]{jk27} for the proof.
\end{proof}

We now use the comparison point $\beta_{M,N}$ to introduce a way of 
comparing models in $\mathcal X$.

\begin{definition}
Let $M$ and $N$ be in $\mathcal X$.
\begin{enumerate}
\item Let $M < N$ if $M \cap \beta_{M,N} \in N$.
\item Let $M \sim N$ if 
$M \cap \beta_{M,N} = N \cap \beta_{M,N}$.
\item Let $M \le N$ if either $M < N$ or $M \sim N$.
\end{enumerate}
\end{definition}

If $M < N$, then by elementarity, the set 
$\cl(M \cap \beta_{M,N})$ is a member of $N$. 
Since $\cl(M \cap \beta_{M,N})$ is countable, it follows that 
$\cl(M \cap \beta_{M,N}) \subseteq N$. 
Also, every initial segment of the set of ordinals 
$M \cap \beta_{M,N}$ is in $N$, since there are only countably 
many initial segments.

\begin{definition}
A finite set $A \subseteq \mathcal X$ is said to be \emph{adequate} 
if for all $M$ and $N$ in $A$, either $M < N$, $M \sim N$, or $N < M$.
\end{definition}

Note that $A$ is adequate iff for all $M$ and $N$ in $A$, 
$\{ M, N \}$ is adequate. 
If $A$ is adequate and $B \subseteq A$, then $B$ is adequate. 
If $M$ and $N$ are in an adequate set $A$, then either 
$M \le N$ or $N \le M$.

\begin{lemma}
Suppose that $M \le N$. 
Then 
$$
M \cap \beta_{M,N} = M \cap N \cap \kappa = 
M \cap N \cap \beta_{M,N}.
$$
\end{lemma}

\begin{proof}
Since $M \le N$, either $M \cap \beta_{M,N} \in N$ or 
$M \cap \beta_{M,N} = N \cap \beta_{M,N}$. 
In either case, $M \cap \beta_{M,N} \subseteq N$. 
So $M \cap \beta_{M,N} \subseteq M \cap N \cap \kappa$. 
Conversely, by Proposition 2.11, 
$M \cap N \cap \kappa \subseteq \beta_{M,N}$, so 
$M \cap N \cap \kappa \subseteq M \cap \beta_{M,N}$. 
This proves the first equality. 
For the second equality, the reverse inclusion is trivial, and the 
forward inclusion follows from Proposition 2.11.
\end{proof}

\begin{lemma}
If $A$ is adequate, $N \in \mathcal X$, and $A \in N$, then 
$A \cup \{ N \}$ is adequate.
\end{lemma}

\begin{proof}
Since $A$ is finite, $A \subseteq N$. 
So it suffices to show that if $M \in N \cap \mathcal X$, then $M < N$. 
But as $M$ is countable, by elementarity every initial segment of 
$M \cap \kappa$ is in $N$, and in particular, $M \cap \beta_{M,N}$ is in $N$ 
(in fact, Proposition 2.11 implies that $M \cap \kappa = M \cap \beta_{M,N}$). 
So $M < N$.
\end{proof}

It turns out that if $\{ M, N \}$ is adequate, then which relation holds between 
$M$ and $N$ is determined by comparing the ordinals 
$M \cap \omega_1$ and $N \cap \omega_1$.

\begin{lemma}
Let $\{ M, N \}$ be adequate. Then:
\begin{enumerate}
\item $M < N$ iff $M \cap \omega_1 < N \cap \omega_1$;
\item $M \sim N$ iff $M \cap \omega_1 = N \cap \omega_1$;
\item $M \le N$ iff $M \cap \omega_1 \le N \cap \omega_1$.
\end{enumerate}
\end{lemma}

\begin{proof}
Suppose that $M < N$, and we will show that 
$M \cap \omega_1 < N \cap \omega_1$. 
Since $\beta_{M,N}$ has uncountable cofinality, 
$\omega_1 \le \beta_{M,N}$. 
Therefore, $M \cap \omega_1$ is an initial segment of $M \cap \beta_{M,N}$, 
and hence is in $N$. 
So $M \cap \omega_1 < N \cap \omega_1$.

Suppose that $M \sim N$, and we will show that 
$M \cap \omega_1 = N \cap \omega_1$. 
Then $M \cap \beta_{M,N} = N \cap \beta_{M,N}$. 
Since $\omega_1 \le \beta_{M,N}$, 
$M \cap \omega_1 = N \cap \omega_1$. 

Conversely, if $M \cap \omega_1 < N \cap \omega_1$, then the implications 
which we just proved rule out the possibilities that 
$N < M$ and $N \sim M$. 
Therefore, $M < N$. 
This completes the proof of (1) and (2), and 
(3) follows immediately.
\end{proof}

\begin{lemma}
Let $A$ be an adequate set. 
Then the relation $<$ is irreflexive and transitive on $A$, 
$\sim$ is an equivalence relation on $A$, $\le$ is 
transitive on $A$, 
and the relations $<$ and $\le$ respect $\sim$.
\end{lemma}

\begin{proof}
Immediate from Lemma 2.17.
\end{proof}

We state a closure property of $\mathcal X$ as an assumption.

\begin{assumption}
Suppose that $M$ and $N$ are in $\mathcal X$ and 
$\{ M, N \}$ is adequate. 
Then $M \cap N \in \mathcal X$.
\end{assumption}

The proof of this assumption depends on the actual definition of 
$\mathcal X$, which we are not giving here. 
The actual definitions of the 
sets of models $\mathcal X$ and $\mathcal Y$ 
which we will use in the main consistency result, as 
well as the proofs that they satisfy the 
assumptions listed in this section, will not appear until Section 13.
We note that the last assumption is fairly easy to verify in the case that 
$\kappa = \lambda$. 
However, when $\lambda > \kappa$, the assumption is highly nontrivial, 
and requires a very careful definition of $\mathcal X$. 
Similar comments apply to the existence of stationarily many simple models, 
which we describe next.

\begin{definition}
A set $N \in \mathcal X$ is said to be \emph{simple} 
if for all $M \in \mathcal X$, if $M < N$ then $M \cap N \in N$.
\end{definition}

\begin{definition}
A set $P \in \mathcal Y$ is said to be \emph{simple} 
if for all $M \in \mathcal X$, $M \cap P \in P$.
\end{definition}

\begin{assumption}
The set of $N \in P_{\omega_1}(H(\lambda))$ such that $N \in \mathcal X$ 
and $N$ is simple is stationary.
\end{assumption}

\begin{assumption}
The set of $P \in P_{\kappa}(H(\lambda))$ such that $P \in \mathcal Y$ 
and $P$ is simple is stationary.
\end{assumption}

We now state the main theorems of the basic theory of adequate sets. 
The proofs of these results depend on technical, and sometimes tedious, 
facts about comparison points, so we will omit them.

First, we handle countable models.

\begin{proposition}
Let $A$ be adequate and $N \in A$. Then 
the set 
$$
B := A \cup \{ M \cap N : M \in A, \ M < N \}
$$
is adequate. 
Moreover, for all $M \in B$, if $M < N$ then $M \cap N \in B$.
\end{proposition}

\begin{proof}
The first statement is proven in \cite[Proposition 1.25]{jk27}. 
The second statement is easy to prove.
\end{proof}

\begin{proposition}[Amalgamation over countable models]
Let $A$ be adequate, $N \in A$, and suppose that for all $M \in A$, 
if $M < N$ then $M \cap N \in A$. 
Assume that $N$ is simple. 
Suppose that $B$ is adequate and 
$$
A \cap N \subseteq B \subseteq N.
$$ 
Then $A \cup B$ is adequate.
\end{proposition}

\begin{proof}
See \cite[Proposition 1.29]{jk27}.
\end{proof}

Let us derive an easy consequence of Proposition 2.24.

\begin{lemma}
Suppose that $M < N$. 
Then $M \sim M \cap N$.
\end{lemma}

Note that by Assumption 2.19, $M \cap N \in \mathcal X$.

\begin{proof}
Applying Proposition 2.24 to the adequate set $\{ M, N \}$, 
we get that the set $\{ M, N, M \cap N \}$ is adequate. 
In particular, $\{ M, M \cap N \}$ is adequate. 
By Lemma 2.15, $M < N$ implies that 
$M \cap \beta_{M,N} = M \cap N \cap \beta_{M,N}$. 
Since $\omega_1 \le \beta_{M,N}$, it follows that 
$M \cap \omega_1 = M \cap N \cap \omega_1$. 
As $\{ M, M \cap N \}$ is adequate, Lemma 2.17(2) implies that 
$M \sim M \cap N$.
\end{proof}

Next, we handle uncountable models.

\begin{proposition}
Let $A$ be adequate and $P \in \mathcal Y$. Then 
the set 
$$
B := A \cup \{ M \cap P : M \in A \}
$$
is adequate. 
Moreover, for all $M \in B$, $M \cap P \in B$.
\end{proposition}

\begin{proof}
The first statement is proven in \cite[Proposition 1.33]{jk27}. 
The second statement is easy to prove.
\end{proof}

\begin{proposition}[Amalgamation over uncountable models]
Let $A$ be adequate, $P \in \mathcal Y$, and suppose that for all $M \in A$, 
$M \cap P \in A$. 
Assume that $P$ is simple. 
Suppose that $B$ is adequate and 
$$
A \cap P \subseteq B \subseteq P.
$$ 
Then $A \cup B$ is adequate.
\end{proposition}

\begin{proof}
See \cite[Proposition 1.35]{jk27}.
\end{proof}

\begin{lemma}
Suppose that $M \in \mathcal X$ and $P \in \mathcal Y$. 
Then $M \sim M \cap P$.
\end{lemma}

Note that by Assumption 2.5(2), $M \cap P \in \mathcal X$.

\begin{proof}
Applying Proposition 2.27 to the adequate set $\{ M \}$, 
we get that $\{ M, M \cap P \}$ is adequate. 
Since $\omega_1 \le P \cap \kappa$, we have that 
$M \cap \omega_1 = M \cap P \cap \omega_1$. 
Hence, by Lemma 2.17(2), $M \sim M \cap P$.
\end{proof}

We will need one more result about simple models.

\begin{lemma}
Suppose that $N \in \mathcal X$ is simple and $P \in \mathcal Y \cap N$ 
is simple. 
Then $N \cap P$ is simple.
\end{lemma}

Note that by Assumption 2.5(2), $N \cap P \in \mathcal X$.

\begin{proof}
Let $M \in \mathcal X$ be such that $M < N \cap P$, 
and we will show that $M \cap (N \cap P) \in N \cap P$. 
It suffices to show that $M \cap N \cap P < N$. 
For then, since $N$ is simple, 
$$
M \cap (N \cap P) = (M \cap N \cap P) \cap N \in N,
$$
and since $P$ is simple,
$$
M \cap (N \cap P) = (M \cap N \cap P) \cap P \in P.
$$
So $M \cap (N \cap P) \in N \cap P$.

Since $M < N \cap P$, we have that 
$M \cap \beta_{M,N \cap P} \in N \cap P$. 
In particular, $M \cap \beta_{M,N \cap P} \in N$. 
We claim that $\beta_{M \cap N \cap P,N} \le \beta_{M,N \cap P}$. 
If not, then by Lemma 2.10(2), we can fix 
$\gamma \in 
(M \cap N \cap P) \cap [\beta_{M,N \cap P},\beta_{M \cap N \cap P,N})$. 
Then by Proposition 2.11, 
$\gamma < \beta_{M,N \cap P}$, which is a contradiction.

Since $M \cap \beta_{M,N \cap P} \in N$ and 
$\beta_{M \cap N \cap P,N} \le \beta_{M,N \cap P}$, it follows that 
$$
M \cap \beta_{M \cap N \cap P,N} \in N.
$$
But $M < N \cap P$ implies that 
$$
M \cap \beta_{M \cap N \cap P,N} \subseteq 
M \cap \beta_{M,N \cap P} \subseteq N \cap P.
$$
Thus, 
$$
M \cap \beta_{M \cap N \cap P,N} = 
(M \cap N \cap P) \cap \beta_{M \cap N \cap P,N},
$$
and so this set is in $N$. 
Hence, $M \cap N \cap P < N$.
\end{proof}

We end this section by deriving some specialized 
consequences of adequacy, which 
will play a role in the arguments concerning 
the forcing poset developed in Part II.

\begin{definition}
A set $a \subseteq P_{\omega_1}(\kappa)$ is said to be an 
\emph{$\in$-chain} if for all $x$ and $y$ in $a$, either 
$x = y$, $x \in Sk(y)$, or $y \in Sk(x)$.
\end{definition}

\begin{lemma}
Suppose that $a \subseteq P_{\omega_1}(\kappa)$ is an $\in$-chain. 
Assume that for all $x \in a$, $Sk(x) \cap \kappa = x$. 
Then for all $x$ and $y$ in $a$, $x \in Sk(y)$ iff $\sup(x) < \sup(y)$, 
and $x = y$ iff $\sup(x) = \sup(y)$.
\end{lemma}

\begin{proof}
Straightforward.
\end{proof}

\begin{lemma}
Let $A$ be an adequate set, $\alpha \in \Lambda$, and assume that 
$Sk(\alpha) \cap \kappa = \alpha$. 
Then the set 
$$
C := \{ M \cap \alpha : M \in A, \ \alpha \in M \}
$$
is a finite $\in$-chain and a subset of $Sk(\alpha)$.
\end{lemma}

\begin{proof}
By Assumption 2.6, $C$ is a subset of $Sk(\alpha)$. 
Let $M$ and $N$ be in $A$ such that $\alpha \in M \cap N$. 
We will show that either $M \cap \alpha = N \cap \alpha$, 
$M \cap \alpha \in Sk(N \cap \alpha)$, or $N \cap \alpha \in Sk(M \cap \alpha)$. 
Without loss of generality, assume that $M \le N$. 
Then either $M \cap \beta_{M,N} = N \cap \beta_{M,N}$, or 
$M \cap \beta_{M,N} \in N$.

Since $\alpha \in M \cap N \cap \kappa$, it follows that 
$\alpha < \beta_{M,N}$ by Proposition 2.11. 
So if $M \cap \beta_{M,N} = N \cap \beta_{M,N}$, then intersecting 
both sides of this equation with $\alpha$ we get that 
$M \cap \alpha = N \cap \alpha$. 
Assume that $M \cap \beta_{M,N} \in N$. 
Then since $M \cap \alpha$ is an initial segment of $M \cap \beta_{M,N}$, 
$M \cap \alpha \in N$. 
By Lemma 2.8(1), $M \cap \alpha \in Sk(N \cap \alpha)$.
\end{proof}

\bigskip

\addcontentsline{toc}{section}{3. Remainders}

\textbf{\S 3. Remainders}

\stepcounter{section}

\bigskip

In the standard development of the basic ideas of adequate sets, the next topic 
which comes up is the idea of a remainder point 
(see \cite[Section 2]{jk27}). 
In this paper, however, only a particular kind of remainder point 
will be relevant, namely, those which are in the set $r^*$ defined next.

\begin{definition}
Let $A$ be an adequate set. 
Define $r^*(A)$ as the set of ordinals $\gamma$ satisfying that for some 
$K$ and $M$ in $A$ with $K \sim M$,
$$
\gamma = \min((M \cap \kappa) \setminus \beta_{K,M}).
$$
\end{definition}

Note that $r^*(A)$ is finite. 
Also, $A \subseteq B$ implies that $r^*(A) \subseteq r^*(B)$.

Before analyzing the set $r^*(A)$, we first prove a very useful lemma.

\begin{lemma}
Suppose that $M$ and $N$ are in $\mathcal X$ and $\{ M, N \}$ is adequate. 
Assume that $\alpha$ and $\gamma$ are uncountable ordinals satisfying:
\begin{enumerate}
\item $\alpha \in M \cap \kappa$;
\item $\gamma \in (N \cap \kappa) \cup \{ \kappa \}$;
\item $\alpha \ne \gamma$;
\item $M \cap \alpha = N \cap \gamma$.
\end{enumerate}
Then $M \sim N$ and 
$\alpha = \min((M \cap \kappa) \setminus \beta_{M,N})$. 
\end{lemma}

\begin{proof}
Since $\alpha$ and $\gamma$ are uncountable, 
$M \cap \alpha = N \cap \gamma$ 
implies that $M \cap \omega_1 = N \cap \omega_1$. 
As $\{ M, N \}$ is adequate, Lemma 2.17(2) implies that $M \sim N$. 
In particular, $M \cap \beta_{M,N} = N \cap \beta_{M,N}$.

Since $\alpha \ne \gamma$, 
either $\gamma < \alpha$ or $\alpha < \gamma$. 
First, assume that $\gamma < \alpha$. 
Then $\gamma \in N \cap \kappa$. 
Since $\gamma < \alpha$ and $M \cap \alpha = N \cap \gamma$, it follows 
that $\gamma \notin M$. 
But $\gamma \in N \cap \kappa$. 
As $M \cap \beta_{M,N} = N \cap \beta_{M,N}$ and $\gamma \notin M$, 
$\beta_{M,N} \le \gamma$. 
So $\beta_{M,N} \le \gamma < \alpha$.

We claim that $\alpha = \min((M \cap \kappa) \setminus \beta_{M,N})$. 
If not, then there is $\alpha_0 \in M \cap \kappa$ such that 
$\beta_{M,N} \le \alpha_0 < \alpha$. 
Then $\alpha_0 \in M \cap \alpha = N \cap \gamma$, so 
$\alpha_0 \in N$. 
By Proposition 2.11, $\alpha_0 \in M \cap N \cap \kappa 
\subseteq \beta_{M,N}$. 
So $\alpha_0 < \beta_{M,N}$, which contradicts the choice of $\alpha_0$.

Secondly, assume that $\alpha < \gamma$. 
Since $M \cap \alpha = N \cap \gamma$, we have that $\alpha \notin N$. 
As $M \cap \beta_{M,N} = N \cap \beta_{M,N}$ and 
$\alpha \in M \cap \kappa$, it follows that 
$\beta_{M,N} \le \alpha$.

We claim that $\alpha = \min((M \cap \kappa) \setminus \beta_{M,N})$. 
If not, then there is $\alpha_0 \in M \cap \kappa$ such that 
$\beta_{M,N} \le \alpha_0 < \alpha$. 
But then $\alpha_0 \in M \cap \alpha = N \cap \gamma$. 
By Proposition 2.11, it follows that 
$\alpha_0 \in M \cap N \cap \kappa \subseteq \beta_{M,N}$. 
So $\alpha_0 < \beta_{M,N}$, which contradicts the choice of $\alpha_0$.
\end{proof}

The main goal of this section is to prove Proposition 3.5, which 
handles models in $\mathcal X$, and Proposition 3.8, which handles 
models in $\mathcal Y$.

First, we analyze $r^*$ in the context of models in $\mathcal X$.

\begin{lemma}
Suppose that $M < N$ and 
$(M \cap \kappa) \setminus \beta_{M,N} \ne \emptyset$. 
Then 
$$
\min((M \cap \kappa) \setminus \beta_{M,N}) = 
\min((M \cap \kappa) \setminus \beta_{M,M \cap N}).
$$
\end{lemma}

Note that by Assumption 2.19, $M < N$ implies that 
$M \cap N \in \mathcal X$.

\begin{proof}
By Lemma 2.26, $M \sim M \cap N$.
Let $\alpha := \min((M \cap \kappa) \setminus \beta_{M,N})$. 
Then by Lemma 2.15 and the minimality of $\alpha$, 
$$
M \cap N \cap \kappa = M \cap \beta_{M,N} = M \cap \alpha.
$$
So $M \cap \alpha = M \cap N \cap \kappa$. 
Applying Lemma 3.2 to $M$, $M \cap N$, $\alpha$, and $\kappa$, 
we get that 
$\alpha = \min((M \cap \kappa) \setminus \beta_{M,M \cap N})$.
\end{proof}

\begin{lemma}
Let $\{ K, M, N \}$ be adequate, and assume that 
$M \in N \cap \mathcal X$ and $K \sim M$. 
Then:
\begin{enumerate}
\item $K < N$ and $M \sim K \cap N$;

\item if $\alpha = \min((M \cap \kappa) \setminus \beta_{K,M})$, then 
$\alpha = \min((M \cap \kappa) \setminus \beta_{M,K \cap N})$;

\item if $\alpha = \min((K \cap \kappa) \setminus \beta_{K,M})$, then 
either $\alpha = \min((K \cap N \cap \kappa) \setminus 
\beta_{K \cap N,M})$, or 
$\alpha = \min((K \cap \kappa) \setminus \beta_{K,K \cap N})$.
\end{enumerate}
\end{lemma}

\begin{proof}
(1) Since $K \sim M$, it follows that 
$K \cap \omega_1 = M \cap \omega_1$ 
by Lemma 2.17(2). 
As $M \in N$, $M \cap \omega_1 < N \cap \omega_1$. 
Therefore, $K \cap \omega_1 < N \cap \omega_1$. 
By Lemma 2.17(1), $K < N$. 
By Proposition 2.24, $\{ K, M, N, K \cap N \}$ is adequate. 
By Lemma 2.26, $K \sim K \cap N$. 
So $M \sim K \sim K \cap N$, which by Lemma 2.18 implies 
that $M \sim K \cap N$.

(2) We apply Lemma 3.2 to the objects $M$, $K \cap N$, 
$\alpha$, and $\gamma$, 
where 
$$
\gamma := \min(((K \cap N \cap \kappa) \cup \{ \kappa \}) \setminus 
\beta_{K,M}).
$$
Provided that the assumptions of Lemma 3.2 are true for these objects, 
we get that $\alpha = \min((M \cap \kappa) \setminus \beta_{M,K \cap N})$, 
which finishes the proof of (2).

Assumptions (1) and (2) of Lemma 3.2 are immediate, and (3) 
follows from Proposition 2.11 since $\alpha$ and $\gamma$ are greater 
than or equal to $\beta_{K,M}$. 
For assumption (4), 
we need to show that $M \cap \alpha = K \cap N \cap \gamma$. 
By the minimality of $\alpha$ and $\gamma$, we have that 
$M \cap \alpha = M \cap \beta_{K,M}$ and 
$K \cap N \cap \gamma = K \cap N \cap \beta_{K,M}$. 
So it suffices to show that 
$$
M \cap \beta_{K,M} = K \cap N \cap \beta_{K,M}.
$$
Since $K \sim M$, $M \cap \beta_{K,M} = K \cap \beta_{K,M}$. 
This last equation implies the reverse inclusion of the displayed equation, 
and noting that 
$M \in N$ implies that $M \subseteq N$, it implies the forward 
inclusion as well.

(3) Let 
$$
\gamma := \min((M \cap \kappa) \cup \{ \kappa \}) \setminus \beta_{K,M}).
$$
Then by the minimality of $\alpha$ and $\gamma$ and the fact that 
$K \sim M$, we have that 
$$
K \cap \alpha = K \cap \beta_{K,M} = M \cap \beta_{K,M} = 
M \cap \gamma.
$$

\bigskip

First, assume that $\alpha \in N$. 
We apply Lemma 3.2 to the objects 
$K \cap N$, $M$, $\alpha$, and $\gamma$. 
Provided that the assumptions of Lemma 3.2 hold for these objects, 
we can conclude that 
$$
\alpha = \min((K \cap N \cap \kappa) \setminus \beta_{K \cap N,M}),
$$
which completes the proof of (3). 
Assumptions (1) and (2) of Lemma 3.2 are immediate, and 
(3) follows from Proposition 2.11 since $\alpha$ and $\gamma$ are 
greater than or equal to $\beta_{K,M}$.

For (4), we need to show that $K \cap N \cap \alpha = M \cap \gamma$. 
From the above, we already know that $K \cap \alpha = M \cap \gamma$, 
so it suffices to show that 
$$
K \cap \alpha = K \cap N \cap \alpha.
$$
Since $K < N$, $K \cap \beta_{K,N} = K \cap N \cap \beta_{K,N}$ 
by Lemma 2.15. 
As $\alpha \in K \cap N \cap \kappa$, we have that 
$\alpha < \beta_{K,N}$ by Proposition 2.11. 
Hence, 
$$
K \cap \alpha = K \cap \beta_{K,N} \cap \alpha = 
K \cap N \cap \beta_{K,N} \cap \alpha = K \cap N \cap \alpha.
$$

\bigskip

Secondly, assume that $\alpha \notin N$. 
Since $K < N$ by (1), we have that 
$K \cap \beta_{K,N} \subseteq N$, and therefore 
$\beta_{K,N} \le \alpha$. 
To show that 
$\alpha = \min((K \cap \kappa) \setminus \beta_{K,K \cap N})$, 
by Lemma 3.3 it suffices to show that 
$\alpha = \min((K \cap \kappa) \setminus \beta_{K,N})$.

Suppose for a contradiction that 
$\alpha$ is not equal to $\min((K \cap \kappa) \setminus \beta_{K,N})$. 
Then there is $\alpha_0 \in K$ such that 
$\beta_{K,N} \le \alpha_0 < \alpha$. 
But then $\alpha_0 \in K \cap \alpha$. 
Since $\alpha = \min((K \cap \kappa) \setminus \beta_{K,M})$, 
it follows that $\alpha_0 < \beta_{K,M}$. 
Now $K \sim M$ means that $K \cap \beta_{K,M} = M \cap \beta_{K,M}$. 
As $\alpha_0 \in K \cap \beta_{K,M}$, we have that $\alpha_0 \in M$. 
But $M \in N$, so $M \subseteq N$, and hence $\alpha_0 \in N$. 
Therefore, $\alpha_0 \in K \cap N \cap \kappa$, which implies that 
$\alpha_0 < \beta_{K,N}$ by Proposition 2.11, contradicting the choice 
of $\alpha_0$.
\end{proof}

\begin{proposition}
Assume that 
$A$ is adequate, $N \in A$ is simple, and for all 
$M \in A$, if $M < N$ then $M \cap N \in A$. 
Let $B$ be adequate, and suppose that 
$$
A \cap N \subseteq B \subseteq N.
$$
Then 
$$
r^*(A \cup B) = r^*(A) \cup r^*(B).
$$
\end{proposition}

\begin{proof}
By Proposition 2.25, $A \cup B$ is adequate. 
The reverse inclusion is immediate. 
For the forward inclusion, it suffices to show that if $K \in A$, $M \in B$, 
and $K \sim M$,  
then $r^*(\{ K, M \}) \subseteq r^*(A) \cup r^*(B)$.

Since $M \in B \subseteq N$, we have that $M \in N$. 
Since $K$ and $N$ are in $A$ and $K \sim M$, Lemma 3.4(1) implies that 
$K < N$ and $M \sim K \cap N$. 
So by our assumptions about $A$, $K \cap N \in A$. 
And since $N$ is simple, $K \cap N \in N$. 
Therefore, $K \cap N \in A \cap N \subseteq B$. 
So $K \cap N$ and $M$ are both in $B$.
Also, by Lemma 2.26, $K \sim K \cap N$.

Suppose that $\alpha \in r^*(\{ K, M \})$, and we will show 
that $\alpha \in r^*(A) \cup r^*(B)$. 
First, assume that $\alpha = \min((M \cap \kappa) \setminus \beta_{K,M})$. 
Then by Lemma 3.4(2), 
$\alpha = \min((M \cap \kappa) \setminus \beta_{M,K \cap N})$. 
Since $K \cap N$ and $M$ are in $B$ and $M \sim K \cap N$, 
$\alpha \in r^*(B)$.

Secondly, assume that 
$\alpha = \min((K \cap \kappa) \setminus \beta_{K,M})$. 
Then by Lemma 3.4(3), 
either $\alpha = \min((K \cap N \cap \kappa) \setminus \beta_{K \cap N,M})$, 
or $\alpha = \min((K \cap \kappa) \setminus \beta_{K,K \cap N})$. 
In the first case, $\alpha \in r^*(B)$, since $K \cap N$ and $M$ 
are in $B$ and $M \sim K \cap N$. 
In the second case, $\alpha \in r^*(A)$, 
since $K$ and $K \cap N$ are in $A$ and $K \sim K \cap N$.
\end{proof}

Next, we analyze the set $r^*$ in the context of models in $\mathcal Y$.

\begin{lemma}
Suppose that $M \in \mathcal X$, $P \in \mathcal Y$, and 
$(M \cap \kappa) \setminus (P \cap \kappa) \ne \emptyset$. 
Then 
$$
\min((M \cap \kappa) \setminus (P \cap \kappa)) = 
\min((M \cap \kappa) \setminus \beta_{M \cap P,M}).
$$
\end{lemma}

\begin{proof}
By Lemma 2.29, $M \sim M \cap P$. 
Let $\alpha := \min((M \cap \kappa) \setminus (P \cap \kappa))$. 
Then by the minimality of $\alpha$, 
$$
M \cap \alpha = M \cap P \cap \kappa. 
$$
Applying Lemma 3.2 to $M$, $M \cap P$, $\alpha$, and $\kappa$, 
we get that 
$\alpha = \min((M \cap \kappa) \setminus \beta_{M,M \cap P})$.
\end{proof}

\begin{lemma}
Let $K \in \mathcal X$, $P \in \mathcal Y$, 
and assume that $M \in P \cap \mathcal X$ and $K \sim M$. 
Then:
\begin{enumerate}
\item $K \cap P \sim M$;

\item if $\alpha = \min((M \cap \kappa) \setminus \beta_{K,M})$, then 
$\alpha = \min((M \cap \kappa) \setminus \beta_{M,K \cap P})$;

\item if $\alpha = \min((K \cap \kappa) \setminus \beta_{K,M})$, then 
either $\alpha = \min((K \cap P \cap \kappa) \setminus 
\beta_{K \cap P,M})$, or 
$\alpha = \min((K \cap \kappa) \setminus \beta_{K \cap P,K})$.
\end{enumerate}
\end{lemma}

\begin{proof}
(1) Since $K \sim M$, $K \cap \beta_{K,M} = M \cap \beta_{K,M}$. 
By Lemma 2.12, since $M \in P$, $\beta_{K,M} < P \cap \kappa$. 
Thus,
$$
K \cap P \cap \beta_{K,M} = M \cap \beta_{K,M}.
$$
By Lemma 2.10(3), $\beta_{K \cap P,M} \le \beta_{K,M}$. 
Therefore,
$$
K \cap P \cap \beta_{K \cap P,M} = M \cap \beta_{K \cap P,M}.
$$
So $K \cap P \sim M$.

(2) We apply Lemma 3.2 to the objects $M$, $K \cap P$, $\alpha$, 
and $\gamma$, where 
$$
\gamma := \min(((K \cap P \cap \kappa) \cup \{ \kappa \}) \setminus 
\beta_{K,M}).
$$
Provided that the assumptions of Lemma 3.2 are true for these objects, 
it follows that 
$\alpha = \min((M \cap \kappa) \setminus \beta_{M,K \cap P})$, and 
we are done.

Assumptions (1) and (2) are immediate, and (3) follows from 
Proposition 2.11 since $\alpha$ and $\gamma$ are greater than or 
equal to $\beta_{K,M}$. 
It remains to show that $M \cap \alpha = K \cap P \cap \gamma$. 
By the minimality of $\alpha$ and $\gamma$, we have that 
$M \cap \alpha = M \cap \beta_{K,M}$ and 
$K \cap P \cap \gamma = K \cap P \cap \beta_{K,M}$. 
So it suffices to show that 
$$
M \cap \beta_{K,M} = K \cap P \cap \beta_{K,M}.
$$

Since $K \sim M$, $M \cap \beta_{K,M} = K \cap \beta_{K,M}$. 
Hence, it is enough to show that 
$$
K \cap \beta_{K,M} = K \cap P \cap \beta_{K,M}.
$$
But this follows immediately from the fact that 
$\beta_{K,M} < P \cap \kappa$, which is true by Lemma 2.12.

(3) Let
$$
\gamma := 
\min(((M \cap \kappa) \cup \{ \kappa \}) \setminus \beta_{K,M}).
$$
By Proposition 2.11, $\alpha \ne \gamma$, since $\alpha$ and $\gamma$ 
are greater than or equal to $\beta_{K,M}$. 
By the minimality of $\alpha$ and $\gamma$ and the fact that 
$K \sim M$, we have that 
$$
K \cap \alpha = K \cap \beta_{K,M} = M \cap \beta_{K,M} = 
M \cap \gamma.
$$

First, assume that $\alpha < P \cap \kappa$. 
Then $K \cap \alpha = K \cap P \cap \alpha$.  
Therefore, by the last paragraph, 
$$
K \cap P \cap \alpha = M \cap \gamma.
$$
Applying Lemma 3.2 to the objects 
$K \cap P$, $M$, $\alpha$, and $\gamma$, it follows that 
$$
\alpha = \min((K \cap P \cap \kappa) \setminus \beta_{K \cap P,M}).
$$

Secondly, assume that $P \cap \kappa \le \alpha$. 
To show that 
$\alpha = \min((K \cap \kappa) \setminus \beta_{K \cap P,K})$, 
by Lemma 3.6 
it suffices to show that 
$\alpha = \min((K \cap \kappa) \setminus (P \cap \kappa))$.

Suppose for a contradiction that $\alpha$ is not equal to 
$\min((K \cap \kappa) \setminus (P \cap \kappa))$. 
Then there is $\alpha_0 \in K \cap \kappa$ such that 
$P \cap \kappa \le \alpha_0 < \alpha$. 
Then $\alpha_0 \in K \cap \alpha$. 
Since $\alpha = \min((K \cap \kappa) \setminus \beta_{K,M})$, it follows that 
$\alpha_0 < \beta_{K,M}$. 
Hence, $\alpha_0 \in K \cap \beta_{K,M}$. 
As $K \sim M$, we have that $K \cap \beta_{K,M} = M \cap \beta_{K,M}$. 
So $\alpha_0 \in M$. 
But $M \in P$, so $M \subseteq P$. 
Therefore, $\alpha_0 \in P \cap \kappa$, which contradicts 
the choice of $\alpha_0$.
\end{proof}

\begin{proposition}
Suppose that $A$ is adequate, 
$P \in \mathcal Y$ is simple, 
and for all $M \in A$, $M \cap P \in A$. 
Let $B$ be adequate, and suppose that 
$$
A \cap P \subseteq B \subseteq P.
$$
Then 
$$
r^*(A \cup B) = r^*(A) \cup r^*(B).
$$
\end{proposition}

\begin{proof}
By Proposition 2.28, $A \cup B$ is adequate. 
The reverse inclusion is immediate. 
For the forward inclusion, it suffices to show that if 
$K \in A$, $M \in B$, and $K \sim M$, 
then $r^*(\{ K, M \}) \subseteq r^*(A) \cup r^*(B)$. 

Since $M \in B \subseteq P$, we have that $M \in P$. 
Thus, the assumptions of Lemma 3.7 are satisfied.
By Lemma 3.7(1), $M \sim K \cap P$. 
By the assumptions on $A$, $K \cap P \in A$. 
As $P$ is simple, $K \cap P \in P$. 
So $K \cap P \in A \cap P \subseteq B$. 
Hence, $K \cap P$ is in $B$. 
By Lemma 2.29, $K \sim K \cap P$.

Suppose that $\alpha \in r^*(\{ K, M \})$, and we will show 
that $\alpha \in r^*(A) \cup r^*(B)$. 
First, assume that 
$\alpha = \min((M \cap \kappa) \setminus \beta_{K,M})$. 
Then by Lemma 3.7(2), 
$\alpha = \min((M \cap \kappa) \setminus \beta_{M,K \cap P})$. 
Since $K \cap P$ and $M$ are in $B$ and $M \sim K \cap P$, it follows that 
$\alpha \in r^*(B)$.

Secondly, assume that 
$\alpha = \min((K \cap \kappa) \setminus \beta_{K,M})$. 
Then by Lemma 3.7(3), either 
$\alpha = \min((K \cap P \cap \kappa) \setminus \beta_{K \cap P,M})$, or 
$\alpha = \min((K \cap \kappa) \setminus \beta_{K \cap P,K})$. 
In the first case, since $K \cap P$ and $M$ are in $B$ and $K \cap P \sim M$, 
it follows that $\alpha \in r^*(B)$. 
In the second case, since $K$ and $K \cap P$ are in $A$ and $K \sim K \cap P$, 
it follows that $\alpha \in r^*(A)$.
\end{proof}

\bigskip

\part{The Single Forcing}

\bigskip

\addcontentsline{toc}{section}{4. The forcing poset}

\textbf{\S 4. The forcing poset}

\stepcounter{section}

\bigskip

We introduce a forcing poset for adding a partial square sequence 
to a stationary set $S \subseteq \kappa \cap \cof(>\! \omega)$. 
This forcing poset will preserve $\omega_1$, is $\kappa$-c.c., and 
if $\kappa > \omega_2$, then 
it will collapse $\kappa$ to become $\omega_2$.

\begin{notation}
Fix, for the remainder of Part 2, 
a set $S \subseteq \Lambda$ which is stationary in $\kappa$ 
and satisfies that for all 
$\alpha \in S$, $Sk(\alpha) \cap \kappa = \alpha$.
\end{notation}

\begin{definition}
Let $\p$ be the forcing poset whose conditions are triples of the form 
$p = (f_p,g_p,A_p)$ satisfying the following 
requirements:\footnote{We will sometimes refer to 
$f_p$, $g_p$, and $A_p$ as the $f$, $g$, and $A$ components of $p$.}
\begin{enumerate}

\item $A_p$ is an adequate set;

\item $f_p$ is a function with a finite domain, 
and for all $x \in \dom(f_p)$, 
either $x \in S$, or there is $M \in A_p$ and 
$$
\alpha \in (M \cap \dom(f_p) \cap S) \cup \{ \kappa \}
$$ 
satisfying that $x = M \cap \alpha$; moreover, for all $x \in \dom(f_p)$, 
$f_p(x)$ is a finite $\in$-chain and a subset of 
$Sk(x)$;

\item if $x \in \dom(f_p)$, then $f_p(x) \subseteq \dom(f_p)$, and for 
all $K \in f_p(x)$, 
$$
f_p(K) = f_p(x) \cap Sk(K);
$$

\item $g_p$ is a function whose domain is the set of all pairs 
$(K,x)$ such that $K \in f_p(x)$, and for all $(K,x) \in \dom(g_p)$, 
$g_p(K,x)$ is a finite subset of $x \setminus \sup(K)$;

\item if $K \in f_p(L)$ and $L \in f_p(x)$, then 
$g_p(K,x) \subseteq g_p(K,L)$;\footnote{Note that if $K \in f_p(L)$ and 
$L \in f_p(x)$, then $K \in f_p(x)$ by requirement (3).}

\item if $\alpha \in \dom(f_p) \cap S$, $M \in A_p$, and $\alpha \in M$, 
then $M \cap \alpha \in f_p(\alpha)$;

\item $r^*(A_p) \cap S \subseteq \dom(f_p)$.
\end{enumerate}
For conditions $p$ and $q$ in $\p$, we let $q \le p$ if:
\begin{enumerate}
\item[(a)] $A_p \subseteq A_q$;

\item[(b)] $\dom(f_p) \subseteq \dom(f_q)$, and 
for all $x \in \dom(f_p)$, $f_p(x) \subseteq f_q(x)$;

\item[(c)] for all 
$(K,x) \in \dom(g_p)$, $g_p(K,x) \subseteq g_q(K,x)$;

\item[(d)] if $K$ and $x$ are in 
$\dom(f_p)$ and $K \in f_q(x)$, then $K \in f_p(x)$.
\end{enumerate}
\end{definition}

Let us summarize some of the main properties which we will prove about $\p$. 
The forcing poset $\p$ will be shown to be strongly proper on a stationary 
set, and thus preserve $\omega_1$, and to be $\kappa$-c.c. 
In particular, $\p$ preserves the stationarity of $S$. 
If $\kappa > \omega_2$, then $\p$ collapses all cardinals $\mu$ such that 
$\omega_1 < \mu < \kappa$ to have size $\omega_1$, and hence 
forces that $\kappa = \omega_2$. 
Finally, $\p$ forces that there exists a partial square sequence on $S$, 
and in particular, forces that $S$ is in the approachability 
ideal $I[\omega_2]$.

The properties of $\p$ just listed will be proved in Sections 4--7. 
In Sections 8 and 9, we will derive some additional information 
about the forcing poset $\p$, and use this information 
to show that certain quotients of $\p$ have the 
$\omega_1$-approximation property.

In the remainder of the current section, we will prove some basic facts 
about $\p$ which we will need.

\bigskip

\begin{lemma}
Let $p \in \p$ and $x \in \dom(f_p)$. Then:
\begin{enumerate}
\item $Sk(x) \cap \kappa = x$;
\item for all $K \in f_p(x)$, $K \subseteq x$;
\item if $N \in \mathcal X$, $\alpha \in S$, $x \in N \setminus S$, and $x \subseteq \alpha$, 
then $x \in Sk(N \cap \alpha)$;
\item if $P \in \mathcal Y$ and $\sup(x) < P \cap \kappa$, then 
$x \in P$.
\end{enumerate}
\end{lemma}

\begin{proof}
(1) By Definition 4.2(2), either $x \in S$, or there is $M \in A_p$ and 
$\alpha \in (M \cap \dom(f_p) \cap S) \cup \{ \kappa \}$ such that 
$x = M \cap \alpha$. 
Then $Sk(x) \cap \kappa = x$ holds by Notation 4.1 in the first case, 
and by Lemma 2.7(2) in the second case. 

(2) Suppose that $K \in f_p(x)$, and we will show that $K \subseteq x$. 
By Definition 4.2(2,3), $K$ is a countable subset of $\kappa$ in $Sk(x)$. 
By the elementarity of $Sk(x)$, $K \subseteq Sk(x)$. 
As $K \subseteq \kappa$, it follows by (1) that 
$K \subseteq Sk(x) \cap \kappa = x$.

(3) Fix $M \in A_p$ and $\beta \in (M \cap \dom(f_p) \cap S) \cup \{ \kappa \}$ 
such that $x = M \cap \beta$. 
Since $M \cap \beta = x \in N$, we have that 
$x = M \cap \beta \in Sk(N \cap \beta)$ by Lemma 2.8(1). 
If $\beta \le \alpha$, then $x \in Sk(N \cap \beta) \subseteq Sk(N \cap \alpha)$, 
and we are done. 
If $\alpha < \beta$, then since $x = M \cap \beta \subseteq \alpha$, 
$M \cap \beta = M \cap \alpha$. 
So $M \cap \alpha = x \in N$. 
Hence, $x \in Sk(N \cap \alpha)$ by Lemma 2.8(1).

(4) If $x \in S$, then $x = \sup(x) \in P$. 
Otherwise by Definition 4.2(2) there is $M \in A_p$ and 
$\beta \in (M \cap \dom(f_p) \cap S) \cup \{ \kappa \}$ such that 
$x = M \cap \beta$. 
Let $\alpha := P \cap \kappa$. 
Then by the elementarity of $P$, 
$Sk(\alpha) \cap \kappa = \alpha$. 
By Assumption 2.6, $M \cap \alpha \in Sk(\alpha) \subseteq P$. 
So $M \cap \alpha \in P$. 
If $\beta \le \alpha$, then $x = M \cap \beta$ is an initial segment of 
$M \cap \alpha$, and hence is in $P$. 
If $\alpha < \beta$, then since $\sup(x) = \sup(M \cap \beta) < \alpha$, 
$M \cap \alpha = M \cap \beta = x$, which is in $P$.
\end{proof}

\begin{lemma}
Let $p \in \p$ and $z \in \dom(f_p)$. 
Then for all $x$ and $y$ in $f_p(z)$, 
$x \in Sk(y)$ iff $\sup(x) < \sup(y)$, and 
$x = y$ iff $\sup(x) = \sup(y)$.
\end{lemma}

\begin{proof}
By Definition 4.2(3), $f_p(z) \subseteq \dom(f_p)$. 
So by Lemma 4.3(1), if $x \in f_p(z)$ then 
$Sk(x) \cap \kappa = x$. 
The lemma now follows from Lemma 2.32, letting $a = f_p(z)$.
\end{proof}

\begin{lemma}
Let $p \in \p$ and $\alpha \in S \setminus \dom(f_p)$. 
If $M \in A_p$ and $\alpha \in M$, 
then $M \cap \alpha$ is not in $\dom(f_p)$.
\end{lemma}

\begin{proof}
Suppose for a contradiction that $M \cap \alpha \in \dom(f_p)$. 
Then by Definition 4.2(2), there is $M_1 \in A_p$ and 
$\beta \in (M_1 \cap \dom(f_p) \cap S) \cup \{ \kappa \}$ 
such that $M \cap \alpha = M_1 \cap \beta$. 
Since $\beta \in \dom(f_p)$ and $\alpha \notin \dom(f_p)$, 
$\alpha \ne \beta$. 
By Lemma 3.2 applied to $M$, $M_1$, $\alpha$, and $\beta$, we have that 
$M \sim M_1$ and 
$$
\alpha = \min((M \cap \kappa) \setminus \beta_{M,M_1}).
$$ 
In particular, $\alpha \in r^*(A_p) \cap S$. 
So by Definition 4.2(7), $\alpha \in \dom(f_p)$, 
which contradicts our assumptions.
\end{proof}

We show next that for any condition $p$ and 
any ordinal $\alpha$ in $S$, there is $q \le p$ with 
$\alpha \in \dom(f_q)$. 
Among other things, this fact will allow us to prove that 
$\p$ adds a partial square sequence whose domain is all of $S$.

\begin{lemma}
Let $p \in \p$, and let $\alpha$ and $\beta$ be distinct ordinals 
in $S \setminus \dom(f_p)$. 
Then the sets 
$$
\dom(f_p), \ \{ M \cap \alpha : M \in A_p, \ \alpha \in M \}, \ 
\{ N \cap \beta : N \in A_p, \ \beta \in N \}
$$
are pairwise disjoint.
\end{lemma}

\begin{proof}
By Lemma 4.5, the first and second sets are disjoint, and the first and third 
sets are disjoint. 
If the second and third sets are not disjoint, 
then $M \cap \alpha = N \cap \beta$, for some 
$M$ and $N$ in $A_p$ with $\alpha \in M$ and $\beta \in N$. 
Applying Lemma 3.2 to $M$, $N$, $\alpha$, and $\beta$, 
we get that $M \sim N$ and 
$\alpha = \min((M \cap \kappa) \setminus \beta_{M,N})$. 
In particular, $\alpha \in r^*(A_p) \cap S$. 
So by Definition 4.2(7), $\alpha \in \dom(f_p)$, 
which contradicts our assumptions.
\end{proof}

\begin{definition}
Let $p \in \p$, and let $x$ be a finite subset of $S \setminus \dom(f_p)$. 
Define $p + x$ as the triple $(f,g,A)$ satisfying:
\begin{enumerate}
\item $A := A_p$;
\item $\dom(f) := \dom(f_p) \cup x \cup 
\{ M \cap \alpha : M \in A_p, \ \alpha \in M \cap x \}$;
\item for each $z \in \dom(f_p)$, $f(z) := f_p(z)$;
\item for each $\alpha \in x$, 
$f(\alpha) :=  \{ M \cap \alpha : M \in A_p, \ \alpha \in M \}$;
\item for each $\alpha \in x$ and $M \in A_p$ with $\alpha \in M$, 
$f(M \cap \alpha) := f(\alpha) \cap Sk(M \cap \alpha)$, where $f(\alpha)$ 
was defined in (4);
\item the domain of $g$ is the set of pairs $(K,z)$ such that 
$K \in f(z)$;
\item $g(K,z) = g_p(K,z)$ if $(K,z) \in \dom(g_p)$, and 
$g(K,z) = \emptyset$ if $(K,z) \in \dom(g) \setminus \dom(g_p)$.
\end{enumerate}
\end{definition}

Note that by Lemma 4.6, $f(M \cap \alpha)$ in (5) is well-defined, 
since the set $M \cap \alpha$ is not in $\dom(f_p)$ and 
uniquely determines $\alpha$.

\begin{lemma}
Let $p \in \p$, and let $x$ be a finite subset of $S \setminus \dom(f_p)$. 
Then $p + x$ is a condition in $\p$, and $p + x \le p$.
\end{lemma}

\begin{proof}
Let $q := p + x$. 
Assuming that $q$ is a condition, it is easy to check that $q \le p$ from 
the definition of $q$ and the fact that for all $M$ and $z$ in $\dom(f_p)$, 
$M \in f_p(z)$ iff $M \in f_q(z)$.

To show that $q$ is a condition, we 
verify requirements (1)--(7) of Definition 4.2. 
(1), (4), (6), and (7) are immediate from the definition of $q$, and 
(3) is easy to check using the definition of $f_q$. 
It remains to prove (2) and (5).

(2) Clearly $f_q$ is a function with a finite domain, and every $z \in \dom(f_q)$ 
has the right form. 
Let $z \in \dom(f_q)$, and we will show that $f_q(z)$ is a finite $\in$-chain 
and $f_q(z) \subseteq Sk(z) \setminus S$. 
This is immediate if $z \in \dom(f_p)$.

If $z = M \cap \alpha$, for some $M \in A_p$ with $\alpha \in M \cap x$, then 
$f_q(M \cap \alpha) = f_q(\alpha) \cap Sk(M \cap \alpha)$ is obviously 
a subset of $Sk(M \cap \alpha)$, and will be a finite $\in$-chain disjoint 
from $S$ provided that $f_q(\alpha)$ is. 
So it suffices to show that 
$f_q(\alpha) = \{ M \cap \alpha : M \in A_p , \ \alpha \in M \}$ 
is a finite $\in$-chain and is a subset of $Sk(\alpha) \setminus S$. 
This set is obviously disjoint from $S$, and it is 
a finite $\in$-chain and a subset of $Sk(\alpha)$ by Lemma 2.33.

(5) Suppose that $K \in f_q(L)$ and $L \in f_q(z)$. 
We will show that $g_q(K,z) \subseteq g_q(K,L)$. 
If $z \in \dom(f_p)$, then $f_q(z) = f_p(z)$; this implies that 
$K$ and $L$ are in $\dom(f_p)$ as well, so $K \in f_p(L)$ 
and $L \in f_p(z)$. 
Hence, $g_q(K,z) = g_p(K,z) \subseteq g_p(K,L) = g_q(K,L)$. 
On the other hand, if $z$ is not in $\dom(f_p)$, then 
$(K,z) \notin \dom(g_p)$. 
Therefore, $g_q(K,z) = \emptyset \subseteq g_q(K,L)$.
\end{proof}

The next lemma will be needed in the amalgamation arguments of Section 7.

\begin{lemma}
Let $p$ be a condition. 
Then there is $q \le p$ satisfying that whenever 
$K \in f_q(x)$ and $x \in f_q(y)$, then 
$$
g_q(K,x) \subseteq g_q(K,y).\footnote{Note that by Definition 4.2(5), 
any such condition satisfies that 
whenever $K \in f_q(x)$ and $x \in f_q(y)$, then 
$g_q(K,x) = g_q(K,y)$.}
$$
Moreover, $f_q = f_p$ and $A_q = A_p$.
\end{lemma}

\begin{proof}
Define $q$ as follows. 
Let $f_q := f_p$ and $A_q := A_p$. 
For any $K$ and $y$ such that $K \in f_p(y)$, define 
$$
g_q(K,y) := \bigcup 
\{ g_p(K,x) : 
x = y, \ \textrm{or} \ (K \in f_p(x) \ \textrm{and} \ 
x \in f_p(y)) \}.
$$

It is trivial to check that if $q$ is a condition, then $q \le p$. 
To show that $q$ is a condition, we verify requirements (1)--(7) 
of Definition 4.2. 
(1), (2), (3), (6), and (7) are immediate. 
It remains to prove (4) and (5).

(4) The domain of $g_q$ is equal to the set of pairs $(K,y)$, where 
$K \in f_q(y)$. 
Let $K \in f_q(y)$, and we will show that $g_q(K,y)$ is a finite 
subset of $y \setminus \sup(K)$. 
By the definition of $g_q(K,y)$ and the fact that $p$ is a condition, 
it is clear that $g_q(K,y)$ is finite and 
every ordinal in $g_q(K,y)$ is greater than or equal to $\sup(K)$. 
It remains to show that $g_q(K,y) \subseteq y$.

Let $\xi \in g_q(K,y)$, and we will show that $\xi \in y$. 
By the definition of $g_q$, 
either $\xi \in g_p(K,y)$, 
or else there is some $x$ satisfying that 
$K \in f_p(x)$, $x \in f_p(y)$, and $\xi \in g_p(K,x)$. 
In the first case, $\xi \in y$, since $p$ is a condition. 
In the second case, $\xi \in g_p(K,x) \subseteq x$. 
But $x \in f_p(y)$, which implies that $x \subseteq y$ by Lemma 4.3(2). 
So $\xi \in y$.

(5) Suppose that $K \in f_q(L)$ and $L \in f_q(y)$. 
We will show that $g_q(K,y) \subseteq g_q(K,L)$. 
Since $f_q = f_p$, we have that 
$K \in f_p(L)$ and $L \in f_p(y)$. 
Let $\xi \in g_q(K,y)$, and we will show that $\xi \in g_q(K,L)$. 
By the definition of $g_q(K,y)$, either
\begin{enumerate}
\item[(a)] $\xi \in g_p(K,y)$, or
\item[(b)] there is $x$ such that $K \in f_p(x)$, 
$x \in f_p(y)$, and $\xi \in g_p(K,x)$. 
\end{enumerate}
In case a, since $p$ is a condition, 
$g_p(K,y) \subseteq g_p(K,L)$. 
Also, $g_p(K,L) \subseteq g_q(K,L)$ by the definition of $g_q$. 
So $\xi \in g_p(K,y) \subseteq g_p(K,L) \subseteq g_q(K,L)$.

Consider case b. 
Since $x$ and $L$ are both in $f_p(y)$, 
either $L \in f_p(x)$, $x = L$, or 
$x \in f_p(L)$. 
Assume that $L \in f_p(x)$. 
Then, since $p$ is a condition, 
$$
\xi \in g_p(K,x) \subseteq g_p(K,L) \subseteq g_q(K,L).
$$
Assume that $x = L$. 
Then 
$$
\xi \in g_p(K,x) = g_p(K,L) \subseteq g_q(K,L).
$$
Finally, assume that $x \in f_p(L)$. 
Then by the definition of $g_q(K,L)$, since 
$K \in f_p(x)$ and $x \in f_p(L)$, $g_p(K,x) \subseteq g_q(K,L)$. 
So $\xi \in g_q(K,L)$. 

\bigskip

This completes the proof that $q$ is a condition. 
To show that $q$ is as required, suppose that 
$K \in f_q(x)$ and $x \in f_q(y)$, and we will show that 
$g_q(K,x) \subseteq g_q(K,y)$. 
Then $K \in f_p(x)$ and $x \in f_p(y)$. 

Let $\xi \in g_q(K,x)$, and we will show that $\xi \in g_q(K,y)$. 
By the definition of $g_q(K,x)$, either $\xi \in g_p(K,x)$, or 
$\xi \in g_p(K,x_0)$ for some $x_0$ with $K \in f_p(x_0)$ 
and $x_0 \in f_p(x)$. 
In the second case, $x_0 \in f_p(x)$ and $x \in f_p(y)$ imply that 
$x_0 \in f_p(y)$. 

Let $x'$ be equal to $x$ or $x_0$ depending on the first or the second case. 
Then in either case, $K \in f_p(x')$ and $x' \in f_p(y)$, and also 
$\xi \in g_p(K,x')$. 
By the definition of $g_q$, 
$g_p(K,x') \subseteq g_q(K,y)$. 
Thus, $\xi \in g_q(K,y)$, as required.
\end{proof}

\bigskip

\addcontentsline{toc}{section}{5. A partial square sequence}

\textbf{\S 5. A partial square sequence}

\stepcounter{section}

\bigskip

In Sections 6 and 7, we will prove that $\p$ preserves $\omega_1$, 
is $\kappa$-c.c., and collapses $\kappa$ to become $\omega_2$. 
The proofs of these facts are quite involved. 
So it makes sense, from 
the expositional point of view, to assume for the time being that they are true, 
and show that the forcing poset 
$\p$ does what it is intended to do, namely, to add a partial 
square sequence on $S$.

To be precise, in this section we will 
assume exactly that $\p$ preserves $\omega_1$, forces that 
$\kappa$ is equal to $\omega_2$, and that 
Lemma 7.1 from Section 7 below holds.

\bigskip

Let $\dot f$ be a $\p$-name for a function 
such that $\p$ forces that for all $x$,
$$
\dot f(x) = \bigcup \{ f_p(x) : p \in \dot G_\p, \ x \in \dom(f_p) \}.
$$
It is easy to check that $\p$ forces that $\dot f(x)$ is an 
$\in$-chain of countable subsets of $\kappa$ 
and is a subset of $Sk(x) \setminus S$. 
Note that by Lemma 4.4, 
$\p$ forces that if $J$ and $K$ are in $\dot f(x)$, 
then $\sup(J) < \sup(K)$ iff $J \in Sk(K)$.

For each $\alpha \in S$, 
let $\dot c_\alpha$ be a $\p$-name such that $\p$ forces that 
$$
\dot c_\alpha = \{ \sup(M) : M \in \dot f(\alpha) \}.
$$
We will prove that $\p$ forces that the sequence 
$$
\langle \dot c_\alpha : \alpha \in S \rangle
$$
is a partial square sequence on $S$.

\begin{lemma}
Let $\alpha \in S$. 
Then $\p$ forces that for all $K \in \dot f(\alpha)$, 
$$
\dot f(K) = \dot f(\alpha) \cap Sk(K)
$$
and 
$$
\dot c_\alpha \cap \sup(K) = \{ \sup(J) : J \in \dot f(K) \}.
$$
\end{lemma}

\begin{proof}
Straightforward.
\end{proof}

\begin{lemma}
Let $\alpha \in S$. 
Then $\p$ forces that $\dot c_\alpha$ is a cofinal subset of $\alpha$ 
with order type $\omega_1$.
\end{lemma}

\begin{proof}
We first show that $\dot c_\alpha$ is forced to be a cofinal 
subset of $\alpha$. 
So let $\gamma < \alpha$ and $p \in \p$. 
Using Lemma 4.8, we can fix $q \le p$ with $\alpha \in \dom(f_q)$. 
By Lemma 7.1, fix $r \le q$ such that for some $N \in \mathcal X$ 
with $\gamma$ and $\alpha$ in $N$, $N \in A_r$. 
By Definition 4.2(6), $N \cap \alpha \in f_r(\alpha)$. 
So $r$ forces that 
$\gamma < \sup(N \cap \alpha) \in \dot c_\alpha$.

Now we show that $\dot c_\alpha$ is forced to have order 
type equal to $\omega_1$. 
Since $\alpha$ has uncountable cofinality and $\dot c_\alpha$ is 
forced to be cofinal in $\alpha$, clearly $\dot c_\alpha$ is forced to have 
an order type of uncountable cofinality. 
If it is not forced to have order type equal to $\omega_1$, then some condition 
forces that it has a proper initial segment of order type $\omega_1$. 
Hence, for some $p \in \p$ and $K \in f_p(\alpha)$, $p$ forces that 
$\dot c_\alpha \cap \sup(K)$ has order type equal to $\omega_1$.

Let $G$ be a generic filter on $\p$ which contains $p$, and let 
$c_\alpha := \dot c_\alpha^G$ and $f := \dot f^G$. 
By Lemma 5.1, $c_\alpha \cap \sup(K) = \{ \sup(J) : J \in f(K) \}$. 
Since $c_\alpha \cap \sup(K)$ is uncountable, 
it follows that $f(K)$ is uncountable. 
But since $K \in f(\alpha)$, by Lemma 5.1 we have that 
$f(K) = f(\alpha) \cap Sk(K) \subseteq Sk(K)$. 
As $Sk(K)$ is countable, so is $f(K)$, and we have a contradiction.
\end{proof}

\begin{proposition}
Let $\alpha \in S$. 
Suppose that $p \in \p$ and $p$ 
forces that $\xi < \alpha$ is a limit point of $\dot c_\alpha$. 
Then there is $q \le p$ such that for some $M \in f_q(\alpha)$, 
$\sup(M) = \xi$. 
In particular, $\p$ forces that $\dot c_\alpha$ is closed.
\end{proposition}

\begin{proof}
Note that by Lemma 5.2, $\xi$ must have cofinality $\omega$. 
Extend $p$ to $q$ so that for some $M \in f_{q}(\alpha)$, 
$q$ forces that $M$ is the membership least element of $\dot f(\alpha)$ 
with $\xi \le \sup(M)$. 
We will prove that $\sup(M) = \xi$, which finishes the proof.

\bigskip

\noindent \emph{Claim 1:} 
If $K \in f_q(\alpha)$ and $\sup(K \cap \xi) < \xi$, then 
$\sup(K) < \xi$.

\bigskip

Suppose for a contradiction that $K \in f_q(\alpha)$, 
$\sup(K \cap \xi) < \xi$, but $\xi \le \sup(K)$. 
Since $\xi$ is forced by $q$ to be a limit point of $\dot c_\alpha$, 
there exist $t \le q$ and $N$ such that 
$N \in f_t(\alpha)$ and 
$$
\sup(K \cap \xi) < \sup(N) < \xi.
$$
As $K$ and $N$ are in $f_t(\alpha)$ and $\sup(N) < \xi \le \sup(K)$, 
it follows that $N \in Sk(K)$ by Lemma 4.4. 
By elementarity, $\sup(N) \in Sk(K)$. 
By Lemma 4.3(1), $Sk(K) \cap \kappa = K$. 
So $\sup(N) \in K$. 
Thus, $\sup(N) \in K \cap \xi$, contradicting that 
$\sup(K \cap \xi) < \sup(N)$.

\bigskip

It easily follows from Claim 1 that $f_q(\alpha)$ is the union of the sets 
$A_1$ and $A_2$ defined by
$$
A_1 := \{ K \in f_q(\alpha) : \sup(K) < \xi \},
$$
$$
A_2 := \{ K \in f_q(\alpha) : \sup(K \cap \xi) = \xi \}.
$$
Namely, if $K \in f_q(\alpha)$, then either 
$\sup(K \cap \xi) < \xi$, in which case $\sup(K) < \xi$ by Claim 1, 
or else $\sup(K \cap \xi) = \xi$. 
Note that since $f_q(\alpha)$ is an $\in$-chain, 
if $K \in A_1$ and $L \in A_2$, then 
$\sup(K) < \sup(L)$, and therefore $K \in Sk(L)$ by Lemma 4.4. 

Observe that since $M \in f_q(\alpha)$ and $\xi \le \sup(M)$, 
we have that $M \in A_2$. 
So for all $K \in A_1$, $K \in Sk(M)$. 
Also, since $M$ is the membership least element of $f_q(\alpha)$ 
with $\xi \le \sup(M)$, we have that for all $N \in A_2$, 
either $M = N$ or $M \in Sk(N)$. 
In particular, for all $N \in A_2$, $M \subseteq N$.

\bigskip

Now we prove the proposition. 
Assume for a contradiction that $\sup(M) \ne \xi$. 
Then $M \cap [\xi,\alpha) \ne \emptyset$.
Fix $\gamma < \xi$ large enough so that if $J \in A_1$, 
then $\sup(J) < \gamma$. 
This is possible since $A_1$ is finite.

As $q$ forces that $\xi$ is a limit point of $\dot c_\alpha$, 
we can fix $s \le q$ and $K$ such that:
\begin{enumerate}
\item $K$ is the membership largest element of $f_s(\alpha)$ with 
$\sup(K) < \xi$;
\item $K \notin \dom(f_q)$;
\item $\gamma < \sup(K)$;
\item $K \cap \omega_1$ is different from $L \cap \omega_1$, for 
all $L \in A_q$.
\end{enumerate} 
By Definition 4.2(2), fix $K_1 \in A_s$ and 
$\beta \in (K_1 \cap \dom(f_s) \cap S) \cup \{ \kappa \}$ 
such that $K = K_1 \cap \beta$.

We use $s$, $K$, $K_1$, and $\beta$ to define an extension 
$r$ of $q$. 
Let 
$$
A_r := A_q \cup \{ K_1 \}.
$$
If $\beta = \kappa$, then let 
$$
\dom(f_r) := \dom(f_q) \cup \{ K \} \cup 
\{ K_1 \cap \delta : \delta \in K_1 \cap \dom(f_q) \cap S \},
$$
and if $\beta < \kappa$, then let 
\begin{multline*}
\dom(f_r) := 
\dom(f_q) \cup \{ K \} \cup 
\{ K_1 \cap \delta : \delta \in K_1 \cap \dom(f_q) \cap S \} \ \cup \\ 
\{ \beta \} \cup \{ L \cap \beta : L \in A_q, \ \beta \in L \}.
\end{multline*}
Note that the domain of $f_r$ is a subset of the domain of $f_s$, 
since $s \le q$, $K \in f_s(\alpha)$, $K_1 \in A_s$, and 
$\beta \in (K_1 \cap \dom(f_s) \cap S) \cup \{ \kappa \}$. 
Thus, it makes sense to define, for each $x \in \dom(f_r)$, 
$$
f_r(x) := f_s(x) \cap \dom(f_r).
$$
Observe that since $K \in f_s(\alpha)$ 
and $K$ and $\alpha$ are in $\dom(f_r)$, we have that 
$K \in f_r(\alpha)$.

Let $J \in f_r(x)$, and we will define $g_r(J,x)$. 
We let $g_r(J,x) := g_s(J,x)$, unless 
$J = K$ and $x \in f_r(\alpha) \cup \{ \alpha \}$, in which 
case we let 
$$
g_r(J,x) := g_s(J,x) \cup \{ \zeta \},
$$
where 
$$
\zeta := \min(M \setminus \xi).
$$
The ordinal $\zeta$ exists because we are assuming for a 
contradiction that $\sup(M) > \xi$. 
Note that in either case, 
we have that $g_s(J,x) \subseteq g_r(J,x)$.

\bigskip

We will prove that $r$ is a condition and $r \le q$. 
Let us see that this gives us a contradiction. 
If $r$ is a condition and $r \le q$, then $r$ forces that $\xi$ is a limit 
point of $\dot c_\alpha$. 
But $\sup(K) < \xi$, so we can find $u \le r$ and $L \in f_u(\alpha)$ 
such that $\sup(K) < \sup(L) < \xi$. 
Then $K$ and $L$ are in $f_u(\alpha)$, and since 
$\sup(K) < \sup(L)$, $K \in Sk(L)$ by Lemma 4.4. 
So $K \in f_u(\alpha) \cap Sk(L) = f_u(L)$. 
Therefore, by Definition 4.2(4,5), 
$$
g_u(K,\alpha) \subseteq g_u(K,L) \subseteq L.
$$
By the definition of $g_r$, $\zeta \in g_r(K,\alpha)$. 
Hence,
$$
\zeta \in g_r(K,\alpha) \subseteq g_u(K,\alpha) \subseteq L.
$$
So $\zeta \in L$. 
This is a contradiction, since $\sup(L) < \xi \le \zeta$.

\bigskip

Suppose for a moment that $r$ is a condition, and let us prove that $r \le q$. 
We verify properties (a)--(d) of Definition 4.2. 
(a,b) By the definition of $r$, $A_q \subseteq A_r$ and 
$\dom(f_q) \subseteq \dom(f_r)$. 
If $x \in \dom(f_q)$, then since $s \le q$, 
$$
f_q(x) \subseteq f_s(x) \cap \dom(f_q) \subseteq 
f_s(x) \cap \dom(f_r) = f_r(x).
$$
(c) Suppose that $(J,x) \in \dom(g_q)$. 
Then $J \ne K$, since $K \notin \dom(f_q)$. 
So by the definition of $g_r$, $g_r(J,x) = g_s(J,x)$. 
Since $s \le q$, we have that 
$g_q(J,x) \subseteq g_s(J,x) = g_r(J,x)$. 
(d) Assume that $J$ and $x$ are in $\dom(f_q)$ and 
$J \in f_r(x)$. 
Then by the definition of $f_r$, $J \in f_s(x)$. 
Since $s \le q$, it follows that $J \in f_q(x)$.

\bigskip

It remains to prove that $r$ is a condition. 
We verify requirements (1)--(7) of Definition 4.2.

\bigskip

(1) We have that $A_r = A_q \cup \{ K_1 \} \subseteq A_s$. 
Since $A_s$ is adequate, so is $A_r$.

\bigskip

(2) It is obvious that $f_r$ is a function with a finite domain. 
Let $x \in \dom(f_r)$, and we will show that either $x \in S$, 
or there is $L \in A_r$ and 
$\delta \in (L \cap \dom(f_r) \cap S) \cup \{ \kappa \}$ such that 
$x = L \cap \delta$. 
If $x \in \dom(f_q)$, then this statement follows from the fact 
that $q$ is a condition. 
If $x = K$, then $x = K_1 \cap \beta$, where $K_1 \in A_r$ 
and $\beta \in (K_1 \cap \dom(f_r) \cap S) \cup \{ \kappa \}$. 
If $x = K_1 \cap \delta$, where $\delta \in K_1 \cap \dom(f_q) \cap S$, 
then we are done since $K_1 \in A_r$ and $\dom(f_q) \subseteq \dom(f_r)$.

If $\beta = \kappa$, then we have already handled all 
possibilities for $x$. 
Suppose that $\beta < \kappa$. 
Then we also have the possibility that $x = \beta$, in which case 
$x \in S$, or $x = L \cap \beta$, where $\beta \in L$ and $L \in A_q$. 
In the second case, $L \in A_r$ and $\beta \in L \cap \dom(f_r) \cap S$, 
so we are done.

Let $x \in \dom(f_r)$, and we will show that 
$f_r(x)$ is a finite $\in$-chain and a subset of $Sk(x) \setminus S$. 
But by the definition of $f_r$, 
$f_r(x) \subseteq f_s(x)$. 
Since $s$ is a condition, $f_s(x)$ is a finite $\in$-chain and a subset 
of $Sk(x) \setminus S$. 
Hence, $f_r(x)$ is as well.

\bigskip

(3) Let $x \in \dom(f_r)$. 
Then $f_r(x) = f_s(x) \cap \dom(f_r)$, and therefore 
$f_r(x) \subseteq \dom(f_r)$. 
Let $K \in f_r(x)$, and we will show that 
$f_r(K) = f_r(x) \cap Sk(K)$. 
But $f_r(K) = f_s(K) \cap \dom(f_r)$, and since $s$ is a condition, 
$f_s(K) = f_s(x) \cap Sk(K)$. 
Thus, 
\begin{multline*}
f_r(K) = f_s(K) \cap \dom(f_r) = (f_s(x) \cap Sk(K)) \cap \dom(f_r) = \\
(f_s(x) \cap \dom(f_r)) \cap Sk(K) = f_r(x) \cap Sk(K).
\end{multline*}

\bigskip

(4) Consider $J \in f_r(x)$, and we will show that $g_r(J,x)$ is a finite 
subset of $x \setminus \sup(J)$. 
If $g_r(J,x) = g_s(J,x)$, then since $s$ is a condition, 
$g_s(J,x)$ is a finite subset of $x \setminus \sup(J)$, 
and we are done.

Suppose that $g_r(J,x) \ne g_s(J,x)$. 
Then by the definition of $g_r$, 
$J = K$, $x \in f_r(\alpha) \cup \{ \alpha \}$, and 
$g_r(J,x) = g_s(K,x) \cup \{ \zeta \}$, where 
$\zeta = \min(M \setminus \xi)$. 
Since $s$ is a condition, 
$g_s(J,x)$ is a finite subset of $x \setminus \sup(J)$. 
So it suffices 
to show that $\zeta \in x \setminus \sup(K)$. 
We already know that $\sup(K) < \xi \le \zeta$, so 
$\zeta \notin \sup(K)$.

It remains to show that $\zeta \in x$. 
If $x = \alpha$, then certainly $\zeta < \alpha$, since 
$\zeta \in M$ and $M \subseteq \alpha$. 
Assume that $x \in f_r(\alpha)$. 
We consider the different possibilities for why $x$ is in $f_r(\alpha)$. 

First, assume that $x \in \dom(f_q)$. 
Then, since $x \in f_r(\alpha) \subseteq f_s(\alpha)$, 
we have that $x \in f_s(\alpha)$. 
But $x$ and $\alpha$ are in $\dom(f_q)$. 
Since $s \le q$, it follows that $x \in f_q(\alpha)$. 
So either $x \in A_1$ or $x \in A_2$.

By the choice of $\gamma$ and $K$, for all $J \in A_1$, 
$\sup(J) < \gamma < \sup(K)$. 
Since $K \in f_r(x)$, $K \in Sk(x)$, and so 
$\sup(K) < \sup(x)$. 
Therefore, $x \notin A_1$.
So $x \in A_2$. 
As noted above, the minimality of $M$ implies that for all 
$J \in A_2$, $M \subseteq J$. 
So $M \subseteq x$. 
Since $\zeta \in M$, $\zeta \in x$, and we are done.

Secondly, assume that $x \notin \dom(f_q)$. 
Since $K \in f_r(x)$, $K \in Sk(x)$. 
So 
$$
K_1 \cap \omega_1 = K \cap \omega_1 < x \cap \omega_1.
$$
It follows that $x$ is not equal to $K$, and 
$x$ is not equal to $K_1 \cap \delta$ for 
any $\delta \in K_1 \cap \dom(f_q) \cap S$.

The remaining possibility is that $\beta < \kappa$ and 
$x = L \cap \beta$, where $L \in A_q$ and $\beta \in L$. 
Since $x \in f_r(\alpha) \subseteq f_s(\alpha)$, we have that
$x \in f_s(\alpha)$. 
As $x$ and $M$ are both in $f_s(\alpha)$, they 
are membership comparable.

Since $K \in f_s(x) = f_s(L \cap \beta)$ and 
$K$ is the membership largest member of $f_s(\alpha)$ with 
$\sup(K) < \xi$, we have that $\xi \le \sup(L \cap \beta)$. 
But recall that 
$q$ forces that $M$ is the membership least element of $\dot f(\alpha)$ 
with $\xi \le \sup(M)$, and $s \le q$. 
Hence, it is not the case that $L \cap \beta \in Sk(M)$. 
So either $L \cap \beta = M$ or $M \in Sk(L \cap \beta)$. 
In either case, $\zeta \in M \subseteq L \cap \beta$, so 
$\zeta \in L \cap \beta = x$.

\bigskip

(5) Suppose that $J \in f_r(L)$ and $L \in f_r(x)$. 
Then $J \in f_s(L)$ and $L \in f_s(x)$. 
We will show that $g_r(J,x) \subseteq g_r(J,L)$. 
If $g_r(J,x) = g_s(J,x)$, then since $s$ is a condition and 
$g_s(J,L) \subseteq g_r(J,L)$, we have that 
$$
g_r(J,x) = g_s(J,x) \subseteq g_s(J,L) \subseteq g_r(J,L).
$$

Assume that $g_r(J,x) \ne g_s(J,x)$. 
Then by the definition of $g_r$, 
we have that $J = K$, $x \in f_r(\alpha) \cup \{ \alpha \}$, 
and $g_r(K,x) = g_s(K,x) \cup \{ \zeta \}$, where 
$\zeta = \min(M \setminus \xi)$. 
Again, $g_s(K,x) \subseteq g_s(K,L) \subseteq g_r(K,L)$. 
So it suffices to show that $\zeta \in g_r(K,L)$. 

By the definition of $g_r$, in order to show that $\zeta \in g_r(K,L)$, 
it is enough to show that $L \in f_r(\alpha) \cup \{ \alpha \}$, for 
then $g_r(K,L)$ is defined as $g_s(K,L) \cup \{ \zeta \}$. 
But $L \in f_r(x)$ and $x \in f_r(\alpha) \cup \{ \alpha \}$. 
So if $x = \alpha$, then $L \in f_r(\alpha)$, and if 
$x \in f_r(\alpha)$, then 
$L \in f_r(x) = f_r(\alpha) \cap Sk(x)$ by requirement (3), 
so $L \in f_r(\alpha)$.

\bigskip

(6) Suppose that 
$\delta \in \dom(f_r) \cap S$, $L \in A_r$, and $\delta \in L$. 
We will show that $L \cap \delta \in f_r(\delta)$. 
Since $\dom(f_r) \subseteq \dom(f_s)$, we have that 
$\delta \in \dom(f_s) \cap S$. 
And $A_r \subseteq A_s$. 
As $s$ is a condition, it follows that $L \cap \delta \in f_s(\delta)$. 
Since $f_r(\delta) = f_s(\delta) \cap \dom(f_r)$, it suffices 
to show that $L \cap \delta \in \dom(f_r)$.

By the definition of $\dom(f_r)$, $\delta$ being in 
$\dom(f_r) \cap S$ implies that either $\delta \in \dom(f_q)$ 
or $\delta = \beta$. 
Also, $A_r = A_q \cup \{ K_1 \}$, so $L$ being in $A_r$ means 
that either $L \in A_q$ or $L = K_1$.

If $\delta \in \dom(f_q)$ and $L \in A_q$, then 
$L \cap \delta \in f_q(\delta) \subseteq
\dom(f_q) \subseteq \dom(f_r)$, 
since $q$ is a condition. 
If $\delta \in \dom(f_q)$ and $L = K_1$, then 
$L \cap \delta = K_1 \cap \delta$ is in $\dom(f_r)$ 
by the definition of $\dom(f_r)$. 
If $\delta = \beta$ and $L \in A_q$, 
then $L \cap \delta = L \cap \beta$ is in 
$\dom(f_r)$ by the definition of $\dom(f_r)$. 
And if $\delta = \beta$ and $L = K_1$, then 
$L \cap \delta = K_1 \cap \beta = K$ is in $\dom(f_r)$ by the 
definition of $\dom(f_r)$.
 
\bigskip

(7) Suppose that $\tau \in r^*(A_r) \cap S$, and we will show that 
$\tau \in \dom(f_r)$. 
Fix $J$ and $L$ in $A_r$ such that $J \sim L$ and 
$\tau = \min((J \cap \kappa) \setminus \beta_{J,L})$. 
Then obviously $J$ and $L$ are different, and since $J \sim L$, 
$J \cap \omega_1 = L \cap \omega_1$ by Lemma 2.17(2). 
Moreover, $A_r = A_q \cup \{ K_1 \}$, and by the choice of $K$, 
$K_1 \cap \omega_1 = K \cap \omega_1$ 
is different from $N \cap \omega_1$ for all $N \in A_q$. 
Therefore, $J$ and $L$ must both be in $A_q$. 
Thus, $\tau \in r^*(A_q) \cap S \subseteq \dom(f_q) \subseteq \dom(f_r)$.
\end{proof}

\begin{proposition}
The forcing poset $\p$ forces that $\langle \dot c_\alpha : 
\alpha \in S \rangle$ is a partial square sequence.
\end{proposition}

\begin{proof}
By Lemma 5.2 and Proposition 5.3, for each $\alpha \in S$, $\p$ 
forces that $\dot c_\alpha$ is a club subset of $\alpha$ with 
order type equal to $\omega_1$. 
Let $G$ be a generic filter on $\p$. 
Consider $\alpha$ and $\beta$ in $S$, 
and let $c_\alpha := \dot c_\alpha^G$ and 
$c_\beta := \dot c_\beta^G$. 
Assume that $\xi$ is a common limit point of $c_\alpha$ and $c_\beta$. 
We will show that $c_\alpha \cap \xi = c_\beta \cap \xi$.

Since $c_\alpha$ and $c_\beta$ are closed, it follows that  
$\xi \in c_\alpha \cap c_\beta$. 
Thus, there are $K \in f(\alpha)$ and $L \in f(\beta)$ such that 
$\sup(K) = \xi = \sup(L)$. 
By Lemma 5.1, 
$$
c_\alpha \cap \xi = c_\alpha \cap \sup(K) = 
\{ \sup(J) : J \in f(K) \},
$$
and 
$$
c_\beta \cap \xi = c_\beta \cap \sup(L) = 
\{ \sup(J) : J \in f(L) \}.
$$
Thus, to show that $c_\alpha \cap \xi = c_\beta \cap \xi$, 
it suffices to show that $f(K) = f(L)$. 
We will prove, in fact, that $K = L$.

Since $K \in f(\alpha)$ and $L \in f(\beta)$, we can fix $p \in G$ 
such that $K \in f_p(\alpha)$ and $L \in f_p(\beta)$. 
Then $K$ and $L$ are in $\dom(f_p)$, by Definition 4.2(3). 
By Definition 4.2(2), fix $K_1$ in $A_p$ and 
$\theta \in (K_1 \cap \dom(f_p) \cap S) \cup \{ \kappa \}$ 
such that $K = K_1 \cap \theta$, and 
$L_1$ in $A_p$ and $\tau \in (L_1 \cap \dom(f_p) \cap S) \cup \{ \kappa \}$ 
such that $L = L_1 \cap \tau$.

Since $\sup(K_1 \cap \theta) = \sup(K) = \xi$, it follows that 
$\xi \le \theta$, 
and similarly, 
$\xi \le \tau$. 
As $\theta$ and $\tau$ are in $S \cup \{ \kappa \}$, 
they have uncountable cofinality. 
On the other hand, since $\xi$ is the supremum of the countable 
set $K$, $\xi$ has countable cofinality. 
Therefore, $\xi < \theta$ and $\xi < \tau$. 
Since the sets $K = K_1 \cap \theta$ and $L = L_1 \cap \tau$ are closed under 
successor ordinals by elementarity, $\xi$ is not a member of 
$K_1$ nor $L_1$. 
Therefore, 
$$
K = K_1 \cap \theta = K_1 \cap \xi,
$$
and 
$$
L = L_1 \cap \tau = L_1 \cap \xi.
$$
The ordinal $\xi$, which is the supremum of $K$ and $L$, 
is a common limit point of $K_1 \cap \kappa$ and $L_1 \cap \kappa$. 
So by Proposition 2.11, $\xi < \beta_{K_1,L_1}$. 

We claim that $K_1 \sim L_1$. 
Suppose not, and without loss of generality, assume that $K_1 < L_1$. 
Then $K_1 \cap \beta_{K_1,L_1}$ is in $L_1$. 
Since $\xi < \beta_{K_1,L_1}$, and $\xi$ is a limit point of 
$K = K_1 \cap \theta$, we have that $\xi$ is a limit point of 
$K_1 \cap \beta_{K_1,L_1}$. 
But $K_1 \cap \beta_{K_1,L_1}$ is in $L_1$, and therefore 
$\xi$ is in $L_1$ by elementarity. 
So $\xi \in L_1 \cap \tau = L$, which is a contradiction.

So indeed, $K_1 \sim L_1$. 
Hence $K_1 \cap \beta_{K_1,L_1} = L_1 \cap \beta_{K_1,L_1}$. 
Since $\xi < \beta_{K_1,L_1}$, it follows that 
$K_1 \cap \xi = L_1 \cap \xi$. 
Thus, $K = K_1 \cap \xi = L_1 \cap \xi = L$.
\end{proof}

We point out that if $S$ is chosen so that 
$S = D \cap \cof(> \! \omega)$, for some club set $D \subseteq \kappa$, 
then adding a partial 
square sequence on $S$ will imply that $\Box_{\omega_1}$ holds 
in the generic extension (see the end of \cite{jk23}).  
So as a special case, our forcing poset provides another way to 
force $\Box_{\omega_1}$ with finite conditions. 

\bigskip

\addcontentsline{toc}{section}{6. Amalgamation over uncountable models}

\textbf{\S 6. Amalgamation over uncountable models}

\stepcounter{section}

\bigskip

We now turn to proving that $\p$ is strongly 
proper on a stationary set, and hence preserves $\omega_1$, and 
is $\kappa$-c.c. 
Strong properness is proven using amalgamation of conditions 
over countable models, and the $\kappa$-c.c.\ is proven using 
amalgamation of conditions over uncountable models. 
The uncountable case is similar to, but not as complicated as, 
the countable case, so we will handle the uncountable case first.

Many of the results which we will prove in Sections 6 and 7 
will be used again in Sections 8 and 9, 
where the approximation property of certain quotients of $\p$ 
is verified. 
For this reason, it will be helpful to develop the notation and results 
of Sections 6 and 7 in great detail.

Let us give a brief outline of the main ideas presented in this section. 
The goal is to show that for any simple model $Q \in \mathcal Y$, 
the maximum condition in $\p$ is strongly $Q$-generic. 
This fact will imply that $\p$ is $\kappa$-c.c. 
Let $D_Q$ denote the set of conditions $s \in \p$ such that for 
all $M \in A_s$, $M \cap Q \in A_s$. 
We will show that $D_Q$ is dense in $\p$. 
For each $s \in D_Q$, we will define a condition $s \restriction Q$ 
in $Q \cap \p$. 
This condition will satisfy that for all $w \le s \restriction Q$ in 
$Q \cap \p$, $w$ and $s$ are compatible. 
Since $D_Q$ is dense, it will follow that the maximum condition of $\p$ 
is strongly $Q$-generic.

\begin{lemma}
Let $q \in \p$ and $Q \in \mathcal Y$. 
Then there is $s \le q$ such that for all $M \in A_s$, 
$M \cap Q \in A_s$. 
Moreover, $A_s = A_q \cup \{ M \cap Q : M \in A_q \}$.
\end{lemma}

Recall that by Assumption 2.5(2), if 
$M \in \mathcal X$ and $Q \in \mathcal Y$, then 
$M \cap Q \in \mathcal X$.

\begin{proof}
By Proposition 2.27, the set 
$A_q \cup \{ M \cap Q : M \in A_q \}$ is adequate. 
Define 
$$
x_0 := r^*(A_q \cup \{ M \cap Q : M \in A_q \}) \cap S,
$$
and define 
$$
x := x_0 \setminus \dom(f_q).
$$
Let $r := q + x$. 
By Definition 4.7 and Lemma 4.8, $r$ is a condition, $r \le q$, and  
$A_r = A_q$. 
Also, $x_0 \subseteq \dom(f_r)$.

Define $s$ as follows. 
Let $f_s := f_r$, $g_s := g_r$, and 
$$
A_s := A_r \cup \{ M \cap Q : M \in A_r \}.
$$
We claim that $s$ is as required. 
By Proposition 2.27, for all $M \in A_s$, 
$M \cap Q \in A_s$. 
It is trivial to check that if $s$ is a condition, then $s \le q$. 

It remains to show that $s$ is a condition. 
We verify requirements (1)--(7) of Definition 4.2. 
(1) follows from Proposition 2.27. 
Requirements (2)--(5) follow immediately from $r$ being a condition, 
together with the fact that 
$f_s = f_r$, $g_s = g_r$, and $A_r \subseteq A_s$.

(6) Suppose that 
$\alpha \in \dom(f_s) \cap S$, $M \in A_s$, and $\alpha \in M$. 
We will show that $M \cap \alpha \in f_s(\alpha)$. 
Since $f_r = f_s$, we have that $\alpha \in \dom(f_r) \cap S$.

First, assume that $M \in A_r$. 
Then since $r$ is a condition, 
$M \cap \alpha \in f_r(\alpha) = f_s(\alpha)$, and we are done. 
Secondly, assume that $M = M_1 \cap Q$ for some $M_1 \in A_r$. 
Then $\alpha \in M = M_1 \cap Q$. 
So $\alpha \in M_1$ and $\alpha < Q \cap \kappa$. 
Thus, 
$$
M \cap \alpha = M_1 \cap Q \cap \alpha = M_1 \cap \alpha.
$$
Since $M_1 \in A_r$, $\alpha \in \dom(f_r) \cap S$, 
and $\alpha \in M_1$, it follows that 
$M \cap \alpha = M_1 \cap \alpha \in f_r(\alpha) = f_s(\alpha)$.

(7) We need to show that $r^*(A_s) \cap S \subseteq \dom(f_s)$. 
But since $A_r = A_q$, we have that 
$$
r^*(A_s) \cap S =  r^*(A_q \cup \{ M \cap Q : M \in A_q \}) \cap S 
= x_0 \subseteq \dom(f_r) = \dom(f_s).
$$
\end{proof}

\begin{definition}
For each $Q \in \mathcal Y$, let $D_Q$ denote the set of conditions 
$q \in \p$ such that for all $M \in A_q$, $M \cap Q \in A_q$.
\end{definition}

\begin{lemma}
Let $Q \in \mathcal Y$. 
Then $D_Q$ is dense in $\p$.
\end{lemma}

\begin{proof}
Immediate from Lemma 6.1.
\end{proof}

Note that if $q \in Q \cap \p$, then $q \in D_Q$. 
Namely, for all $M \in A_q$, $M \in Q$, and therefore 
$M \cap Q = M \in A_q$.

\begin{definition}
Suppose that $Q \in \mathcal Y$ is simple and $q \in D_Q$. 
Let $q \restriction Q$ denote the triple $(f,g,A)$ satisfying:
\begin{enumerate}
\item $\dom(f) := \dom(f_q) \cap Q$, and for all 
$x \in \dom(f)$, $f(x) := f_q(x)$;
\item $\dom(g) := \dom(g_q) \cap Q$, and for all 
$(K,x) \in \dom(g)$, $g(K,x) := g_q(K,x)$;
\item $A := A_q \cap Q$.
\end{enumerate}
\end{definition}

Note that in (1) above, if $x \in \dom(f_q) \cap Q$, then 
$Sk(x) \subseteq Q$. 
So $f_q(x) = f(x)$ is a finite subset of $Q$, and therefore 
is in $Q$. 
In (2), if $K \in f_q(x)$ and $K$ and $x$ are in $Q$, 
then $g_q(K,x) \subseteq x \subseteq Q$. 
So $g_q(K,x)$ is a finite subset of $Q$, and hence is in $Q$. 
Similarly, $A = A_q \cap Q$ is in $Q$. 
It easily follows from these observations 
that $q \restriction Q \in Q$.

\begin{lemma}
Let $Q \in \mathcal Y$ be simple and $q \in D_Q$. 
Then $q \restriction Q$ is in $Q \cap \p$ and 
$q \le q \restriction Q$.
\end{lemma}

\begin{proof}
Let $q \restriction Q = (f,g,A)$. 
We already observed that $q \restriction Q \in Q$.
It is trivial to check that 
if $q \restriction Q$ is a condition, then $q \le q \restriction Q$. 
So it suffices to show that $q \restriction Q$ is a condition. 
We verify requirements (1)--(7) of Definition 4.2. 
(1), (5), (6), and (7) are immediate. 
It remains to prove (2), (3), and (4).

\bigskip

(2) Obviously $f$ is a function with a finite domain. 
Let $x \in \dom(f) = \dom(f_q) \cap Q$. 
Since $f(x) = f_q(x)$, it follows that 
$f(x)$ is a finite $\in$-chain and a subset 
of $Sk(x) \setminus S$. 
We claim that either $x \in S$, or there is $M \in A$ and 
$\alpha \in (M \cap \dom(f) \cap S) \cup \{ \kappa \}$ 
such that $x = M \cap \alpha$.

Since $x \in \dom(f_q)$ and $q$ is a condition, 
we have that either $x \in S$, 
or there is $M_1 \in A_q$ and 
$\alpha \in (M_1 \cap \dom(f_q) \cap S) \cup \{ \kappa \}$ 
such that $x = M_1 \cap \alpha$. 
If $x \in S$, then we are done, so assume the second case.

Since $q \in D_Q$ and $Q$ is simple, 
$$
M_1 \cap Q \in A_q \cap Q = A.
$$
So it suffices to show that $x = M_1 \cap Q \cap \beta$, 
for some $\beta$ with 
$$
\beta \in (M_1 \cap Q \cap \dom(f) \cap S) \cup \{ \kappa \}.
$$
We split the proof into the cases of whether  
$\alpha < Q \cap \kappa$ or $Q \cap \kappa \le \alpha$.

First, assume that $\alpha < Q \cap \kappa$. 
Then $\alpha \in \dom(f_q) \cap S \cap Q = \dom(f) \cap S$, 
and 
$$
x = M_1 \cap \alpha = M_1 \cap \alpha \cap Q 
= M_1 \cap Q \cap \alpha.
$$
Hence, $x = M_1 \cap Q \cap \beta$, where 
$\beta = \alpha \in M_1 \cap Q \cap \dom(f) \cap S$.

Secondly, assume that $Q \cap \kappa \le \alpha$. 
Since $x \in Q$, 
$x = M_1 \cap \alpha \subseteq Q \cap \kappa$. 
Therefore, 
$$
x = M_1 \cap \alpha = (M_1 \cap \alpha) \cap (Q \cap \kappa) = 
M_1 \cap Q \cap \kappa.
$$
Thus, $x = M_1 \cap Q \cap \beta$, where $\beta = \kappa$.

\bigskip

(3) Let $x \in \dom(f) = \dom(f_q) \cap Q$, and we will 
show that $f(x) \subseteq \dom(f)$. 
We have that $f_q(x) \subseteq Sk(x) \subseteq Q$. 
And since $q$ is a condition, 
$$
f(x) = f_q(x) \subseteq \dom(f_q) \cap Q = \dom(f).
$$
So $f(x) \subseteq \dom(f)$, as required. 
Now consider $K \in f(x) = f_q(x)$. 
Then since $q$ is a condition, 
$$
f(K) = f_q(K) = f_q(x) \cap Sk(K) = f(x) \cap Sk(K).
$$

\bigskip

(4) We have that 
$(K,x) \in \dom(g)$ iff $(K,x) \in \dom(g_q) \cap Q$ iff 
($K$ and $x$ are in $Q$ and $K \in f_q(x)$) iff $K \in f(x)$. 
For each $(K,x) \in \dom(g)$, 
$g(K,x) = g_q(K,x) \subseteq x \setminus \sup(K)$. 
\end{proof}

The next lemma will not be used until Section 8.

\begin{lemma}
Let $Q \in \mathcal Y$ be simple.
\begin{enumerate}
\item Suppose that $p \in Q \cap \p$, $q \in D_Q$, and $q \le p$. 
Then $q \restriction Q \le p$.
\item Suppose that $q$ and $r$ are in $D_Q$ and 
$r \le q \restriction Q$. 
Then $r \restriction Q \le q \restriction Q$.
\item Suppose that $p$ and $q$ are in $D_Q$ and $q \le p$. 
Then $q \restriction Q \le p \restriction Q$.
\end{enumerate}
\end{lemma}

\begin{proof}
(1) We verify properties (a)--(d) of Definition 4.2. 
(a) Since $p \in Q$, $A_p \subseteq Q$. 
As $q \le p$, we have that 
$$
A_p \subseteq A_q \cap Q = A_{q \restriction Q}.
$$
(b) Since $p \in Q$, $\dom(f_p) \subseteq Q$. 
As $q \le p$, we have that 
$$
\dom(f_p) \subseteq \dom(f_q) \cap Q = \dom(f_{q \restriction Q}).
$$ 
Let $x \in \dom(f_p)$. 
Then 
$$
f_p(x) \subseteq f_q(x) = f_{q \restriction Q}(x).
$$
(c) Let $(K,x) \in \dom(g_p)$. 
Since $q \le p$, 
$$
g_p(K,x) \subseteq g_q(K,x) = g_{q \restriction Q}(K,x).
$$
(d) Assume that $K$ and $x$ are in $\dom(f_p)$ 
and $K \in f_{q \restriction Q}(x)$. 
Then $K \in f_{q \restriction Q}(x) = f_q(x)$. 
Since $q \le p$, it follows that $K \in f_p(x)$.

\bigskip

(2) We know that $q \restriction Q \in Q \cap \p$, 
$r \in D_Q$, and $r \le q \restriction Q$. 
By (1), it follows that $r \restriction Q \le q \restriction Q$.

\bigskip

(3) By Lemma 6.5, we know that $p \le p \restriction Q$. 
So $q \le p \le p \restriction Q$. 
Hence, $q \le p \restriction Q$. 
Thus, $p$ and $q$ are in $D_Q$ and $q \le p \restriction Q$. 
By (2), it follows that $q \restriction Q \le p \restriction Q$. 
\end{proof}

We will now begin analyzing the situation where $q \in D_Q$  
and $w \le q \restriction Q$ is in $Q \cap \p$.

\begin{lemma}
Let $Q \in \mathcal Y$ be simple and $q \in D_Q$. 
Suppose that $w \in Q \cap \p$ and 
$w \le q \restriction Q$. 
Then:
\begin{enumerate}
\item $A_q \cap Q \subseteq A_w$;
\item $\dom(f_q) \cap Q \subseteq \dom(f_w)$, and 
for all $x \in \dom(f_q) \cap Q$, 
$f_q(x) \subseteq f_w(x)$;
\item $\dom(g_q) \cap Q \subseteq \dom(g_w)$, and for all 
$(K,x) \in \dom(g_q) \cap Q$, 
$g_q(K,x) \subseteq g_w(K,x)$.
\end{enumerate}
\end{lemma}

\begin{proof}
Immediate from the definition of $q \restriction Q$ and the 
fact that $w \le q \restriction Q$.
\end{proof}

As discussed at the beginning of the section, we are going to show that 
whenever $w \le q \restriction Q$, where $q \in D_Q$ and 
$w \in Q \cap \p$, then $w$ and $q$ are compatible. 
We now begin the construction of a specific lower bound of $w$ and $q$, 
which we will denote by $w \oplus_Q q$. 
In order to define the amalgam $w \oplus_Q q$, we will need to 
define the $f$, $g$, and $A$ components of $w \oplus_Q q$. 
The amalgam of the $A$-components will be $A_w \cup A_q$. 
We handle the $f$-components next.

\begin{definition}
Let $Q \in \mathcal Y$ be simple and $q \in D_Q$. 
Suppose that $w \in Q \cap \p$ and $w \le q \restriction Q$. 
Define $f_w \oplus_Q f_q = f$ as follows.

The domain of $f$ is equal to $\dom(f_w) \cup \dom(f_q)$. 
The values of $f$ are defined by the following cases:
\begin{enumerate}
\item for all $x \in \dom(f_w)$, $f(x) := f_w(x)$;
\item for all $x \in \dom(f_q) \setminus Q$, if $f_q(x) \cap Q = \emptyset$, 
then $f(x) := f_q(x)$;
\item for all $x \in \dom(f_q) \setminus Q$, if $f_q(x) \cap Q \ne \emptyset$, 
then $f(x) := f_q(x) \cup f_w(M)$, where $M$ is the membership 
largest element of $f_q(x) \cap Q$.
\end{enumerate}
\end{definition}

It is easy to see that cases 1--3 describe all of the 
possibilities for a set being 
in $\dom(f)$, since $\dom(f_q) \cap Q \subseteq \dom(f_w)$ 
by Lemma 6.7(2). 
Moreover, cases 1--3 are obviously disjoint. 

The next three lemmas describe some important properties of 
$f_w \oplus_Q f_q$. 
The first two lemmas are easy, but the third is quite involved.

\begin{lemma}
Let $Q \in \mathcal Y$ be simple and $q \in D_Q$. 
Suppose that $w \in Q \cap \p$ and $w \le q \restriction Q$. 
Let $f := f_w \oplus_Q f_q$. Then:
\begin{enumerate}
\item if $x \in \dom(f_w)$, then $f_w(x) = f(x)$;
\item if $x \in \dom(f_q)$, then $f_q(x) \subseteq f(x)$.
\end{enumerate}
\end{lemma}

\begin{proof}
(1) is by Definition 6.8(1), and (2) follows immediately 
from Definition 6.8(2,3).
\end{proof}

\begin{lemma}
Let $Q \in \mathcal Y$ be simple and $q \in D_Q$. 
Suppose that $w \in Q \cap \p$ and $w \le q \restriction Q$. 
Let $f := f_w \oplus_Q f_q$. Then:
\begin{enumerate}
\item $\dom(f) \cap Q = \dom(f_w)$;
\item if $K \in f(x)$ and $K$ and $x$ are in $\dom(f_w)$, then 
$K \in f_w(x)$;
\item if $K \in f(x)$ and $K$ and $x$ are in $\dom(f_q)$, 
then $K \in f_q(x)$.
\end{enumerate}
\end{lemma}

\begin{proof}
(1) By Lemma 6.7(2), $\dom(f_q) \cap Q \subseteq \dom(f_w)$.  
Hence, 
$$
\dom(f) \cap Q = (\dom(f_w) \cup \dom(f_q)) \cap Q = \dom(f_w),
$$
where the last equality follows from the fact that 
$\dom(f_w) \cap Q = \dom(f_w)$ and 
$\dom(f_q) \cap Q \subseteq \dom(f_w)$.

(2) Suppose that $K \in f(x)$ and $K$ and $x$ are in $\dom(f_w)$. 
By Definition 6.8(1), $f(x) = f_w(x)$, so $K \in f_w(x)$.

(3) Assume that $K \in f(x)$ and $K$ and $x$ are in $\dom(f_q)$. 
We will show that $K \in f_q(x)$. 
The proof splits into the three cases of Definition 6.8 for how 
$f(x)$ is defined.
In case 2, $f(x) = f_q(x)$, so $K \in f_q(x)$. 

In case 1, $f(x) = f_w(x)$. 
So $K \in f_w(x)$. 
In particular, $K$ and $x$ are in $Q$. 
So $K$ and $x$ are in 
$\dom(f_q) \cap Q = \dom(f_{q \restriction Q})$. 
Since $w \le q \restriction Q$ and $K \in f_w(x)$, it follows 
that $K \in f_{q \restriction Q}(x) = f_q(x)$.

In case 3, $f(x) = f_q(x) \cup f_w(M)$, where $M$ is the membership 
largest element of $f_q(x) \cap Q$. 
So either $K \in f_q(x)$, or $K \in f_w(M)$. 
In the first case we are done, so assume that $K \in f_w(M)$. 
Then $K$ and $M$ are in 
$\dom(f_q) \cap Q = \dom(f_{q \restriction Q})$. 
Since $w \le q \restriction Q$ and $K \in f_w(M)$, it follows 
that $K \in f_{q \restriction Q}(M) = f_q(M)$. 
So $K \in f_q(M)$ and $M \in f_q(x)$. 
Therefore, $K \in f_q(x)$.
\end{proof}

The next lemma will be used to verify that 
$f_w \oplus_Q f_q$ 
satisfies requirements (2) and (3) of Definition 4.2 for 
$w \oplus_Q q$.

\begin{lemma}
Let $Q \in \mathcal Y$ be simple and $q \in D_Q$. 
Suppose that $w \in Q \cap \p$ and 
$w \le q \restriction Q$. 
Let $f := f_w \oplus_Q f_q$. 
Then:
\begin{enumerate}
\item $f$ is a function with a finite domain, and for all 
$x \in \dom(f)$, either $x \in S$, or there is $M \in A_w \cup A_q$ 
and 
$$
\alpha \in (M \cap \dom(f) \cap S) \cup \{ \kappa \}
$$ 
such that $x = M \cap \alpha$; moreover, for all $x \in \dom(f)$, 
$f(x)$ is a finite $\in$-chain and $f(x) \subseteq Sk(x) \setminus S$;

\item if $x \in \dom(f)$, then $f(x) \subseteq \dom(f)$, and for all 
$K \in f(x)$, $f(K) = f(x) \cap Sk(K)$.
\end{enumerate}
\end{lemma}

\begin{proof}
(1) The domain of $f$ is equal to $\dom(f_w) \cup \dom(f_q)$, 
which is finite. 
Let $x \in \dom(f)$. 
Then either $x \in \dom(f_w)$ or $x \in \dom(f_q)$.

If $x \in \dom(f_w)$, then either $x \in S$, 
or there is $M \in A_w$ 
and $\alpha \in (M \cap \dom(f_w) \cap S) \cup \{ \kappa \}$ 
such that $x = M \cap \alpha$. 
If $x \in \dom(f_q)$, then either $x \in S$, or there is $M \in A_q$ 
and $\alpha \in (M \cap \dom(f_q) \cap S) \cup \{ \kappa \}$ 
such that $x = M \cap \alpha$. 
In either case, either $x \in S$, or there is 
$M \in A_w \cup A_q$ and 
$\alpha \in (M \cap \dom(f) \cap S) \cup \{ \kappa \}$ 
such that $x = M \cap \alpha$.

Let $x \in \dom(f)$, and we will show that $f(x)$ is a finite 
$\in$-chain and a subset of $Sk(x) \setminus S$. 
We consider the three cases in the definition of $f(x)$ given 
in Definition 6.8. 
In cases 1 and 2, $f(x)$ is equal to either $f_w(x)$ or $f_q(x)$. 
Since $w$ and $q$ are conditions, then in either case, 
$f(x)$ is a finite $\in$-chain and a subset of $Sk(x) \setminus S$.

Consider case 3, which says that $x \in \dom(f_q) \setminus Q$ 
and $f(x) = f_q(x) \cup f_w(M)$, where $M$ is the membership 
largest element of $f_q(x) \cap Q$. 
Since $w$ and $q$ are conditions, it follows that 
$$
f(x) \subseteq (Sk(x) \cup Sk(M)) \setminus S.
$$
But $M \in f_q(x)$ implies that $M \in Sk(x)$, and therefore 
$Sk(M) \subseteq Sk(x)$. 
Hence,
$$
f(x) \subseteq Sk(x) \setminus S.
$$

Since $w$ and $q$ are conditions, $f_q(x)$ and $f_w(M)$ are 
each finite $\in$-chains. 
So to prove that $f(x)$ is a finite $\in$-chain, it suffices to show that  
whenever $K \in f_w(M)$ and 
$L \in f_q(x) \setminus f_w(M)$, then $K \in Sk(L)$.

If $L = M$, then since $K \in f_w(M)$, $K \in Sk(M) = Sk(L)$, 
and we are done. 
Suppose that $L \ne M$. 
As $L$ and $M$ are different elements of $f_q(x)$, either 
$L \in f_q(M)$ or $M \in f_q(L)$. 
But $f_q(M) \subseteq f_w(M)$ by Lemma 6.7(2), 
and we assumed that $L \notin f_w(M)$. 
Thence, $M \in f_q(L)$. 
But $K \in f_w(M)$ implies that $K \in Sk(M)$, and 
$M \in f_q(L)$ implies that $M \in Sk(L)$. 
Therefore, $K \in Sk(M) \subseteq Sk(L)$, so $K \in Sk(L)$.

\bigskip

(2) Let $x \in \dom(f)$. 
We claim that $f(x) \subseteq \dom(f)$. 
In cases 1 and 2 of Definition 6.8, either 
$f(x) = f_w(x)$ or $f(x) = f_q(x)$. 
Since $w$ and $q$ are conditions, 
$f(x) \subseteq \dom(f_w) \subseteq \dom(f)$ in the first case, 
and $f(x) \subseteq \dom(f_q) \subseteq \dom(f)$ in the second case.

In the third case, $f(x) = f_q(x) \cup f_w(M)$, where $M$ is the 
membership largest element of $f_q(x) \cap Q$. 
Since $w$ and $q$ are conditions, 
$f(x) \subseteq \dom(f_q) \cup \dom(f_w) = \dom(f)$.

\bigskip

Assume that $K \in f(x)$, and we will show that 
$f(K) = f(x) \cap Sk(K)$. 
We split the proof into the three cases of Definition 6.8 
for the definition of $f(x)$.

In case 1, $x \in \dom(f_w)$ and $f(x) = f_w(x)$. 
So $K \in f_w(x) \subseteq \dom(f_w)$.  
Hence, $f(K) = f_w(K)$. 
Since $w$ is a condition,  
$$
f(K) = f_w(K) = f_w(x) \cap Sk(K) = f(x) \cap Sk(K).
$$

In case 2, $x \in \dom(f_q) \setminus Q$, $f_q(x) \cap Q = \emptyset$, 
and $f(x) = f_q(x)$. 
Since $K \in f(x) = f_q(x)$, we have that 
$f_q(K) = f_q(x) \cap Sk(K)$, since 
$q$ is a condition. 
In particular, $K \in \dom(f_q) \setminus Q$ and 
$f_q(K) \cap Q = \emptyset$. 
Therefore, by definition, $f(K) = f_q(K)$. 
So 
$$
f(K) = f_q(K) = f_q(x) \cap Sk(K) = f(x) \cap Sk(K).
$$

In case 3, $x \in \dom(f_q) \setminus Q$ and 
$f(x) = f_q(x) \cup f_w(M)$, where $M$ is the membership largest 
element of $f_q(x) \cap Q$. 
Then either $K \in f_q(x)$, or $K \in f_w(M)$. 
Since $M$ is the largest element of $f_q(x) \cap Q$ and 
$q$ is a condition,
$$
f_q(x) \cap Q = (f_q(x) \cap Sk(M)) \cup \{ M \} = 
f_q(M) \cup \{ M \} \subseteq f_w(M) \cup \{ M \},
$$
where the inclusion holds by Lemma 6.7(2). 
So $f_q(x) \cap Q \subseteq f_w(M) \cup \{ M \}$. 
It easily follows that either $K \in f_q(x) \setminus Q$, $K = M$, or 
$K \in f_w(M)$.

First, assume that $K \in f_q(x) \setminus Q$. 
Then, since $f_q(K) = f_q(x) \cap Sk(K)$, 
$M$ is the membership largest element of $f_q(K) \cap Q$. 
So by definition, $f(K) = f_q(K) \cup f_w(M)$. 
Since $f_w(M) \subseteq Sk(M) \subseteq Sk(K)$, we have that  
\begin{multline*}
f(K) = f_q(K) \cup f_w(M) = 
(f_q(x) \cap Sk(K)) \cup f_w(M) = \\
(f_q(x) \cup f_w(M)) \cap Sk(K) = f(x) \cap Sk(K).
\end{multline*}

Secondly, assume that $K = M$. 
Then $f(K) = f(M) = f_w(M)$. 
Hence, it suffices to show that 
$$
f_w(M) = (f_q(x) \cup f_w(M)) \cap Sk(M).
$$
The forward inclusion is immediate. 
For the reverse inclusion, 
let $J \in (f_q(x) \cup f_w(M)) \cap Sk(M)$, and we will show 
that $J \in f_w(M)$. 
So either $J \in f_q(x) \cap Sk(M)$ or $J \in f_w(M) \cap Sk(M)$. 
In the latter case, we are done. 
In the former case, by Lemma 6.7(2) we have that 
$$
J \in f_q(x) \cap Sk(M) = f_q(M) \subseteq f_w(M),
$$
so $J \in f_w(M)$.

Thirdly, assume that $K \in f_w(M)$. 
Then $K \in \dom(f_w)$, so $f(K) = f_w(K)$. 
Since $w$ is a condition, $f_w(K) = f_w(M) \cap Sk(K)$. 
But by the case in the previous paragraph, 
$f_w(M) = f(M) = f(x) \cap Sk(M)$. 
And since $K \in Sk(M)$, $Sk(K) \subseteq Sk(M)$. 
Therefore,
$$
f(K) = f_w(K) = f_w(M) \cap Sk(K) = 
(f(x) \cap Sk(M)) \cap Sk(K) = f(x) \cap Sk(K).
$$
\end{proof}

We now handle the amalgamation of the $g$-components of 
$w$ and $q$.

\begin{definition}
Let $Q \in \mathcal Y$ be simple and $q \in D_Q$. 
Suppose that $w \in Q \cap \p$ and $w \le q \restriction Q$. 
Let $f := f_w \oplus_Q f_q$.

Define $g_w \oplus_Q g_q$ as the function $g$ with domain 
equal to the set of pairs $(K,x)$ such that $K \in f(x)$, such that 
for all $(K,x) \in \dom(g)$,
$$
g(K,x) := \bigcup \{ g_w(K,y) \cup g_q(K,y) : x = y, \ 
\textrm{or} \ x \in f(y) \}.\footnote{When working with the $g$-components 
of a condition, we will adopt the convention that when $(K,x)$ is not 
a member of the domain of $g$, then 
$g(K,x)$ will denote the empty set. 
In particular, when verifying an inclusion of the form 
$g(K,x) \subseteq y$, the inclusion is trivial in the case that 
$(K,x)$ is not in $\dom(g)$.}
$$
\end{definition}

Note that $g(K,x)$ is finite. 

The next lemma will be used to show that $g_w \oplus_Q g_q$ 
satisfies requirement (4) of Definition 4.2 for 
$w \oplus_Q q$.

\begin{lemma}
Let $Q \in \mathcal Y$ be simple and $q \in D_Q$. 
Suppose that $w \in Q \cap \p$ and $w \le q \restriction Q$. 
Let $g := g_w \oplus_Q g_q$. 
Then for any $(K,x) \in \dom(g)$,
$$
g(K,x) \subseteq x \setminus \sup(K).
$$
\end{lemma}

\begin{proof}
Let $(K,x) \in \dom(g)$, which means that $K \in f(x)$. 
By definition, any ordinal in $g(K,x)$ is either in 
$g_w(K,x) \cup g_q(K,x)$, or in 
$g_w(K,y) \cup g_q(K,y)$ for some $y$ with $x \in f(y)$. 
In the first case, the ordinal is in 
$x \setminus \sup(K)$, since $w$ and $q$ are conditions. 
In the second case, the ordinal is not less than $\sup(K)$, because 
$w$ and $q$ are conditions. 
Thus, it suffices to show that whenever $x \in f(y)$, then
$$
g_w(K,y) \cup g_q(K,y) \subseteq x.
$$
We split the proof into the three cases of 
Definition 6.8 for how $f(y)$ is defined. 
Note that by Lemma 6.11(2), $K \in f(x)$ and $x \in f(y)$ 
implies that $K \in f(y)$.

\bigskip

(1) Suppose that $y \in \dom(f_w)$, so that $f(y) = f_w(y)$. 
Then $x \in f(y) = f_w(y) \subseteq \dom(f_w)$. 
So $x \in \dom(f_w)$, and therefore $f(x) = f_w(x)$ 
Hence, $K \in f_w(x)$ and $x \in f_w(y)$. 
By Definition 4.2(4,5), 
$$
g_w(K,y) \subseteq g_w(K,x) \subseteq x.
$$
Now assume that $(K,y) \in \dom(g_q)$, and we will show that 
$g_q(K,y) \subseteq x$. 
Then $(K,y) \in \dom(g_q) \cap Q$, so by Lemma 6.7(3), 
$g_q(K,y) \subseteq g_w(K,y)$. 
But we just proved that $g_w(K,y) \subseteq x$. 
So $g_q(K,y) \subseteq x$.

\bigskip

(2) Suppose that $y \in \dom(f_q) \setminus Q$ and 
$f_q(y) \cap Q = \emptyset$. 
Then $f(y) = f_q(y)$, so $x \in f_q(y)$.  
And since $K$ and $x$ are in $f(y) = f_q(y)$, 
$K$ and $x$ are not in $Q$. 

Since $f_q(x) = f_q(y) \cap Sk(x)$, 
we have that $f_q(x) \cap Q = \emptyset$. 
Therefore, by definition, $f(x) = f_q(x)$. 
Hence, $K \in f_q(x)$ and $x \in f_q(y)$. 
By Definition 4.2(4,5), it follows that 
$$
g_q(K,y) \subseteq g_q(K,x) \subseteq x.
$$
On the other hand, 
since $y$ is not in $Q$, $(K,y)$ is not in the domain of $g_w$. 
So the inclusion $g_w(K,y) \subseteq x$ is trivial. 

\bigskip

(3) Suppose that $y \in \dom(f_q) \setminus Q$, and 
$f(y) = f_q(y) \cup f_w(M)$, where $M$ is the membership largest 
element of $f_q(y) \cap Q$. 
Since $y \notin Q$, $(K,y)$ is not in $\dom(g_w)$, and therefore 
$g_w(K,y) = \emptyset$, which is a subset of $x$. 
It remains to show that $g_q(K,y) \subseteq x$. 
This is trivial if $(K,y) \notin \dom(g_q)$, so 
assume that $(K,y) \in \dom(g_q)$, which means 
that $K \in f_q(y)$.

Since $f(y) = f_q(y) \cup f_w(M)$, either $x \in f_q(y)$, or 
$x \in f_w(M)$. 
First, assume that $x \in f_q(y)$. 
Then $K \in f_q(y)$, $x \in f_q(y)$, and $K \in Sk(x)$, which implies 
that $K \in f_q(x)$, since $q$ is a condition. 
Therefore, 
$$
g_q(K,y) \subseteq g_q(K,x) \subseteq x,
$$
since $q$ is a condition.

Secondly, assume that $x \in f_w(M)$. 
So $x \in Q$, and since $K \in Sk(x)$, it follows that $K \in Q$ as well. 
Now $x \in f_w(M)$ implies that $x$, and hence $K$, are in $Sk(M)$. 
Also, $K$ and $M$ are both in $f_q(y)$. 
So $K \in f_q(M)$. 
By Lemma 6.7(2), 
$f_q(M) \subseteq f_w(M)$, so $K \in f_w(M)$. 

Since $K$ and $M$ are in $f_q(y) \cap Q$ and $K \in f_q(M)$, 
$$
g_q(K,y) \subseteq g_q(K,M) \subseteq g_w(K,M),
$$
where the last inclusion holds by Lemma 6.7(3). 
Also, $K \in f_w(M)$, $x \in f_w(M)$, and $K \in Sk(x)$ imply 
that $K \in f_w(x)$, since $w$ is a condition. 
So 
$$
g_w(K,M) \subseteq g_w(K,x) \subseteq x,
$$
since $w$ is a condition. 
Thence, $g_q(K,y) \subseteq g_w(K,M) \subseteq x$.
\end{proof}

We are ready to define the amalgam $w \oplus_Q q$.

\begin{definition}
Let $Q \in \mathcal Y$ be simple and $q \in D_Q$. 
Suppose that $w \in Q \cap \p$ and $w \le q \restriction Q$. 
Let $w \oplus_Q q$ be the triple $(f,g,A)$ defined by:
\begin{enumerate}
\item $f := f_w \oplus_Q f_q$;
\item $g := g_w \oplus_Q g_q$;
\item $A := A_w \cup A_q$.
\end{enumerate}
\end{definition}

We will now show that $w \oplus_Q q$ is a condition below $w$ and $q$. 
We have done most of the work of the proof already.

\begin{proposition}
Let $Q \in \mathcal Y$ be simple and $q \in D_Q$. 
Suppose that $w \in Q \cap \p$ and $w \le q \restriction Q$. 
Then $w$ and $q$ are compatible.  
In fact, $w \oplus_Q q$ is in $\p$ and $w \oplus_Q q \le w, q$.
\end{proposition}

\begin{proof}
We will prove that $w \oplus_Q q$ is a condition and that 
$w \oplus_Q q$ is below $w$ and $q$. 
Let $w \oplus_Q q = (f,g,A)$. 

To show that $w \oplus_Q q$ is a condition, we verify requirements 
(1)--(7) of Definition 4.2.

\bigskip

(1) We apply Proposition 2.28. 
Since $q \in D_Q$, we have that for all $M \in A_q$, 
$M \cap Q \in A_q$. 
Also, $A_w$ is adequate, and by Lemma 6.7(1), 
$$
A_q \cap Q \subseteq A_w \subseteq Q.
$$
By Proposition 2.28, $A_w \cup A_q = A$ is adequate.

\bigskip

(2,3) These statements are immediate from Lemma 6.11.

\bigskip

(4) By Definition 6.12, $g$ is a function whose domain is the set 
of pairs $(K,x)$ such that $K \in f(x)$. 
And by Lemma 6.13, for all $(K,x) \in \dom(g)$, 
$g(K,x) \subseteq x \setminus \sup(K)$. 
Also $g(K,x)$ is finite, by Definition 6.12.

\bigskip

(5) Let $K \in f(L)$ and $L \in f(x)$, and 
we will show that $g(K,x) \subseteq g(K,L)$. 
Let $\xi \in g(K,x)$. 
Then by Definition 6.12, either $\xi \in g_w(K,x) \cup g_q(K,x)$, 
or for some $y$ with $x \in f(y)$, 
$\xi \in g_w(K,y) \cup g_q(K,y)$. 
In the second case, $L \in f(x)$ and $x \in f(y)$ imply by 
requirement (3) that $L \in f(y)$. 
So letting $z := x$ in the first case, and $z := y$ in the second case, 
we have that $L \in f(z)$ and 
$\xi \in g_w(K,z) \cup g_q(K,z)$. 
By Definition 6.12, it follows that $\xi \in g(K,L)$.

\bigskip

(6) Let $\alpha \in \dom(f) \cap S$, $M \in A$, and suppose that 
$\alpha \in M$. 
We will show that $M \cap \alpha \in f(\alpha)$. 
Since $\dom(f) = \dom(f_w) \cup \dom(f_q)$, either 
$\alpha \in \dom(f_w)$ or $\alpha \in \dom(f_q)$. 
As $A = A_w \cup A_q$, either $M \in A_w$ or $M \in A_q$. 

First, assume that $\alpha \in \dom(f_w)$. 
Then $f(\alpha) = f_w(\alpha)$ by Definition 6.8. 
If $M \in A_w$, then $M \cap \alpha \in f_w(\alpha)$, since 
$w$ is a condition. 
But $f(\alpha) = f_w(\alpha)$, so $M \cap \alpha \in f(\alpha)$.

Suppose that $M \in A_q$. 
Since $q \in D_Q$ and $Q$ is simple, 
$M \cap Q \in A_q \cap Q \subseteq A_w$, by Lemma 6.7(1). 
So $M \cap Q \in A_w$. 
As $\alpha \in \dom(f_w)$, $\alpha \in Q$. 
Thus, $\alpha \in (M \cap Q) \cap \dom(f_w)$. 
Since $w$ is a condition, it follows that 
$(M \cap Q) \cap \alpha \in f_w(\alpha) = f(\alpha)$. 
But $\alpha \in Q$ implies that 
$M \cap Q \cap \alpha = M \cap \alpha$. 
So $M \cap \alpha \in f(\alpha)$, as required.

Secondly, assume that $\alpha \in \dom(f_q) \setminus \dom(f_w)$. 
Then $\alpha \notin Q$, for otherwise 
$\alpha \in \dom(f_q) \cap Q \subseteq \dom(f_w)$ 
by Lemma 6.7(2). 
Note that this implies that $M \notin A_w$. 
For otherwise, 
$\alpha \in M \in A_w \subseteq Q$, which implies that $\alpha \in Q$. 
So $M \in A_q$. 
Since $q$ is a condition, $M \cap \alpha \in f_q(\alpha)$. 
But $f_q(\alpha) \subseteq f(\alpha)$, by cases 2 and 3 of 
Definition 6.8. 
So $M \cap \alpha \in f(\alpha)$.

\bigskip

(7) As in (1) above, the assumptions of Proposition 3.8 hold for 
$A_q$ and $A_w$. 
Therefore, 
$$
r^*(A) = r^*(A_w \cup A_q) = 
r^*(A_w) \cup r^*(A_q).
$$
As $w$ and $q$ are conditions, 
$$
r^*(A_w) \cap S \subseteq \dom(f_w), \ \ \ 
r^*(A_q) \cap S \subseteq \dom(f_q).
$$
But $\dom(f_w) \subseteq \dom(f)$ and 
$\dom(f_q) \subseteq \dom(f)$. 
Hence, 
$$
r^*(A) \cap S = (r^*(A_w) \cap S) \cup (r^*(A_q) \cap S) 
\subseteq \dom(f).
$$

\bigskip

This completes the proof that $w \oplus_Q q$ is a condition. 
Now we show that $w \oplus_Q q \le w, q$. 
First, we prove that $w \oplus_Q q \le w$ by verifying 
properties (a)--(d) of Definition 4.2 for $w$. 
(a) Since $A = A_w \cup A_q$, clearly $A_w \subseteq A$. 
(b) follows from Lemma 6.9(1), 
(c) is immediate from Definition 6.12, and (d) was 
proved in Lemma 6.10(2).

Secondly, we prove that $w \oplus_Q q \le q$ by 
verifying properties (a)--(d) of Definition 4.2 for $q$. 
(a) Since $A = A_w \cup A_q$, clearly $A_q \subseteq A$. 
(b) was proved in Lemma 6.9(2),  
(c) is immediate from Definition 6.12, and (d) was 
proved in Lemma 6.10(3).
\end{proof}

\begin{corollary}
The forcing poset $\p$ is $\kappa$-c.c.
\end{corollary}

\begin{proof}
Let $A$ be an antichain of $\p$, and suppose for a contradiction 
that $A$ has size at least $\kappa$. 
Without loss of generality, assume that $A$ is maximal. 
By Assumption 2.23, there are stationarily many simple models 
in $\mathcal Y$, so we can fix a simple model $Q \in \mathcal Y$ 
such that $Q \prec (H(\lambda),\in,\p,A)$. 
As $A$ has size at least $\kappa$ and $|Q| < \kappa$, we can 
fix $s \in A \setminus Q$.

By Lemma 6.3, fix $q \le s$ such that $q \in D_Q$. 
Then $q \restriction Q$ is a condition in $Q \cap \p$. 
By the elementarity of $Q$ and the maximality of $A$, 
there is $t \in A \cap Q$ which is compatible with $q \restriction Q$. 
By elementarity, fix $w \in Q \cap \p$ such that $w \le q \restriction Q, t$. 

By Proposition 6.15, $w$ and $q$ are compatible, so fix  
$v \le w, q$. 
Then $v \le w \le t$, and $v \le q \le s$. 
Hence, $s$ and $t$ are compatible. 
But $s$ and $t$ are in $A$ and $A$ is an antichain. 
Therefore, $s = t$. 
This is impossible, since $t \in Q$ and $s \notin Q$.
\end{proof}

\bigskip

\addcontentsline{toc}{section}{7. Amalgamation over countable models}

\textbf{\S 7. Amalgamation over countable models}

\stepcounter{section}

\bigskip

In this section we will prove that the forcing poset $\p$ 
is strongly proper on a stationary set. 
We will show that for any simple model $N \in \mathcal X$, for 
any $p \in N \cap \p$, there is $q \le p$ which is 
strongly $N$-generic.

For each $p \in N \cap \p$, we will show that 
there is $q \le p$ such that $N \in A_q$. 
We will argue that $q$ is strongly $N$-generic as follows. 
We will define a set $D_N$ which satisfies, among other things, 
that for all $r \in D_N$, 
if $M \in A_r$ and $M < N$, then $M \cap N \in A_r$. 
The set $D_N$ will be dense below $q$. 
For each $r \in D_N$, we will define a condition 
$r \restriction N$ in $N$ satisfying that 
for all $w \in N \cap \p$, if 
$w \le r \restriction N$, then $w$ and $r$ are compatible.

The arguments given in this section are very similar to those 
in the previous section. 
However, since $N \cap \kappa$ is a countable set, rather than an ordinal 
as in the uncountable case, the arguments given in this section 
are more complicated.

\bigskip

The first thing we will prove is that any condition $p \in N$ can be extended 
to a condition containing $N$.

\begin{lemma}
Let $p \in \p$, $N \in \mathcal X$, and suppose that $p \in N$. 
Then there is $q \le p$ such that $N \in A_q$.
\end{lemma}

\begin{proof}
Define $q$ as follows. 
Let $A_q := A_p \cup \{ N \}$. 
Define $f_q$ by letting 
$$
\dom(f_q) := \dom(f_p) \cup \{ N \cap \alpha : \alpha \in \dom(f_p) \cap S \}.
$$
For each $\alpha \in \dom(f_p) \cap S$, define 
$$
f_q(\alpha) := f_p(\alpha) \cup \{ N \cap \alpha \}.
$$
For each $M \in \dom(f_p) \setminus S$, define $f_q(M) := f_p(M)$. 
Finally, for each $\alpha \in \dom(f_p) \cap S$, define 
$$
f_q(N \cap \alpha) := f_p(\alpha).
$$

For $K \in f_p(x)$, define $g_q(K,x) := g_p(K,x)$. 
Let $\alpha \in \dom(f_p) \cap S$. 
For $K \in f_p(\alpha)$, let $g_q(K,N \cap \alpha) := g_p(K,\alpha)$ and 
$g_q(N \cap \alpha,\alpha) := \emptyset$.

\bigskip

It is easy to verify that if $q$ is a condition, then $q \le p$. 
Also, $N \in A_q$ by definition. 
It remains to prove that $q$ is a condition. 
We verify requirements (1)--(7) of Definition 4.2. 
For (1), $A_q$ is adequate by Lemma 2.16. 
(4), (5), and (6) are easy. 
It remains to prove (2), (3), and (7).

\bigskip

(2) Clearly $f_q$ is a function, and every member of $\dom(f_q)$ is 
of the required form. 
Let $x \in \dom(f_q)$, and we will show that $f_q(x)$ is a finite 
$\in$-chain and a subset of $Sk(x) \setminus S$. 
If $x \in \dom(f_p) \setminus S$, then $f_q(x) = f_p(x)$, 
so we are done since $p$ is a condition.

Suppose that $x = \alpha \in \dom(f_p) \cap S$. 
Then $f_q(\alpha) = f_p(\alpha) \cup \{ N \cap \alpha \}$. 
Since $p$ is a condition, $f_p(\alpha)$ is a finite $\in$-chain and a subset of 
$Sk(\alpha) \setminus S$. 
So it suffices to show that $f_p(\alpha) \subseteq Sk(N \cap \alpha)$ and 
$N \cap \alpha \in Sk(\alpha)$.

By Assumption 2.6, $N \cap \alpha \in Sk(\alpha)$. 
Let $K \in f_p(\alpha)$, and we will show that $K \in Sk(N \cap \alpha)$. 
By Definition 4.2(2), fix $K_1 \in A_p$ and 
$\beta \in (K_1 \cap \dom(f_p) \cap S) \cup \{ \kappa \}$ such that 
$K = K_1 \cap \beta$. 
Since $p \in N$, we have that $K_1$ and $\alpha$ are in $N$. 
Therefore, $K_1 \cap \alpha \in N$. 
By Lemma 2.8(1), $K_1 \cap \alpha \in Sk(N \cap \alpha)$. 

Note that $K = K_1 \cap \beta$ is an initial segment of $K_1 \cap \alpha$  
if $\beta \le \alpha$, and since $K \subseteq \alpha$, 
$K = K_1 \cap \alpha$ if $\alpha < \beta$. 
In either case, $K$ is an initial segment of $K_1 \cap \alpha$. 
Since $K_1 \cap \alpha$ is in $Sk(N \cap \alpha)$, by elementarity so 
is $K$.

Finally, assume that $x = N \cap \alpha$, 
where $\alpha \in \dom(f_p) \cap S$. 
Then $f_q(x) = f_p(\alpha)$. 
We just showed that $f_p(\alpha)$ is a subset of $Sk(N \cap \alpha)$, 
and since $p$ is a condition, it is a finite $\in$-chain disjoint from $S$.

\bigskip

(3) Let $x \in \dom(f_q)$. 
It is easy to check by cases that $f_q(x) \subseteq \dom(f_q)$. 
Let $K \in f_q(x)$, and we will show that $f_q(K) = f_q(x) \cap Sk(K)$.

First, assume that $x \in \dom(f_p) \setminus S$. 
Then $f_q(x) = f_p(x)$, 
so $K \in f_p(x) \subseteq \dom(f_p) \setminus S$. 
So by definition, $f_q(K) = f_p(K)$. 
Therefore, 
$$
f_q(K) = f_p(K) = f_p(x) \cap Sk(K) = f_q(x) \cap Sk(K).
$$

Secondly, assume that $x = \alpha \in \dom(f_p) \cap S$. 
Then $f_q(\alpha) = f_p(\alpha) \cup \{ N \cap \alpha \}$. 
If $K \in f_p(\alpha)$, then $K \in \dom(f_p) \setminus S$, so 
by definition, $f_q(K) = f_p(K)$. 
Thence,
$$
f_q(K) = f_p(K) = f_p(\alpha) \cap Sk(K) = f_q(\alpha) \cap Sk(K),
$$
where the last equality follows from the fact that 
$N \cap \alpha \notin Sk(K)$. 
If $K = N \cap \alpha$, 
then 
$$
f_q(K) = f_q(N \cap \alpha) = f_p(\alpha) = f_q(\alpha) \cap Sk(N \cap \alpha),
$$
where the last equality follows from the fact that 
$f_p(\alpha) \subseteq Sk(N \cap \alpha)$, as shown above. 

Thirdly, assume that $x = N \cap \alpha$. 
Then $f_q(N \cap \alpha) = f_p(\alpha)$. 
Hence, $K \in f_p(\alpha) \setminus S$. 
So by definition, $f_q(K) = f_p(K)$. 
Thus, 
$$
f_q(K) = f_p(K) = f_p(\alpha) \cap Sk(K) = f_q(N \cap \alpha) \cap Sk(K).
$$

\bigskip

(7) Note that $r^*(A_p \cup \{ N \}) = r^*(A_p)$. 
Namely, if $\gamma \in r^*(A_p \cup \{ N \})$, then 
$\gamma \in r^*(\{ K, M \})$ for some distinct $K$ and $M$ in 
$A_p \cup \{ N \}$ such that $K \sim M$. 
But for all $K \in A_p$, $K \not \sim N$. 
Hence, $K$ and $M$ are in $A_p$, and $\gamma \in r^*(A_p)$. 
It follows that 
$$
r^*(A_q) \cap S = r^*(A_p \cup \{ N \}) \cap S = 
r^*(A_p) \cap S \subseteq \dom(f_p) \subseteq \dom(f_q).
$$
\end{proof}

The next lemma will be used to show that the set $D_N$, which we will 
define shortly, is dense below 
any condition which contains $N$.

\begin{lemma}
Let $q \in \p$ and let $N \in A_q$. 
Then there is $s \le q$ such that for all $M \in A_s$, if 
$M < N$ then $M \cap N \in A_s$.
\end{lemma}

Recall that if $M < N$ are in $\mathcal X$, then $\{ M, N \}$ is adequate, and 
therefore $M \cap N \in \mathcal X$ by Assumption 2.19.

\begin{proof}
By Proposition 2.24, the set 
$A_q \cup \{ M \cap N : M \in A_q, \ M < N \}$ is adequate. 
Define 
$$
x_0 := r^*(A_q \cup \{ M \cap N : M \in A_q, \ M < N \}) \cap S,
$$
and define 
$$
x := x_0 \setminus \dom(f_q).
$$
Let $r := q + x$. 
By Lemma 4.8, $r$ is a condition and $r \le q$. 
By Definition 4.7, $A_r = A_q$, 
and easily, $x_0 \subseteq \dom(f_r)$.

Define $s$ as follows. 
Let $f_s := f_r$, $g_s := g_r$, and 
$$
A_s := A_r \cup \{ M \cap N : M \in A_r, \ M < N \}.
$$
We claim that $s$ is as required.

By Proposition 2.24, for all $M \in A_s$, if $M < N$ then 
$M \cap N \in A_s$. 
It is trivial to check that if $s$ is a condition, then $s \le r$, and 
therefore $s \le q$.

It remains to prove that $s$ is a condition. 
We verify requirements (1)--(7) of Definition 4.2. 
(1) follows from Proposition 2.24. 
(2)--(5) follow immediately from $r$ being a condition, 
together with the fact that 
$f_s = f_r$, $g_s = g_r$, and $A_r \subseteq A_s$.

(6) Suppose that $\alpha \in \dom(f_s) \cap S$, 
$M \in A_s$, and $\alpha \in M$. 
We will show that $M \cap \alpha \in f_s(\alpha)$. 
Since $f_s = f_r$, we have that $\alpha \in \dom(f_r) \cap S$. 
So if $M \in A_r$, then $M \cap \alpha \in f_r(\alpha) = f_s(\alpha)$, since 
$r$ is a condition.

Assume that $M \in A_s \setminus A_r$, which means that 
$M = M_1 \cap N$ for some $M_1 \in A_r$ with $M_1 < N$. 
Then $\alpha \in M \cap \kappa = M_1 \cap N \cap \kappa$. 
By Proposition 2.11, it follows that $\alpha < \beta_{M_1,N}$. 
Since $M_1 < N$, $M_1 \cap \beta_{M_1,N} \in N$, so 
$M_1 \cap \alpha \in N$. 
In particular, $M_1 \cap \alpha \subseteq N$. 
So 
$$
M \cap \alpha = M_1 \cap N \cap \alpha = M_1 \cap \alpha.
$$
But $M_1 \in A_r$ and $\alpha \in M_1$. 
Since $r$ is a condition, 
$$
M \cap \alpha = M_1 \cap \alpha \in f_r(\alpha) = f_s(\alpha).
$$

(7) We need to show that $r^*(A_s) \cap S \subseteq \dom(f_s)$. 
Since $f_s = f_r$, it suffices to show that 
$r^*(A_s) \cap S \subseteq \dom(f_r)$. 
But by the definition of $A_s$ and since $A_r = A_q$, we have that 
$$
r^*(A_s) \cap S = r^*(A_q \cup \{ M \cap N : M \in A_q, \ M < N \}) \cap S = x_0,
$$
and as noted above, $x_0 \subseteq \dom(f_r)$. 
So $r^*(A_s) \cap S \subseteq \dom(f_r)$.
\end{proof}

\begin{definition}
For any $N \in \mathcal X$, let $D_N$ denote the set of conditions 
$r \in \p$ satisfying:
\begin{enumerate}
\item $N \in A_r$;
\item for all $M \in A_r$, if $M < N$ then $M \cap N \in A_r$;
\item whenever $K \in f_r(x)$ and $x \in f_r(y)$, then 
$$
g_r(K,x) \subseteq g_r(K,y).
$$
\end{enumerate}
\end{definition}

Note that by Definition 4.2(5), 
the conclusion of (3) is equivalent to 
$g_r(K,x) = g_r(K,y)$.

The next lemma says that $D_N$ is dense below any condition which 
contains $N$.

\begin{lemma}
Let $N \in \mathcal X$. 
Then for any condition $q \in \p$, if $N \in A_q$, then there is 
$s \le q$ such that $s \in D_N$.
\end{lemma}

\begin{proof}
Let $q \in \p$ be such that $N \in A_q$. 
By Lemma 7.2, there is $r \le q$ such that for all $M \in A_r$, if 
$M < N$ then $M \cap N \in A_r$. 
By Lemma 4.9, there is $s \le r$ such that 
$f_s = f_r$, $A_s = A_r$, 
and whenever $K \in f_s(x)$ and $x \in f_s(y)$, 
then $g_s(K,x) \subseteq g_s(K,y)$. 
Then $s \le q$ and $s \in D_N$.
\end{proof}

\begin{definition}
Suppose that $N \in \mathcal X$ is simple and $r \in D_N$. 
Define $r \restriction N$ as the triple $(f,g,A)$ satisfying:
\begin{enumerate}
\item $\dom(f) = \dom(f_r) \cap N$, and for all 
$x \in \dom(f)$, $f(x) := f_r(x) \cap N$;
\item $\dom(g) = \dom(g_r) \cap N$, and for all 
$(K,x) \in \dom(g)$, $g(K,x) := g_r(K,x)$;
\item $A := A_r \cap N$.
\end{enumerate}
\end{definition}

Observe that in (1), if $x \in \dom(f) \setminus S$, then $x \in N$ 
implies that $Sk(x) \subseteq N$. 
Therefore, $f_r(x) \subseteq N$. 
So in this case, $f(x) = f_r(x) \cap N$ is equal to $f_r(x)$.

Let us prove that $r \restriction N$ is in $N$. 
Obviously $A = A_r \cap N$ and $\dom(f) = \dom(f_r) \cap N$ are in $N$, 
and for all $x \in \dom(f)$, $f(x) = f_r(x) \cap N$ is in $N$. 
Consider $K \in f(x)$. 
Then $K$ and $x$ are in $N$. 
If $x \notin S$, then $x \subseteq N$. 
Therefore, $g(K,x) = g_r(K,x) \subseteq x \subseteq N$. 
So $g(K,x)$ is a finite subset of $N$, and hence is in $N$.

Finally, suppose that $x = \alpha \in S$. 
Then $\alpha \in N$, and as $N \in A_r$, $N \cap \alpha$ is in 
$f_r(\alpha)$ by Definition 4.2(6). 
Also, 
$$
g_r(K,\alpha) \subseteq g_r(K,N \cap \alpha) \subseteq N \cap \alpha
$$
by Definition 4.2(4,5). 
Hence, $g(K,\alpha) = g_r(K,\alpha)$ is a finite subset of $N$, 
and hence is in $N$.

We have proven that all of the components of $r \restriction N$ are 
in $N$. 
Therefore, $r \restriction N$ is in $N$.

\begin{lemma}
Suppose that $N \in \mathcal X$ is simple and $r \in D_N$. 
Then $r \restriction N$ is in $N \cap \p$ and $r \le r \restriction N$.
\end{lemma}

\begin{proof}
Let $r \restriction N = (f,g,A)$. 
We have already observed that $r \restriction N \in N$.
It is trivial to check that 
if $r \restriction N$ is a condition, then $r \le r \restriction N$. 
So it suffices to show that $r \restriction N$ is a condition. 
We verify requirements (1)--(7) of Definition 4.2. 
(1), (4), (5), (6), and (7) are easy to check. 
It remains to prove (2) and (3).

(2) Obviously $f$ is a function with a finite domain. 
Let $x \in \dom(f) = \dom(f_r) \cap N$. 
Then $x \in N$. 
We will show that either $x \in S$, or there is $M \in A$ and 
$\alpha \in (M \cap \dom(f) \cap S) \cup \{ \kappa \}$ 
such that $x = M \cap \alpha$, and moreover, 
$f(x)$ is a finite $\in$-chain and $f(x) \subseteq Sk(x) \setminus S$.

We begin by showing that $f(x)$ is 
a finite $\in$-chain and $f(x) \subseteq Sk(x) \setminus S$.
Since $r$ is a condition, $f_r(x)$ is a finite $\in$-chain and 
$f_r(x) \subseteq Sk(x) \setminus S$. 
But $f(x) = f_r(x) \cap N \subseteq f_r(x)$. 
Therefore, $f(x)$ is a finite $\in$-chain 
and $f(x) \subseteq Sk(x) \setminus S$.

Now we show that 
either $x \in S$, or there is $M \in A$ and 
$\alpha \in (M \cap \dom(f) \cap S) \cup \{ \kappa \}$ 
such that $x = M \cap \alpha$. 
Since $x \in \dom(f_r)$ and $r$ is a condition, 
we have that either $x \in S$, 
or there is $M_1 \in A_r$ and 
$\alpha \in (M_1 \cap \dom(f_r) \cap S) \cup \{ \kappa \}$ 
such that $x = M_1 \cap \alpha$. 
In the first case, we are done, so assume the second case.

Since $x = M_1 \cap \alpha \in N$, clearly 
$M_1 \cap \omega_1 \in N$. 
By Lemma 2.17(1), it follows that $M_1 < N$. 
Therefore, as $r \in D_N$, we have that $M_1 \cap N \in A_r$. 
But $N$ is simple, so also $M_1 \cap N \in N$. 
So $M_1 \cap N \in A_r \cap N = A$. 
Hence, to complete the proof, it suffices to show that 
$x = M_1 \cap N \cap \beta$, for some 
$\beta \in (M_1 \cap N \cap \dom(f) \cap S) \cup \{ \kappa \}$.

If $\alpha = \kappa$, then $x = M_1 \cap \kappa \in N$, so 
$M_1 \cap \kappa \subseteq N$. 
Hence 
$$
x = M_1 \cap \kappa = M_1 \cap N \cap \kappa,
$$
and we are done.

Suppose that $\alpha < \kappa$. 
First, assume that $\alpha \in N$. 
Then $\alpha \in \dom(f_r) \cap S \cap N = \dom(f) \cap S$. 
Also, $\alpha \in M_1 \cap N \cap \kappa$, which implies that 
$\alpha < \beta_{M_1,N}$ by Proposition 2.11. 
Since $M_1 < N$, it follows that $M_1 \cap \alpha \subseteq N$. 
Hence, 
$$
x = M_1 \cap \alpha = M_1 \cap N \cap \alpha.
$$
As $\alpha \in M_1 \cap N \cap \dom(f) \cap S$, we are done.

Secondly, assume that $\alpha \notin N$. 
We claim that 
$$
M_1 \cap \alpha = M_1 \cap N \cap \kappa,
$$
which will finish the proof. 
Since $x = M_1 \cap \alpha$ is in $N$, the forward inclusion is immediate. 
For the reverse inclusion, 
let $\gamma \in M_1 \cap N \cap \kappa$, and we 
will show that $\gamma \in M_1 \cap \alpha$. 
Then $\gamma < \beta_{M_1,N}$ by Proposition 2.11, 
and therefore, since $M_1 < N$, we have that 
$$
M_1 \cap \gamma \subseteq M_1 \cap \beta_{M_1,N} \subseteq N.
$$
Since $\alpha \notin N$, it follows that $\alpha \notin M_1 \cap \gamma$. 
But $\alpha \in M_1$. 
Therefore, $\gamma \le \alpha$. 
Also, $\gamma \in N$ and $\alpha \notin N$ implies that 
$\gamma \ne \alpha$, so $\gamma < \alpha$. 
Thus, $\gamma \in M_1 \cap \alpha$, completing the proof.

(3) Let $x \in \dom(f) = \dom(f_r) \cap N$. 
Since $r$ is a condition, we have that 
$$
f(x) = f_r(x) \cap N \subseteq \dom(f_r) \cap N = \dom(f).
$$
Thus, $f(x) \subseteq \dom(f)$.

Let $K \in f(x) = f_r(x) \cap N$, and we will show that 
$f(K) = f(x) \cap Sk(K)$. 
Since $K \in f_r(x)$, $K \notin S$, 
so $Sk(K) \subseteq N$. 
Therefore, $f_r(K) \subseteq Sk(K) \subseteq N$. 
So 
$$
f(K) = f_r(K) \cap N = f_r(K).
$$
As $r$ is a condition, we have that 
$f(K) = f_r(K) = f_r(x) \cap Sk(K)$. 
But $Sk(K) \subseteq N$ implies that 
$$
f_r(x) \cap Sk(K) = f_r(x) \cap Sk(K) \cap N = f(x) \cap Sk(K).
$$
Thus, $f(K) = f(x) \cap Sk(K)$.
\end{proof}

The next lemma will not be used until Section 8.

\begin{lemma}
Suppose that $N \in \mathcal X$ is simple and $r \in D_N$. 
Assume that $p \in N \cap \p$ and $r \le p$. 
Then $r \restriction N \le p$.
\end{lemma}

\begin{proof}
We verify properties (a)--(d) of Definition 4.2. 
(a) Since $p \in N$, $A_p \subseteq N$. 
As $r \le p$, we have that 
$$
A_p \subseteq A_r \cap N = A_{r \restriction N}.
$$
(b) Since $p \in N$, $\dom(f_p) \subseteq N$ and for all $x \in \dom(f_p)$, 
$f_p(x) \subseteq N$. 
As $r \le p$, we have that 
$$
\dom(f_p) \subseteq \dom(f_r) \cap N = \dom(f_{r \restriction N}).
$$
Let $x \in \dom(f_p)$. 
Then 
$$
f_p(x) \subseteq f_r(x) \cap N = f_{r \restriction N}(x).
$$
(c) Let $(K,x) \in \dom(g_p)$. 
Then since $r \le p$, 
$$
g_p(K,x) \subseteq g_r(K,x) = g_{r \restriction N}(K,x).
$$
(d) Assume that $K$ and $x$ are in $\dom(f_p)$ 
and $K \in f_{r \restriction N}(x)$. 
We claim that $K \in f_p(x)$. 
But 
$$
K \in f_{r \restriction N}(x) = f_r(x) \cap N \subseteq f_r(x).
$$
So $K \in f_r(x)$. 
Since $r \le p$, it follows that $K \in f_p(x)$.
\end{proof}

We will now begin analyzing the situation where 
$r \in D_N$ and $w \le r \restriction N$ is in $N \cap \p$.

\begin{lemma}
Let $N \in \mathcal X$ be simple and $r \in D_N$. 
Suppose that $w \in N \cap \p$ and $w \le r \restriction N$. 
Then:
\begin{enumerate}
\item $A_r \cap N \subseteq A_w$;
\item $\dom(f_r) \cap N \subseteq \dom(f_w)$, and for all 
$x \in \dom(f_r) \cap N$, $f_r(x) \cap N \subseteq f_w(x)$;
\item $\dom(g_r) \cap N \subseteq \dom(g_w)$, and for all 
$(K,x) \in \dom(g_r) \cap N$, $g_r(K,x) \subseteq g_w(K,x)$.
\end{enumerate}
\end{lemma}

\begin{proof}
Immediate from the definition of $r \restriction N$ and the fact that 
$w \le r \restriction N$.
\end{proof}

As we discussed at the beginning of the section, we are going to show 
that whenever $w \le r \restriction N$, where $r \in D_N$ and 
$w \in N \cap \p$, then $w$ and $r$ are compatible. 
As in the previous section, we will construct a specific lower bound 
$w \oplus_N r$ of $w$ and $r$. 
We will describe separately the $f$, $g$, and $A$ components 
of $w \oplus_N r$.
The $A$-component of $w \oplus_N r$ will be defined as 
$A_w \cup A_r$. 

We handle the $f$-component next.  
Unfortunately, the definition of $f_w \oplus_N f_r$ is much more 
complicated than in the previous section. 
The domain of $f_w \oplus_N f_r$ will include not only 
$\dom(f_w) \cup \dom(f_r)$, but also some additional sets.

Before defining $f_w \oplus_N f_r$, 
we prove two lemmas which will help us handle its domain.

\begin{lemma}
Let $N \in \mathcal X$ be simple and $r \in D_N$. 
Suppose that $w \in N \cap \p$ and $w \le r \restriction N$. 
Then the set 
$$
\{ M \cap \alpha : M \in A_r, \ N \le M, \ 
\alpha \in (M \cap \dom(f_w) \cap S) \setminus \dom(f_r) \}
$$
is disjoint from $N$, $\dom(f_w)$, and $\dom(f_r)$. 
\end{lemma}

\begin{proof}
It is clear that any member $M \cap \alpha$ of the displayed set is 
not in $N$, since $N \le M$ implies that 
$N \cap \omega_1 \le M \cap \omega_1 = M \cap \alpha \cap \omega_1$. 
Since $w \in N$, and hence $\dom(f_w) \subseteq N$, 
it follows that the displayed set is disjoint from $\dom(f_w)$.

Suppose for a contradiction that for some 
$M \in A_r$ with $N \le M$ and some 
$\alpha \in (M \cap \dom(f_w) \cap S) \setminus \dom(f_r)$, 
$M \cap \alpha$ is in $\dom(f_r)$. 
By Definition 4.2(2), fix $M_1 \in A_r$ and 
$\beta \in (M_1 \cap \dom(f_r) \cap S) \cup \{ \kappa \}$ 
such that $M \cap \alpha = M_1 \cap \beta$. 
Since $\alpha \in \kappa \setminus \dom(f_r)$ and 
$\beta \in \dom(f_r) \cup \{ \kappa \}$, $\alpha \ne \beta$. 

Applying Lemma 3.2 to $M$, $M_1$, $\alpha$, and $\beta$, 
we get that $M \sim M_1$ and 
$\alpha = \min((M \cap \kappa) \setminus \beta_{M,M_1})$. 
Since $M$ and $M_1$ are in $A_r$ and $\alpha \in S$, we have that 
$\alpha \in r^*(A_r) \cap S$. 
By Definition 4.2(7), it follows that $\alpha \in \dom(f_r)$. 
But this contradicts the choice of $\alpha$.
\end{proof}

\begin{lemma}
Suppose that $r \in \p$, $N \in A_r$, and $x \in \dom(f_r)$. 
Then there is at most one ordinal $\alpha$ such that 
$\alpha \in \dom(f_r) \cap S \cap N$ and 
$N \cap \alpha \in f_r(x) \cup \{ x \}$. 
\end{lemma}

\begin{proof}
Suppose for a contradiction that $\alpha < \beta$ are in 
$\dom(f_r) \cap S \cap N$, and $N \cap \alpha$ and 
$N \cap \beta$ are both in $f_r(x) \cup \{ x \}$. 
Then $N \cap \alpha$ and $N \cap \beta$ are membership comparable. 
Since $N \cap \alpha$ and $N \cap \beta$ 
have the same intersection with $\omega_1$, 
they must be equal. 
But then $\alpha \in N \cap \beta = N \cap \alpha$, which is impossible.
\end{proof}

We are ready to define $f_w \oplus_N f_r$.

\begin{definition}
Let $N \in \mathcal X$ be simple and $r \in D_N$. 
Suppose that $w \in N \cap \p$ and $w \le r \restriction N$. 
Define $f_w \oplus_N f_r = f$ as follows.

The domain of $f$ is equal to the union of 
$\dom(f_w)$, $\dom(f_r)$, and the set
$$
\{ M \cap \alpha : M \in A_r, \ N \le M, \ 
\alpha \in (M \cap \dom(f_w) \cap S) \setminus \dom(f_r) \}.
$$
The values of $f$ are defined by the following cases:
\begin{enumerate}
\item for all $x \in \dom(f_w) \setminus S$, 
$$
f(x) := f_w(x);
$$

\item for all $\alpha \in \dom(f_w) \cap S \cap \dom(f_r)$, 
$$
f(\alpha) := f_w(\alpha) \cup f_r(\alpha);
$$

\item for all $\alpha \in (\dom(f_w) \cap S) \setminus \dom(f_r)$, 
$$
f(\alpha) := f_w(\alpha) \cup 
\{ M \cap \alpha : M \in A_r, \ N \le M, \ \alpha \in M \};
$$

\item if $x \in \dom(f_r) \setminus N$, 
and for some $\alpha \in \dom(f_r) \cap S \cap N$, 
$N \cap \alpha \in f_r(x) \cup \{ x \}$, then 
$$
f(x) := f_w(\alpha) \cup f_r(x);
$$

\item for all $x \in \dom(f_r) \setminus N$ such that (4) fails, 
if $f_r(x) \cap N = \emptyset$, then 
$$
f(x) := f_r(x);
$$

\item for all $x \in \dom(f_r) \setminus N$ such that (4) fails, 
if $f_r(x) \cap N \ne \emptyset$, then 
$$
f(x) := f_r(x) \cup f_w(M),
$$
where $M$ is the membership largest element of 
$f_r(x) \cap N$;

\item for a set of the form $M \cap \alpha$, where $M \in A_r$, $N \le M$, 
and $\alpha \in (M \cap \dom(f_w) \cap S) \setminus \dom(f_r)$, 
$$
f(M \cap \alpha) := f(\alpha) \cap Sk(M \cap \alpha),
$$
where $f(\alpha)$ was defined in (3).
\end{enumerate}
\end{definition}

It is easy to see that cases 1--7 describe all of the possibilities for a set 
being in $\dom(f)$, using the fact that 
$\dom(f_r) \cap N \subseteq \dom(f_w)$ by Lemma 7.8(2). 
Moreover, cases 1--6 are obviously disjoint, and they are also disjoint 
from case 7 by Lemma 7.9. 
Finally, note that the ordinal $\alpha$ in case 4 is unique by Lemma 7.10, 
so $f(x)$ is well-defined in this case.

The next four lemmas describe some important 
properties of $f_w \oplus_N f_r$. 
Lemmas 7.12 and 7.13 will be used to show that 
$w \oplus_N r$ is below $w$ and $r$ in $\p$.

\begin{lemma}
Let $N \in \mathcal X$ be simple and $r \in D_N$. 
Suppose that $w \in N \cap \p$ and $w \le r \restriction N$. 
Let $f := f_w \oplus_N f_r$. Then: 
\begin{enumerate}
\item $\dom(f_w) \cup \dom(f_r) \subseteq \dom(f)$;
\item if $x \in \dom(f_w)$, then $f_w(x) \subseteq f(x)$;
\item if $x \in \dom(f_r)$, then $f_r(x) \subseteq f(x)$.
\end{enumerate}
\end{lemma}

\begin{proof}
(1,2) By the definition of $f$ in Definition 7.11, 
$\dom(f_w)$ and $\dom(f_r)$ are subsets of $\dom(f)$. 
By Cases 1, 2, and 3 of Definition 7.11, 
for all $x \in \dom(f_w)$, $f_w(x) \subseteq f(x)$. 

(3) Suppose that $x \in \dom(f_r)$, and we will show 
that $f_r(x) \subseteq f(x)$. 
We consider each of the cases 1--7 of Definition 7.11 
in the definition of $f(x)$.

If $f(x)$ is defined by cases 2, 4, 5, or 6, then 
$f_r(x) \subseteq f(x)$ by definition. 
Since $x \in \dom(f_r)$, case 3 does not hold, and case 7 does not 
hold by Lemma 7.9. 
It remains to consider case 1.

For case 1, suppose that $x \in \dom(f_w) \setminus S$ and 
$f(x) = f_w(x)$. 
Then $x \in N$. 
So $f_r(x) \subseteq Sk(x) \subseteq N$. 
Hence, by Lemma 7.8(2), 
$f_r(x) = f_r(x) \cap N \subseteq f_w(x) = f(x)$.
\end{proof}

\begin{lemma}
Let $N \in \mathcal X$ be simple and $r \in D_N$. 
Suppose that $w \in N \cap \p$ and $w \le r \restriction N$. 
Let $f := f_w \oplus_N f_r$. Then: 
\begin{enumerate}
\item $\dom(f) \cap N = \dom(f_w)$;
\item if $K \in f(x)$ and $K$ and $x$ are in $\dom(f_w)$, then 
$K \in f_w(x)$;
\item if $K \in f(x)$ and $K$ and $x$ are in $\dom(f_r)$, 
then $K \in f_r(x)$.
\end{enumerate}
\end{lemma}

\begin{proof}
(1) The inclusion 
$\dom(f_w) \subseteq \dom(f) \cap N$ is immediate, 
so it suffices 
to show that if $x \in \dom(f) \cap N$, then $x \in \dom(f_w)$. 
Note that $x$ is not equal to 
$M \cap \alpha$, for any $M \in A_r$ with $N \le M$ and 
$\alpha \in (M \cap \dom(f_w) \cap S) \setminus \dom(f_r)$, 
since such a set is not in $N$ by Lemma 7.9. 
So by the definition of $\dom(f)$, we have that either $x \in \dom(f_w)$, 
in which case we are done, or $x \in \dom(f_r)$. 
In the second case, $x \in \dom(f_r) \cap N \subseteq \dom(f_w)$ 
by Lemma 7.8(2).

\bigskip

(2) Suppose that $K \in f(x)$ and $K$ and $x$ are in $\dom(f_w)$. 
We will show that $K \in f_w(x)$. 
Since $K$ and $x$ are in $\dom(f_w)$, they are in $N$. 
Hence, we are in cases 1, 2, or 3 of Definition 7.11. 
In case 1, $f(x) = f_w(x)$, so $K \in f_w(x)$. 
In case 2, $f(x) = f_w(\alpha) \cup f_r(\alpha)$, where 
$x = \alpha \in \dom(f_w) \cap S \cap \dom(f_r)$. 
So either $K \in f_w(\alpha)$, in which case we are done, or 
$K \in f_r(\alpha)$. 
In the second case, by Lemma 7.8(2) we have that 
$$
K \in f_r(\alpha) \cap N \subseteq f_w(\alpha) = f_w(x).
$$
In case 3, $f(x)$ is equal to 
the union of $f_w(x)$ together with a collection of sets 
which are not members of $N$. 
Since $K \in N$, $K \in f_w(x)$.

\bigskip

(3) Suppose that $K \in f(x)$ and $K$ and $x$ are in $\dom(f_r)$. 
We will show that $K \in f_r(x)$.
First, assume that $K$ and $x$ are both in $N$. 
Then $K$ and $x$ are in 
$\dom(f_r) \cap N \subseteq \dom(f_w)$ by Lemma 7.8(2). 
By (2) just proven, $K \in f_w(x)$. 
Yet $K$ and $x$ are in $\dom(f_r) \cap N = \dom(f_{r \restriction N})$. 
Since $w \le r \restriction N$, it follows that 
$K \in f_{r \restriction N}(x) \subseteq f_r(x)$. 
So $K \in f_r(x)$.

Next, we consider each of the cases 1--7 of Definition 7.11 
for $f(x)$. 
Case 1 is immediate, since it implies that $K$ and $x$ are in $N$. 
In case 2, $x = \alpha$ is in $N$ 
and $f(x) = f_w(\alpha) \cup f_r(\alpha)$. 
If $K \in f_w(\alpha)$, then $K$ and $x$ are both in $N$, and we are done. 
Otherwise $K \in f_r(\alpha)$, and we are also done.

Case 3 does not apply, since it says that $x = \alpha$ is not in $\dom(f_r)$. 
In case 5, $f(x) = f_r(x)$, so $K \in f_r(x)$, and we are done.
Case 7 does not apply, since any set of the form described there is 
not in $\dom(f_r)$ by Lemma 7.9. 
It remains to consider cases 4 and 6.

In case 4, there is $\alpha \in \dom(f_r) \cap S \cap N$ such that 
$N \cap \alpha \in f_r(x) \cup \{ x \}$, and 
$f(x) = f_w(\alpha) \cup f_r(x)$. 
So either $K \in f_w(\alpha)$, or $K \in f_r(x)$. 
In the second case we are done, so assume that $K \in f_w(\alpha)$. 
By Lemma 7.12(2), $f_w(\alpha) \subseteq f(\alpha)$, so $K \in f(\alpha)$. 
Since $K$ and $\alpha$ are in $\dom(f_r)$, it follows 
by case 2 just handled that 
$K \in f_r(\alpha)$. 
But also $K \in f_w(\alpha)$ means that $K \in N$. 
Since $N \cap \alpha \in f_r(\alpha)$ by Definition 4.2(6) and $K \in N$, 
we must have that $K \in f_r(N \cap \alpha)$. 
But $N \cap \alpha \in f_r(x) \cup \{ x \}$, so $K \in f_r(x)$.

In case 6, $f(x) = f_r(x) \cup f_w(M)$, 
where $M$ is the membership largest element of $f_r(x) \cap N$. 
If $K \in f_r(x)$ then we are done, so assume that $K \in f_w(M)$. 
Since $M$ is in $\dom(f_r) \cap N \subseteq \dom(f_w)$ by 
Lemma 7.8(2), $M$ is in $\dom(f_w) \setminus S$. 
So $f(M)$ is defined as in case 1 of Definition 7.11. 
Hence, $f(M) = f_w(M)$. 
Therefore, $K \in f(M)$. 
Since $K$ and $M$ are in $\dom(f_r)$, it follows 
by case 1 handled above that $K \in f_r(M)$. 
Thus, $K \in f_r(M)$ and $M \in f_r(x)$, so 
$K \in f_r(x)$.
\end{proof}

The next lemma will be used to show that $f_w \oplus_N f_r$ satisfies 
requirement (2) of Definition 4.2 for $w \oplus_N r$.

\begin{lemma}
Let $N \in \mathcal X$ be simple and $r \in D_N$. 
Suppose that $w \in N \cap \p$ and $w \le r \restriction N$. 
Let $f := f_w \oplus_N f_r$. Then 
$f$ is a function with a finite domain, and for all $x \in \dom(f)$, 
either $x \in S$, or there is $M \in A_w \cup A_r$ and 
$$
\alpha \in (M \cap \dom(f) \cap S) \cup \{ \kappa \}
$$
satisfying that 
$x = M \cap \alpha$; moreover, for all $x \in \dom(f)$, 
$f(x)$ is a finite $\in$-chain and $f(x) \subseteq Sk(x) \setminus S$.
\end{lemma}

\begin{proof}
It is immediate that $f$ is a function with a finite domain. 
By the definition of the domain of $f$ in Definition 7.11, together with 
the fact that $\dom(f_w) \cup \dom(f_r) \subseteq \dom(f)$ and 
$w$ and $r$ are conditions, it is easy to see that if $x \in \dom(f)$, 
then either $x \in S$, or there is $M \in A_w \cup A_r$ and 
$\alpha \in (M \cap \dom(f) \cap S) \cup \{ \kappa \}$ satisfying that 
$x = M \cap \alpha$.

Let $x \in \dom(f)$, and we will show that $f(x)$ is a finite 
$\in$-chain and $f(x) \subseteq Sk(x) \setminus S$. 
We will consider each of the cases 1--7 of Definition 7.11. 
We are done if either $f(x) = f_w(x)$ or $f(x) = f_r(x)$ as in 
cases 1 and 5, since 
$w$ and $r$ are conditions. 
In case 7, $f(x) \subseteq Sk(x)$ by definition, and 
$f(x)$ is a finite $\in$-chain and disjoint from $S$ provided that 
the result is true for case 3. 
It remains to handle cases 2, 3, 4, and 6.

\bigskip

\noindent \emph{Case 2:} $x = \alpha \in \dom(f_w) \cap S \cap \dom(f_r)$ 
and $f(\alpha) = f_w(\alpha) \cup f_r(\alpha)$. 
Since $w$ and $r$ are conditions, $f_w(\alpha)$ and $f_r(\alpha)$ are 
themselves finite $\in$-chains and subsets of $Sk(\alpha) \setminus S$. 
So it suffices to show that if $K \in f_w(\alpha)$ and 
$M \in f_r(\alpha) \setminus f_w(\alpha)$, then 
$K \in Sk(M)$.

Since $\alpha \in \dom(f_w)$, $\alpha$ is in $N$. 
Thus, $\alpha \in \dom(f_r) \cap S \cap N$. 
Since $N \in A_r$, we have that $N \cap \alpha \in f_r(\alpha)$, 
because $r$ is a condition. 
As $M$ is also in $f_r(\alpha)$, $M$ and $N \cap \alpha$ 
are membership comparable. 
We claim that $M$ is not in $Sk(N \cap \alpha)$. 
Otherwise, $M \in N$, so 
$M \in f_r(\alpha) \cap N \subseteq f_w(\alpha)$ by Lemma 7.8(2). 
Hence, $M \in f_w(\alpha)$, which contradicts the choice of $M$.

Thus, either $M = N \cap \alpha$, or $N \cap \alpha \in Sk(M)$. 
In either case, $Sk(N \cap \alpha) \subseteq Sk(M)$. 
Since $K \in f_w(\alpha)$, $K \subseteq \alpha$. 
Therefore, $K \in N$ implies that $K \in Sk(N \cap \alpha)$ 
by Lemma 4.3(3). 
Hence, $K \in Sk(N \cap \alpha) \subseteq Sk(M)$, 
so $K \in Sk(M)$.

\bigskip

\noindent \emph{Case 3:} $x = \alpha \in (\dom(f_w) \cap S) \setminus \dom(f_r)$ and 
$$
f(\alpha) = f_w(\alpha) \cup 
\{ M \cap \alpha : M \in A_r, \ N \le M, \ \alpha \in M \}.
$$
Since $w$ is a condition, $f_w(\alpha)$ is a finite $\in$-chain and a subset 
of $Sk(\alpha) \setminus S$. 
By Lemma 2.33, the second set in the above union is also a finite 
$\in$-chain and a subset of $Sk(\alpha)$, and it is obviously 
disjoint from $S$. 
So it suffices to show that if $K \in f_w(\alpha)$, $M \in A_r$, $N \le M$, 
and $\alpha \in M$, then $K \in Sk(M \cap \alpha)$.

Since $K \in f_w(\alpha)$, we have that $K \in N$ and $K \subseteq \alpha$. 
By Lemma 4.3(3), $K \in Sk(N \cap \alpha)$. 
Since $\alpha \in M \cap N$, $\alpha < \beta_{M,N}$ by Proposition 2.11. 
As $N \le M$, we have that $N \cap \alpha \subseteq M \cap \alpha$. 
Thus, $Sk(N \cap \alpha) \subseteq Sk(M \cap \alpha)$. 
So $K \in Sk(M \cap \alpha)$.

\bigskip

\noindent \emph{Case 4:} $x \in \dom(f_r) \setminus N$, 
$\alpha \in \dom(f_r) \cap S \cap N$, 
and $N \cap \alpha \in f_r(x) \cup \{ x \}$. 
Then $f(x) = f_w(\alpha) \cup f_r(x)$. 
Since $w$ and $r$ are conditions, $f_w(\alpha)$ and $f_r(x)$ are themselves 
finite $\in$-chains, and subsets of $Sk(\alpha) \setminus S$ and 
$Sk(x) \setminus S$ respectively.
Consider $K \in f_w(\alpha)$, and we will show that $K \in Sk(x)$. 
So $K \in N$ and $K \subseteq \alpha$. 
By Lemma 4.3(3), $K \in Sk(N \cap \alpha)$. 
But $N \cap \alpha \in f_r(x) \cup \{ x \}$, so in particular, 
$Sk(N \cap \alpha) \subseteq Sk(x)$. 
Thus, $K \in Sk(x)$.
This completes the proof that $f(x) \subseteq Sk(x) \setminus S$.

Since $f_w(\alpha)$ and $f_r(x)$ are themselves finite $\in$-chains, 
in order to show that $f(x)$ is a finite $\in$-chain, it suffices to show that if 
$K \in f_w(\alpha)$ and $L \in f_r(x) \setminus f_w(\alpha)$, then 
$K \in Sk(L)$.

Since $K \in N$ and $K \subseteq \alpha$, $K \in Sk(N \cap \alpha)$ 
by Lemma 4.3(3). 
As $L \in f_r(x)$ and $N \cap \alpha \in f_r(x) \cup \{ x \}$, 
$L$ and $N \cap \alpha$ are membership comparable. 
We claim that $L$ is not in $Sk(N \cap \alpha)$. 
Suppose for a contradiction that $L \in Sk(N \cap \alpha)$. 
Then $L \in N$, and $L \in f_r(N \cap \alpha) \subseteq f_r(\alpha)$. 
So $L \in f_r(\alpha) \cap N \subseteq f_w(\alpha)$ by Lemma 7.8(2). 
Hence, $L \in f_w(\alpha)$, which contradicts the choice of $L$. 
Thence, either $L = N \cap \alpha$ or $N \cap \alpha \in Sk(L)$. 
In either case, $Sk(N \cap \alpha) \subseteq Sk(L)$, and in particular, 
$K \in Sk(L)$.

\bigskip

\noindent \emph{Case 6:} $x \in \dom(f_r) \setminus N$ and 
$f(x) = f_r(x) \cup f_w(M)$, where $M$ is the membership largest 
element of $f_r(x) \cap N$. 
Since $r$ and $w$ are conditions, $f_r(x)$ is a finite $\in$-chain and a subset 
of $Sk(x) \setminus S$, 
and $f_w(M)$ is a finite $\in$-chain and a subset of $Sk(M) \setminus S$. 
Since $M \in f_r(x)$, $M \in Sk(x)$. 
Therefore, $f_w(M) \subseteq Sk(M) \subseteq Sk(x)$. 
Hence, both $f_r(x)$ and $f_w(M)$ are finite $\in$-chains and subsets of 
$Sk(x) \setminus S$. 
So it suffices to show that 
if $K \in f_w(M)$ and $L \in f_r(x) \setminus f_w(M)$, then $K \in Sk(L)$.

Since $L$ and $M$ are both in $f_r(x)$, they are membership comparable. 
But if $L \in Sk(M)$, then by Lemma 7.8(2),  
$$
L \in (f_r(x) \cap Sk(M)) = f_r(M) \subseteq f_w(M).
$$
So $L \in f_w(M)$, which contradicts the choice of $L$. 
Therefore, either $L = M$, or $M \in Sk(L)$. 
In either case, $Sk(M) \subseteq Sk(L)$. 
As $K \in f_w(M)$, $K \in Sk(M)$, so 
$K \in Sk(M) \subseteq Sk(L)$.
\end{proof}

The next lemma will be used to show that $f_w \oplus_N f_r$ satisfies 
requirement (3) of Definition 4.2 for $w \oplus_N r$.

\begin{lemma}
Let $N \in \mathcal X$ be simple and $r \in D_N$. 
Suppose that $w \in N \cap \p$ and $w \le r \restriction N$. 
Let $f := f_w \oplus_N f_r$. 
If $x \in \dom(f)$, then $f(x) \subseteq \dom(f)$, and 
for all $K \in f(x)$, $f(K) = f(x) \cap Sk(K)$.
\end{lemma}

\begin{proof}
We begin by proving that $f(x) \subseteq \dom(f)$. 
This is immediate in cases 1, 2, 4, 5, and 6 of Definition 7.11, since 
the fact that $w$ and $r$ are conditions implies that 
for any $y \in \dom(f_w)$ and $z \in \dom(f_r)$, 
$f_w(y) \subseteq \dom(f_w) \subseteq \dom(f)$ and 
$f_r(z) \subseteq \dom(f_r) \subseteq \dom(f)$. 
Also, case 7 follows from case 3. 
It remains to consider case 3.

In case 3, 
$\alpha$ is in $(\dom(f_w) \cap S) \setminus \dom(f_r)$ 
and $f(\alpha) = f_w(\alpha) \cup \{ M \cap \alpha : M \in A_r, \ N \le M, \ 
\alpha \in M \}$. 
Again, we know that $f_w(\alpha) \subseteq \dom(f_w) \subseteq \dom(f)$. 
And if $M \in A_r$, $N \le M$, and $\alpha \in M$, then 
$M \cap \alpha$ is in $\dom(f)$ by Definition 7.11.

\bigskip

Let $K \in f(x)$, and we will prove that $f(K) = f(x) \cap Sk(K)$. 
The proof splits into the seven cases of Definition 7.11 
in the definition of $f(x)$. 
Note that since $K \in f(x)$, 
it follows that $K \notin S$ by Lemma 7.14.

\bigskip

\noindent \emph{Case 1:} $x \in \dom(f_w) \setminus S$ and 
$f(x) = f_w(x)$. 
Then $K \in f_w(x)$. 
So $K \in \dom(f_w) \setminus S$ as well. 
Hence, $f(K) = f_w(K)$. 
Therefore, 
$$
f(K) = f_w(K) = f_w(x) \cap Sk(K) = f(x) \cap Sk(K).
$$

\bigskip

\noindent \emph{Case 2:} $x = \alpha \in \dom(f_w) \cap S \cap \dom(f_r)$ 
and $f(\alpha) = f_w(\alpha) \cup f_r(\alpha)$. 
So either $K \in f_w(\alpha)$ or $K \in f_r(\alpha)$.

First, assume that $K \notin N$. 
Then $K \in f_r(\alpha) \setminus N$. 
By Definition 4.2(6), $N \cap \alpha \in f_r(\alpha)$. 
So both $K$ and $N \cap \alpha$ are in $f_r(\alpha)$. 
As $K \notin N$, $N \cap \alpha \in f_r(K) \cup \{ K \}$. 
So we are in case 4 of Definition 7.11 in the definition of $f(K)$. 
Therefore, $f(K) = f_w(\alpha) \cup f_r(K)$. 
So it suffices to show that 
$$
f_w(\alpha) \cup f_r(K) = 
(f_w(\alpha) \cup f_r(\alpha)) \cap Sk(K).
$$

Since $w \in N$, $f_w(\alpha) \subseteq N \cap Sk(\alpha) = Sk(N \cap \alpha)$ 
by Lemma 2.7(1). 
Since $N \cap \alpha \in f_r(K) \cup \{ K \}$, 
$N \cap \alpha \subseteq K$, so $Sk(N \cap \alpha) \subseteq Sk(K)$. 
Hence, $f_w(\alpha) \subseteq Sk(K)$. 
Also, as $K \in f_r(\alpha)$, 
$f_r(K) = f_r(\alpha) \cap Sk(K)$. 
Thence, 
\begin{multline*}
$$
(f_w(\alpha) \cup f_r(\alpha)) \cap Sk(K) = \\
= (f_w(\alpha) \cap Sk(K)) \cup (f_r(\alpha) \cap Sk(K)) = 
f_w(\alpha) \cup f_r(K).
\end{multline*}

Secondly, assume that $K \in N$. 
Then either $K \in f_w(\alpha)$, or 
$K \in f_r(\alpha) \cap N \subseteq f_w(\alpha)$ by Lemma 7.8(2). 
In either case, $K \in f_w(\alpha)$. 
So $K \in \dom(f_w) \setminus S$, and therefore, by definition, 
$f(K) = f_w(K)$. 
So 
$$
f(K) = f_w(K) = f_w(\alpha) \cap Sk(K).
$$
Since $K$ and $w$ are $N$, and 
$f_r(\alpha) \cap N \subseteq f_w(\alpha)$ by Lemma 7.8(2), 
we have that 
\begin{multline*}
f(\alpha) \cap Sk(K) \subseteq f(\alpha) \cap N = \\
= (f_w(\alpha) \cup f_r(\alpha)) \cap N \subseteq 
f_w(\alpha) \cup (f_r(\alpha) \cap N) \subseteq f_w(\alpha).
\end{multline*}
So $f(\alpha) \cap Sk(K) \subseteq f_w(\alpha)$. 
On the other hand, by Lemma 7.12(2), $f_w(\alpha) \subseteq f(\alpha)$. 
Therefore,
$$
f(\alpha) \cap Sk(K) = f_w(\alpha) \cap Sk(K) = f(K).
$$

\bigskip

\noindent \emph{Case 3:} $x = \alpha \in (\dom(f_w) \cap S) \setminus \dom(f_r)$ and 
$$
f(\alpha) = f_w(\alpha) \cup 
\{ M \cap \alpha : M \in A_r, \ N \le M, \ \alpha \in M \}.
$$
First, assume that $K \in f_w(\alpha)$. 
Then $f(K)$ is defined as in case 1 of Definition 7.11, so $f(K) = f_w(K)$. 
Since $K \in N$, 
and the members of the second set in the displayed union are not in $N$ 
by Lemma 7.9, and hence not in $Sk(K)$, we have that 
$$
f(\alpha) \cap Sk(K) = f_w(\alpha) \cap Sk(K) = f_w(K) = f(K).
$$

Secondly, assume that $K = M \cap \alpha$, 
where $M \in A_r$, $N \le M$, 
and $\alpha \in M$. 
Then $f(K)$ is defined as in case 7 of Definition 7.11, namely, 
$$
f(K) = f(M \cap \alpha) = f(\alpha) \cap Sk(M \cap \alpha) = 
f(\alpha) \cap Sk(K).
$$

\bigskip

\noindent \emph{Case 4:} $x \in \dom(f_r) \setminus N$, 
$\alpha \in \dom(f_r) \cap S \cap N$, and 
$N \cap \alpha \in f_r(x) \cup \{ x \}$. 
Then $f(x) = f_w(\alpha) \cup f_r(x)$. 
Note that by Definition 4.2(6), $N \cap \alpha \in f_r(\alpha)$.

First, assume that $K \in N$. 
Then either $K \in f_w(\alpha)$, or $K \in f_r(x) \cap N$. 
Consider the second case. 
Since $K$ and $N \cap \alpha$ are both in 
$f_r(x) \cup \{ x \}$ and 
$K \in N$, we must have that $K \in f_r(N \cap \alpha)$. 
Then $K \in f_r(N \cap \alpha) \subseteq f_r(\alpha)$. 
So $K \in f_r(\alpha) \cap N \subseteq f_w(\alpha)$ by Lemma 7.8(2). 
Hence, in either case, $K \in f_w(\alpha)$.

By Definition 7.11(1), $f(K) = f_w(K)$. 
Since $K \in f_w(\alpha)$, 
$f(K) = f_w(K) = f_w(\alpha) \cap Sk(K)$. 
Therefore,
\begin{multline*}
f(x) \cap Sk(K) = (f_w(\alpha) \cup f_r(x)) \cap Sk(K) = \\
= (f_w(\alpha) \cap Sk(K)) \cup (f_r(x) \cap Sk(K)) = 
f_w(K) \cup (f_r(x) \cap Sk(K)).
\end{multline*}
Since $f(K) = f_w(K)$, it suffices to show that 
$f_r(x) \cap Sk(K) \subseteq f_w(K)$, for then 
$$
f(x) \cap Sk(K) = f_w(K) \cup (f_r(x) \cap Sk(K)) = f_w(K) = f(K).
$$

Let $z \in f_r(x) \cap Sk(K)$, and we will show that 
$z \in f_w(K)$. 
Since $K \in N$, $z \in N$. 
As $z \in f_r(x) \cap N$ and $N \cap \alpha \in f_r(x) \cup \{ x \}$, 
we have that $z \in f_r(N \cap \alpha)$. 
So $z \in f_r(N \cap \alpha) \cap N \subseteq f_r(\alpha) \cap N 
\subseteq f_w(\alpha)$ by Lemma 7.8(2). 
So $z \in f_w(\alpha)$. 
Hence, $z \in f_w(\alpha) \cap Sk(K) = f_w(K)$, which completes the proof.

Secondly, assume that $K \notin N$. 
Then since $K \in f(x) = f_w(\alpha) \cup f_r(x)$, and $w \in N$, we must have that 
$K \in f_r(x)$. 
As $N \cap \alpha \in f_r(x) \cup \{ x \}$ and $K \notin N$, it follows that 
$N \cap \alpha \in f_r(K) \cup \{ K \}$. 
So we are in case 4 of Definition 7.11 in the definition of $f(K)$. 
That means that $f(K) = f_w(\alpha) \cup f_r(K)$. 
We have that 
\begin{multline*}
f(x) \cap Sk(K) = (f_w(\alpha) \cup f_r(x)) \cap Sk(K) = \\
= (f_w(\alpha) \cap Sk(K)) \cup (f_r(x) \cap Sk(K)) = (f_w(\alpha) \cap Sk(K)) \cup f_r(K).
\end{multline*}
So $f(x) \cap Sk(K) = (f_w(\alpha) \cap Sk(K)) \cup f_r(K)$. 
On the other hand, 
\begin{multline*}
f(K) = f(K) \cap Sk(K) = 
(f_w(\alpha) \cup f_r(K)) \cap Sk(K) = \\
= (f_w(\alpha) \cap Sk(K)) \cup (f_r(K) \cap Sk(K)) = 
(f_w(\alpha) \cap Sk(K)) \cup f_r(K).
\end{multline*}
So $f(K) = (f_w(\alpha) \cap Sk(K)) \cup f_r(K) = f(x) \cap Sk(K)$.

\bigskip

\noindent \emph{Case 5:} $x \in \dom(f_r) \setminus N$, 
case 4 fails for $x$, 
$f_r(x) \cap N = \emptyset$, and $f(x) = f_r(x)$. 
Then $K \in f_r(x)$. 
Since $f_r(x) \cap N = \emptyset$, $K \notin N$. 
Also, since $f_r(K) = f_r(x) \cap Sk(K)$, 
we have that $f_r(K) \cap N = \emptyset$. 
And as 
$f_r(K) \cup \{ K \} \subseteq f_r(x)$, we cannot have that 
$N \cap \alpha \in f_r(K) \cup \{ K \}$ 
for any $\alpha \in \dom(f_r) \cap S \cap N$, 
since otherwise case 4 would be true for $x$. 
So we are in case 5 of Definition 7.11 in the definition of $f(K)$, and therefore 
$f(K) = f_r(K)$. 
So 
$$
f(K) = f_r(K) = f_r(x) \cap Sk(K) = f(x) \cap Sk(K).
$$

\bigskip

\noindent \emph{Case 6:} $x \in \dom(f_r) \setminus N$, 
case 4 fails, $M$ is the largest member of $f_r(x) \cap N$, and 
$f(x) = f_r(x) \cup f_w(M)$. 
By the maximality of $M$, and since 
$f_r(M) = f_r(x) \cap Sk(M)$, we have that 
$$
f_r(x) \cap N = f_r(M) \cup \{ M \} \subseteq f_w(M) \cup \{ M \},
$$
where the inclusion holds by Lemma 7.8(2) and the fact 
that $f_r(M) = f_r(M) \cap N$. 
Therefore, 
$$
f(x) \cap N = (f_r(x) \cup f_w(M)) \cap N = 
f_w(M) \cup \{ M \}.
$$
Also, since $M \in N$, clearly $f_r(M) \setminus N = \emptyset$, and therefore, 
$$
f(x) \setminus N = (f_r(x) \cup f_w(M)) \setminus N = f_r(x) \setminus N.
$$
To summarize, 
$$
f(x) \cap N = f_w(M) \cup \{ M \}, \ \ \ 
f(x) \setminus N = f_r(x) \setminus N.
$$

We split into the two cases of whether $K \in N$ or $K \notin N$. 
First, assume that $K \in N$. 
Then $K \in f_w(M) \cup \{ M \}$. 
So $f(K)$ is defined as in case 1 of Definition 7.11 and $f(K) = f_w(K)$. 
So it suffices to show that 
$$
f_w(K) = f(x) \cap Sk(K).
$$
For the forward inclusion, $K \in f_w(M) \cup \{ M \}$ implies that 
$$
f_w(K) \subseteq f_w(M) \subseteq f_r(x) \cup f_w(M) = f(x).
$$
So $f_w(K) \subseteq f(x) \cap Sk(K)$. 
For the reverse inclusion, let $J \in f(x) \cap Sk(K)$, and we will show 
that $J \in f_w(K)$. 
Then 
$$
J \in f(x) \cap N = f_w(M) \cup \{ M \}.
$$
Since $K \in f_w(M) \cup \{ M \}$ and $J \in Sk(K)$, clearly 
$J \ne M$. 
Therefore, $J \in f_w(M)$. 
Since $w$ is a condition, 
$$
J \in f_w(M) \cap Sk(K) = f_w(K).
$$

Secondly, assume that $K \notin N$. 
Then $K \in f(x) \setminus N = f_r(x) \setminus N$.
Hence, $f_r(K) = f_r(x) \cap Sk(K)$, since $r$ is a condition. 
As $K$ and $M$ are both in $f_r(x)$, they are membership comparable. 
But as $M \in N$ and $K \notin N$, clearly we must have that 
$M \in Sk(K)$. 
Since $M$ is the membership largest element of $f_r(x) \cap N$ and 
$M \in f_r(K)$, $M$ is the membership largest 
element of $f_r(K) \cap N$. 
Moreover, since $K \in f_r(x)$, 
$f_r(K) \cup \{ K \} \subseteq f_r(x)$, so there is no 
$\alpha \in \dom(f_r) \cap S \cap N$ with $N \cap \alpha \in f_r(K) \cup \{ K \}$.  
So $f(K)$ is defined as in case 6 of Definition 7.11, which means that 
$f(K) = f_r(K) \cup f_w(M)$. 
So
\begin{multline*}
f(K) = f_r(K) \cup f_w(M) = (f_r(x) \cap Sk(K)) \cup f_w(M) = \\
= (f_r(x) \cup f_w(M)) \cap (Sk(K) \cup f_w(M)) 
= f(x) \cap Sk(K),
\end{multline*}
where the last equality follows from the definition of $f(x)$ and 
the fact that $f_w(M) \subseteq Sk(M) \subseteq Sk(K)$. 
So $f(K) = f(x) \cap Sk(K)$.

\bigskip

\noindent \emph{Case 7:} $x = M \cap \alpha$, where $M \in A_r$, $N \le M$, 
and $\alpha \in (M \cap \dom(f_w) \cap S) \setminus \dom(f_r)$. 
Then 
$$
f(x) = f(\alpha) \cap Sk(M \cap \alpha),
$$
where $f(\alpha)$ is defined as in case 3 of Definition 7.11. 
So $K \in f(\alpha)$. 
Also, $x = M \cap \alpha \in f(\alpha)$ by the definition of $f(\alpha)$. 
By case 3 handled above, 
$$
f(K) = f(\alpha) \cap Sk(K), \ \ \ f(x) = f(\alpha) \cap Sk(x).
$$
We also know from Lemma 7.14 that 
$K \in f(x)$ implies that $K \in Sk(x)$, and hence 
$Sk(K) \subseteq Sk(x)$. 
Therefore, $Sk(x) \cap Sk(K) = Sk(K)$. 
Consequently, 
$$
f(K) = f(\alpha) \cap Sk(K) = 
f(\alpha) \cap Sk(x) \cap Sk(K) = f(x) \cap Sk(K).
$$
\end{proof}

This completes our analysis of $f_w \oplus_N f_r$. 
We now turn to amalgamating the $g$-components of $w$ and $r$.

\begin{definition}
Let $N \in \mathcal X$ be simple and $r \in D_N$. 
Suppose that $w \in N \cap \p$ and $w \le r \restriction N$. 
Let $f := f_w \oplus_N f_r$. 
Define $g_w \oplus_N g_r$ as the function $g$ with domain equal 
to the set of pairs $(K,x)$ such that $K \in f(x)$, such that for all 
$(K,x) \in \dom(g)$, 
$$
g(K,x) := \bigcup \{ g_w(K,y) \cup g_r(K,y) : 
x = y, \textrm{or} \ x \in f(y) \}.\footnote{As in the previous section, we let 
$g_w(K,y)$ denote the empty set in the case that $(K,y) \notin \dom(g_w)$, 
and similarly with $g_r$.}
$$
\end{definition}

The next lemma will be used to show that 
$g_w \oplus_N g_r$ satisfies requirement (4) of Definition 4.2 
for $w \oplus_N r$. 

\begin{lemma}
Let $N \in \mathcal X$ be simple and $r \in D_N$. 
Suppose that $w \in N \cap \p$ and $w \le r \restriction N$. 
Let $g := g_w \oplus_N g_r$. 
Then for any $(K,x)$ in $\dom(g)$, 
$$
g(K,x) \subseteq x \setminus \sup(K).
$$
\end{lemma}

\begin{proof}
Let $(K,x) \in \dom(g)$. 
Then $K \in f(x)$. 
By definition, we have that 
$$
g(K,x) = \bigcup \{ g_w(K,y) \cup g_r(K,y) : 
x = y, \textrm{or} \ x \in f(y) \}.
$$
Since $g_w(K,x)$ and $g_r(K,x)$ are subsets of $x \setminus \sup(K)$, 
it suffices to show that if $x \in f(y)$, then 
$g_w(K,y)$ and $g_r(K,y)$ are subsets of $x \setminus \sup(K)$. 
Since $w$ are $r$ are conditions, $g_w(K,y)$ and 
$g_r(K,y)$ are subsets of $y \setminus \sup(K)$. 
Hence, it suffices to show that $g_w(K,y)$ 
and $g_r(K,y)$ are subsets of $x$. 

\bigskip

\noindent \emph{Claim 1:} The result holds if $K$, $x$, and $y$ are all in $N$.

\bigskip

If $K$, $x$, and $y$ are all in $N$, then 
$K \in f(x)$ and $x \in f(y)$ imply by Lemma 7.13(1,2) that 
$K \in f_w(x)$ and $x \in f_w(y)$. 
Since $w$ is a condition, it follows that 
$g_w(K,y) \subseteq g_w(K,x) \subseteq x$. 
To show that $g_r(K,y) \subseteq x$, assume that $(K,y) \in \dom(g_r)$, 
which means that $K \in f_r(y)$. 
Then by Lemma 7.8(3), $g_r(K,y) \subseteq g_w(K,y)$. 
But we just showed that $g_w(K,y) \subseteq x$, so we are done.

\bigskip

\noindent \emph{Claim 2:} The result holds if $K$, $x$, and $y$ are all in 
$\dom(f_r)$. 

\bigskip

If $K$, $x$, and $y$ are all in $\dom(f_r)$, then 
$K \in f(x)$ and $x \in f(y)$ imply by Lemma 7.13(3) that 
$K \in f_r(x)$ and $x \in f_r(y)$. 
Since $r$ is a condition, 
$$
g_r(K,y) \subseteq g_r(K,x) \subseteq x.
$$
To show that $g_w(K,y) \subseteq x$, 
assume that $(K,y) \in \dom(g_w)$, which means that $K \in f_w(y)$. 
In particular, $K$ and $y$ are in $N$. 
So if $x$ is also in $N$, then we are done by Claim 1. 
Assume that $x \notin N$.

Note that $y \in S$. 
For if $y \notin S$, then $y$ is countable. 
So $y \in N$ implies that $Sk(y) \subseteq N$, 
and hence $x \in N$, which is false. 
Therefore, $y = \alpha$ for some $\alpha \in \dom(f_r) \cap S$. 
Since $N \in A_r$ and $\alpha \in N$, $N \cap \alpha \in f_r(\alpha)$ 
by Definition 4.2(6). 
As $x \in f_r(\alpha)$, $x$ and $N \cap \alpha$ are membership comparable. 
But $x \notin N$. 
Therefore, either $x = N \cap \alpha$ or 
$N \cap \alpha \in Sk(x)$. 
In either case, $N \cap \alpha \subseteq x$. 
As $w \in N$, we have that 
$$
g_w(K,y) = g_w(K,\alpha) \subseteq N \cap \alpha \subseteq x.
$$

\bigskip

The rest of the proof splits up into the cases 1--7 of Definition 7.11 
in the definition of $f(y)$.

\bigskip

\noindent \emph{Case 1:} $y \in \dom(f_w) \setminus S$ and $f(y) = f_w(y)$. 
Then $K$, $x$, and $y$ are in $N$. 
So we are done by Claim 1.

\bigskip

\noindent \emph{Case 2:} $y = \alpha \in \dom(f_w) \cap S \cap \dom(f_r)$ 
and $f(y) = f_w(\alpha) \cup f_r(\alpha)$. 
Then $y \in N$. 
By Definition 4.2(6), $N \cap \alpha \in f_r(\alpha)$.

Assume that $(K,\alpha) \in \dom(g_w)$, which means that 
$K \in f_w(\alpha)$. 
We will show that $g_w(K,\alpha) \subseteq x$. 
In particular, $K \in N$. 
If $x$ is also in $N$, then we are done by Claim 1. 
So assume that $x \notin N$.

Since $x \in f(y) = f_w(\alpha) \cup f_r(\alpha)$ and $x \notin N$, 
we have that $x \in f_r(\alpha)$. 
So both $N \cap \alpha$ and $x$ are in $f_r(\alpha)$. 
As $x \notin N$, either $x = N \cap \alpha$ or $N \cap \alpha \in Sk(x)$. 
In either case, $N \cap \alpha \subseteq x$. 
But $g_w(K,\alpha) \subseteq N \cap \alpha$, since $w \in N$. 
Therefore, $g_w(K,\alpha) \subseteq x$.

Assume that $(K,\alpha) \in \dom(g_r)$, 
which means that $K \in f_r(\alpha)$. 
We will show that $g_r(K,\alpha) \subseteq x$. 
So $K$ and $\alpha$ are both in $\dom(f_r)$. 
If $x$ is also in $\dom(f_r)$, then we are done by Claim 2. 
So assume that $x \notin \dom(f_r)$. 
But $x \in f(y) = f_w(\alpha) \cup f_r(\alpha)$, so 
$x \in f_w(\alpha)$. 
So $x \in N$. 
As $K \in Sk(x)$, also $K \in N$. 
So $K$, $x$, and $\alpha$ are all in $N$, and we are done by Claim 1.

\bigskip

\noindent \emph{Case 3:} $y = \alpha \in (\dom(f_w) \cap S) \setminus \dom(f_r)$. 
Since $\alpha \notin \dom(f_r)$, we have that $(K,\alpha) \notin \dom(g_r)$. 
Therefore, $g_r(K,\alpha) = \emptyset \subseteq x$.

Assume that $(K,\alpha) \in \dom(g_w)$, which means that 
$K \in f_w(\alpha)$. 
We will show that $g_w(K,\alpha) \subseteq x$. 
In particular, $K$ and $\alpha$ are both in $N$. 
So if $x$ is also in $N$, then we are done by Claim 1. 

Suppose that $x \notin N$. 
Then by the definition of $f(\alpha)$, we have that 
$x = M \cap \alpha$, for some $M \in A_r$ with $N \le M$ and 
$\alpha \in M$. 
Then $\alpha \in M \cap N \cap \kappa$, so $\alpha < \beta_{M,N}$ 
by Proposition 2.11. 
Since $N \le M$, we have that 
$N \cap \alpha \subseteq M \cap \alpha = x$. 
As $w \in N$, it follows that 
$g_w(K,\alpha) \subseteq N \cap \alpha \subseteq x$.

\bigskip

Before handling cases 4--7, let us note that in each of these cases, $y \notin N$. 
This is immediate in cases 4, 5, and 6, and follows from Lemma 7.9 in case 7. 
Consequently, $(K,y) \notin \dom(g_w)$, and hence 
$g_w(K,y) = \emptyset \subseteq x$. 
Thus, we only need to show that 
$g_r(K,y) \subseteq x$.

Assume that $(K,y) \in \dom(g_r)$, which means that 
$K \in f_r(y)$. 
We will show that $g_r(K,y) \subseteq x$. 
If $x \in \dom(f_r)$, then we are done by Claim 2. 
Hence, we may assume that $x \notin \dom(f_r)$.

\bigskip

\noindent \emph{Case 4:} For some $\alpha \in \dom(f_r) \cap S \cap N$, 
$N \cap \alpha \in f_r(y) \cup \{ y \}$, and 
$f(y) = f_w(\alpha) \cup f_r(y)$. 
By Definition 4.2(6), $N \cap \alpha \in f_r(\alpha)$. 
Since $x \notin \dom(f_r)$, it follows that 
$x \in f_w(\alpha)$. 
In particular, $x \in N$, and therefore also $K \in N$.

By Lemma 7.12(2), since $\alpha \in \dom(f_w)$, 
$f_w(\alpha) \subseteq f(\alpha)$. 
Hence, $x \in f(\alpha)$. 
We also know that $K \in f(x)$. 
Thus, by case 2 handled above, 
$$
g_r(K,\alpha) \subseteq x.
$$
Since $K \in f_r(y)$, $N \cap \alpha \in f_r(y) \cup \{ y \}$, and $K \in N$, 
we have that $K \in f_r(N \cap \alpha)$. 
Therefore, $K \in f_r(\alpha)$. 
By Definition 4.2(5), it follows that 
$$
g_r(K,y) \subseteq g_r(K,N \cap \alpha).
$$
Since $r \in D_N$, by Definition 7.3(3) we have that 
$$
g_r(K,N \cap \alpha) \subseteq g_r(K,\alpha).
$$
Putting it all together, 
$$
g_r(K,y) \subseteq g_r(K,N \cap \alpha) \subseteq g_r(K,\alpha) \subseteq x.
$$

\bigskip

\noindent \emph{Case 5:} 
$y \in \dom(f_r) \setminus N$, $f_r(y) \cap N = \emptyset$, and 
$f(y) = f_r(y)$. 
Since $x \in f(y) = f_r(y)$, we have that $x \in \dom(f_r)$. 
But this contradicts the fact that $x \notin \dom(f_r)$.

\bigskip

\noindent \emph{Case 6:} 
$y \in \dom(f_r) \setminus N$, and 
$f(y) = f_r(y) \cup f_w(M)$, 
where $M$ is the membership largest element of $f_r(y) \cap N$. 
So either $x \in f_r(y)$ or $x \in f_w(M)$. 
Since $x \notin \dom(f_r)$, we have that $x \in f_w(M)$. 
Then $x \in Sk(M)$, and thus $K \in Sk(M)$. 

Since $K$ and $M$ are in $f_r(y)$ and $K \in Sk(M)$, 
it follows that $K \in f_r(M)$. 
Since $r$ is a condition, $g_r(K,y) \subseteq g_r(K,M)$. 
Now $x \in f_w(M)$, and $f(M) = f_w(M)$ by Definition 7.11(1). 
Thus, $K \in f(x)$ and $x \in f(M)$. 
By case 1 handled above, 
$g_r(K,M) \subseteq x$. 
Thus, $g_r(K,y) \subseteq g_r(K,M) \subseteq x$.

\bigskip

\noindent \emph{Case 7:} 
$y = M \cap \alpha$, where $M \in A_r$, $N \le M$, and 
$\alpha \in (M \cap \dom(f_w) \cap S) \setminus \dom(f_r)$. 
By Lemma 7.9, $y$ is not in $\dom(f_r)$. 
But we assumed that $K \in f_r(y)$, so we have a contradiction. 
Thus, this case does not occur.
\end{proof}

We are ready to define the amalgam $w \oplus_N r$.

\begin{definition}
Let $N \in \mathcal X$ be simple and $r \in D_N$. 
Suppose that $w \in N \cap \p$ and $w \le r \restriction N$. 
Define $w \oplus_N r$ as the object $(f,g,A)$ satisfying:
\begin{enumerate}
\item $f := f_w \oplus_N f_r$;
\item $g := g_w \oplus_N g_r$;
\item $A := A_w \cup A_r$.
\end{enumerate}
\end{definition}

We will now show that $w \oplus_N r$ is a condition 
which is below $w$ and $r$. 
We have completed most of the work for this proof in the 
preceding lemmas.

\begin{proposition}
Let $N \in \mathcal X$ be simple and $r \in D_N$. 
Then for all $w \le r \restriction N$ in $N \cap \p$, 
$w$ and $r$ are compatible. 
In fact, $w \oplus_N r$ is in $\p$ and $w \oplus_N r \le w, r$.
\end{proposition}

\begin{proof}
Let $w \oplus_N r = (f,g,A)$. 
We will prove that $w \oplus_N r$ is a condition which is below 
$w$ and $r$. 
To show that $w \oplus_N r$ is a condition, 
we verify requirements (1)--(7) of Definition 4.2.

\bigskip

(1) We apply Proposition 2.25. 
Since $r \in D_N$, we have that for all $M \in A_r$, if $M < N$ then 
$M \cap N \in A_r$. 
Also, $A_w$ is adequate, and by Lemma 7.8(1), 
$$
A_r \cap N \subseteq A_w \subseteq N.
$$
By Proposition 2.25, $A_w \cup A_r = A$ is adequate.

\bigskip

(2,3) These statements are immediate from Lemmas 7.14 and 7.15.

\bigskip

(4) By Definition 7.16, $g$ is a function whose domain is the set of 
pairs $(K,x)$ such that $K \in f(x)$. 
And by Lemma 7.17, for all $(K,x) \in \dom(g)$, 
$g(K,x) \subseteq x \setminus \sup(K)$.
Also, $g(K,x)$ is finite by Definition 7.16.

\bigskip

(5) Let $K \in f(L)$ and $L \in f(x)$. 
We will show that $g(K,x) \subseteq g(K,L)$. 
So let $\xi \in g(K,x)$. 
Then by Definition 7.16, either $\xi \in g_w(K,x) \cup g_r(K,x)$, 
or for some $y$ with $x \in f(y)$, 
$\xi \in g_w(K,y) \cup g_r(K,y)$. 
In the second case, $L \in f(x)$ and $x \in f(y)$ imply by 
requirement (3) that $L \in f(y)$. 
So letting $z := x$ in the first case, and $z := y$ in the second case, 
we have that $L \in f(z)$ and 
$\xi \in g_w(K,z) \cup g_r(K,z)$. 
By Definition 7.16, it follows that $\xi \in g(K,L)$.

\bigskip

(6) Let $\alpha \in \dom(f) \cap S$ and $M \in A$ with $\alpha \in M$. 
We will show that $M \cap \alpha \in f(\alpha)$. 
By the definition of $\dom(f)$ given in Definition 7.11, clearly 
$\alpha \in \dom(f_w) \cup \dom(f_r)$. 
Also, $A = A_w \cup A_r$, so either $M \in A_w$ or $M \in A_r$.

First, assume that $\alpha \in \dom(f_r)$. 
If $M \in A_r$, then since $r$ is a condition, 
$M \cap \alpha \in f_r(\alpha) \subseteq f(\alpha)$. 
Suppose that $M \in A_w$. 
Then $\alpha \in M \in N$, 
so $\alpha \in \dom(f_r) \cap N \subseteq \dom(f_w)$ 
by Lemma 7.8(2). 
Since $w$ is a condition, 
$M \cap \alpha \in f_w(\alpha) \subseteq f(\alpha)$. 

Secondly, assume that $\alpha \in \dom(f_w) \setminus \dom(f_r)$. 
If $M \in A_w$, then since $w$ is a condition, 
$M \cap \alpha \in f_w(\alpha) \subseteq f(\alpha)$. 
Suppose that $M \in A_r$. 
If $N \le M$, then $M \cap \alpha \in f(\alpha)$ by 
Definition 7.11(3).

Suppose that $M < N$. 
Then since $r \in D_N$, we have that 
$M \cap N \in A_r \cap N \subseteq A_w$ 
by Lemma 7.8(1). 
Since $\alpha \in \dom(f_w)$ and $\alpha \in M \cap N$, 
it follows that 
$M \cap N \cap \alpha \in f_w(\alpha) \subseteq f(\alpha)$. 
Now $\alpha \in M \cap N \cap \kappa$ implies that 
$\alpha < \beta_{M,N}$ by Proposition 2.11. 
Since $M < N$, Lemma 2.15 implies that 
$$
M \cap \alpha = M \cap \beta_{M,N} \cap \alpha = 
M \cap N \cap \alpha.
$$
Thus, $M \cap \alpha \in f(\alpha)$.

\bigskip

(7) As in (1) above, the assumptions of Proposition 3.5 hold 
for $A_r$ and $A_w$. 
Therefore, 
$$
r^*(A) = r^*(A_w \cup A_r) = r^*(A_w) \cup r^*(A_r).
$$
As $w$ and $r$ are conditions, 
$$
r^*(A_w) \cap S \subseteq \dom(f_w), \ \ \ 
r^*(A_r) \cap S \subseteq \dom(f_r).
$$
But $\dom(f_w) \subseteq \dom(f)$ and 
$\dom(f_r) \subseteq \dom(f)$. 
Hence, 
$$
r^*(A) \cap S = 
(r^*(A_w) \cap S) \cup (r^*(A_r) \cap S) 
\subseteq \dom(f).
$$

\bigskip

This completes the proof that $w \oplus_N r$ is a condition. 
Now we show that $w \oplus_N r \le w, r$. 
First, we prove that $w \oplus_N r \le w$ by verifying 
properties (a)--(d) of Definition 4.2 for $w$. 
(a) Since $A = A_w \cup A_r$, clearly $A_w \subseteq A$. 
(b) follows from Lemma 7.12(1,2). 
(c) is immediate from Definition 7.16, and (d) was 
proved in Lemma 7.13(2).

Secondly, we prove that $w \oplus_N r \le r$ by 
verifying properties (a)--(d) of Definition 4.2 for $r$. 
(a) Since $A = A_w \cup A_r$, clearly $A_r \subseteq A$. 
(b) follows from Lemma 7.12(1,3). 
(c) is immediate from Definition 7.16, and (d) was 
proved in Lemma 7.13(3).
\end{proof}

\begin{corollary}
The forcing poset $\p$ is strongly proper on a stationary set. 
In particular, it preserves $\omega_1$.
\end{corollary}

\begin{proof}
By Assumption 2.22, the set of $N \in \mathcal X$ such that $N$ is simple 
is stationary in $P_{\omega_1}(H(\lambda))$. 
So it suffices to show that for all simple $N \in \mathcal X$, for all 
$p \in N \cap \p$, there is $q \le p$ such that $q$ is strongly $N$-generic. 

Let $p \in N \cap \p$. 
By Lemma 7.1, fix $q \le p$ with $N \in A_q$. 
We claim that $q$ is strongly $N$-generic.
So let $D$ be a dense subset of $N \cap \p$, and we will show that 
$D$ is predense below $q$. 
Let $r \le q$, and we will find $w \in D$ which is compatible 
with $r$.

Since $N \in A_{r}$, we can apply Lemma 7.4 and fix $s \le r$ such that 
$s \in D_N$. 
Then by Lemma 7.6, $s \restriction N$ is in $N \cap \p$. 
As $D$ is dense in $N \cap \p$, fix $w \le s \restriction N$ in $D$. 
By Proposition 7.19, $w$ and $s$ are compatible. 
Since $s \le r$, it follows that $w$ and $r$ are compatible.
\end{proof}

It follows from Corollaries 6.16 and 7.20 that $\p$ preserves 
$\omega_1$ and $\kappa$. 
By the next lemma, no cardinal in between $\omega_1$ and $\kappa$ 
survives.

\begin{lemma}
If $\mu$ is a cardinal and $\omega_1 < \mu < \kappa$, then 
$\p$ forces that $\mu$ is not a cardinal.
\end{lemma}

\begin{proof}
Let $G$ be a generic filter on $\p$. 
Define 
$$
A := \{ M : \exists p \in G \ M \in A_p, \ \mu \in M \}.
$$
Consider $M$ and $N$ in $A$. 
Then there is $p \in G$ with $M$ and $N$ in $A_p$. 
Since $\mu \in M \cap N \cap \kappa$, 
$\mu < \beta_{M,N}$ by Proposition 2.11.
If $M \cap \omega_1 = N \cap \omega_1$, then 
$M \sim N$ by Lemma 2.17, and 
therefore, $M \cap \beta_{M,N} = N \cap \beta_{M,N}$. 
Since $\mu < \beta_{M,N}$, it follows that 
$M \cap \mu = N \cap \mu$. 
This proves that the map which sends a member of the set 
$$
A^* := \{ M \cap \mu : M \in A \}
$$
to its intersection with $\omega_1$ is an injective function 
from $A^*$ into $\omega_1$. 
Hence, in $V[G]$, $A^*$ has size less than or equal to $\omega_1$.

A density argument using Lemma 7.1 shows that 
for all $\xi < \mu$, there is $N \in A$ with 
$\xi \in N$. 
It follows that $\bigcup A^* = \mu$.
So in $V[G]$, 
$\mu$ is the union of a collection of countable sets of size 
at most $\omega_1$. 
This implies that in $V[G]$, $\mu$ has size at most $\omega_1$. 
Since $\omega_1 < \mu$, $\mu$ is not a cardinal in $V[G]$.
\end{proof}

\begin{corollary}
The forcing poset $\p$ forces that $\kappa = \omega_2$.
\end{corollary}

\begin{proof}
Immediate from Corollaries 6.16 and 7.20 and 
Lemma 7.21.
\end{proof}

\bigskip

\addcontentsline{toc}{section}{8. Further analysis}

\textbf{\S 8. Further analysis}

\stepcounter{section}

\bigskip

The goal of the next two sections is to prove that certain quotients 
of the forcing poset $\p$ satisfy the $\omega_1$-approximation property. 
This fact will follow from the equation 
$$
(q \oplus_N p) \restriction Q = 
(q \restriction Q) \oplus_{N \cap Q} (p \restriction Q),
$$
which is proved in Proposition 8.6. 
Lemmas 8.1 and 8.2 provide some additional information about 
$\p$ which we will need to prove the approximation property in Section 9. 
Then Lemmas 8.3, 8.4, and 8.5 prepare us for proving Proposition 8.6.
The information which we provide here will be used again in Part III to analyze 
products of the partial square forcing poset.

\begin{lemma}
Let $N \in \mathcal X$ be simple and $r \in D_N$. 
Suppose that $v$ and $w$ are in $N \cap \p$ and 
$$
w \le v \le r \restriction N.
$$
Then $w \oplus_N r \le v \oplus_N r$.
\end{lemma}

The proof of this lemma is straightforward, but due to the 
multitude of cases in Definition 7.11, it is also somewhat lengthy.

\begin{proof}
Let $s := v \oplus_N r$ and $t := w \oplus_N r$. 
We will prove that $t \le s$. 
We verify properties (a)--(d) of Definition 4.2 for $s$ and $t$.

\bigskip

(a) By Definition 7.18, $A_s = A_v \cup A_r$ and $A_t = A_w \cup A_r$. 
Since $w \le v$, $A_v \subseteq A_w$. 
Therefore, $A_s \subseteq A_t$.

\bigskip

(b,d) The domain of $f_s$ is equal to the union of $\dom(f_v)$, $\dom(f_r)$, 
and the set 
$$
\{ M \cap \alpha : M \in A_r, \ N \le M, \ 
\alpha \in (M \cap \dom(f_v) \cap S) \setminus \dom(f_r) \}.
$$
The domain of $f_t$ is equal to the union of $\dom(f_w)$, $\dom(f_r)$, 
and the set 
$$
\{ M \cap \alpha : M \in A_r, \ N \le M, \ 
\alpha \in (M \cap \dom(f_w) \cap S) \setminus \dom(f_r) \}.
$$
But $w \le v$ implies that $\dom(f_v) \subseteq \dom(f_w)$. 
It is easy to check from the above definitions and the fact that 
$\dom(f_v) \subseteq \dom(f_w)$ that $\dom(f_s) \subseteq \dom(f_t)$.

\bigskip

Let $x \in \dom(f_s)$, and we will prove that $f_s(x) \subseteq f_t(x)$. 
At the same time, we will also show that if $K$ is in $\dom(f_s)$ and 
$K \in f_t(x)$, then $K \in f_s(x)$. 
Note that these assertions imply (b) and (d). 
The proof will split into the seven cases of Definition 7.11 
for how $f_s(x)$ is defined. 
First, we prove a claim.

\bigskip

\noindent \emph{Claim 1:}  If $K \in \dom(f_s)$ and 
$K \in f_w(x)$, then $K \in f_s(x)$.

\bigskip

Since $K \in f_w(x)$, $K$ and $x$ are in $N$. 
By Lemma 7.13(1), 
$K$ and $x$ are in $\dom(f_s) \cap N = \dom(f_v)$. 
So $K$ and $x$ are in $\dom(f_v)$ and $K \in f_w(x)$. 
Since $w \le v$, it follows that $K \in f_v(x)$. 
But $s \le v$ implies that $f_v(x) \subseteq f_s(x)$. 
So $K \in f_s(x)$.

\bigskip

\noindent \emph{Case 1:} $x \in \dom(f_v) \setminus S$ 
and $f_s(x) = f_v(x)$. 
Since $w \le v$, $x \in \dom(f_w) \setminus S$, and so $f_t(x) = f_w(x)$. 
But $w \le v$ implies that $f_v(x) \subseteq f_w(x)$, 
hence $f_s(x) \subseteq f_t(x)$.

Suppose that $K \in \dom(f_s)$ and $K \in f_t(x)$. 
Then $K \in f_t(x) = f_w(x)$. 
So $K \in f_s(x)$ by Claim 1.

\bigskip

\noindent \emph{Case 2:} $x = \alpha \in \dom(f_v) \cap S \cap \dom(f_r)$ 
and $f_s(\alpha) = f_v(\alpha) \cup f_r(\alpha)$. 
Since $\dom(f_v) \subseteq \dom(f_w)$, 
$\alpha \in \dom(f_w) \cap S \cap \dom(f_r)$, 
so $f_t(\alpha) = f_w(\alpha) \cup f_r(\alpha)$. 
Also, $w \le v$ implies that 
$f_v(\alpha) \subseteq f_w(\alpha)$. 
Therefore, 
$f_s(\alpha) \subseteq f_t(\alpha)$.

Suppose that $K \in \dom(f_s)$ and 
$K \in f_t(\alpha)$. 
Then $K \in f_t(\alpha) =  f_w(\alpha) \cup f_r(\alpha)$. 
So either $K \in f_w(\alpha)$ or $K \in f_r(\alpha)$. 
If $K \in f_w(\alpha)$, then $K \in f_s(\alpha)$ by Claim 1. 
If $K \in f_r(\alpha)$, then $K \in f_s(\alpha)$ by definition.

\bigskip

\noindent \emph{Case 3:} $x = \alpha \in (\dom(f_v) \cap S) \setminus \dom(f_r)$ and 
$$
f_s(\alpha) = f_v(\alpha) \cup \{ M \cap \alpha : M \in A_r, \ N \le M, \ 
\alpha \in M \}.
$$
Since $\dom(f_v) \subseteq \dom(f_w)$, 
$\alpha \in (\dom(f_w) \cap S) \setminus \dom(f_r)$, and therefore 
$$
f_t(\alpha) = f_w(\alpha) \cup 
\{ M \cap \alpha : M \in A_r, \ N \le M, \ \alpha \in M \}.
$$
Since $f_v(\alpha) \subseteq f_w(\alpha)$, it follows that 
$f_s(\alpha) \subseteq f_t(\alpha)$.

Suppose that $K \in \dom(f_s)$ and $K \in f_t(\alpha)$. 
Then $K \in f_t(\alpha) = 
f_w(\alpha) \cup \{ M \cap \alpha : M \in A_r, \ N \le M, \ \alpha \in M \}$. 
If $K$ is in the second set of this union, then 
$K \in f_s(\alpha)$ by definition. 
If $K \in f_w(\alpha)$, then $K \in f_s(\alpha)$ by Claim 1.

\bigskip

\noindent \emph{Case 4:} $x \in \dom(f_r) \setminus N$, 
and for some $\alpha \in \dom(f_r) \cap S \cap N$, 
$N \cap \alpha \in f_r(x) \cup \{ x \}$. 
Then $f_s(x) = f_v(\alpha) \cup f_r(x)$. 
Clearly we are also in case 4 in the definition of $f_t(x)$. 
So $f_t(x) = f_w(\alpha) \cup f_r(x)$. 
Since $w \le v$, $f_v(\alpha) \subseteq f_w(\alpha)$. 
Therefore, $f_s(x) \subseteq f_t(x)$.

Suppose that $K \in \dom(f_s)$ and $K \in f_t(x)$. 
Then $K \in f_w(\alpha) \cup f_r(x)$. 
If $K \in f_r(x)$, then $K \in f_s(x)$ by definition. 
Suppose that $K \in f_w(\alpha)$. 
Then $K$ is in $N$. 
So $K$ and $\alpha$ are in $\dom(f_s) \cap N = \dom(f_v)$ by 
Lemma 7.13(1). 
Since $K \in f_w(\alpha)$ and $w \le v$, it follows that 
$K \in f_v(\alpha)$. 
So $K \in f_s(x)$ by definition.

\bigskip

\noindent \emph{Case 5:} $x \in \dom(f_r) \setminus N$, case 4 is false, 
and $f_r(x) \cap N = \emptyset$. 
Then clearly we are also in case 5 in the definition of $f_t(x)$. 
So by definition, $f_s(x)$ and $f_t(x)$ are both equal to $f_r(x)$. 
In particular, $f_s(x) \subseteq f_t(x)$. 
Also, if $K \in \dom(f_s)$ and $K \in f_t(x)$, then 
$K \in f_t(x) = f_s(x)$.

\bigskip

\noindent \emph{Case 6:} $x \in \dom(f_r) \setminus N$, case 4 is false, 
and $f_s(x) = f_r(x) \cup f_v(M)$, 
where $M$ is the membership largest element of $f_r(x) \cap N$. 
Then we are obviously also in case 6 in the definition of $f_t$. 
So $f_t(x) = f_r(x) \cup f_w(M)$. 
Since $w \le v$, we have that $f_v(M) \subseteq f_w(M)$. 
Therefore, $f_s(x) \subseteq f_t(x)$.

Suppose that $K \in \dom(f_s)$ and $K \in f_t(x)$. 
Then $K \in f_t(x) = f_r(x) \cup f_w(M)$. 
If $K \in f_r(x)$, then $K \in f_s(x)$ by definition. 
Otherwise, $K \in f_w(M)$. 
Since $M \in \dom(f_r) \subseteq \dom(f_s)$, 
Claim 1 implies that $K \in f_s(M)$. 
But $M \in f_r(x) \subseteq f_s(x)$. 
So $K \in f_s(M)$ and $M \in f_s(x)$. 
Since $s$ is a condition, $K \in f_s(x)$.

\bigskip

\noindent \emph{Case 7:} $x = M \cap \alpha$, 
where $M \in A_r$, $N \le M$, 
and $\alpha \in (M \cap \dom(f_v) \cap S) \setminus \dom(f_r)$. 
Then 
$$
f_s(M \cap \alpha) = f_s(\alpha) \cap Sk(M \cap \alpha).
$$
Since $w \le v$, $\alpha \in (M \cap \dom(f_w) \cap S) \setminus \dom(f_r)$. 
So we are also in case 7 in the definition of $f_t(x)$. 
Therefore, 
$$
f_t(M \cap \alpha) = f_t(\alpha) \cap Sk(M \cap \alpha).
$$
Since $f_s(\alpha) \subseteq f_t(\alpha)$ by case 3 handled above, 
it follows that $f_s(M \cap \alpha) \subseteq f_t(M \cap \alpha)$.

Suppose that $K \in \dom(f_s)$ and $K \in f_t(M \cap \alpha)$. 
Then by the definition of $f_t(M \cap \alpha)$, $K \in f_t(\alpha)$ 
and $K \in Sk(M \cap \alpha)$. 
Since $\alpha \in \dom(f_v)$ and $s \le v$, $\alpha \in \dom(f_s)$. 
By case 3 handled above, it follows that 
$K \in f_s(\alpha)$. 
So $K \in f_s(\alpha) \cap Sk(M \cap \alpha) = f_s(M \cap \alpha)$.

This completes the proof of (b) and (d).

\bigskip

(c) Let $(K,x) \in \dom(g_s)$, and we will show that 
$g_s(K,x) \subseteq g_t(K,x)$. 
By Definition 7.16, $K \in f_s(x)$ and 
$$
g_s(K,x) = \bigcup \{ g_v(K,y) \cup g_r(K,y) : x = y, \textrm{or} \ x \in f_s(y) \}.
$$
Since $f_s(x) \subseteq f_t(x)$ as just shown, $K \in f_t(x)$, and 
by Definition 7.16,
$$
g_t(K,x) = \bigcup \{ g_w(K,y) \cup g_r(K,y) : x = y, \textrm{or} \ x \in f_t(y) \}.
$$
Let $\xi \in g_s(K,x)$. 
To show that $\xi \in g_t(K,x)$, we will consider the different possibilities for 
why $\xi$ is in $g_s(K,x)$.

\bigskip

\noindent \emph{Case 1:} $\xi \in g_v(K,x)$. 
Then $(K,x) \in \dom(g_v)$, which means that $K \in f_v(x)$. 
Since $w \le v$, $K \in f_w(x)$ and 
$g_v(K,x) \subseteq g_w(K,x)$. 
So $\xi \in g_w(K,x) \subseteq g_t(K,x)$.

\bigskip

\noindent \emph{Case 2:} $\xi \in g_v(K,y)$, where $x \in f_s(y)$. 
Then $(K,y) \in \dom(g_v)$, which means that $K \in f_v(y)$. 
Since $w \le v$, $K \in f_w(y)$ and 
$g_v(K,y) \subseteq g_w(K,y)$. 
So $\xi \in g_w(K,y)$. 
But $x \in f_s(y) \subseteq f_t(y)$. 
So by definition, 
$g_w(K,y) \subseteq g_t(K,x)$. 
So $\xi \in g_t(K,x)$.

\bigskip

\noindent \emph{Case 3:} $\xi \in g_r(K,x)$. 
Since $t \le r$, it follows that 
$g_r(K,x) \subseteq g_t(K,x)$, so $\xi \in g_t(K,x)$.

\bigskip

\noindent \emph{Case 4:} $\xi \in g_r(K,y)$, where $x \in f_s(y)$. 
By (b), $f_s(y) \subseteq f_t(y)$. 
So $x \in f_t(y)$. 
Since $\xi \in g_r(K,y)$ and $x \in f_t(y)$, it follows by definition that 
$\xi \in g_t(K,x)$.
\end{proof}

\begin{lemma}
Let $N \in \mathcal X$ and $Q \in \mathcal Y$. 
Let $p \in \p$, and suppose that $N \in A_p$. 
Then there is $s \le p$ such that $s \in D_N \cap D_Q$.
\end{lemma}

\begin{proof}
We construct $s$ in several steps. 
We begin by applying Lemma 7.2 to find $q \le p$ such that for all 
$M \in A_q$, if $M < N$ then $M \cap N \in A_q$. 
Next, we apply Lemma 6.1 to fix $r \le q$ such that for all 
$M \in A_r$, $M \cap Q \in A_r$, and moreover, 
$A_r = A_q \cup \{ M \cap Q : M \in A_q \}$.

We claim that for all $M \in A_r$, if $M < N$ then $M \cap N \in A_r$. 
This is certainly true if $M \in A_q$, so assume that 
$M = M_1 \cap Q$, where $M_1 \in A_q$. 
By Lemma 2.29, $M_1 \sim M_1 \cap Q = M$. 
Since $M < N$, it follows that $M_1 < N$ by Lemma 2.18. 
As $M_1 \in A_q$, we have that 
$M_1 \cap N \in A_q$ by the choice of $q$. 
Hence, 
$$
M \cap N = (M_1 \cap Q) \cap N = (M_1 \cap N) \cap Q.
$$
But $M_1 \cap N \in A_q$ implies that 
$M \cap N = (M_1 \cap N) \cap Q \in A_r$ by the definition of $A_r$.

Now apply Lemma 4.9 to find $s \le r$ such that $A_s = A_r$ 
and whenever $K \in f_s(x)$ and $x \in f_s(y)$, then 
$g_s(K,x) \subseteq g_s(K,y)$. 
Then $s \le p$ and $s \in D_N \cap D_Q$.
\end{proof}

The next three lemmas will prepare us for proving Proposition 8.6.

\begin{lemma}
Let $N \in \mathcal X$ be simple, $Q \in \mathcal Y \cap N$ 
be simple, and 
$p \in D_N \cap D_Q$. 
Then $p \restriction N \in D_Q$ and $p \restriction Q \in D_{N \cap Q}$.
\end{lemma}

Recall that by Lemma 2.30, $N \cap Q$ is simple.

\begin{proof}
We prove first that $p \restriction N \in D_Q$, which by 
Definition 6.2 means that for all $M \in A_{p \restriction N}$, 
$M \cap Q \in A_{p \restriction N}$. 
So let $M \in A_{p \restriction N}$. 
Then by Definition 7.5, 
$M \in A_{p \restriction N} = A_p \cap N$. 
As $p \in D_Q$, we have that $M \cap Q \in A_p$. 
And since $M$ and $Q$ are in $N$, $M \cap Q \in N$. 
Thus, $M \cap Q \in A_p \cap N = A_{p \restriction N}$.

Next, we prove that $p \restriction Q \in D_{N \cap Q}$. 
First, we show that $N \cap Q \in A_{p \restriction Q}$. 
Since $A_{p \restriction Q} = A_p \cap Q$ by Definition 6.4, 
we need to show that 
$N \cap Q \in A_p \cap Q$. 
Since $Q$ is simple, $N \cap Q \in Q$. 
As $p \in D_N$, $N \in A_p$, and since $p \in D_Q$, 
it follows that $N \cap Q \in A_p$. 
So $N \cap Q \in A_p \cap Q$.

Secondly, let $M \in A_{p \restriction Q} = A_p \cap Q$ 
be such that $M < N \cap Q$, 
and we will show that 
$$
M \cap (N \cap Q) \in A_{p \restriction Q} = A_p \cap Q.
$$
Since $M \cap N \in \mathcal X$ and $Q$ is simple, we have that 
$M \cap N \cap Q \in Q$. 
By Lemma 2.29, $N \sim N \cap Q$. 
Since $M < N \cap Q$, Lemma 2.18 implies that $M < N$. 
As $p \in D_N$ and $M \in A_p$, it follows that $M \cap N \in A_p$. 
Since $p \in D_Q$, $M \cap N \cap Q \in A_p$. 
So $M \cap N \cap Q \in A_p \cap Q$, as required.

Thirdly, let $K \in f_{p \restriction Q}(x)$ and 
$x \in f_{p \restriction Q}(y)$, 
and we will show that 
$g_{p \restriction Q}(K,x) \subseteq g_{p \restriction Q}(K,y)$. 
By the definition of $p \restriction Q$, we have that 
$K \in f_p(x)$ and $x \in f_p(y)$. 
Since $p \in D_N$, $g_p(K,x) \subseteq g_p(K,y)$. 
As $g_p(K,x) = g_{p \restriction Q}(K,x)$ and 
$g_p(K,y) = g_{p \restriction Q}(K,y)$ by the definition of 
$p \restriction Q$, we are done.
\end{proof}

\begin{lemma}
Let $N \in \mathcal X$ be simple, $Q \in \mathcal Y \cap N$ be simple, 
and $p \in D_N \cap D_Q$. 
Then 
$$
(p \restriction N) \restriction Q = 
(p \restriction Q) \restriction (N \cap Q).
$$
\end{lemma}

Note that we needed Lemma 8.3 to see that 
$(p \restriction N) \restriction Q$ and  
$(p \restriction Q) \restriction (N \cap Q)$ are defined.

\begin{proof}
By Definitions 6.4 and 7.5, we have that 
\begin{multline*}
\dom(f_{(p \restriction N) \restriction Q}) = 
\dom(f_{p \restriction N}) \cap Q = 
\dom(f_p) \cap N \cap Q = \\ 
= (\dom(f_p) \cap Q) \cap (N \cap Q) = 
\dom(f_{p \restriction Q}) \cap (N \cap Q) = 
\dom(f_{(p \restriction Q) \restriction (N \cap Q)}).
\end{multline*}
And for each $x \in \dom(f_{(p \restriction N) \restriction Q})$, 
$$
f_{(p \restriction N) \restriction Q}(x) = 
f_{p \restriction N}(x) = f_{p}(x) \cap N = 
f_{p \restriction Q}(x) \cap N,
$$
which, since $f_{p \restriction Q}(x) \subseteq Q$, is equal to 
$$
f_{p \restriction Q}(x) \cap (N \cap Q) = 
f_{(p \restriction Q) \restriction (N \cap Q)}(x).
$$
Thus, $f_{(p \restriction N) \restriction Q} = 
f_{(p \restriction Q) \restriction (N \cap Q)}$.

Again by Definitions 6.4 and 7.5, we have that 
\begin{multline*}
\dom(g_{(p \restriction N) \restriction Q}) = 
\dom(g_{p \restriction N}) \cap Q = 
\dom(g_p) \cap N \cap Q = \\
= (\dom(g_p) \cap Q) \cap (N \cap Q) = 
\dom(g_{p \restriction Q}) \cap (N \cap Q) = 
\dom(g_{(p \restriction Q) \restriction (N \cap Q)}).
\end{multline*}
And for each $(K,x) \in \dom(g_{(p \restriction N) \restriction Q})$, 
$$
g_{(p \restriction N) \restriction Q}(K,x) = 
g_{p \restriction N}(K,x) = g_p(K,x) = 
g_{p \restriction Q}(K,x) = 
g_{(p \restriction Q) \restriction (N \cap Q)}(K,x).
$$
This proves that 
$g_{(p \restriction N) \restriction Q} = 
g_{(p \restriction Q) \restriction (N \cap Q)}$.

Finally, by Definitions 6.4 and 7.5, 
\begin{multline*}
A_{(p \restriction N) \restriction Q} = 
A_{p \restriction N} \cap Q = 
A_p \cap N \cap Q = \\ 
= (A_p \cap Q) \cap (N \cap Q) = 
A_{p \restriction Q} \cap (N \cap Q) = 
A_{(p \restriction Q) \restriction (N \cap Q)}.
\end{multline*}
\end{proof}

\begin{lemma}
Let $N \in \mathcal X$ be simple, $Q \in \mathcal Y \cap N$ be simple, 
and $p \in D_N \cap D_Q$. 
Suppose that $q \in N \cap D_Q$ and 
$q \le p \restriction N$. 
Then:
\begin{enumerate}
\item $q \oplus_{N} p$ is in $D_Q$;
\item $q \restriction Q \in N \cap Q$ and 
$$
q \restriction Q \le (p \restriction Q) \restriction (N \cap Q).
$$
\end{enumerate}
\end{lemma}

\begin{proof}
(1) Since $q \le p \restriction N$, $q \oplus_N p$ is a condition which 
is below $q$ and $p$. 
We claim that $q \oplus_N p$ is in $D_Q$, which means that for all 
$M \in A_{q \oplus_N p}$, $M \cap Q \in A_{q \oplus_N p}$. 
Now $A_{q \oplus_N p} = A_q \cup A_p$. 
So if $M \in A_{q \oplus_N p}$, then either $M \in A_q$ or $M \in A_p$. 
But $q$ and $p$ are both in $D_Q$, so in the first case, $M \cap Q \in A_q$, 
and in the second case, $M \cap Q \in A_p$. 
In either case, $M \cap Q \in A_q \cup A_p = A_{q \oplus_N p}$.

(2) Since $q$ and $Q$ are in $N$, $q \restriction Q \in N$. 
Also, $q \restriction Q \in Q$, so $q \restriction Q \in N \cap Q$. 
By Lemmas 8.3 and 8.4, $p \restriction N \in D_Q$ and 
$$
(p \restriction N) \restriction Q = (p \restriction Q) \restriction (N \cap Q).
$$
Now $q \le p \restriction N$, so by Lemma 6.6(3), we have that 
$$
q \restriction Q \le (p \restriction N) \restriction Q = 
(p \restriction Q) \restriction (N \cap Q).
$$
\end{proof}

\begin{proposition}
Let $N \in \mathcal X$ be simple, $Q \in \mathcal Y \cap N$ be simple, 
and suppose that $Q \cap \kappa \notin S$. 
Let $p \in D_N \cap D_Q$. 
Assume that $q \in N \cap D_Q$ 
and $q \le p \restriction N$. 
Then 
$$
(q \oplus_N p) \restriction Q = 
(q \restriction Q) \oplus_{N \cap Q} (p \restriction Q). 
$$
\end{proposition}

Note that Lemma 8.5 implies that 
$(q \oplus_N p) \restriction Q$ and 
$(q \restriction Q) \oplus_{N \cap Q} (p \restriction Q)$ are defined.

\begin{proof}
Let 
$$
s := (q \oplus_N p) \restriction Q
$$
and 
$$
t := (q \restriction Q) \oplus_{N \cap Q} (p \restriction Q).
$$
Our goal is to prove that $s = t$. 
The proof will be split into three steps, namely, showing that 
$A_s = A_t$, $f_s = f_t$, and $g_s = g_t$.

\bigskip

By Definitions 6.4 and 7.18, we have that 
\begin{multline*}
A_s = A_{(q \oplus_N p) \restriction Q} = A_{q \oplus_N p} \cap Q = 
(A_q \cup A_p) \cap Q = \\ 
= (A_q \cap Q) \cup (A_p \cap Q) = 
A_{q \restriction Q} \cup A_{p \restriction Q} = 
A_{(q \restriction Q) \oplus_{N \cap Q} (p \restriction Q)} = A_t.
\end{multline*}
So $A_s = A_t$.

\bigskip

We begin the proof of $f_s = f_t$ by showing that 
$\dom(f_t) \subseteq \dom(f_s)$. 
So let $x \in \dom(f_t)$, and we will show that $x \in \dom(f_s)$. 
By Definition 7.11, either 
\begin{enumerate}
\item[(a)] $x \in \dom(f_{q \restriction Q})$, or
\item[(b)] $x \in \dom(f_{p \restriction Q})$, or 
\item[(c)] $x = M \cap \alpha$, where 
$M \in A_{p \restriction Q}$, $N \cap Q \le M$, and 
$\alpha \in (M \cap \dom(f_{q \restriction Q}) \cap S) \setminus 
\dom(f_{p \restriction Q})$.
\end{enumerate}
By Definition 6.4, the domain of $f_s$ is equal to 
$\dom(f_{q \oplus_N p}) \cap Q$. 
So it suffices to show that $x \in \dom(f_{q \oplus_N p})$ and
$x \in Q$. 
We consider cases a, b, and c separately.

(a) $x \in \dom(f_{q \restriction Q}) = \dom(f_q) \cap Q$. 
So $x \in Q$. 
Also, $\dom(f_q) \subseteq \dom(f_{q \oplus_N p})$, so 
$x \in \dom(f_{q \oplus_N p})$.

(b) $x \in \dom(f_{p \restriction Q}) = \dom(f_p) \cap Q$. 
Then $x \in Q$. 
Also, $\dom(f_p) \subseteq \dom(f_{q \oplus_N p})$, so 
$x \in \dom(f_{q \oplus_N p})$.

(c) $x = M \cap \alpha$, where 
$M \in A_{p \restriction Q}$, $N \cap Q \le M$, and 
$\alpha \in (M \cap \dom(f_{q \restriction Q}) \cap S) \setminus 
\dom(f_{p \restriction Q})$. 
Then $M \in A_{p \restriction Q} = A_p \cap Q$ and 
$\alpha \in \dom(f_{q \restriction Q}) = \dom(f_q) \cap Q$. 
In particular, $M$ and $\alpha$ are in $Q$, so $M \cap \alpha = x$ 
is in $Q$.

It remains to show that $M \cap \alpha \in \dom(f_{q \oplus_N p})$. 
By the definition of the domain of $f_{q \oplus_N p}$ in Definition 7.11, 
it suffices to show that $M \in A_p$, $N \le M$, and 
$\alpha \in (M \cap \dom(f_q) \cap S) \setminus \dom(f_p)$. 
We know that $M \in A_p$ from the last paragraph. 
By Lemma 2.29, $N \sim N \cap Q$, and since $N \cap Q \le M$, it 
follows by Lemma 2.18 that $N \le M$. 
In the previous paragraph, we observed that $\alpha \in \dom(f_q) \cap Q$. 
Also, by the choice of $\alpha$, $\alpha \in M \cap S$ and 
$\alpha \notin \dom(f_{p \restriction Q}) = \dom(f_p) \cap Q$. 
Since $\alpha \in Q$, the latter statement implies that 
$\alpha \notin \dom(f_p)$. 
To summarize, 
$\alpha \in (M \cap \dom(f_q) \cap S) \setminus \dom(f_p)$, as required.

This completes the proof that $\dom(f_t) \subseteq \dom(f_s)$. 
To finish the proof that $f_s = f_t$, we show that if 
$x \in \dom(f_s)$, then $x \in \dom(f_t)$ and 
$f_s(x) = f_t(x)$.

Let $x \in \dom(f_s)$. 
Then 
$$
x \in \dom(f_s) = \dom(f_{q \oplus_N p}) \cap Q
$$
and 
$$
f_s(x) = f_{(q \oplus_N p) \restriction Q}(x) = 
f_{q \oplus_N p}(x).
$$
Thus, it suffices to prove that $x \in \dom(f_t)$ and 
$$
f_{q \oplus_N p}(x) = 
f_{(q \restriction Q) \oplus_{N \cap Q} (p \restriction Q)}(x).
$$ 
The proof splits into the seven cases of Definition 7.11 for how 
$f_{q \oplus_N p}(x)$ is defined. 
We remind the reader that $x \in Q$, as just noted.

\bigskip

(1) $x \in \dom(f_q) \setminus S$ and $f_s(x) = f_q(x)$. 
Then 
$$
x \in \dom(f_q) \cap Q = \dom(f_{q \restriction Q}) \subseteq \dom(f_t).
$$
Also, $x \in \dom(f_{q \restriction Q}) \setminus S$ implies that 
$f_t(x)$ is defined as in case 1 of Definition 7.11. 
So $f_t(x) = f_{q \restriction Q}(x) = f_q(x) = f_s(x)$. 

\bigskip

(2) $x = \alpha \in \dom(f_q) \cap S \cap \dom(f_p)$ and 
$f_s(\alpha) = f_q(\alpha) \cup f_p(\alpha)$. 
Then $\alpha \in \dom(f_q) \cap Q = \dom(f_{q \restriction Q}) 
\subseteq \dom(f_t)$ and 
$\alpha \in \dom(f_p) \cap Q = \dom(f_{p \restriction Q})$. 
Hence,
$$
\alpha \in \dom(f_{q \restriction Q}) \cap S \cap \dom(f_{p \restriction Q}),
$$
and therefore $f_t(\alpha)$ is defined as in case 2 of Definition 7.11. 
So 
$$
f_t(\alpha) = f_{q \restriction Q}(\alpha) \cup f_{p \restriction Q}(\alpha) = 
f_q(\alpha) \cup f_p(\alpha) = f_s(\alpha).
$$

\bigskip

(3) $x = \alpha \in (\dom(f_q) \cap S) \setminus \dom(f_p)$ and 
$$
f_s(x) = f_q(\alpha) \cup \{ M \cap \alpha : M \in A_p, \ N \le M, \ 
\alpha \in M \}.
$$
Then $\alpha \in \dom(f_q) \cap Q = \dom(f_{q \restriction Q}) 
\subseteq \dom(f_t)$. 
Also, $\alpha \notin \dom(f_p)$ implies that 
$\alpha \notin \dom(f_p) \cap Q = \dom(f_{p \restriction Q})$.

To summarize, we have that 
$\alpha \in (\dom(f_{q \restriction Q}) \cap S) \setminus 
\dom(f_{p \restriction Q})$, which means that we are in case 3 
in the definition of $f_t(\alpha)$.
So 
$$
f_t(\alpha) = f_{q \restriction Q}(\alpha) \cup 
\{ K \cap \alpha : K \in A_{p \restriction Q}, \ N \cap Q \le K, \ 
\alpha \in K \}.
$$
Since $f_{q \restriction Q}(\alpha) = f_q(\alpha)$, in order to show that 
$f_s(x) = f_t(x)$, the above equations imply that it suffices to show that 
\begin{multline*}
\{ M \cap \alpha : M \in A_p, \ N \le M, \ 
\alpha \in M \} = \\ 
= \{ K \cap \alpha : K \in A_{p \restriction Q}, \ N \cap Q \le K, \ 
\alpha \in K \}.
\end{multline*}

Let $K \cap \alpha$ be in the set on the right, where $K \in A_{p \restriction Q}$, 
$N \cap Q \le K$, and $\alpha \in K$. 
Then $K \in A_{p \restriction Q} = A_p \cap Q$. 
By Lemma 2.29, $N \sim N \cap Q$. 
Since $N \cap Q \le K$, Lemma 2.18 implies that $N \le K$. 
So $K \in A_p$, $N \le K$, and $\alpha \in K$. 
Thus, $K \cap \alpha$ is in the set on the left.

Conversely, let $M \cap \alpha$ be a member of the set on the left, 
where $M \in A_p$, $N \le M$, and $\alpha \in M$. 
We will show that $M \cap \alpha$ is a member of the set on the right. 
It suffices to show that $M \cap \alpha$ is equal to $K \cap \alpha$, 
for some $K \in A_{p \restriction Q}$ such that $N \cap Q \le K$ and 
$\alpha \in K$.

Let $K := M \cap Q$. 
Since $p \in D_Q$ and $M \in A_p$, $M \cap Q \in A_p$. 
As $Q$ is simple, $M \cap Q \in Q$. 
So $K = M \cap Q \in A_p \cap Q = A_{p \restriction Q}$. 
By Lemma 2.29, $M \sim M \cap Q$ and $N \sim N \cap Q$. 
Since $N \le M$, Lemma 2.18 implies that $N \cap Q \le M \cap Q = K$. 
And $\alpha \in Q$, so $\alpha \in M \cap Q = K$. 
It remains to show that $M \cap \alpha = K \cap \alpha$. 
But $\alpha \in Q$ implies that 
$K \cap \alpha = M \cap Q \cap \alpha = M \cap \alpha$.

\bigskip

(4) $x \in \dom(f_p) \setminus N$, and for some 
$\alpha \in \dom(f_p) \cap S \cap N$, 
$N \cap \alpha \in f_p(x) \cup \{ x \}$. 
Then $f_s(x) = f_q(\alpha) \cup f_p(x)$.

We claim that $\alpha \in Q$. 
Suppose for a contradiction that $\alpha \notin Q$. 
Since $\alpha \in S$ and $Q \cap \kappa \notin S$, it follows that 
$Q \cap \kappa < \alpha$. 
As $N \cap \alpha \in f_p(x) \cup \{ x \}$, 
$N \cap \alpha$ is either equal to $x$ or is in $f_p(x)$. 
In particular, since $x \in Q$, we have that 
$$
\sup(N \cap \alpha) \le \sup(x) < Q \cap \kappa < \alpha.
$$
So $\sup(N \cap \alpha) < Q \cap \kappa$.  
But $Q \in N$, so $Q \cap \kappa \in N \cap \alpha$, which contradicts 
that $\sup(N \cap \alpha) < Q \cap \kappa$.
This completes the proof that $\alpha \in Q$. 

Now we argue that we are in case 4 of Definition 7.11 for $f_t(x)$. 
We know that $\alpha$ and $x$ are in 
$\dom(f_p) \cap Q = \dom(f_{p \restriction Q})$. 
So $x \in \dom(f_{p \restriction Q}) \setminus (N \cap Q)$, 
$\alpha \in \dom(f_{p \restriction Q}) \cap S \cap (N \cap Q)$, and 
$N \cap \alpha \in f_p(x) \cup \{ x \}$. 
Since $\alpha \in Q$, $N \cap \alpha = N \cap Q \cap \alpha$, so 
$N \cap Q \cap \alpha \in f_p(x) \cup \{ x \} = 
f_{p \restriction Q}(x) \cup \{ x \}$. 
This proves that 
we are in case 4 in the definition of $f_t(x)$. 
So 
$$
f_t(x) = f_{q \restriction Q}(\alpha) \cup f_{p \restriction Q}(x) = 
f_q(\alpha) \cup f_p(x) = f_s(x).
$$

\bigskip

(5,6) $x \in \dom(f_p) \setminus N$ and case 4 fails. 
Then $x \in \dom(f_p) \cap Q = \dom(f_{p \restriction Q})$, 
so $x \in \dom(f_{p \restriction Q}) \setminus (N \cap Q)$. 
In particular, $x \in \dom(f_t)$.

We claim that case 4 of Definition 7.11 fails for $f_t(x)$. 
Suppose for a contradiction that for some 
$\alpha \in \dom(f_{p \restriction Q}) \cap S \cap (N \cap Q)$, 
we have that 
$(N \cap Q) \cap \alpha \in f_{p \restriction Q}(x) \cup \{ x \}$. 
Then $\alpha \in \dom(f_{p \restriction Q}) = \dom(f_p) \cap Q$. 
Since $\alpha \in Q$, $N \cap Q \cap \alpha = N \cap \alpha$. 
So $\alpha \in \dom(f_p) \cap S \cap N$ and 
$N \cap \alpha \in f_{p \restriction Q}(x) \cup \{ x \} = 
f_p(x) \cup \{ x \}$. 
Hence, we are in case 4 of Definition 7.11 in the definition of 
$f_{q \oplus_N p}(x)$, which is a contradiction.

It follows that we are in 
either case 5 or 6 in the definition of $f_t(x)$. 
Assume that we are in case 5 in the definition of $f_{q \oplus_N p}(x)$. 
Then 
$f_p(x) \cap N = \emptyset$ and $f_s(x) = f_p(x)$. 
It follows that 
$$
f_{p \restriction Q}(x) \cap (N \cap Q) = 
f_p(x) \cap (N \cap Q) \subseteq f_p(x) \cap N = \emptyset.
$$
So we are also in case 5 in the definition of $f_t(x)$. 
Hence, 
$$
f_t(x) = f_{p \restriction Q}(x) = f_p(x) = f_s(x).
$$

Now assume that we are in case 6 in the definition of $f_{q \oplus_N p}(x)$. 
Then $f_s(x) = f_p(x) \cup f_q(M)$, where $M$ is the 
membership largest element of $f_p(x) \cap N$. 
Since $M \in f_p(x)$, it follows that $M \in Sk(x)$. 
And since $x \in Q$, $M \in Q$. 
Thus, $M \in f_p(x) \cap Q = f_{p \restriction Q}(x) \cap Q$. 
Since $M \in N$, $M$ is in $f_{p \restriction Q}(x) \cap (N \cap Q)$. 
In particular, $f_{p \restriction Q}(x) \cap (N \cap Q)$ is nonempty, 
so we are in case 6 in the definition of $f_t(x)$.

We claim that $M$ is the membership largest element of 
$f_{p \restriction Q}(x) \cap (N \cap Q)$. 
Otherwise there is $M_1 \in f_{p \restriction Q}(x) \cap (N \cap Q)$ 
such that $M \in Sk(M_1)$. 
But then 
$$
M_1 \in f_{p \restriction Q}(x) \cap (N \cap Q) = 
f_p(x) \cap N \cap Q \subseteq f_p(x) \cap N.
$$
Since $M$ is the membership largest element of $f_p(x) \cap N$, this 
is a contradiction. 
By the definition of $f_t(x)$, we have that 
$$
f_t(x) = f_{p \restriction Q}(x) \cup f_{q \restriction Q}(M) = 
f_p(x) \cup f_q(M) = f_s(x).
$$

\bigskip

(7) $x = M \cap \alpha$, where $M \in A_p$, 
$N \le M$, and 
$\alpha \in (M \cap \dom(f_q) \cap S) \setminus \dom(f_p)$. 
Then $f_s(x) = f_s(\alpha) \cap Sk(M \cap \alpha)$.

We claim that $M \cap \alpha$ is equal to $K \cap \alpha$, for 
some $K \in A_{p \restriction Q}$ such that $N \cap Q \le K$ and 
$\alpha \in (K \cap \dom(f_{q \restriction Q}) \cap S) \setminus 
\dom(f_{p \restriction Q})$. 
Let $K := M \cap Q$. 
Since $M \in A_p$ and $p \in D_Q$, $M \cap Q \in A_p$. 
As $Q$ is simple, $M \cap Q \in Q$. 
So $K = M \cap Q \in A_p \cap Q = A_{p \restriction Q}$. 
By Lemma 2.29, $M \sim M \cap Q$ and $N \sim N \cap Q$. 
Since $N \le M$, it follows by Lemma 2.18 
that $N \cap Q \le M \cap Q = K$. 

Suppose for a moment that $\alpha \in Q$. 
Then $\alpha \in M \cap Q = K$ and 
$\alpha \in \dom(f_q) \cap Q = \dom(f_{q \restriction Q})$. 
Also, since $\alpha \notin \dom(f_p)$ and 
$\dom(f_{p \restriction Q}) = \dom(f_p) \cap Q$, 
it follows that $\alpha \notin \dom(f_{p \restriction Q})$. 
So $\alpha \in (K \cap \dom(f_{q \restriction Q}) \cap S) \setminus 
\dom(f_{p \restriction Q})$. 
Finally, assuming that $\alpha \in Q$, we have that 
$$
K \cap \alpha = M \cap Q \cap \alpha = M \cap \alpha.
$$

Thus, assuming that $\alpha \in Q$, we have shown that 
$K \cap \alpha = M \cap \alpha$ 
is in the domain of $f_t$ and $f_t(K \cap \alpha)$ is defined as in 
case 7 of Definition 7.11. 
So $f_t(K \cap \alpha) = f_t(\alpha) \cap Sk(K \cap \alpha)$. 
By case 3 handled above, $f_s(\alpha) = f_t(\alpha)$. 
So 
\begin{multline*}
f_t(M \cap \alpha) = 
f_t(K \cap \alpha) = f_t(\alpha) \cap Sk(K \cap \alpha) = \\ 
= f_s(\alpha) \cap Sk(M \cap \alpha) = f_s(M \cap \alpha).
\end{multline*}

It remains to prove that $\alpha \in Q$. 
Suppose for a contradiction that $\alpha \notin Q$. 
Since $p \in D_Q$ and $M \in A_p$, we have that 
$M \cap Q \in A_p$. 
By Lemma 2.29, $M \sim M \cap Q$. 
As $M \cap \alpha = x \in Q$ and $\alpha \notin Q$, clearly 
$\alpha = \min((M \cap \kappa) \setminus (Q \cap \kappa))$. 
By Lemma 3.6, we have that 
$$
\alpha = \min((M \cap \kappa) \setminus \beta_{M \cap Q,M}).
$$
Now $M$ and $M \cap Q$ are both in $A_p$ and $M \sim M \cap Q$. 
Therefore, $\alpha \in r^*(A_p) \cap S$. 
Since $p$ is a condition, $\alpha \in \dom(f_p)$. 
But this contradicts our original choice of $\alpha$.

\bigskip

This completes the proof that $f_s = f_t$.

\bigskip

Now we show that $g_s = g_t$. 
Since the domain of $g_s$ is equal to the set of pairs $(K,x)$ such that 
$K \in f_s(x)$, and the domain of $g_t$ is equal to 
the set of pairs $(K,x)$ such that 
$K \in f_t(x)$, the fact that $f_s = f_t$ implies that 
$\dom(g_s) = \dom(g_t)$.

Let $(K,x) \in \dom(g_s)$, and we will show 
that $g_s(K,x) = g_t(K,x)$. 
So $K \in f_s(x) = f_{q \oplus_N p}(x)$. 
By Definition 6.4, we have that 
$$
g_s(K,x) = g_{(q \oplus_N p) \restriction Q}(K,x) = 
g_{q \oplus_N p}(K,x).
$$
Hence, by Definition 7.16, 
$$
g_s(K,x) = 
\bigcup \{ g_q(K,y) \cup g_p(K,y) : 
x = y, \textrm{or} \ x \in f_{q \oplus_N p}(y) \}.
$$
Also, since $g_{q \restriction Q} = g_q \restriction Q$ and 
$g_{p \restriction Q} = g_p \restriction Q$, we have that 
\begin{multline*}
g_t(K,x) = g_{(q \restriction Q) \oplus_{N \cap Q} (p \restriction Q)}(K,x) = \\ 
= \bigcup \{ g_{q}(K,y) \cup g_p(K,y) : x = y, \textrm{or} \ x \in f_t(y) \}.
\end{multline*}
Therefore, to show that $g_s(K,x) = g_t(K,x)$, it suffices to show 
that for any ordinal $\xi$, the following are equivalent:
\begin{enumerate}
\item $\xi \in g_q(K,y) \cup g_p(K,y)$, for some $y$ such that either $x = y$ or 
$x \in f_{q \oplus_N p}(y)$;
\item $\xi \in g_q(K,y) \cup g_p(K,y)$, for some $y$ such that either $x = y$ or  
$x \in f_t(y)$.
\end{enumerate}
Obviously (1) and (2) are equivalent in the special case that $x = y$.

\bigskip

(2) implies (1): Suppose that $\xi \in g_q(K,y) \cup g_p(K,y)$, where 
$x \in f_t(y)$. 
Then 
$$
x \in f_t(y) = f_s(y) = f_{(q \oplus_N p) \restriction Q}(y) = 
f_{q \oplus_N p}(y).
$$
So $\xi \in g_q(K,y) \cup g_p(K,y)$, where 
$x \in f_{q \oplus_N p}(y)$. 
Hence, (1) holds.

\bigskip

(1) implies (2): Suppose that $\xi \in g_q(K,y) \cup g_p(K,y)$, where 
$x \in f_{q \oplus_N p}(y)$. 
If $y \in Q$, then 
$$
y \in \dom(f_{q \oplus_N p}) \cap Q = 
\dom(f_{(q \oplus_N p) \restriction Q}) = \dom(f_s)
$$
and 
$$
x \in f_{q \oplus_N p}(y) = 
f_{(q \oplus_N p) \restriction Q}(y) = f_s(y) = f_t(y).
$$
Therefore, (2) holds.

The more difficult case is when $y$ is not in $Q$. 
We split the proof into the two cases of whether $\xi$ is in $g_q(K,y)$ 
or in $g_p(K,y)$.

\bigskip

\noindent \emph{Case 1:} $\xi \in g_q(K,y)$. 
Then in particular, $g_q(K,y)$ is nonempty, which implies that 
$K \in f_q(y)$. 
So $K$ and $y$ are in $\dom(f_q)$. 
Since $q \in N$, $K$ and $y$ are in $N$.

We claim that $x \in \dom(f_q)$. 
Assume for a moment that this claim is true, and we finish the proof. 
We have that $K$, $x$, and $y$ are in $\dom(f_q)$, and also, 
$K \in f_s(x) = f_{q \oplus_N p}(x)$ and 
$x \in f_{q \oplus_N p}(y)$. 
Since $q \oplus_N p \le q$, it follows that 
$K \in f_q(x)$ and $x \in f_q(y)$. 
As $q$ is a condition, we have that $g_q(K,y) \subseteq g_q(K,x)$. 
Thus, $\xi \in g_q(K,x)$, which implies that (2) holds, as required.

Suppose for a contradiction that $x$ is not in $\dom(f_q)$. 
If $x \in N$, then $x \in \dom(f_{q \oplus_N p}) \cap N = \dom(f_q)$ 
by Lemma 7.13(1). 
So $x \notin N$. 
Since $x \in f_{q \oplus_N p}(y)$ and $y \in N$, the only way that 
$x$ would not be in $N$ is if $y \in S$. 

So $y = \alpha$, for some $\alpha \in \dom(f_q) \cap S$. 
Since $N \in A_p \subseteq A_{q \oplus_N p}$, we have that 
$N \cap \alpha \in f_{q \oplus_N p}(\alpha)$. 
And since $x \in f_{q \oplus_N p}(\alpha)$ and $x \notin N$, 
$N \cap \alpha$ is either equal to $x$ or is in $f_{q \oplus_N p}(x)$. 
In particular, $\sup(N \cap \alpha) \le \sup(x)$.

Since $s \in Q$ and $x \in \dom(f_s)$, $x \in Q$. 
Hence, $\sup(x) < Q \cap \kappa$. 
Since $y = \alpha$ is not in $Q$, $Q \cap \kappa \le \alpha$. 
But $\alpha \in S$ and $Q \cap \kappa \notin S$, so 
$Q \cap \kappa < \alpha$. 
Therefore, 
$$
\sup(N \cap \alpha) \le \sup(x) < Q \cap \kappa < \alpha.
$$
But $Q \in N$ implies that $Q \cap \kappa \in N \cap \alpha$, which 
contradicts that $\sup(N \cap \alpha) < Q \cap \kappa$.

\bigskip

\noindent \emph{Case 2:}  $\xi \in g_p(K,y)$. 
Then in particular, $g_p(K,y)$ is nonempty, which implies that $K \in f_p(y)$. 
So $K$ and $y$ are in $\dom(f_p)$.

The easier case is when $x \in \dom(f_p)$. 
Then since $K \in f_s(x) = f_{q \oplus_N p}(x)$, 
$x \in f_{q \oplus_N p}(y)$, and $q \oplus_N p \le p$, 
it follows that $K \in f_p(x)$ and $x \in f_p(y)$. 
Since $p$ is a condition, it follows that 
$g_p(K,y) \subseteq g_p(K,x)$. 
So $\xi \in g_p(K,x)$, and (2) holds.

\bigskip

Assume that $x \notin \dom(f_p)$. 
Since $x \in \dom(f_{q \oplus_N p})$, Definition 7.11 implies that 
either $x \in \dom(f_q)$, or 
$x = M \cap \alpha$ for some $M \in A_p$ with 
$N \le M$ and some 
$\alpha \in (M \cap \dom(f_q) \cap S) \setminus \dom(f_p)$.

\bigskip

\noindent \emph{Case 2a:} $x \in \dom(f_q)$. 
Then $x \in N$. 
Since $K \in f_s(x)$ and $x \notin S$, $K \in N$ as well. 
By Lemma 7.8(2), 
$K \in \dom(f_p) \cap N \subseteq \dom(f_q)$. 
So $K \in \dom(f_q)$. 
Since $K$ and $x$ are in $\dom(f_q)$, 
$K \in f_{q \oplus_N p}(x)$, 
and $q \oplus_N p \le q$, it follows that 
$K \in f_q(x)$.

\bigskip

First, assume that $y \in N$. 
By Lemma 7.8(2), $y \in \dom(f_p) \cap N \subseteq \dom(f_q)$. 
So $x$ and $y$ are in $\dom(f_q)$. 
Since $x \in f_{q \oplus_N p}(y)$ 
and $q \oplus_N p \le q$, we have that $x \in f_q(y)$. 
So $K \in f_q(x)$ and $x \in f_q(y)$. 
Since $q$ is a condition, $g_q(K,y) \subseteq g_q(K,x)$. 
On the other hand, by Lemma 7.8(3), 
$(K,y) \in \dom(g_p) \cap N \subseteq \dom(g_q)$ and 
$g_p(K,y) \subseteq g_q(K,y)$. 
Thus, 
$$
\xi \in g_p(K,y) \subseteq g_q(K,y) \subseteq g_q(K,x).
$$ 
So (2) holds.

\bigskip

Secondly, assume that $y \notin N$. 
Let us consider cases 1--7 of Definition 7.11 
in the definition of $f_{q \oplus_N p}(y)$. 
Since $y \notin N$, cases 1, 2, and 3 are false. 
As $K \in f_p(y) \cap N$, case 5 is false. 
Case 7 is false by Lemma 7.9, since $y \in \dom(f_p)$. 
So we are left with cases 4 and 6.

In case 4, $y \in \dom(f_p) \setminus N$, 
and $N \cap \alpha \in f_p(y) \cup \{ y \}$ 
for some $\alpha \in \dom(f_p) \cap S \cap N$. 
Then $f_{q \oplus_N p}(y) = f_q(\alpha) \cup f_p(y)$. 
Since $x \in f_{q \oplus_N p}(y)$ and $x \notin \dom(f_p)$, 
it follows that $x \in f_q(\alpha)$.

Since $K \in f_p(y) \cap N$ and $N \cap \alpha \in f_p(y) \cup \{ y \}$, 
we have that $K \in f_p(N \cap \alpha)$. 
Since $p$ is a condition, $g_p(K,y) \subseteq g_p(K,N \cap \alpha)$. 
As $N \cap \alpha \in f_p(\alpha)$ and $p \in D_N$, it follows that 
$g_p(K,N \cap \alpha) \subseteq g_p(K,\alpha)$. 
So $g_p(K,y) \subseteq g_p(K,\alpha)$. 
Since $K$ and $\alpha$ are in $N$, 
$g_p(K,\alpha) \subseteq g_q(K,\alpha)$ by Lemma 7.8(3). 
Now $K \in f_q(x)$ and $x \in f_q(\alpha)$, so 
$g_q(K,\alpha) \subseteq g_q(K,x)$ since $q$ is a condition. 
Putting it all together, 
$$
\xi \in g_p(K,y) \subseteq g_p(K,\alpha) \subseteq g_q(K,\alpha) 
\subseteq g_q(K,x).
$$
So (2) holds.

In case 6, we have that case 4 fails, and 
$f_{q \oplus_N p}(y) = f_p(y) \cup f_q(M)$, where $M$ 
is the membership largest element of $f_p(y) \cap N$. 
Since $x \in f_{q \oplus_N p}(y)$ and $x \notin \dom(f_p)$, 
we have that $x \in f_q(M)$. 
Also, $K \in f_p(y) \cap N$, so $K \in f_p(M) \cup \{ M \}$. 
But $K \in Sk(x)$ and $x \in f_q(M)$, so $K \ne M$. 
Hence, $K \in f_p(M)$. 
As $p$ is a condition, $g_p(K,y) \subseteq g_p(K,M)$. 
Also, since $K \in f_p(M)$ and $K$ and $M$ are in $N$, 
$K \in f_q(M)$ by Lemma 7.8(2). 
As $x \in f_q(M)$ and $K \in Sk(x)$, $K \in f_q(x)$. 
Since $q$ is a condition, $g_q(K,M) \subseteq g_q(K,x)$. 
And by Lemma 7.8(3), $g_p(K,M) \subseteq g_q(K,M)$. 
Therefore,
$$
\xi \in g_p(K,y) \subseteq g_p(K,M) \subseteq 
g_q(K,M) \subseteq g_q(K,x).
$$
So (2) holds.

\bigskip

\noindent \emph{Case 2b:} $x = M \cap \alpha$ 
for some $M \in A_p$ with $N \le M$ and 
some $\alpha \in (M \cap \dom(f_q) \cap S) \setminus \dom(f_p)$. 
Note that $\alpha \in \dom(f_q)$ implies that $\alpha \in N$. 
We will show, in fact, that this case is impossible.

We claim that $\alpha < Q \cap \kappa$. 
If not, then since $\alpha \in S$ and $Q \cap \kappa \notin S$, 
we have that $Q \cap \kappa < \alpha$. 
Now $M \cap \alpha = x \in Q$, so $\sup(x) < Q \cap \kappa$. 
And $\alpha \in M \cap N \cap \kappa$, so 
$\alpha < \beta_{M,N}$ by Proposition 2.11. 
Therefore, $Q \cap \kappa < \alpha < \beta_{M,N}$. 
Since $N \le M$, we have that $N \cap \beta_{M,N} \subseteq M$. 
So $N \cap \alpha \subseteq M \cap \alpha$. 
Therefore,
$$
\sup(N \cap \alpha) \le \sup(M \cap \alpha) = \sup(x) < Q \cap \kappa < 
\alpha.
$$
So $\sup(N \cap \alpha) < Q \cap \kappa$. 
Since $Q \in N$, we have that 
$Q \cap \kappa \in N \cap \alpha$, which contradicts that 
$\sup(N \cap \alpha) < Q \cap \kappa$.

So indeed, $\alpha < Q \cap \kappa$. 
Since $y \notin Q$, Lemma 4.3(4) implies that 
$Q \cap \kappa \le \sup(y)$. 
Hence, $\alpha < \sup(y)$. 
But we know that $M \cap \alpha = x \in f_{q \oplus_N p}(y)$. 
And by Definition 7.11 and Lemma 4.6, the only value of $y$ for which 
$M \cap \alpha$ is in $f_{q \oplus_N p}(y)$ 
is either $y = \alpha$ or $y = L \cap \alpha$ for some $L$. 
Neither of these cases is possible, since $\alpha < \sup(y)$.
\end{proof}

\bigskip

\addcontentsline{toc}{section}{9. The approximation property}

\textbf{\S 9. The approximation property}

\stepcounter{section}

\bigskip

In this section we will prove that certain quotients of the forcing poset 
$\p$ have the $\omega_1$-approximation property.

\begin{lemma}
Suppose that $Q \in \mathcal Y$ is simple and 
$Q \prec (H(\lambda),\in,\p)$. 
Then $Q \cap \p$ is a regular suborder of $\p$.
\end{lemma}

\begin{proof}
If $p$ and $q$ are in $Q \cap \p$ and are compatible in $\p$, then 
by the elementarity of $Q$, there is $r \in Q \cap \p$ such that $r \le p, q$. 
So $p$ and $q$ are compatible in $Q \cap \p$.

Let $B$ be a maximal antichain of $Q \cap \p$, 
and we will show that $B$ is predense in $\p$. 
So let $p \in \p$, and we will find $s \in B$ which is 
compatible with $p$. 
Since the set $D_Q$ is dense by Lemma 6.3, fix $q \le p$ in $D_Q$. 
Then $q \restriction Q$ exists and is in $Q \cap \p$.

Since $B$ is maximal, we can find $s \in B$ such that 
$s$ and $q \restriction Q$ are compatible in $Q \cap \p$. 
Fix $w \le s, q \restriction Q$ in $Q \cap \p$. 
By Proposition 6.15, $w$ and $q$ are compatible in $\p$. 
Fix $t \le w, q$. 
Then $t \le q \le p$ and $t \le w \le s$. 
So $p$ and $s$ are compatible. 
\end{proof}

\begin{thm}
Suppose that $Q \in \mathcal Y$ is simple, 
$Q \prec (H(\lambda),\in,\p)$, and $Q \cap \kappa \notin S$. 
Then $Q \cap \p$ forces that $\p / \dot G_{Q \cap \p}$ has the 
$\omega_1$-approximation property.\footnote{As 
discussed in the introduction of the paper, the proof of this result is 
similar to the proof of \cite[Lemma 2.22]{mitchell}, 
except that we replace the tidy 
condition property of \cite[Definition 2.20]{mitchell} with Proposition 8.6.}
\end{thm}

\begin{proof}
By Lemma 1.4, 
it suffices to show that $\p$ forces that the pair 
$$
(V[\dot G_\p \cap Q],V[\dot G_\p])
$$
has the $\omega_1$-approximation property. 
So let $p$, $\mu$, and $\dot k$ be given such that $\mu$ is an ordinal, 
and $p$ forces in $\p$ that $\dot k : \mu \to On$ is a function satisfying that 
for any 
countable set $a$ in $V[\dot G_\p \cap Q]$, 
$\dot k \restriction a \in V[\dot G_\p \cap Q]$. 
We will find an extension of $p$ which forces that 
$\dot k$ is in $V[\dot G_\p \cap Q]$.

Fix a regular cardinal $\theta$ large enough so that 
$\p$, $\mu$, and $\dot k$ are members of $H(\theta)$. 
By the stationarity of the simple models in $\mathcal X$ 
as described in Assumption 2.22, we can fix a 
countable set $M^* \prec H(\theta)$ which contains 
the parameters 
$\p$, $Q$, $p$, $\mu$, $\dot k$, and satisfies that the set 
$M^* \cap H(\lambda)$ is in $\mathcal X$ and is simple. 

Let $M := M^* \cap H(\lambda)$. 
Note that since $\p \subseteq H(\lambda)$, 
$M^* \cap \p = M \cap \p$. 
In particular, $p \in M \cap \p$. 
Also, observe that $Q \in M^* \cap H(\lambda) = M$.

By Lemma 7.1, we can fix $p_0 \le p$ such that 
$M \in A_{p_0}$. 
Since $M^* \cap \mu$ is in $V$, by the choice of $p$ and $\dot k$ we can fix 
$p_1 \le p_0$ and 
a $(Q \cap \p )$-name $\dot s$ such that 
$$
p_1 \Vdash_\p \dot k \restriction (M^* \cap \mu) = 
\dot s^{\dot G_\p \cap Q}.
$$ 
Since $M \in A_{p_1}$, apply Lemma 8.2 
to fix $p_2 \le p_1$ such that $p_2 \in D_{M} \cap D_Q$.

Note that since $p_2 \le p$ and $p \in M$, it follows that 
$p_2 \restriction M \le p$ by Lemma 7.7. 
So it will suffice to prove that $p_2 \restriction M$ forces that 
$\dot k$ is in $V[\dot G_\p \cap Q]$.

\bigskip

\noindent \emph{Claim 1:} If $t \le p_2$ is in $D_Q$, 
$\nu \in M^* \cap \mu$, and 
$t \Vdash_{\p} \dot k(\nu) = x$ 
(or $t \Vdash_{\p} \dot k(\nu) \ne x$, respectively) 
then $t \restriction Q \Vdash_{Q \cap \p} \dot s(\nu) = x$ (or 
$t \restriction Q \Vdash_{Q \cap \p} \dot s(\nu) \ne x$, respectively).

\bigskip

We will prove only the main part of Claim 1, since the 
parenthetical part has the essentially the same proof. 
So assume that $t \Vdash_\p \dot k(\nu) = x$. 
If $t \restriction Q \not \Vdash_{Q \cap \p} \dot s(\nu) = x$, 
then there is $u \le t \restriction Q$ in $Q \cap \p$ such that 
$u \Vdash_{Q \cap \p} \dot s(\nu) \ne x$. 
By Proposition 6.15, $u$ and $t$ are compatible in $\p$. 
Fix $z \in \p$ such that $z \le u, t$.

Fix a generic filter $G$ on $\p$ with $z \in G$. 
Then $t \in G$, which implies that $\dot k^G(\nu) = x$. 
Also $p_1 \in G$, which implies that 
$$
\dot k^G \restriction (M^* \cap \mu) = \dot s^{G \cap Q}.
$$
It follows that $\dot s^{G \cap Q}(\nu) = x$. 
But $z \le u$, so $u \in G \cap Q$. 
By the choice of $u$, $\dot s^{G \cap Q}(\nu) \ne x$, 
which is a contradiction. 
This completes the proof of Claim 1.

\bigskip

\noindent \emph{Claim 2:} For all 
$q \le p_2 \restriction M$ in $D_Q$, 
$\nu < \mu$, and $x$, 
$$
q \Vdash_{\p} \dot k(\nu) = x \implies \forall r \in \p 
( ( r \le p_2 \restriction M \land r \le q \restriction Q ) \implies 
(r \Vdash_\p \dot k(\nu) = x)).
$$

\bigskip

Note that $p_2 \restriction M$, $Q$, 
$D_Q$, $\mu$, $\dot k$, and $\p$ are in $M^*$. 
So by the elementarity of $M^*$, it suffices to show that the statement 
holds in $M^*$.

Suppose for a contradiction that there exists 
$q \le p_2 \restriction M$ 
in $M^* \cap D_Q$, $\nu \in M^* \cap \mu$, and $x \in M^*$ such that 
$$
q \Vdash_\p \dot k(\nu) = x,
$$
but there is $r_0 \in M^* \cap \p$ with 
$r_0 \le p_2 \restriction M$ and 
$r_0 \le q \restriction Q$ such that 
$$
r_0 \not \Vdash_\p \dot k(\nu) = x.
$$
By the elementarity of $M^*$, 
we can fix $r \le r_0$ in $M^* \cap D_Q$ such that 
$$
r \Vdash_\p \dot k(\nu) \ne x.
$$
Then $r \le p_2 \restriction M$ and $r \le q \restriction Q$. 
Since $r \le q \restriction Q$, it follows that 
$r \restriction Q \le q \restriction Q$ 
by Lemma 6.6(2). 

We have that $q \le p_2 \restriction M$ and 
$q \in M^* \cap \p = M \cap \p$. 
By Proposition 7.19, $q \oplus_{M} p_2$ is a condition which is 
below $q$ and $p_2$. 
Similarly, $r \le p_2 \restriction M$ and $r \in M \cap \p$. 
By Proposition 7.19, $r \oplus_{M} p_2$ is a condition which is below 
$r$ and $p_2$. 

By Proposition 8.6, we have that 
$$
(q \oplus_{M} p_2) \restriction Q = 
(q \restriction Q) \oplus_{M \cap Q} (p_2 \restriction Q),
$$
and 
$$
(r \oplus_{M} p_2) \restriction Q = 
(r \restriction Q) \oplus_{M \cap Q} (p_2 \restriction Q).
$$

We would like to apply Lemma 8.1 to $M \cap Q$, 
$p_2 \restriction Q$, $q \restriction Q$, and 
$r \restriction Q$. 
Let us check that the assumptions of Lemma 8.1 hold for these objects. 
By Lemma 2.30, $M \cap Q$ is simple. 
Since $p_2 \in D_{M} \cap D_Q$, 
it follows that $p_2 \restriction Q \in D_{M \cap Q}$ by Lemma 8.3. 
As $q$, $r$, and $Q$ are in $M^*$, we have that 
$q \restriction Q$ and $r \restriction Q$ are in 
$M^* \cap Q = M \cap Q$. 
We observed above that $r \restriction Q \le q \restriction Q$. 
Finally, $q \le p_2 \restriction M$ implies that 
$q \restriction Q \le (p_2 \restriction Q) \restriction (M \cap Q)$ 
by Lemma 8.5.

Thus, all of the assumptions of Lemma 8.1 hold. 
So by Lemma 8.1,
$$
(r \restriction Q) \oplus_{M \cap Q} (p_2 \restriction Q) \le 
(q \restriction Q) \oplus_{M \cap Q} (p_2 \restriction Q).
$$
Combining this with the equalities above, we have that 
$$
(r \oplus_{M} p_2) \restriction Q \le (q \oplus_{M} p_2) \restriction Q.
$$

We claim that this last inequality is impossible. 
In fact, we will show that $(r \oplus_{M} p_2) \restriction Q$ and 
$(q \oplus_{M} p_2) \restriction Q$ are incompatible. 
This contradiction will complete the proof of Claim 2.

We know that $r \Vdash_\p \dot k(\nu) \ne x$, and therefore, since 
$r \oplus_{M} p_2 \le r$, 
$r \oplus_{M} p_2 \Vdash_\p \dot k(\nu) \ne x$. 
By Lemma 8.5(1), $r \oplus_M p_2$ is in $D_Q$. 
So by Claim 1, 
$$
(r \oplus_{M} p_2) \restriction Q \Vdash_{Q \cap \p} \dot s(\nu) \ne x.
$$
Similarly, $q \Vdash_\p \dot k(\nu) = x$, and therefore, since 
$q \oplus_{M} p_2 \le q$, 
$q \oplus_{M} p_2 \Vdash_\p \dot k(\nu) = x$. 
By Lemma 8.5(1), $q \oplus_M p_2$ is in $D_Q$. 
So by Claim 1,
$$
(q \oplus_{M} p_2) \restriction Q \Vdash_{Q \cap \p} \dot s(\nu) = x.
$$
So indeed, 
$(r \oplus_{M} p_2) \restriction Q$ and 
$(q \oplus_{M} p_2) \restriction Q$ are incompatible, since they 
force contradictory information.

\bigskip

Now we finish the proof that $p_2 \restriction M$ forces that 
$\dot k \in V[\dot G_\p \cap Q]$. 
Let $G$ be a generic filter on $\p$ with $p_2 \restriction M \in G$, 
and we will show that $\dot k^G \in V[G \cap Q]$.

In the model $V[G \cap Q]$, 
define a partial function $h : \mu \to V$ by letting, for 
every $\nu < \mu$, $h(\nu) = x$ iff 
there exists $t \in G \cap Q$ such that for every $r \in \p$, 
if $r$ is below both $p_2 \restriction M$ and $t$, 
then 
$$
r \Vdash_\p^V \dot k(\nu) = x. 
$$
We claim that $h = \dot k^G$. 

First, let us show that $h$ is well-defined. 
So assume that $t_0$ and $t_1$ are in $G \cap Q$ and 
witness respectively that $h(\nu) = x_0$ and $h(\nu) = x_1$. 
We will show that $x_0 = x_1$. 
Fix $u \le t_0, t_1$ in $G \cap Q$. 
Since $p_2 \restriction M$ is in $G$, 
we can fix $r \le u, p_2 \restriction M$ in $G$. 
Now $r$ is below both $p_2 \restriction M$ and $t_0$, 
so by the choice of $t_0$, 
$r \Vdash_\p^V \dot k(\nu) = x_0$. 
Similarly, $r$ is below both $p_2 \restriction M$ and $t_1$, 
so by the choice of $t_1$, 
$r \Vdash_\p^V \dot k(\nu) = x_1$. 
Thus, $r \Vdash_\p^V x_0 = x_1$, which implies that $x_0 = x_1$.

Secondly, we prove that $h = \dot k^G$. 
As just shown, 
$h$ is a well-defined function whose domain is a subset of $\mu$. 
So it suffices to show that for all $\nu < \mu$, 
$\nu \in \dom(h)$ and $h(\nu) = \dot k^G(\nu)$. 
Fix $\nu < \mu$, and let $x := \dot k^G(\nu)$.
Fix $q \le p_2 \restriction M$ in $G \cap D_Q$ 
such that $q \Vdash_\p \dot k(\nu) = x$. 
Then by Claim 2, for all $r \in \p$, if $r$ is an extension of 
both $p_2 \restriction M$ and $q \restriction Q$, then 
$r \Vdash_\p \dot k(\nu) = x$.

Let $t := q \restriction Q$. 
By Lemma 6.5, $q \le t$. 
Since $q \in G$, it follows that $t \in G \cap Q$. 
So by the definition of $h$, in order 
to show that $h(\nu) = x$, it suffices 
to show that for all $r \in \p$, if 
$r$ is an extension of both $p_2 \restriction M$ and $t = q \restriction Q$, 
then $r \Vdash_\p^V \dot k(\nu) = x$. 
But this statement is exactly what we observed to be true 
at the end of the previous paragraph.
\end{proof}

Recall that in Theorem 9.2 we assumed that $Q \cap \kappa \notin S$. 
Suppose, on the other hand, that $Q \in \mathcal Y$ is simple, 
$Q \prec (H(\lambda),\in,\p)$, and $Q \cap \kappa \in S$. 
Let $G$ be a generic filter on $\p$. 
Let $\langle c_\alpha : \alpha \in S \rangle$ be the partial square sequence 
in $V[G]$ as defined in Section 5. 
Since $Q \cap \kappa \in S$, we have that 
$c_{Q \cap \kappa}$ is defined. 
Using the coherence property of the partial square sequence, 
one can show that every proper initial segment of $c_{Q \cap \kappa}$ 
is in $V[G \cap Q]$. 
But by a density argument, $c_{Q \cap \kappa}$ is not in $V[G \cap Q]$. 
Thus, the quotient 
$\p / (G \cap Q)$ does not have the $\omega_1$-approximation 
property in $V[G \cap Q]$.

\bigskip

\part{Combining Forcings}

\bigskip

\addcontentsline{toc}{section}{10. A product forcing}

\textbf{\S 10. A product forcing}

\stepcounter{section}

\bigskip

We now develop a forcing poset which simultaneously adds partial 
square sequences on multiple stationary subsets of $\kappa$. 
This forcing poset will be a kind of side condition product forcing. 
Before we get started, we need to make some additional assumptions.

\begin{assumption}
The cardinal $\lambda$ introduced in Notation 2.1 is at least $\kappa^+$.
\end{assumption}

\begin{notation}
Fix an ordinal $\lambda^* \le \lambda$. 
Fix a sequence $\langle S_i : i < \lambda^* \rangle$ 
such that for all $i < \lambda^*$, $S_i \subseteq \Lambda$ is stationary 
in $\kappa$ and 
for all $\alpha \in S_i$, $Sk(\alpha) \cap \kappa = \alpha$; 
moreover, for all $i < j < \lambda^*$, $S_i \cap S_j$ is nonstationary. 
\end{notation}

\begin{assumption}
For all $i < j < \lambda^*$, there is a club set $C_{i,j} \subseteq \kappa$ 
satisfying that $S_i \cap S_j \cap C_{i,j} = \emptyset$, and moreover, 
$C_{i,j}$ is definable in the structure $\mathcal A$ of Notation 2.3 
from $i$ and $j$.
\end{assumption}

\begin{notation}
For each $i < \lambda^*$, let $\p_i$ denote the forcing poset defined 
in Definition 4.2 which adds a partial square sequence on $S_i$.
\end{notation}

We introduce a side condition product forcing which combines 
the forcing posets $\p_i$, for all $i < \lambda^*$. 
Before giving the definition, we need the following technical lemma to 
make sure that the definition makes sense.

\begin{lemma}
Suppose that $M$ and $N$ are in $\mathcal X$, $M \le N$, and 
$\gamma = \min((M \cap \kappa) \setminus \beta_{M,N})$. 
Then there is at most one ordinal $i < \lambda^*$ 
such that $i \in M \cap N$ and $\gamma \in S_i$.
\end{lemma}

\begin{proof}
Suppose for a contradiction that there are 
$i < j$ in $M \cap N \cap \lambda^*$ 
such that $\gamma \in S_i \cap S_j$. 
By Assumption 10.3, $S_i \cap S_j \cap C_{i,j} = \emptyset$, and 
by the elementarity of $M$ and $N$, $C_{i,j} \in M \cap N$. 
By the elementarity of $M \cap N$, it is easy to show that 
$C_{i,j}$ is cofinal in $M \cap N \cap \kappa$. 
Since $M \le N$, by Lemma 2.15 and the minimality of $\gamma$, 
$$
M \cap N \cap \kappa = M \cap \beta_{M,N} = M \cap \gamma.
$$
Therefore, $C_{i,j}$ is cofinal in $M \cap \gamma$. 
Since $C_{i,j} \in M$, 
by the elementarity of $M$ it follows that 
$C_{i,j}$ is cofinal in $\gamma$. 
As $C_{i,j}$ is a club, $\gamma \in C_{i,j}$. 
But then $\gamma \in S_i \cap S_j \cap C_{i,j}$, contradicting 
the choice of $C_{i,j}$.
\end{proof}

\begin{definition}
Let $\q$ be the forcing poset consisting of pairs $p = (F_p,A_p)$ satisfying:
\begin{enumerate}
\item $A_p$ is an adequate set;

\item $F_p$ is a function whose domain is a finite subset of $\lambda^*$;

\item for all $i \in \dom(F_p)$, $F_p(i) \in \p_i$ and 
$$
\{ M \in A_p : i \in M \} \subseteq A_{F_p(i)};\footnote{Recall from 
Definition 4.2 that $F_p(i) = (f_{F_p(i)},g_{F_p(i)},A_{F_p(i)})$.}
$$
\item if $M$ and $N$ are in $A_p$, $M \sim N$, 
$i \in M \cap N \cap \lambda^*$, 
and the ordinal $\min((M \cap \kappa) \setminus \beta_{M,N})$ 
exists and is in $S_i$, then $i \in \dom(F_p)$.
\end{enumerate}
Let $q \le p$ if $A_p \subseteq A_q$, $\dom(F_p) \subseteq \dom(F_q)$, 
and for all $i \in \dom(F_p)$, $F_q(i) \le F_p(i)$ in $\p_i$. 
\end{definition}

Note that since $A_p$ is finite, Lemma 10.5 implies that there are only finitely many ordinals $i$ as described in (4).

Let us see that we can add any ordinal in $\lambda^*$ to the domain of $F_p$.

\begin{definition}
Let $p \in \q$, and let $x$ be a finite subset of $\lambda^* \setminus \dom(F_p)$. 
Define 
$$
p \uplus x
$$
as the pair $(F,A)$ satisfying:
\begin{enumerate}
\item $A := A_p$;
\item $\dom(F) := \dom(F_p) \cup x$;
\item for all $j \in \dom(F_p)$, $F(j) := F_p(j)$, and for all $i \in x$, 
$$
F(i) := (\emptyset,\emptyset,B_i),
$$
where 
$$
B_i := \{ M \in A_p : i \in M \}.
$$
\end{enumerate}
\end{definition}

\begin{lemma}
Let $p \in \q$, and let $x$ be a finite subset of 
$\lambda^* \setminus \dom(F_p)$. Then:
\begin{enumerate}
\item $p \uplus x \in \q$;
\item $p \uplus x \le p$;
\item $x \subseteq \dom(F_{p \uplus x})$;
\item whenever $q \le p$ and $x \subseteq \dom(F_q)$, then 
$q \le p \uplus x$.
\end{enumerate}
\end{lemma}

\begin{proof}
(3) is immediate. 

\bigskip

(1) To see that $p \uplus x$ is a condition, 
requirements (1), (2), and (4) of Definition 10.6 are immediate. 
For requirement (3), for all $j \in \dom(F_p)$, 
$F_{p \uplus x}(j) = F_p(j) \in \p_j$ and 
$\{ M \in A_p : j \in M \} \subseteq A_{F_p(j)} = A_{F_{p \uplus x}(j)}$, 
since $p$ is a condition. 
Consider $i \in x$. 
Then by definition, 
$F_{p \uplus x}(i) = (\emptyset,\emptyset,B_i)$, where 
$B_i = \{ M \in A_p : i \in M \}$. 
Thus, it suffices to show that $(\emptyset,\emptyset,B_i) \in \p_i$.

We check that $(\emptyset,\emptyset,B_i)$ 
satisfies properties (1)--(7) of Definition 4.2. 
(1) $B_i$ is adequate, because it is a subset of $A_p$. 
(2)--(6) are vacuously true. 
We claim that (7) is vacuously true as well.

Suppose that $\gamma \in r^*(B_i) \cap S_i$. 
Then for some $M$ and $N$ in $B_i$, $M \sim N$ and 
$\gamma = \min((M \cap \kappa) \setminus \beta_{M,N})$. 
By the definition of $B_i$, $M$ and $N$ are in $A_p$ and 
$i \in M \cap N \cap \lambda^*$. 
So $M$ and $N$ are in $A_p$, $M \sim N$, 
$i \in M \cap N \cap \lambda^*$, 
and $\min((M \cap \kappa) \setminus \beta_{M,N})$ exists 
and is in $S_i$. 
Since $p$ is a condition, Definition 10.6(4) implies that 
$i \in \dom(F_p)$. 
But $i \in x$ and $x \cap \dom(F_p) = \emptyset$, which is a contradiction.

\bigskip

(2) It is trivial to check that $p \uplus x \le p$.

\bigskip

(4) Assume that $q \le p$ and $x \subseteq \dom(F_q)$, and 
we will show that $q \le p \uplus x$. 
Since $q \le p$, $A_{p \uplus x} = A_p \subseteq A_q$ and 
$\dom(F_p) \subseteq \dom(F_q)$. 
Since $x \subseteq \dom(F_q)$, $\dom(F_{p \uplus x}) = 
\dom(F_p) \cup x \subseteq \dom(F_q)$.

Let $j \in \dom(F_{p \uplus x})$, and we will show that 
$F_q(j) \le F_{p \uplus x}(j)$. 
If $j \in \dom(F_p)$, then by definition, 
$F_{p \uplus x}(j) = F_p(j)$. 
And since $q \le p$, $F_q(j) \le F_p(j)$. 
So $F_q(j) \le F_{p \uplus x}(j)$. 

Suppose that $i \in x$. 
We claim that $F_q(i) \le F_{p \uplus x}(i)$, that is, 
$F_q(i) \le (\emptyset,\emptyset,B_i)$, 
where $B_i = \{ M \in A_p : i \in M \}$. 
We verify properties (a)--(d) of Definition 4.2. 
Note that (b), (c), and (d) are vacuously true. 
For (a), since $i \in \dom(F_q)$, by Definition 10.6(3) we have that 
$\{ M \in A_q : i \in M \} \subseteq A_{F_q(i)}$. 
But $q \le p$ implies that $A_p \subseteq A_q$. 
Hence, $B_i = \{ M \in A_p : i \in M \} \subseteq 
\{ M \in A_q : i \in M \} \subseteq A_{F_q(i)}$. 
\end{proof}

The next two easy lemmas will be useful in what follows.

\begin{lemma}
Let $p \in \q$. 
For each $i \in \dom(F_p)$, suppose that 
$r_i$ is a condition in $\p_i$ such that 
$r_i \le F_p(i)$ in $\p_i$. 
Define $r$ as follows:
\begin{enumerate}
\item $A_r := A_p$;
\item $\dom(F_r) := \dom(F_p)$;
\item for all $i \in \dom(F_r)$, $F_r(i) := r_i$.
\end{enumerate}
Then $r \in \q$ and $r \le p$.
\end{lemma}

\begin{proof}
Straightforward.
\end{proof}

\begin{lemma}
Let $x$ be a finite subset of $\lambda^*$, and assume that for each 
$i \in x$, $D_i$ is a dense subset of $\p_i$. 
Then for any $p \in \q$, there is $r \le p$ satisfying:
\begin{enumerate}
\item $A_r = A_p$;
\item $\dom(F_r) = \dom(F_p) \cup x$;
\item for each $i \in x$, $F_r(i) \in D_i$.
\end{enumerate}
\end{lemma}

\begin{proof}
Let $q := p \uplus (x \setminus \dom(F_p))$. 
Then $q$ is a condition, $A_q = A_p$, 
$\dom(F_q) = \dom(F_p) \cup x$, and $q \le p$.

For each $i \in x$, fix $r_i \le F_q(i)$ in $D_i$. 
By Lemma 10.9, there is $r \le q$ such that 
$A_r = A_q = A_p$, $\dom(F_r) = \dom(F_q) = \dom(F_p) \cup x$, 
and for all $i \in \dom(F_q)$, 
if $i \in x$ then $F_r(i) = r_i$, and if $i \notin x$ then 
$F_r(i) = F_q(i)$. 
Then $r$ is as required.
\end{proof}

The next result justifies our informal 
use of the word ``product'' to describe $\q$.

\begin{proposition}
For each $i < \lambda^*$, 
there is a projection mapping from a dense subset of $\q$ into $\p_i$.
\end{proposition}

\begin{proof}
Let $D$ be the set of conditions $q \in \q$ such that $i \in \dom(F_q)$, 
together with the maximum condition $(\emptyset,\emptyset)$. 
By Lemma 10.8, if $p \in \q$, then there is $q \le p$ such that $i \in \dom(F_q)$. 
Thus, $D$ is dense in $\q$.

Define $\pi_i : D \to \p_i$ as follows. 
Let $\pi_i(\emptyset,\emptyset)$ be the maximum condition 
of $\p_i$, namely, $(\emptyset,\emptyset,\emptyset)$. 
If $q \in D$ and $q$ is not the maximum condition, then $i \in \dom(F_q)$. 
In that case, let $\pi_i(q) := F_q(i)$.

We claim that $\pi_i$ is a projection mapping. 
Obviously, $\pi_i$ maps the maximum condition of $\q$ to the 
maximum condition of $\p_i$.

Suppose that $q \le p$ in $D$, 
and we will show that $\pi_i(q) \le \pi_i(p)$ in $\p_i$. 
This is immediate if $p$ is the maximum condition of $\q$, so 
assume that $i \in \dom(F_p)$. 
Then since $q \le p$, we have that 
$\pi_i(q) = F_q(i) \le F_p(i) = \pi_i(p)$.

Suppose that $v \le \pi_i(p)$ in $\p_i$, and we will find 
$r \le p$ in $D$ such that $\pi_i(r) = v$. 
First, assume that $p$ is not the maximum condition of $\q$. 
Then $v \le \pi_i(p) = F_p(i)$. 
By Lemma 10.9, there exists $r \le p$ satisfying that 
$F_r(i) = v$ and $F_r(j) = F_p(j)$ for all 
$j \in \dom(F_p) \setminus \{ i \}$. 
Then $\pi_i(r) = v$, as required.

Secondly, assume that $p$ is the maximum condition of $\q$. 
We construct a condition $r$ as follows. 
Let $A_r := \emptyset$, and let $F_r$ be the function with domain 
equal to $\{ i \}$ such that $F_r(i) = v$. 
Since $A_r = \emptyset$, it is easy to check that $r$ is a condition, 
with most properties of Definition 10.6 being vacuously true. 
Also, $r \le p$, since $p$ is the maximum condition, 
and $\pi_i(r) = F_r(i) = v$.
\end{proof}

It follows that if $G$ is a generic filter on $\q$, then 
$\pi_i[G \cap D]$ generates a generic filter on $\p_i$, where $D$ is 
the dense subset of $\q$ which is the domain of $\pi_i$. 
We will prove in Section 11 that $\q$ preserves $\omega_1$ and is 
$\kappa$-c.c. 
It follows by Corollary 7.22 
that $\q$ collapses $\kappa$ to become $\omega_2$. 
And by Proposition 5.4, $\q$ adds a partial square sequence on $S_i$, 
for all $i < \lambda^*$.  
Since $\q$ is $\kappa$-c.c., it also preserves the stationarity of $S_i$. 
Hence, for all $i < \lambda^*$, $\q$ forces that $S_i$ is a stationary 
subset of $\omega_2 \cap \cof(\omega_1)$ in the approachability 
ideal $I[\omega_2]$. 
See Corollary 11.22 below for more details.

\bigskip

We conclude this section by introducing a set $s^*$ which is analogous 
to the set $r^*$ from Parts I and II.

\begin{definition}
Let $A$ be an adequate set. 
Define $s^*(A)$ as the set of $i < \lambda^*$ 
such that for some $M$ and $N$ in $A$, 
$M \sim N$, $i \in M \cap N$, and the ordinal 
$\min((M \cap \kappa) \setminus \beta_{M,N})$ exists and is in $S_i$.
\end{definition}

Note that requirement (4) of Definition 10.6 is equivalent to the 
statement that $s^*(A_p) \subseteq \dom(F_p)$.

Observe that if $A \subseteq B$, then $s^*(A) \subseteq s^*(B)$.

The following is an analogue of Propositions 3.5 and 3.8 
for $s^*$. 

\begin{proposition}
Let $A$ be an adequate set, $N \in \mathcal X \cup \mathcal Y$ be simple, 
and suppose that:
\begin{enumerate}
\item if $N \in \mathcal X$, then $N \in A$ and for all $M \in A$, 
if $M < N$ then $M \cap N \in A$;
\item if $N \in \mathcal Y$, then for all $M \in A$, $M \cap N \in A$.
\end{enumerate}
Let $B$ be an adequate set such that 
$$
A \cap N \subseteq B \subseteq N.
$$
Then 
$$
s^*(A \cup B) = s^*(A) \cup s^*(B).
$$
\end{proposition}

\begin{proof}
By Propositions 2.25 and 2.28, $A \cup B$ is adequate. 
The reverse inclusion is immediate. 
For the forward inclusion, 
let $K \in A$ and $M \in B$, and assume that 
$K \sim M$, $i \in K \cap M \cap \lambda^*$, $\gamma$ is equal to 
either $\min((K \cap \kappa) \setminus \beta_{K,M})$ or 
$\min((M \cap \kappa) \setminus \beta_{K,M})$, 
and $\gamma \in S_i$. 
We will show that $i \in s^*(A) \cup s^*(B)$.

We claim that $K \cap N \in B$. 
First, assume that $N \in \mathcal X$. 
Since $M \in N$, $M \cap \omega_1 < N \cap \omega_1$. 
As $K \sim M$, $K \cap \omega_1 = M \cap \omega_1$ by Lemma 2.17(2). 
So $K \cap \omega_1 < N \cap \omega_1$, and hence $K < N$ 
by Lemma 2.17(1). 
By (1), $K \cap N \in A$. 
As $N$ is simple, $K \cap N \in N$. 
So $K \cap N \in A \cap N \subseteq B$. 
Secondly, assume that $N \in \mathcal Y$. 
Then since $K \in A$, $K \cap N \in A$ by (2). 
As $N$ is simple, $K \cap N \in N$. 
So $K \cap N \in A \cap N \subseteq B$.

Since $M \in N$, $i \in N$. 
So $i \in (K \cap N) \cap M$. 
Also note that by Lemmas 2.18, 2.26, and 2.29, 
$K \sim K \cap N \sim M$.

We consider the two possibilities for $\gamma$. 
Suppose that $\gamma = \min((K \cap \kappa) \setminus \beta_{K,M})$. 
Then by Lemma 3.4(3) in the case that $N \in \mathcal X$ 
and Lemma 3.7(3) in the case that $N \in \mathcal Y$, either 
$\gamma = \min((K \cap N \cap \kappa) \setminus \beta_{K \cap N,M})$, 
or $\gamma = \min((K \cap \kappa) \setminus \beta_{K,K \cap N})$. 
In the first case, since $\gamma \in S_i$, $i \in (K \cap N) \cap M$, 
and $K \cap N$ and $M$ are in $B$, it follows that $i \in s^*(B)$. 
In second case, since $\gamma \in S_i$, $i \in (K \cap N) \cap K$, 
and $K \cap N$ and $K$ are in $A$, we have that $i \in s^*(A)$.

Now suppose that $\gamma = \min((M \cap \kappa) \setminus \beta_{K,M})$. 
Then by Lemma 3.4(2) in the case that $N \in \mathcal X$ and 
Lemma 3.7(2) in the case that $N \in \mathcal Y$, 
$\gamma = \min((M \cap \kappa) \setminus \beta_{M,K \cap N})$. 
Since $\gamma \in S_i$, $i \in (K \cap N) \cap M$, 
and $K \cap N$ and $M$ are in $B$, it follows that $i \in s^*(B)$.
\end{proof}

\bigskip

\addcontentsline{toc}{section}{11. Amalgamation}

\textbf{\S 11. Amalgamation}

\stepcounter{section}

\bigskip

In this section we will prove 
cardinal preservation results for $\q$, namely, that 
$\q$ is strongly proper on a stationary set and is $\kappa$-c.c. 
The arguments are simpler than those from Sections 6 and 7, and as a result 
we are able to handle the amalgamation arguments 
for countable and uncountable models 
at the same time. 
The order of topics and results is similar to that of Sections 6 and 7.

\begin{lemma}
Let $p \in \q$, $N \in \mathcal X$, and suppose that $p \in N$. 
Then there is $r \le p$ such that $N \in A_r$.
\end{lemma}

\begin{proof}
Since $p \in N$, for all $i \in \dom(F_p)$, 
$F_p(i) \in N$. 
So we can apply Lemma 7.1 and fix, for each $i \in \dom(F_p)$, 
a condition $q_i \le F_p(i)$ in $\p_i$ such that 
$N \in A_{q_i}$. 
Now apply Lemma 10.9 and fix $q \le p$ such that 
$A_q = A_p$, $\dom(F_q) = \dom(F_p)$, and 
for each $i \in \dom(F_p)$, $F_q(i) = q_i$.

Define $r$ by letting $F_r := F_q$ and $A_r := A_q \cup \{ N \}$. 
It is easy to see that if $r$ is a condition, then $r \le q$ and $N \in A_r$. 
So we will be done if we can prove that $r$ is a condition. 
We verify requirements (1)--(4) of Definition 10.6. 

(1) Since $A_p = A_q$ and $p \in N$, 
$A_r$ is adequate by Lemma 2.16. 
(2) is immediate. 
(4) Since $M \in N$ for all $M \in A_q$, easily 
$s^*(A_r) = s^*(A_q \cup \{ N \}) = s^*(A_q)$. 
Since $q$ is a condition, 
$s^*(A_q) \subseteq \dom(F_q) = \dom(F_r)$.

(3) Let $i \in \dom(F_r)$. 
Then $F_r(i) = F_q(i) = q_i$, which is in $\p_i$. 
Let $M \in A_r = A_q \cup \{ N \}$ and suppose that $i \in M$. 
We will show that $M \in A_{q_i}$. 
Since $M \in A_r$, either $M \in A_q$ or $M = N$. 

First, assume that $M = N$. 
Then by the choice of $q_i$, $M = N \in A_{q_i}$, and we are done. 
Secondly, assume that $M \in A_q$. 
Now $A_q = A_p$ and $\dom(F_r) = \dom(F_q) = \dom(F_p)$. 
So $M \in A_p$ and $i \in \dom(F_p)$. 
Since $p$ is a condition, $M \in A_{F_p(i)}$. 
But $q_i \le F_p(i)$, so $A_{F_p(i)} \subseteq A_{q_i}$. 
Hence, $M \in A_{q_i}$.
\end{proof}

\begin{lemma}
Let $p \in \q$ and $N \in \mathcal X \cup \mathcal Y$. 
Suppose that if $N \in \mathcal X$, then $N \in A_p$. 
Then there is $s \le p$ satisfying:
\begin{enumerate}
\item if $N \in \mathcal X$, then for all $M \in A_s$, if 
$M < N$ then $M \cap N \in A_s$, and moreover, 
$A_s = A_p \cup \{ M \cap N : M \in A_p, \ M < N \}$;
\item if $N \in \mathcal Y$, then for all $M \in A_s$, 
$M \cap N \in A_s$, and moreover, 
$A_s = A_p \cup \{ M \cap N : M \in A_p \}$.
\end{enumerate}
\end{lemma}

\begin{proof}
Define 
$$
B := A_p \cup \{ M \cap N : M \in A_p, \ M < N \}
$$
in the case that $N \in \mathcal X$, and 
$$
B := A_p \cup \{ M \cap N : M \in A_p \}
$$
in the case that $N \in \mathcal Y$. 
By Propositions 2.24 and 2.27, $B$ is adequate. 
Define 
$$
q := p \uplus (s^*(B) \setminus \dom(F_p)).
$$
By Definition 10.7 and Lemma 10.8, 
$q \in \q$, $q \le p$, $A_q = A_p$, and 
$\dom(F_q) = \dom(F_p) \cup s^*(B)$.

If $N \in \mathcal X$, then since $N \in A_q$, 
Definition 10.6(3) implies that 
for all $i \in \dom(F_q) \cap N$, $N \in A_{F_q(i)}$. 
Applying Lemma 7.2 in the case that $N \in \mathcal X$ and 
Lemma 6.1 in the case that $N \in \mathcal Y$, we can fix, 
for each $i \in \dom(F_q) \cap N$, 
a condition $r_i \le F_q(i)$ in $\p_i$ 
satisfying:
\begin{enumerate}
\item if $N \in \mathcal X$, then for all 
$M \in A_{r_i}$, if $M < N$ then $M \cap N \in A_{r_i}$;
\item if $N \in \mathcal Y$, then for all $M \in A_{r_i}$, 
$M \cap N \in A_{r_i}$.
\end{enumerate}
Now apply Lemma 10.9 and fix $r \le q$ such that $A_r = A_q$, 
$\dom(F_r) = \dom(F_q)$, for each $i \in \dom(F_q) \cap N$, 
$F_r(i) = r_i$, and for each $i \in \dom(F_q) \setminus N$, 
$F_r(i) = F_q(i)$.

Finally, define 
$$
s := (F_r, B).
$$
We claim that $s$ is as required. 
Note that if $s$ is a condition, then clearly $s \le r$. 
Also, by Propositions 2.24 and 2.27, since $A_p = A_q = A_r$, 
$A_s = B$ satisfies (1) and (2) of the lemma.

It remains to show that $s$ is a condition. 
We verify requirements (1)--(4) of Definition 10.6. 
(1) We already observed that $B$ is adequate. 
(2) is immediate.
(4) We have that 
$$
s^*(A_s) = s^*(B) \subseteq \dom(F_q) = \dom(F_r) = \dom(F_s).
$$

(3) Let $i \in \dom(F_s)$. 
Then $F_s(i) = F_r(i) \in \p_i$, since $r$ is a condition. 
Suppose that $M \in A_s = B$ and $i \in M$. 
We will show that $M \in A_{F_s(i)}$. 
If $M \in A_p = A_r$, then since 
$i \in \dom(F_s) = \dom(F_r)$, it follows that 
$M \in A_{F_r(i)} = A_{F_s(i)}$ since $r$ is a condition.

Suppose that $M \in B \setminus A_p$. 
Then $M = M_1 \cap N$ for some $M_1 \in A_p$, 
where $M_1 < N$ in the case that $N \in \mathcal X$. 
Then $i \in M = M_1 \cap N$, so $i \in M_1$ and $i \in N$. 
Now $M_1$ is in $A_p = A_r$. 
Since $i$ is in $M_1 \cap \dom(F_r)$ and $r$ is a condition, 
it follows that $M_1$ is in $A_{F_r(i)}$. 
Also, $i \in \dom(F_r) \cap N = \dom(F_q) \cap N$, 
so by the choice of $r$, $F_r(i) = r_i$. 
And by the choice of $r_i$, $M_1 \cap N \in A_{r_i}$. 
So $M = M_1 \cap N \in A_{r_i} = A_{F_r(i)} = A_{F_s(i)}$.
\end{proof}

\begin{notation}
Let $N \in \mathcal X$. 
For each $i < \lambda^*$, let $D_{i,N}$ denote the set of conditions 
in $\p_i$ defined as $D_N$ in Definition 7.3.
\end{notation}

We introduce an analogue $D(N)$ of $D_{i,N}$ for $\q$.

\begin{definition}
Let $N \in \mathcal X$. 
Define $D(N)$ as the set of conditions $p \in \q$ satisfying:
\begin{enumerate}
\item $N \in A_p$;
\item for all $M \in A_p$, if $M < N$ then $M \cap N \in A_p$;
\item for all $i \in \dom(F_p) \cap N$, $F_p(i) \in D_{i,N}$.
\end{enumerate}
\end{definition}

\begin{lemma}
Let $N \in \mathcal X$. 
Then for any condition $p \in \q$, if $N \in A_p$, then there is $r \le p$ 
such that $r \in D(N)$.
\end{lemma}

\begin{proof}
Let $p \in \p$. 
By Lemma 11.2(1), fix $q \le p$ such that for all $M \in A_q$, 
if $M < N$ then $M \cap N \in A_q$. 
Since $N \in A_q$, for each $i \in \dom(F_q) \cap N$, 
$N \in A_{F_q(i)}$ by Definition 10.6(3). 
So by Lemma 7.4, 
for each $i \in \dom(F_q) \cap N$, 
we can fix $r_i \le F_q(i)$ such that $r_i \in D_{i,N}$. 
By Lemma 10.9, there is $r \le q$ such that $A_r = A_q$, $\dom(F_r) = \dom(F_q)$, 
for all $i \in \dom(F_r) \cap N$, $F_r(i) = r_i$, and for all 
$i \in \dom(F_r) \setminus N$, $F_r(i) = F_q(i)$. 
Then $r \le p$ and $r \in D(N)$.
\end{proof}

\begin{lemma}
Let $N \in \mathcal X$ and $q \in D(N)$. 
Let $x$ be a finite subset of $\lambda^* \setminus \dom(F_q)$. 
Then $q \uplus x \in D(N)$.
\end{lemma}

\begin{proof}
Let $r := q \uplus x$. 
Since $q \in D(N)$ and $A_q = A_r$, we have that $N \in A_r$, 
and for all $M \in A_r$, if $M < N$ then $M \cap N \in A_r$. 
Also, since $q \in D(N)$, for all $i \in \dom(F_q) \cap N$, 
$F_r(i) = F_q(i) \in D_{i,N}$. 
It remains to show that for all $i \in x \cap N$, 
$F_r(i) \in D_{i,N}$.

Let $i \in x \cap N$. 
By Definition 10.7, 
$F_r(i) = (\emptyset,\emptyset,B_i)$, where 
$B_i = \{ M \in A_q : i \in M \}$. 
We will verify that $(\emptyset,\emptyset,B_i)$ 
satisfies requirements (1), (2), and (3) of Definition 7.3.

(1) Since $q \in D(N)$, $N \in A_q$. 
As $i \in N$, $N \in B_i$ by definition. 
(3) is immediate, since $f_{F_r(i)} = \emptyset$.
(2) Suppose that $M \in B_i$ and $M < N$, and we will show that 
$M \cap N \in B_i$. 
By the definition of $B_i$, $M \in A_q$ and $i \in M$. 
So $i \in M \cap N$. 
Since $q \in D(N)$ and $M < N$, $M \cap N \in A_q$. 
So $M \cap N \in A_q$ and $i \in M \cap N$, which 
by definition implies that $M \cap N \in B_i$. 
\end{proof}

\begin{notation}
Let $P \in \mathcal Y$. 
For each $i < \lambda^*$, let $D_{i,P}$ denote the set of conditions 
in $\p_i$ defined as $D_P$ in Definition 6.2.
\end{notation}

We introduce an analogue $D(P)$ of $D_{i,P}$ for $\q$.

\begin{definition}
Let $P \in \mathcal Y$. 
Define $D(P)$ as the set of conditions $p \in \q$ satisfying:
\begin{enumerate}
\item for all $M \in A_p$, $M \cap P \in A_p$;
\item for all $i \in \dom(F_p) \cap P$, $F_p(i) \in D_{i,P}$.
\end{enumerate}
\end{definition}

\begin{lemma}
Let $P \in \mathcal Y$. 
Then $D(P)$ is dense in $\q$.
\end{lemma}

\begin{proof}
Let $p \in \q$. 
By Lemma 11.2(2), we can find $q \le p$ such that 
for all $M \in A_q$, $M \cap P \in A_q$. 
Let $x := \dom(F_q) \cap P$. 
For each $i \in x$, the set $D_{i,P}$ is dense in $\p_i$ by Lemma 6.3. 
By Lemma 10.10, fix $r \le q$ such that $A_r = A_q$, 
$\dom(F_r) = \dom(F_q) \cup x = \dom(F_q)$, and for each $i \in x$, 
$F_r(i) \in D_{i,P}$. 
Then $r \le p$ and $r \in D(P)$.
\end{proof}

\begin{lemma}
Let $P \in \mathcal Y$ and $q \in D(P)$. 
Let $x$ be a finite subset of $\lambda^* \setminus \dom(F_q)$. 
Then $q \uplus x \in D(P)$.
\end{lemma}

\begin{proof}
Let $r := q \uplus x$. 
Then by Definition 10.7, $A_r = A_q$. 
Since $q \in D(P)$, it follows that for all $M \in A_r$, $M \cap P \in A_r$. 
It remains to show that for all $i \in \dom(F_r) \cap P$, 
$F_r(i) \in D_{i,P}$. 

By Definition 10.7, $\dom(F_r) = \dom(F_q) \cup x$, for 
all $i \in \dom(F_q)$, $F_r(i) = F_q(i)$, and for all $i \in x$, 
$F_r(i) = (\emptyset,\emptyset,B_i)$, where 
$B_i = \{ M \in A_q : i \in M \}$. 
Since $q \in D(P)$, for all $i \in \dom(F_q) \cap P$, 
$F_r(i) = F_q(i) \in D_{i,P}$. 

It remains to show that for all $i \in x \cap P$, 
$F_r(i) = (\emptyset,\emptyset,B_i)$ is in $D_{i,P}$. 
By Definition 6.2, we need to show that for all 
$M \in B_i$, $M \cap P \in B_i$. 
So let $M \in B_i$. 
Then $M \in A_q$ and $i \in M$. 
So $i \in M \cap P$. 
Since $q \in D(P)$, $M \cap P \in A_q$. 
Hence, $M \cap P \in A_q$ and $i \in M \cap P$, which 
means that $M \cap P \in B_i$.
\end{proof}

\begin{definition}
Suppose that $N \in \mathcal X \cup \mathcal Y$ is simple 
and $q \in D(N)$. 
Define $q \restriction N$ as the pair $(F,A)$ satisfying:
\begin{enumerate}
\item $A := A_q \cap N$;
\item $\dom(F) := \dom(F_q) \cap N$;
\item for all $i \in \dom(F)$, 
$F(i) := F_q(i) \restriction N$, as defined in Definition 7.5 if $N \in \mathcal X$, 
and as defined in Definition 6.4 if $N \in \mathcal Y$.
\end{enumerate}
\end{definition}

Note that (3) makes sense because $F_q(i) \in D_{i,N}$, for all 
$i \in \dom(F_q) \cap N$.

\begin{lemma}
Suppose that $N \in \mathcal X \cup \mathcal Y$ is simple and $q \in D(N)$. 
Then $q \restriction N$ is in $N \cap \q$ and $q \le q \restriction N$.
\end{lemma}

\begin{proof}
Let $q \restriction N = (F,A)$. 
Then $A = A_q \cap N$ and $\dom(F) = \dom(F_q) \cap N$ are 
finite subsets of $N$, and hence are in $N$. 
For each $i \in \dom(F)$, $F(i) = F_q(i) \restriction N$ is in $N$ by 
Lemmas 6.5 and 7.6. 
So $F$ is in $N$. 
Since $A$ and $F$ are in $N$, so is $q \restriction N$.

To prove that $q \restriction N$ is in $\q$, we verify requirements 
(1)--(4) of Definition 10.6. 
(1) and (2) are immediate. 
For (3), let $i \in \dom(F)$. 
Then $F(i) = F_q(i) \restriction N$ is in $\p_i$ by Lemmas 6.5 and 7.6. 
Suppose that $M \in A$ and $i \in M$, and we will show 
that $M \in A_{F(i)}$. 
Then $M \in A = A_q \cap N$, so $M \in A_q$ and $M \in N$. 
Since $q$ is a condition, the fact that 
$M \in A_q$ and $i \in M \cap \dom(F_q)$ implies that $M \in A_{F_q(i)}$. 
Since $F(i) = F_q(i) \restriction N$, Definitions 6.4 and 7.5 imply that 
$$
M \in A_{F_q(i)} \cap N = A_{F_q(i) \restriction N} = A_{F(i)}.
$$

For (4), suppose that $K$ and $L$ are in $A$, 
$K \sim L$, $i \in K \cap L \cap \lambda^*$, and 
the ordinal $\min((K \cap \kappa) \setminus \beta_{K,L})$ exists 
and is in $S_i$. 
We will show that $i \in \dom(F)$. 
Since $A = A_q \cap N$, $K$ and $L$ are in $A_q$ and in $N$. 
As $q$ is a condition, $i$ must be in $\dom(F_q)$. 
Since $K \in N$ and $i \in K$, we have that $i \in N$. 
So $i \in \dom(F_q) \cap N = \dom(F)$.

This completes the proof that $q \restriction N$ is a condition. 
Now we show that $q \le q \restriction N$. 
We have that $A = A_q \cap N \subseteq A_q$ 
and $\dom(F) = \dom(F_q) \cap N \subseteq \dom(F_q)$. 
Let $i \in \dom(F)$, and we will show that 
$F_q(i) \le F(i)$ in $\p_i$. 
But $F(i) = F_q(i) \restriction N$ and 
$F_q(i) \le F_q(i) \restriction N$ in $\p_i$ by Lemmas 6.5 and 7.6.
\end{proof}

The next lemma will be needed in Section 12.

\begin{lemma}
Let $N \in \mathcal X \cup \mathcal Y$ be simple and $q \in D(N)$. 
\begin{enumerate}
\item Suppose that $p \in N \cap \q$ and $q \le p$. 
Then $q \restriction N \le p$.
\item Suppose that $p$ is in $D(N)$ and $q \le p$. 
Then $q \restriction N \le p \restriction N$.
\end{enumerate}
\end{lemma}

\begin{proof}
(1) Since $p \in N$ and $q \le p$, 
we have that $A_{p} \subseteq A_q \cap N = A_{q \restriction N}$ 
and $\dom(F_p) \subseteq \dom(F_q) \cap N = \dom(F_{q \restriction N})$. 
Let $i \in \dom(F_p)$, and we will show that 
$F_{q \restriction N}(i) \le F_p(i)$ in $\p_i$. 
Since $q \le p$, we have that $F_q(i) \le F_p(i)$. 
As $F_p(i) \in N \cap \p_i$, it follows that 
$F_{q \restriction N}(i) = 
F_q(i) \restriction N \le F_p(i)$ by Lemmas 6.6(1) and 7.7.

(2) By Lemma 11.12, $p \le p \restriction N$. 
So $q \le p \restriction N$. 
Now $p \restriction N \in N$, so  
by (1), $q \restriction N \le p \restriction N$.
\end{proof}

We will now begin analyzing the situation where $q \in D(N)$ and 
$w \le q \restriction N$ is in $N \cap \q$.

\begin{lemma}
Let $N \in \mathcal X \cup \mathcal Y$ be simple and $q \in D(N)$. 
Suppose that $w \in N \cap \q$ and $w \le q \restriction N$. 
Then:
\begin{enumerate}
\item $A_q \cap N \subseteq A_w$;
\item $\dom(F_q) \cap N \subseteq \dom(F_w)$, and for all 
$i \in \dom(F_q) \cap N$, $F_w(i) \le F_q(i) \restriction N$ in $\p_i$.
\end{enumerate}
\end{lemma}

\begin{proof}
Immediate from the definition of $q \restriction N$ and the fact 
that $w \le q \restriction N$.
\end{proof}

Note that in (2) above, if $i \in \dom(F_q) \cap N$, then 
since $w \in N$, $F_w(i)$ is a condition in $N \cap \p_i$ which is 
below $F_q(i) \restriction N$ in $\p_i$. 
As $F_q(i) \in D_{i,N}$, it follows by 
Propositions 6.15 and 7.19 that $F_w(i) \oplus_N F_q(i)$ is a condition 
in $\p_i$ which is below $F_w(i)$ and $F_q(i)$.

As in Sections 6 and 7, we are going to show that whenever 
$w \le q \restriction N$, where $q \in D(N)$ and $w \in N \cap \q$, 
then $w$ and $q$ are compatible. 
We will define a specific lower bound of $w$ and $q$, namely, 
$w \oplus^N q$. 
However, unlike the situation in Sections 6 and 7, 
the condition $w \oplus^N q$ will exist only 
under the assumption that $\dom(F_w) \subseteq \dom(F_q)$.

\begin{definition}
Let $N \in \mathcal X \cup \mathcal Y$ be simple and $q \in D(N)$. 
Suppose that $w \in N \cap \q$ and $w \le q \restriction N$. 
Assume, moreover, that $\dom(F_w) \subseteq \dom(F_q)$. 
Define $w \oplus^N q$ as the pair $(F,A)$ satisfying:
\begin{enumerate}
\item $A := A_w \cup A_q$;
\item $\dom(F) := \dom(F_q)$;
\item for all $i \in \dom(F_q) \setminus \dom(F_w)$, 
$F(i) := F_q(i)$, and for all $i \in \dom(F_w) \cap \dom(F_q)$, 
$F(i) := F_w(i) \oplus_N F_q(i)$, 
as defined in Definition 7.18 if $N \in \mathcal X$, 
and as defined in Definition 6.14 if $N \in \mathcal Y$.
\end{enumerate}
\end{definition}

\begin{proposition}
Let $N \in \mathcal X \cup \mathcal Y$ be simple and $q \in D(N)$. 
Suppose that $w \in N \cap \q$ and $w \le q \restriction N$. 
Assume, moreover, that $\dom(F_w) \subseteq \dom(F_q)$. 
Then $w$ and $q$ are compatible. 
In fact, $w \oplus^N q$ is in $\q$ and $w \oplus^N q \le w, q$.
\end{proposition}

\begin{proof}
Let $w \oplus^N q = (F,A)$. 
To prove that $w \oplus^N q$ is a condition, we verify requirements 
(1)--(4) of Definition 10.6. 
For (1), the set $A = A_w \cup A_q$ is adequate by Propositions 
2.25 and 2.28. 
For (2), obviously $F$ is a function whose domain is a finite subset of 
$\lambda^*$.

For (4), by Proposition 10.13 we have that 
\begin{multline*}
s^*(A) = s^*(A_w \cup A_q) = 
s^*(A_w) \cup s^*(A_q) \subseteq \\ 
\subseteq \dom(F_w) \cup \dom(F_q) = \dom(F_q) = \dom(F).
\end{multline*}

It remains to prove (3). 
Let $i \in \dom(F)$. 
If $i \in \dom(F_q) \setminus \dom(F_w)$, then 
$F(i) = F_q(i)$, which is in $\p_i$ since $q$ is a condition. 
If $i \in \dom(F_q) \cap \dom(F_w)$, then 
$F(i) = F_w(i) \oplus_N F_q(i)$, which is in $\p_i$ by 
Propositions 6.15 and 7.19.

Assume that $i \in \dom(F)$, and we will show that 
$$
\{ M \in A_{w \oplus^N q} : i \in M \} \subseteq A_{F(i)},
$$
that is,
$$
\{ M \in A_w \cup A_q : i \in M \} \subseteq A_{F(i)}.
$$

First, assume that $i \in \dom(F_q) \setminus \dom(F_w)$, so 
$F(i) = F_q(i)$. 
Since $\dom(F_q) \cap N \subseteq \dom(F_w)$ 
by Lemma 11.14(2), it follows that $i \notin N$. 
In particular, if $M \in A_w$ then $M \in N$, so $i$ cannot be in $M$ 
since otherwise it would be in $N$. 
It follows that 
$$
\{ M \in A_w \cup A_q : i \in M \} = \{ M \in A_q : i \in M \}.
$$
Since $F(i) = F_q(i)$, $q$ being a condition implies that 
$$
\{ M \in A_q : i \in M \} \subseteq A_{F_q(i)} = A_{F(i)}.
$$

Secondly, assume that $i \in \dom(F_q) \cap \dom(F_w)$. 
Then $F(i) = F_w(i) \oplus_N F_q(i)$. 
By Definitions 6.14 and 7.18, 
$$
A_{F(i)} = A_{F_w(i)} \cup A_{F_q(i)}.
$$
Since $w$ and $q$ are conditions, 
$$
\{ M \in A_w : i \in M \} \subseteq A_{F_w(i)}, \ \ \ 
\{ M \in A_q : i \in M \} \subseteq A_{F_q(i)}.
$$
Therefore, 
$$
\{ M \in A_w \cup A_q : i \in M \} \subseteq 
A_{F_w(i)} \cup A_{F_q(i)} = A_{F(i)}.
$$
This completes the proof that $w \oplus^N q$ is in $\q$.

\bigskip

It remains to show that 
$w \oplus^N q \le q, w$. 
First, we prove that $w \oplus^N q \le w$. 
We have that $A_w \subseteq A_w \cup A_q = A$. 
Since $\dom(F_w) \subseteq \dom(F_q)$ by assumption, 
it follows that $\dom(F_w) \subseteq \dom(F_q) = \dom(F)$. 
Let $i \in \dom(F_w)$, and we will show that $F(i) \le F_w(i)$ in $\p_i$. 
But $F(i) = F_w(i) \oplus_N F_q(i)$, which is less than or equal to 
$F_w(i)$ in $\p_i$ by Propositions 6.15 and 7.19.

Secondly, we prove that $w \oplus^N q \le q$. 
We have that $A_q \subseteq A_w \cup A_q = A$, and 
$\dom(F) = \dom(F_q)$. 
Let $i \in \dom(F_q)$, and we will show that $F(i) \le F_q(i)$ in $\p_i$. 
If $i \notin \dom(F_w)$, then $F(i) = F_q(i)$, and we are done. 
If $i \in \dom(F_w)$, then $F(i) = F_w(i) \oplus_N F_q(i)$, which is 
less than or equal to $F_q(i)$ in $\p_i$ by Propositions 6.15 and 7.19.
\end{proof}

In the above amalgamation result, we assumed that 
$\dom(F_w) \subseteq \dom(F_q)$. 
To prove the amalgamation result in general, we need a lemma.

\begin{lemma}
Suppose that 
$N \in \mathcal X \cup \mathcal Y$ is simple and $q \in D(N)$. 
Let $x$ be a finite subset of $\lambda^* \setminus \dom(F_q)$. 
Then 
$$
(q \uplus x) \restriction N = (q \restriction N) \uplus (x \cap N).
$$
\end{lemma}

Recall that by Lemmas 11.6 and 11.10, 
$q \uplus x$ is in $D(N)$.
 
\begin{proof}
By Definitions 10.7 and 11.11,
$$
A_{(q \uplus x) \restriction N} = 
A_{q \uplus x} \cap N = 
A_q \cap N = 
A_{(q \restriction N)} =
A_{(q \restriction N) \uplus (x \cap N)}.
$$
Also, 
\begin{multline*}
\dom(F_{(q \uplus x) \restriction N}) = 
\dom(F_{q \uplus x}) \cap N = 
(\dom(F_q) \cup x) \cap N = \\
(\dom(F_q) \cap N) \cup (x \cap N) = \dom(F_{q \restriction N}) \cup (x \cap N) = 
\dom(F_{(q \restriction N) \uplus (x \cap N)}).
\end{multline*}
Let $i \in \dom(F_{(q \uplus x) \restriction N})$, and we will show that 
$$
F_{(q \uplus x) \restriction N}(i) = 
F_{(q \restriction N) \uplus (x \cap N)}(i).
$$
By the above equalities, we have that either 
$i \in \dom(F_{q \restriction N})$ or $i \in x \cap N$.

First, assume that $i \in \dom(F_{q \restriction N})$. 
Then by Definitions 10.7 and 11.11, we have that 
$F_{(q \restriction N) \uplus (x \cap N)}(i) = 
F_{q \restriction N}(i) = F_q(i) \restriction N$. 
On the other hand, since $\dom(F_{q \restriction N}) = 
\dom(F_q) \cap N$, it follows that $i \in \dom(F_q)$, and hence 
$F_{(q \uplus x) \restriction N}(i) = 
F_{q \uplus x}(i) \restriction N = F_q(i) \restriction N$.

Now assume that $i \in x \cap N$. 
Then by definition, 
$$
F_{(q \restriction N) \uplus (x \cap N)}(i) = (\emptyset,\emptyset,B),
$$
where 
$$
B = \{ M \in A_{q \restriction N} : i \in M \}.
$$
Also,
$$
F_{(q \uplus x) \restriction N}(i) = 
F_{q \uplus x}(i) \restriction N = 
(\emptyset,\emptyset,C) \restriction N,
$$
where 
$$
C = \{ M \in A_q : i \in M \}.
$$
Hence, it suffices to show that 
$$
(\emptyset,\emptyset,B) = (\emptyset,\emptyset,C) \restriction N.
$$
By Definitions 6.4 and 7.5, 
$$
(\emptyset,\emptyset,C) \restriction N = 
(\emptyset,\emptyset,C \cap N).
$$
So it suffices to show that 
$$
B = C \cap N.
$$
But $M \in B$ iff ($M \in A_{q \restriction N} = A_q \cap N$ and $i \in M$) iff 
$M \in C \cap N$. 
\end{proof}

\begin{proposition}
Let $N \in \mathcal X \cup \mathcal Y$ be simple and $q \in D(N)$. 
Then for all $w \le q \restriction N$ in $N \cap \q$, $w$ and $q$ are compatible. 
In fact, let $x := \dom(F_w) \setminus \dom(F_q)$. 
Then $w \le (q \uplus x) \restriction N$, 
$\dom(F_w) \subseteq \dom(F_{q \uplus x})$, and 
$w \oplus^N (q \uplus x)$ is less than or equal to $w$, $q \uplus x$, and $q$.
\end{proposition}

\begin{proof}
Note that $x \subseteq N$. 
By Lemma 11.17,
$$
(q \uplus x) \restriction N = (q \restriction N) \uplus (x \cap N).
$$
By Lemma 10.8(4), since $w \le q \restriction N$ and 
$x \cap N = x \subseteq \dom(F_w)$, it follows that 
$$
w \le (q \restriction N) \uplus (x \cap N).
$$
Hence, 
$w \le (q \uplus x) \restriction N$. 
As $x = \dom(F_w) \setminus \dom(F_q)$, 
clearly $\dom(F_w) \subseteq \dom(F_q) \cup x = \dom(F_{q \uplus x})$. 
By Proposition 11.16, it follows that 
$w \oplus^N (q \uplus x)$ is a condition which is less than or 
equal to $w$ and $q \uplus x$. 
Since $q \uplus x \le q$, also $w \oplus^N (q \uplus x) \le q$.
\end{proof}

\begin{corollary}
The forcing poset $\q$ is strongly proper on a stationary set. 
In particular, it preserves $\omega_1$.
\end{corollary}

\begin{proof}
By Assumption 2.22, the set of $N \in \mathcal X$ such that $N$ is simple 
is stationary. 
So it suffices to show that for all simple $N \in \mathcal X$, for all 
$p \in N \cap \q$, there is $q \le p$ such that $q$ is strongly $N$-generic. 

Let $p \in N \cap \q$. 
By Lemma 11.1, fix $q \le p$ with $N \in A_q$. 
We claim that $q$ is strongly $N$-generic. 
So let $D$ be a dense subset of $N \cap \q$, and we will show that 
$D$ is predense below $q$. 
Let $r \le q$, and we will find $w \in D$ which is compatible 
with $r$.

Since $N \in A_r$, we can apply Lemma 11.5 to fix $s \le r$ such that 
$s \in D(N)$. 
Then by Lemma 11.12, $s \restriction N$ is in $N \cap \q$. 
As $D$ is dense in $N \cap \q$, fix $w \le s \restriction N$ in $D$. 
By Proposition 11.18, $w$ and $s$ are compatible. 
Since $s \le r$, it follows that $w$ and $r$ are compatible.
\end{proof}

\begin{corollary}
Suppose that $P \in \mathcal Y$ is simple and 
$P \prec (H(\lambda),\in,\q)$. 
Then the maximum condition of $\q$ is strongly $P$-generic. 
Moreover, $P \cap \q$ is a regular suborder of $\q$.
\end{corollary}

\begin{proof}
Let $D$ be a dense subset of $P \cap \q$, and we will show that $D$ is predense in $\q$. 
So let $p \in \q$, and we will find $w$ in $D$ which is compatible with $p$. 
By Lemma 11.9, fix $q \le p$ in $D(P)$. 
Then $q \restriction P$ is in $P \cap \q$ by Lemma 11.12. 
Since $D$ is dense in $P \cap \q$, fix $w \le q \restriction P$ in $D$. 
By Proposition 11.18, $w$ and $q$ are compatible. 
Since $q \le p$, it follows that $w$ and $p$ are compatible. 
This completes the proof that the maximum condition in $\q$ 
is strongly $P$-generic.

Now we show that $P \cap \q$ is a regular suborder of $\q$. 
If $p$ and $q$ are in $P \cap \q$ and are compatible in $\q$, then 
by the elementarity of $P$, there is $r \in P \cap \q$ with $r \le p, q$. 
So $p$ and $q$ are compatible in $P \cap \q$.

Let $B \subseteq P \cap \q$ be a maximal antichain of $P \cap \q$, and 
we will prove that $B$ is predense in $\q$. 
Let $D$ be the set of conditions in $P \cap \q$ which are below some member of $B$. 
Then $D$ is dense in $P \cap \q$. 
Since the maximum condition is strongly $P$-generic, $D$ is predense in $\q$. 
It easily follows that $B$ is predense in $\q$.
\end{proof}

\begin{corollary}
The forcing poset $\q$ is $\kappa$-c.c.
\end{corollary}

\begin{proof}
Let $A$ be an antichain of $\q$, and suppose for a contradiction 
that $A$ has size at least $\kappa$. 
Without loss of generality, assume that $A$ is maximal. 
By Assumption 2.23, there are stationarily many simple models 
in $\mathcal Y$. 
So we can fix a simple model $P \in \mathcal Y$ 
such that $P \prec (H(\lambda),\in,\q,A)$. 
As $A$ has size at least $\kappa$ and $|P| < \kappa$, 
fix $s \in A \setminus P$.

By Lemma 11.9, fix $q \le s$ such that $q \in D(P)$. 
By Lemma 11.12, $q \restriction P$ is a condition in $P \cap \q$. 
By the elementarity of $P$ and the maximality of $A$, 
there is $t \in A \cap P$ which is compatible with $q \restriction P$. 
By elementarity, fix $w \in P \cap \q$ such that $w \le q \restriction P, t$. 

By Proposition 11.18, $w$ and $q$ are compatible. 
Fix $v \le w, q$. 
Then $v \le w \le t$ and $v \le q \le s$. 
Hence, $s$ and $t$ are compatible. 
But $s$ and $t$ are in $A$ and $A$ is an antichain. 
Therefore, $s = t$. 
This is impossible, since $t \in P$ and $s \notin P$.
\end{proof}

The next result summarizes the main properties which we have 
proven about the forcing poset $\q$.

\begin{corollary}
The forcing poset $\q$ preserves $\omega_1$, is $\kappa$-c.c., 
forces that $\kappa$ is equal to $\omega_2$, and 
forces that for all $i < \lambda^*$, $S_i \in I[\omega_2]$.
\end{corollary}

\begin{proof}
By Corollaries 11.19 and 11.21, $\q$ preserves $\omega_1$ and is 
$\kappa$-c.c. 
Let $i < \lambda^*$, and consider a generic filter $G$ on $\q$. 
Then by Proposition 10.11 and the comments which follow, 
there is a generic filter $H$ on $\p_i$ 
such that $V[H] \subseteq V[G]$. 
By Corollary 7.22, $\kappa$ is equal to $\omega_2$ in $V[H]$. 
Since $V[H] \subseteq V[G]$, it follows that any cardinal $\mu$ such that  
$\omega_1 < \mu < \kappa$ has size $\omega_1$ 
in $V[G]$. 
Therefore, $\kappa = \omega_2$ in $V[G]$.

By Proposition 5.4, there is a partial square sequence on 
$S_i$ in $V[H]$. 
But being a partial square sequence is upwards absolute between 
$V[H]$ and $V[G]$, since they have the same $\omega_1$ and $\omega_2$. 
So there is a partial square sequence on $S_i$ in $V[G]$. 
Therefore, $S_i \in I[\omega_2]$ in $V[G]$.
\end{proof}

\bigskip

\addcontentsline{toc}{section}{12. Approximation}

\textbf{\S 12. Approximation}

\stepcounter{section}

\bigskip

In this section we will prove that if $P \in \mathcal Y$ is simple, 
$P \prec (H(\lambda),\in,\q)$, and for all $i \in P \cap \lambda^*$, 
$P \cap \kappa \notin S_i$, then $P \cap \q$ forces that 
$\q / \dot G_{P \cap \q}$ has the $\omega_1$-approximation property.
The proof is similar to the analogous result 
given in Sections 8 and 9 for the forcing poset $\p$, 
albeit somewhat easier. 
The order of topics and results follows along the same lines as in those 
previous sections.

\begin{lemma}
Let $N \in \mathcal X$ be simple and $q \in D(N)$. 
Suppose that $v$ and $w$ are in $N \cap \q$ and 
$$
w \le v \le q \restriction N.
$$
Assume, moreover, that $\dom(F_w) \subseteq \dom(F_q)$. 
Then $w \oplus^N q \le v \oplus^N q$.
\end{lemma}

Note that since $w \le v$, 
$\dom(F_v) \subseteq \dom(F_w) \subseteq \dom(F_q)$. 
So $v \oplus^N q$ is defined.

\begin{proof}
Let $s := v \oplus^N q$ and $t := w \oplus^N q$. 
We will prove that $t \le s$. 
Since $w \le v$, $A_v \subseteq A_w$. 
By Definition 11.15, we have that 
$$
A_s = A_v \cup A_q \subseteq A_w \cup A_q = A_t.
$$
So $A_s \subseteq A_t$. 
Also, by Definition 11.15, 
$\dom(F_s) = \dom(F_q) = \dom(F_t)$.

Let $i \in \dom(F_s)$, and we will show that 
$F_t(i) \le F_s(i)$ in $\p_i$. 
First, assume that $i \notin N$. 
Then by Definition 11.15, $F_s(i) = F_q(i)$ and $F_t(i) = F_q(i)$, and 
we are done.

Secondly, assume that $i \in N$. 
Then by Lemma 11.14, 
$i \in \dom(F_q) \cap N \subseteq \dom(F_v) \subseteq \dom(F_w)$. 
So by Definition 11.15, $F_s(i) = F_v(i) \oplus_N F_q(i)$ and 
$F_t(i) = F_w(i) \oplus_N F_q(i)$.
Also, since $v$, $w$, and $i$ are in $N$, so are 
$F_v(i)$ and $F_w(i)$.

As $q \in D(N)$ and $i \in \dom(F_q) \cap N$, $F_q(i) \in D_{i,N}$. 
Also, since $w \le v \le q \restriction N$, we have that in $\p_i$, 
$$
F_w(i) \le F_v(i) \le F_{q \restriction N}(i) = F_q(i) \restriction N.
$$
By Lemma 8.1, it follows that 
$$
F_t(i) = F_w(i) \oplus_N F_q(i) \le F_v(i) \oplus_N F_q(i) = F_s(i).
$$
\end{proof}

\begin{lemma}
Let $N \in \mathcal X$ and $P \in \mathcal Y$. 
Let $p \in \q$, and suppose that $N \in A_p$. 
Then there is $s \le p$ such that $s \in D(N) \cap D(P)$.
\end{lemma}

\begin{proof}
By Lemma 11.2(1), there is $q \le p$ such that for all $M \in A_q$, 
if $M < N$ then $M \cap N \in A_q$. 
By Lemma 11.2(2), there is $r \le q$ such that for all $M \in A_r$, 
$M \cap P \in A_r$, and moreover, 
$$
A_r = A_q \cup \{ M \cap P : M \in A_q \}.
$$

We claim that for all $M \in A_r$, if $M < N$ then $M \cap N \in A_r$. 
This is certainly true if $M \in A_q$, so assume that 
$M = M_1 \cap P$, where $M_1 \in A_q$. 
By Lemma 2.29, $M_1 \sim M_1 \cap P = M$. 
Since $M < N$, it follows that $M_1 < N$ by Lemma 2.18. 
As $M_1 \in A_q$, we have that 
$M_1 \cap N \in A_q$ by the choice of $q$. 
Now 
$M \cap N = (M_1 \cap P) \cap N = (M_1 \cap N) \cap P$. 
But $M_1 \cap N \in A_q$ implies that 
$M \cap N = (M_1 \cap N) \cap P \in A_r$, by the definition of $A_r$.

Since $r$ is a condition and $N \in A_r$, we have that 
for all $i \in \dom(F_r) \cap N$, $N \in A_{F_r(i)}$. 
Let $x_1 := \dom(F_r) \cap N$ and $x_2 := \dom(F_r) \cap P$. 
Then:
\begin{enumerate}
\item For each $i \in x_1 \setminus x_2$, since $N \in A_{F_r(i)}$ we 
can fix, by Lemma 7.4, a condition $s_i \le F_r(i)$ in $D_{i,N}$.

\item For each $i \in x_1 \cap x_2$, since $N \in A_{F_r(i)}$ we 
can fix, by Lemma 8.2, a condition $s_i \le F_r(i)$ in 
$D_{i,N} \cap D_{i,P}$.

\item For each $i \in x_2 \setminus x_1$, we can fix, by Lemma 6.3, 
a condition $s_i \le F_r(i)$ in $D_{i,P}$.
\end{enumerate}

Now apply Lemma 10.9 to fix $s \le r$ satisfying that 
$A_s = A_r$, $\dom(F_s) = \dom(F_r)$, for all 
$i \in x_1 \cup x_2$, $F_s(i) = s_i$, and for all $i \in \dom(F_r) \setminus (x_1 \cup x_2)$, 
$F_s(i) = F_r(i)$. 
Then $s \le p$ and $s \in D(N) \cap D(P)$.
\end{proof}

The next three lemmas will be used in the proof of Proposition 12.6.

\begin{lemma}
Let $N \in \mathcal X$ be simple, $P \in \mathcal Y \cap N$ be simple, 
and $p \in D(N) \cap D(P)$. 
Then $p \restriction N \in D(P)$ 
and $p \restriction P \in D(N \cap P)$.
\end{lemma}

\begin{proof}
We prove first that $p \restriction N \in D(P)$, which means 
that for all $M \in A_{p \restriction N}$, 
$M \cap P \in A_{p \restriction N}$, and 
for all $i \in \dom(F_{p \restriction N}) \cap P$, 
$F_{p \restriction N}(i) \in D_{i,P}$. 
Let $M \in A_{p \restriction N}$. 
Then $M \in A_{p \restriction N} = A_p \cap N$, so $M \in A_p \cap N$. 
Since $p \in D(P)$, we have that $M \cap P \in A_p$. 
And as $M$ and $P$ are in $N$, $M \cap P \in N$. 
Therefore, $M \cap P \in A_p \cap N = A_{p \restriction N}$.

Now let $i \in \dom(F_{p \restriction N}) \cap P = 
\dom(F_p) \cap N \cap P$, and we will show that 
$F_{p \restriction N}(i) \in D_{i,P}$. 
Since $i \in N \cap P$ and 
$p \in D(N) \cap D(P)$, $F_p(i) \in D_{i,N} \cap D_{i,P}$. 
By Lemma 8.3, it follows that $F_p(i) \restriction N \in D_{i,P}$. 
But $F_p(i) \restriction N = F_{p \restriction N}(i)$. 
This completes the proof that $p \restriction N \in D(P)$.

Next, we prove that $p \restriction P \in D(N \cap P)$. 
First, we show that $N \cap P \in A_{p \restriction P}$. 
Since $p \in D(N)$, $N \in A_p$. 
As $p \in D(P)$, $N \cap P \in A_p$. 
Since $P$ is simple, $N \cap P \in P$. 
So $N \cap P \in A_p \cap P = A_{p \restriction P}$. 

Secondly, we prove that if $M \in A_{p \restriction P}$ and 
$M < N \cap P$, then $M \cap N \cap P \in A_{p \restriction P}$. 
Let $M \in A_{p \restriction P}$, and assume that $M < N \cap P$. 
Then $M \in A_{p \restriction P} = A_p \cap P$. 
By Lemma 2.29, $N \sim N \cap P$. 
Since $M < N \cap P$, it follows by Lemma 2.18 that 
$M < N$. 
Since $p \in D(N)$ and $M < N$, it follows that $M \cap N \in A_p$. 
And as $p \in D(P)$, $M \cap N \cap P \in A_p$. 
Since $P$ is simple, $M \cap N \cap P \in P$. 
So $M \cap N \cap P \in A_p \cap P = A_{p \restriction P}$.

Thirdly, we show that if $i \in \dom(F_{p \restriction P}) \cap (N \cap P)$, 
then $F_{p \restriction P}(i) \in D_{i,N \cap P}$. 
Since $p \in D(P) \cap D(N)$ and $i \in N \cap P$, 
we have that $F_p(i) \in D_{i,N} \cap D_{i,P}$. 
By Lemma 8.3, $F_p(i) \restriction P$ 
is in $D_{i,N \cap P}$. 
But $F_p(i) \restriction P = F_{p \restriction P}(i)$.
\end{proof}

\begin{lemma}
Let $N \in \mathcal X$ be simple, $P \in \mathcal Y \cap N$ be simple, 
and $p \in D(N) \cap D(P)$. 
Then 
$$
(p \restriction N) \restriction P = 
(p \restriction P) \restriction (N \cap P).
$$
\end{lemma} 

Note that we needed Lemma 12.3 to see that 
$(p \restriction N) \restriction P$ and  
$(p \restriction P) \restriction (N \cap P)$ are defined.

\begin{proof}
By Definition 11.11, 
we have that 
\begin{multline*}
A_{(p \restriction N) \restriction P} = 
A_{p \restriction N} \cap P = 
A_p \cap N \cap P = \\
(A_p \cap P) \cap (N \cap P) = A_{p \restriction P} \cap (N \cap P) = 
A_{(p \restriction P) \restriction (N \cap P)}.
\end{multline*}
And 
\begin{multline*}
\dom(F_{(p \restriction N) \restriction P}) = 
\dom(F_{p \restriction N}) \cap P = 
\dom(F_p) \cap N \cap P = \\
(\dom(F_p) \cap P) \cap (N \cap P) = 
\dom(F_{p \restriction P}) \cap (N \cap P) = 
\dom(F_{(p \restriction P) \restriction (N \cap P)}).
\end{multline*}
Let $i \in \dom(F_{(p \restriction N) \restriction P})$, and we 
will show that 
$$
F_{(p \restriction N) \restriction P}(i) = 
F_{(p \restriction P) \restriction (N \cap P)}(i).
$$
By the above equations, $i \in N \cap P$. 
Since $p \in D(N) \cap D(P)$, we have that 
$F_p(i) \in D_{i,N} \cap D_{i,P}$. 
So by Definition 11.11 and Lemma 8.4, 
\begin{multline*}
F_{(p \restriction N) \restriction P}(i) = 
F_{p \restriction N}(i) \restriction P = 
(F_{p}(i) \restriction N) \restriction P = \\
(F_{p}(i) \restriction P) \restriction (N \cap P) = 
F_{p \restriction P}(i) \restriction (N \cap P) = 
F_{(p \restriction P) \restriction (N \cap P)}(i).
\end{multline*}
\end{proof}

\begin{lemma}
Let $N \in \mathcal X$ be simple, $P \in \mathcal Y \cap N$ be simple, 
and $p \in D(N) \cap D(P)$. 
Suppose that $q \in N \cap D(P)$, $q \le p \restriction N$, and 
$\dom(F_q) \subseteq \dom(F_p)$. 
Then:
\begin{enumerate}
\item $q \oplus^{N} p$ is in $D(P)$;
\item $q \restriction P \in N \cap P$ and 
$$
q \restriction P \le (p \restriction P) \restriction (N \cap P).
$$
\end{enumerate}
\end{lemma}

\begin{proof}
(1) Let us prove that $q \oplus^N p$ is in $D(P)$, which means that for all 
$M \in A_{q \oplus^N p}$, $M \cap P \in A_{q \oplus^N p}$, and for all 
$i \in \dom(F_{q \oplus^N p}) \cap P$, 
$F_{q \oplus^N p}(i) \in D_{i,P}$. 
Now $A_{q \oplus^N p} = A_q \cup A_p$. 
So if $M \in A_{q \oplus^N p}$, then either $M \in A_q$ or $M \in A_p$. 
But $q$ and $p$ are both in $D(P)$, so in the first case, 
$M \cap P \in A_q$, 
and in the second case, $M \cap P \in A_p$. 
In either case, $M \cap P \in A_q \cup A_p = A_{q \oplus^N p}$.

Now let $i \in \dom(F_{q \oplus^N p}) \cap P$, 
and we will show that $F_{q \oplus^N p}(i) \in D_{i,P}$. 
By Definition 11.15, $\dom(F_{q \oplus^N p}) = \dom(F_p)$, for all 
$i \in \dom(F_p) \setminus \dom(F_q)$, $F_{q \oplus^N p}(i) = F_p(i)$, 
and for all $i \in \dom(F_q) \cap \dom(F_p)$, 
$F_{q \oplus^N p}(i) = F_q(i) \oplus_N F_p(i)$. 

First, assume that $i \in \dom(F_p) \setminus \dom(F_q)$. 
Then $F_{q \oplus^N p}(i) = F_p(i)$. 
Since $p \in D(P)$ and $i \in \dom(F_p) \cap P$, we have that 
$F_p(i) \in D_{i,P}$. 

Secondly, assume that $i \in \dom(F_p) \cap \dom(F_q)$. 
Since $q \in N$, we have that $i \in N$. 
So $F_{q \oplus^N p}(i) = F_q(i) \oplus_N F_p(i)$. 
Thus, it suffices to show that 
$F_q(i) \oplus_N F_p(i) \in D_{i,P}$. 
This will follow from Lemma 8.5, provided that the assumptions of 
this lemma are true for $F_p(i)$ and $F_q(i)$.

Since $i \in N \cap P$ and $p \in D(N) \cap D(P)$, 
$F_p(i) \in D_{i,N} \cap D_{i,P}$. 
As $q$ and $i$ are in $N$, $F_q(i) \in N$, and since $q \in D(P)$ and 
$i \in P$, $F_q(i) \in N \cap D_{i,P}$. 
Finally, as $q \le p \restriction N$, 
$F_q(i) \le F_{p \restriction N}(i) = F_p(i) \restriction N$. 
This completes the verification of the assumptions of 
Lemma 8.5. 
By Lemma 8.5(1), we have that 
$F_q(i) \oplus_N F_p(i)$ is in $D_{i,P}$.

(2) Since $q$ and $P$ are in $N$, $q \restriction P \in N$. 
Also, $q \restriction P \in P$, so $q \restriction P \in N \cap P$. 
By Lemmas 12.3 and 12.4, we have that $p \restriction N \in D(P)$ and 
$$
(p \restriction N) \restriction P = (p \restriction P) \restriction (N \cap P).
$$
As $q \le p \restriction N$, it follows by Lemma 11.13(2) that 
$$
q \restriction P \le (p \restriction N) \restriction P = 
(p \restriction P) \restriction (N \cap P).
$$
\end{proof}

\begin{proposition}
Let $N \in \mathcal X$ be simple, 
$P \in \mathcal Y \cap N$ be simple, 
and suppose that $P \cap \kappa \notin S_i$, 
for all $i \in P \cap \lambda^*$. 
Let $p \in D(N) \cap D(P)$, $q \in N \cap D(P)$, and 
$q \le p \restriction N$. 
Assume, moreover, that $\dom(F_q) \subseteq \dom(F_p)$. 
Then 
$$
(q \oplus^N p) \restriction P = 
(q \restriction P) \oplus^{N \cap P} (p \restriction P). 
$$
\end{proposition}

Note that since $\dom(F_q) \subseteq \dom(F_p)$, we also have that 
$$
\dom(F_{q \restriction P}) = \dom(F_q) \cap P \subseteq 
\dom(F_p) \cap P = \dom(F_{p \restriction P}).
$$
By this fact and Lemma 12.5, it follows that 
$(q \oplus^N p) \restriction P$ and 
$(q \restriction P) \oplus^{N \cap P} (p \restriction P)$ are defined.

\begin{proof}
Let 
$$
s := (q \oplus^N p) \restriction P
$$
and 
$$
t := (q \restriction P) \oplus^{N \cap P} (p \restriction P).
$$
Our goal is to prove that $s = t$.

\bigskip

We have that 
\begin{multline*}
A_s = A_{(q \oplus^N p) \restriction P} = A_{q \oplus^N p} \cap P = 
(A_q \cup A_p) \cap P = \\
= (A_q \cap P) \cup (A_p \cap P) = 
A_{q \restriction P} \cup A_{p \restriction P} = 
A_{(q \restriction P) \oplus^{N \cap P} (p \restriction P)} = A_t.
\end{multline*}
Thus, $A_s = A_t$.

Similarly,
\begin{multline*}
\dom(F_s) = \dom(F_{(q \oplus^N p) \restriction P}) = 
\dom(F_{q \oplus^N p}) \cap P = \\
= \dom(F_p) \cap P = \dom(F_{p \restriction P}) 
= \dom(F_{(q \restriction P) \oplus^{N \cap P} (p \restriction P)}) 
= \dom(F_t).
\end{multline*}
So $\dom(F_s) = \dom(F_t)$.

Let $i \in \dom(F_s)$, and we will show that 
$F_s(i) = F_t(i)$. 
Note that $\dom(F_s) \subseteq P$, so $i \in P$. 
By definition, we have that 
$$
F_s(i) = F_{(q \oplus^N p) \restriction P}(i) = 
F_{q \oplus^N p}(i) \restriction P.
$$
The definition of $F_{q \oplus^N p}(i)$ splits into two cases, 
depending on whether $i \in \dom(F_p) \setminus \dom(F_q)$, 
or $i \in \dom(F_p) \cap \dom(F_q)$. 

First, assume that $i \in \dom(F_p) \setminus \dom(F_q)$. 
Then $F_{q \oplus^N p}(i) = F_p(i)$ by Definition 11.15. 
Thus, by the above, 
$$
F_s(i) = F_p(i) \restriction P.
$$
Since $i \notin \dom(F_q)$, also $i \notin \dom(F_{q}) \cap P = 
\dom(F_{q \restriction P})$. 
Thus, by Definition 11.15,
$$
F_t(i) = F_{(q \restriction P) \oplus^{N \cap P} (p \restriction P)}(i) = 
F_{p \restriction P}(i) = F_p(i) \restriction P = F_s(i).
$$

Secondly, assume that $i \in \dom(F_p) \cap \dom(F_q)$. 
Then $i \in N$. 
By Definition 11.15 and the above,
$$
F_s(i) = F_{q \oplus^N p}(i) \restriction P = 
(F_q(i) \oplus_N F_p(i)) \restriction P.
$$
Also, $i \in \dom(F_p) \cap \dom(F_q) \cap P = 
\dom(F_{p \restriction P}) \cap \dom(F_{q \restriction P})$. 
So by Definition 11.15, 
\begin{multline*}
F_t(i) = F_{(q \restriction P) \oplus^{N \cap P} (p \restriction P)}(i) =
F_{q \restriction P}(i) \oplus_{N \cap P} F_{p \restriction P}(i) = \\ 
(F_{q}(i) \restriction P) \oplus_{N \cap P} 
(F_p(i) \restriction P).
\end{multline*}

Thus, to show that $F_s(i) = F_t(i)$, it suffices to show that 
$$
(F_q(i) \oplus_N F_p(i)) \restriction P = 
(F_{q}(i) \restriction P) \oplus_{N \cap P} 
(F_p(i) \restriction P).
$$
This equation follows immediately from Proposition 8.6 for the conditions 
$F_q(i)$ and $F_p(i)$, so it is enough to verify that the assumptions of 
Proposition 8.6 hold.

Since $i \in P$, $P \cap \kappa \notin S_i$. 
As $i \in N \cap P$ and $p \in D(N) \cap D(P)$, 
$F_p(i) \in D_{i,N} \cap D_{i,P}$. 
Since $q \in N \cap D(P)$, $F_q(i) \in N \cap D_{i,P}$. 
And as $q \le p \restriction N$, 
$F_q(i) \le F_{p \restriction N}(i) = F_p(i) \restriction N$. 
Thus, all of the assumptions of Proposition 8.6 are true, and we are done.
\end{proof}

\begin{thm}
Let $P \in \mathcal Y$ be simple, 
$P \prec (H(\lambda),\in,\q)$, and 
suppose that for all $i \in P \cap \lambda^*$, 
$P \cap \kappa \notin S_i$. 
Then $P \cap \q$ forces that $\q / \dot G_{P \cap \q}$ has the 
$\omega_1$-approximation property.
\end{thm}

Recall that by Corollary 11.20, $P \cap \q$ is a regular suborder of $\q$.

The proof of this theorem is almost identical in several places 
to the proof of Theorem 9.2. 
In those places, we will ask the reader to refer to the proof of Theorem 9.2 
for some of the details instead of repeating everything here.

\begin{proof}
By Lemma 1.4, 
it suffices to show that $\q$ forces that the pair 
$$
(V[\dot G_\q \cap P],V[\dot G_\q])
$$
has the 
$\omega_1$-approximation property. 
So let $p$, $\mu$, and $\dot k$ be given such that 
$\mu$ is an ordinal, and 
$p$ forces in $\q$ that $\dot k : \mu \to On$ is a function satisfying that 
for any countable set $a$ in $V[\dot G_\q \cap P]$, 
$\dot k \restriction a \in V[\dot G_\q \cap P]$. 
We will find an extension of $p$ which forces that 
$\dot k$ is in $V[\dot G_\q \cap P]$. 

Fix a regular cardinal $\theta$ large enough so that 
$\q$, $\mu$, and $\dot k$ are members of $H(\theta)$. 
By the stationarity of the simple models in $\mathcal X$ as described 
in Assumption 2.22, fix a countable set 
$M^* \prec H(\theta)$ such that 
$M^*$ contains the parameters $\q$, $P$, $p$, $\mu$, and $\dot k$, and 
satisfies that $M^* \cap H(\lambda)$ is in $\mathcal X$ and is simple.

Let $M := M^* \cap H(\lambda)$. 
Note that since $\q \subseteq H(\lambda)$, 
$M \cap \q = M^* \cap \q$. 
In particular, $p \in M \cap \q$. 
Also, note that since $P \in H(\lambda)$, we have that 
$P \in M$ and $M \cap P = M^* \cap P$.

By Lemma 11.1, fix $p_0 \le p$ such that $M \in A_{p_0}$. 
By the choice of $p$ and $\dot k$, and since $M^* \cap \mu$ is 
in $V$, 
we can fix $p_1 \le p_0$ and a $(P \cap \q)$-name $\dot s$ such that 
$$
p_1 \Vdash_\q \dot k \restriction (M^* \cap \mu) = 
\dot s^{\dot G_\q \cap P}.
$$
Since $M \in A_{p_1}$, 
by Lemma 12.2 we can fix $p_2 \le p_1$ such that 
$p_2 \in D(M) \cap D(P)$.

Since $p_2 \le p$ and $p \in M$, 
it follows that $p_2 \restriction M \le p$ by Lemma 11.13(1). 
So it suffices to prove that 
$p_2 \restriction M$ forces that $\dot k$ is in 
$V[\dot G_\q \cap P]$.

\bigskip

\noindent \emph{Claim 1:} If $t \le p_2$ is in $D(P)$, 
$\nu \in M^* \cap \mu$, and 
$t \Vdash_{\q} \dot k(\nu) = x$ 
(or $t \Vdash_{\q} \dot k(\nu) \ne x$, respectively) 
then $t \restriction P \Vdash_{P \cap \q} \dot s(\nu) = x$ (or 
$t \restriction P \Vdash_{P \cap \q} \dot s(\nu) \ne x$, respectively).

\bigskip

The proof of Claim 1 is identical to the proof of Claim 1 of 
Theorem 9.2, except that the reference to Proposition 6.15 is 
replaced with a reference to Proposition 11.18.

\bigskip

\noindent \emph{Claim 2:} For all 
$q \le p_2 \restriction M$ in $D(P)$, 
$\nu < \mu$, and $x$, 
$$
q \Vdash_{\q} \dot k(\nu) = x \implies \forall r \in \q 
( ( r \le p_2 \restriction M \land r \le q \restriction P ) \implies 
(r \Vdash_\q \dot k(\nu) = x)).
$$

\bigskip

Note that $p_2 \restriction M$, $P$, 
$D(P)$, $\mu$, $\dot k$, and $\q$ are in $M^*$. 
So by the elementarity of $M^*$, it suffices to show that the statement 
holds in $M^*$.

Suppose for a contradiction that there exists 
$q \le p_2 \restriction M$ 
in $M^* \cap D(P)$, $\nu \in M^* \cap \mu$, and $x \in M^*$ such that 
$$
q \Vdash_\q \dot k(\nu) = x,
$$
but there is $r_0 \in M^* \cap \q$ with $r_0 \le p_2 \restriction M$ and 
$r_0 \le q \restriction P$ such that 
$$
r_0 \not \Vdash_\q \dot k(\nu) = x.
$$
By the elementarity of $M^*$, 
we can fix $r \le r_0$ in $M^* \cap D(P)$ such that 
$$
r \Vdash_\q \dot k(\nu) \ne x.
$$
Then $r \le p_2 \restriction M$ and $r \le q \restriction P$. 
Since $r \le q \restriction P$ and $q \restriction P \in P$, it follows that 
$r \restriction P \le q \restriction P$ 
by Lemma 11.13(1). 

Observe that if we let 
$$
q' := q \uplus (\dom(F_r) \setminus \dom(F_q))$$
and 
$$
r' := r \uplus (\dom(F_q) \setminus \dom(F_r)),
$$
then $q'$ and $r'$ satisfy exactly the 
same properties which we stated that 
$q$ and $r$ satisfy, and moreover, $\dom(F_{q'}) = \dom(F_{r'})$. 
Let us check this observation carefully.

Since $q$ and $r$ are in $M^*$, so are $q'$ and $r'$. 
And by Definition 10.7, 
$$
\dom(F_{q'}) = \dom(F_q) \cup \dom(F_r) = 
\dom(F_{r'}).
$$
Since $q' \le q \le p_2 \restriction M$, we have that 
$q' \le p_2 \restriction M$. 
As $q$ and $r$ are in $D(P)$, so are $q'$ and $r'$ by Lemma 11.10. 
And since $q' \le q$ and $r' \le r$, we have that 
$q' \Vdash_\q \dot k(\nu) = x$ and 
$r' \Vdash_\q \dot k(\nu) \ne x$. 
Finally, $r' \le r \le p_2 \restriction M$ implies that 
$r' \le p_2 \restriction M$. 
Letting $y := \dom(F_r) \setminus \dom(F_q)$, 
the fact that $r \le q \restriction P$ implies by Lemmas 10.8(4) 
and 11.17 that 
$$
r' \le r \le (q \restriction P) \uplus (y \cap P) = 
(q \uplus y) \restriction P = q' \restriction P,
$$
so $r' \le q' \restriction P$. 
And by Lemma 11.13(1), this last inequality implies that 
$r' \restriction P \le q' \restriction P$.

By replacing $q$ and $r$ with $q'$ and $r'$ respectively if necessary, 
we can assume without loss of generality that 
$\dom(F_q) = \dom(F_r)$. 
Let $x := \dom(F_q) \setminus \dom(F_{p_2})$. 
Define 
$$
p_3 := p_2 \uplus x.
$$
Then by Proposition 11.18, 
$q$ and $r$ are below $p_3 \restriction M$. 
Also, $q \oplus^{M} p_3$ is a condition below $q$ and $p_3$, and 
$r \oplus^{M} p_3$ is a condition below $r$ and $p_3$. 
Also, by Lemmas 11.6 and 11.10, $p_3 \in D(M) \cap D(P)$.

Since $\dom(F_q)$ and $\dom(F_r)$ are subsets of 
$\dom(F_{p_3})$, by Proposition 12.6 
we have that 
$$
(q \oplus^{M} p_3) \restriction P = 
(q \restriction P) \oplus^{M \cap P} (p_3 \restriction P)
$$
and 
$$
(r \oplus^{M} p_3) \restriction P = 
(r \restriction P) \oplus^{M \cap P} (p_3 \restriction P).
$$

We would like to apply Lemma 12.1 to $M \cap P$, 
$p_3 \restriction P$, $q \restriction P$, and 
$r \restriction P$. 
Let us check that the assumptions of Lemma 12.1 hold for these objects. 
By Lemma 2.30, $M \cap P$ is simple. 
Since $p_3 \in D(M) \cap D(P)$, 
it follows that $p_3 \restriction P \in D(M \cap P)$ by Lemma 12.3. 
As $q$, $r$, and $P$ are in $M^*$, we have that 
$q \restriction P$ and $r \restriction P$ are in 
$M^* \cap P = M \cap P$. 
Finally, we observed above that $r \restriction P \le q \restriction P$, and 
$q \le p_3 \restriction M$ implies that 
$q \restriction P \le (p_3 \restriction P) \restriction (M \cap P)$ 
by Lemma 12.5.

Thus, all of the assumptions of Lemma 12.1 hold. 
Consequently,
$$
(r \restriction P) \oplus^{M \cap P} (p_3 \restriction P) \le 
(q \restriction P) \oplus^{M \cap P} (p_3 \restriction P).
$$
Combining this with the equalities above, we have that 
$$
(r \oplus^{M} p_3) \restriction P \le (q \oplus^{M} p_3) \restriction P.
$$

We claim that this last inequality is impossible. 
In fact, we will show that $(r \oplus^M p_3) \restriction P$ and 
$(q \oplus^M p_3) \restriction P$ are incompatible. 
This contradiction will complete the proof of Claim 2.

We know that $r \Vdash_\q \dot k(\nu) \ne x$, and therefore, since 
$r \oplus^{M} p_3 \le r$, we have that 
$r \oplus^{M} p_3 \Vdash_\q \dot k(\nu) \ne x$. 
By Claim 1, 
$$
(r \oplus^{M} p_3) \restriction P \Vdash_{P \cap \q} \dot s(\nu) \ne x.
$$
Similarly, $q \Vdash_\q \dot k(\nu) = x$, and therefore, since 
$q \oplus^M p_3 \le q$, we have that 
$q \oplus^M p_3 \Vdash_\q \dot k(\nu) = x$. 
By Claim 1,
$$
(q \oplus^{M} p_3) \restriction P \Vdash_{P \cap \q} \dot s(\nu) = x.
$$
Thus, indeed 
$(r \oplus^M p_3) \restriction P$ and 
$(q \oplus^M p_3) \restriction P$ are incompatible, since they 
force contradictory information.
This completes the proof of Claim 2.

\bigskip

The proof that $p_2 \restriction M$ forces that 
$\dot k$ is in $V[\dot G_\q \cap P]$ follows from Claim 
2 in exactly the same way that the analogous conclusion in Theorem 9.2 
followed from Claim 2 there.
\end{proof}

\bigskip

\addcontentsline{toc}{section}{13. The consistency result}

\textbf{\S 13. The consistency result}

\stepcounter{section}

\bigskip

We now fulfill the mission of the paper and prove that 
it is consistent, relative to the consistency of a 
greatly Mahlo cardinal, that the approachability ideal $I[\omega_2]$ 
does not have a maximal set modulo clubs.

We work in a ground model $V$ in which $\kappa$ is a greatly Mahlo cardinal, 
$2^\kappa = \kappa^+$, and $\Box_\kappa$ holds. 
The consistency of a greatly Mahlo cardinal easily implies the 
consistency of these assumptions.

It is a standard fact that $\kappa$ being greatly Mahlo implies that there 
exists a sequence 
$\langle B_i : i < \kappa^+ \rangle$ 
of stationary subsets of $\kappa$ 
satisfying the following properties:
\begin{enumerate}
\item for each $i < \kappa^+$, for all $\beta \in B_i$, $\beta$ 
is strongly inaccessible;

\item for each $i < \kappa^+$, for all $\beta \in B_{i+1}$, 
$B_i \cap \beta$ is stationary in $\beta$;

\item for all $i < j < \kappa^+$, there is a club set $C \subseteq \kappa$ 
such that $B_j \cap C \subseteq B_i$;

\item for each $i < \kappa^+$, $B_i \setminus B_{i+1}$ 
is stationary.
\end{enumerate}
Such a sequence is obtained by iterating the Mahlo operation 
$$
M(A) := \{ \alpha \in \kappa \cap \cof(>\! \omega) : 
A \cap \alpha \ \textrm{is stationary in $\alpha$} \},
$$
starting with the set of inaccessibles in $\kappa$, 
and taking diagonal intersections 
of some form at limit stages. 
We refer the reader to \cite[Section 4]{baumgartner} 
for more information about greatly Mahlo cardinals.

\bigskip

The results of this paper up to now were made in the context of several 
fixed objects, together with some assumptions about these objects. 
Specifically, in Section 2 we fixed $\kappa$, $\lambda$, $\Lambda$, 
$\mathcal A$, $\mathcal X$, and $\mathcal Y$, 
satisfying the properties described in 
Notations 2.1, 2.2, 2.3, and 2.4, and Assumptions 
2.5, 2.6, 2.19, 2.22, 2.23, and 10.1. 
In addition, in Section 10 we fixed an ordinal $\lambda^*$ and a 
sequence $\langle S_i : i < \lambda^* \rangle$ satisfying 
Notation 10.2 and Assumption 10.3.

We now specify such objects explicitly and justify the properties which we 
have been assuming about them. 
We will refer to our previous paper \cite{jk27} for some of the 
definitions and proofs.

The greatly Mahlo cardinal $\kappa$ which we fixed at the beginning of this section 
is the cardinal described in Notation 2.1. 
The cardinal $\lambda$ described in Notation 2.1 is equal to $\kappa^+$.

We refer to \cite[Notation 1.7]{jk27} for the definition of $\Lambda$. 
In that paper, we have that $\Lambda = C^* \cap \cof(>\! \omega)$, 
where $C^*$ is a club subset of $\kappa$. 
The club set $C^*$, in turn, is defined in terms of a thin stationary 
set $T^* \subseteq P_{\omega_1}(\kappa)$. 
We must justify, therefore, the existence of a thin stationary set. 
But $\kappa$ is strongly inaccessible, so we can let $T^*$ be 
equal to the entire set $P_{\omega_1}(\kappa)$. 
The properties of $\kappa$, $\lambda$, and $\Lambda$ described 
in Notations 2.1 and 2.2 and Assumption 10.1 are now immediate.

We refer to \cite[Section 7]{jk27} for the definitions of 
$\mathcal A$, $\mathcal X$, and $\mathcal Y$. 
At the beginning of that section, it is assumed that 
$2^\kappa = \kappa^+$ and $\Box_\kappa$, which are exactly the same 
assumptions which we made above. 
Let $\mathcal A$ denote the structure which is obtained by 
expanding the structure on $H(\kappa^+)$ 
specified in \cite[Notation 7.6]{jk27} by adding the sequence 
$\langle B_i : i < \kappa^+ \rangle$ as a predicate. 
This structure has a 
well-ordering of $H(\kappa^+)$ as a predicate, 
and therefore has definable Skolem functions. 
It also has $\kappa$ and $\Lambda$ as constants. 
Thus, the description of $\mathcal A$ made in 
Notation 2.3 is satisfied.

We define $\mathcal X$ exactly as in \cite[Notation 7.7]{jk27}. 
Then by definition, 
for all $M \in \mathcal X$, $M$ is a countable elementary substructure 
of $\mathcal A$. 
We define $\mathcal Y$ as the set of models $P$ which are in the set 
defined in \cite[Notation 7.8]{jk27} and also satisfy that 
$\cf(P \cap \kappa) > \omega$. 
Then by definition, for all $P \in \mathcal Y$, 
$P$ is an elementary substructure of $\mathcal A$, 
$|P| < \kappa$, and $P \cap \kappa \in \kappa$. 
Thus, the properties of $\mathcal X$ and $\mathcal Y$ described in 
Notation 2.4 are satisfied.

The next lemma verifies Assumptions 2.5, 2.6, and 2.19.

\begin{lemma}
\begin{enumerate}
\item If $P$ and $Q$ are in $\mathcal Y$, then $P \cap Q \in \mathcal Y$;

\item if $M \in \mathcal X$ and $P \in \mathcal Y$, then 
$M \cap P \in \mathcal X$;

\item if $M$ and $N$ are in $\mathcal X$ and $\{ M, N \}$ 
is adequate, then $M \cap N \in \mathcal X$;

\item if $M \in \mathcal X$, $\alpha \in \Lambda \cup \{ \kappa \}$, and 
$Sk(\alpha) \cap \kappa = \alpha$ if $\alpha < \kappa$, then 
$M \cap \alpha \in Sk(\alpha)$.
\end{enumerate}
\end{lemma}

\begin{proof}
(1), (2), and (3) follow immediately from 
\cite[Lemma 7.16]{jk27}. 
(4) Note that $Sk(\kappa) \cap \kappa = \kappa$. 
Let $M \in \mathcal X$ and 
$\alpha \in \Lambda \cup \{ \kappa \}$ be as in (4). 
Since $\kappa$ is strongly inaccessible, 
for all $\beta < \alpha$, the cardinality of $P(\beta)$ 
is in $Sk(\alpha) \cap \kappa = \alpha$ by elementarity. 
So again by elementarity, 
$P(\beta) \subseteq Sk(\alpha)$. 
As $M$ is countable and $\cf(\alpha) > \omega$, 
$M \cap \alpha \in P(\beta)$ for some $\beta < \alpha$. 
Hence, $M \cap \alpha \in Sk(\alpha)$.
\end{proof}

In \cite[Definition 7.18]{jk27}, the notion of a simple model 
in $\mathcal X \cup \mathcal Y$ is defined. 
This notion is different from what we are calling simple in this paper, so 
let us momentarily refer to the property from 
\cite[Definition 7.18]{jk27} as strongly simple. 
By \cite[Lemma 8.2]{jk27}, a set in $\mathcal X \cup \mathcal Y$ 
which is strongly simple is also simple in the sense that we are using in 
the present paper.

The next lemma verifies Assumption 2.22.

\begin{lemma}
There are stationarily many sets 
$N \in P_{\omega_1}(H(\kappa^+))$ such that 
$N \in \mathcal X$ and $N$ is simple.
\end{lemma}

\begin{proof}
By \cite[Proposition 7.20]{jk27}, there are stationarily many 
strongly simple models in $\mathcal X$. 
Since strongly simple implies simple, there are stationarily many 
simple models in $\mathcal X$.
\end{proof}

The next lemma gives a sufficient criterion for a set $P$ being a simple 
model in $\mathcal Y$.

\begin{lemma}
Suppose that $P \in P_{\kappa}(H(\kappa^+))$ and $P$ satisfies:
\begin{enumerate}
\item $P \prec \mathcal A$;
\item $P \cap \kappa \in \kappa$;
\item $\cf(\sup(P \cap \kappa^+)) = P \cap \kappa$.
\end{enumerate}
Then $P \in \mathcal Y$ and $P$ is simple.
\end{lemma} 

\begin{proof}
By \cite[Lemma 7.15]{jk27}, assumptions (1), (2), and (3) imply that 
$P \in \mathcal Y$. 
By \cite[Lemma 8.3]{jk27}, assumption (3) implies that 
$P$ is strongly simple, and hence simple.
\end{proof}

The next lemma justifies Assumption 2.23.

\begin{lemma}
The set of $P \in P_{\kappa}(H(\kappa^+))$ such that 
$P \in \mathcal Y$ and $P$ is simple is stationary.
\end{lemma}

\begin{proof}
Given a function $F : H(\kappa^+)^{<\omega} \to 
H(\kappa^+)$, 
build a membership increasing and continuous chain 
$\langle P_i : i < \kappa \rangle$ of elementary substructures of 
$\mathcal A$ which have size less than $\kappa$ and are closed 
under $F$. 
Then there is a club of $\alpha$ such that 
$P_\alpha \cap \kappa = \alpha$. 
Fix a strongly inaccessible cardinal $\alpha$ in this club. 
Then $P_\alpha$ is closed under $F$ and 
satisfies properties (1), (2), and (3) of 
Lemma 13.3. 
Hence, $P_\alpha$ is in $\mathcal Y$ and is simple.
\end{proof}

Let the ordinal $\lambda^*$ from Notation 10.2 be equal to 
$\kappa^+$. 
Define, for each $i < \kappa^+$, 
$$
S_i := (B_i \setminus B_{i+1}) \cap C^* \cap C',
$$
where $C^*$ and $C'$ are club subsets of $\kappa$ such that 
$\Lambda = C^* \cap \cof(>\! \omega)$, and for all 
$\alpha \in C'$, $Sk(\alpha) \cap \kappa = \alpha$. 
By the properties described at the beginning of the section 
for $\langle B_i : i < \kappa^+ \rangle$, 
each $S_i$ is a stationary subset of 
$\kappa \cap \cof(>\!\omega)$, and 
for all $\alpha \in S_i$, $\alpha \in \Lambda$ and 
$Sk(\alpha) \cap \kappa = \alpha$.

Consider $i < j < \kappa^+$. 
We claim that there is a club set 
$C_{i,j} \subseteq \kappa$, which is definable 
in $\mathcal A$ from $i$ and $j$, such that 
$S_i \cap S_j \cap C_{i,j} = \emptyset$. 
Since $i +1 \le j$, 
we know that there exists a club $C$ such that 
$B_j \cap C \subseteq B_{i+1}$. 
Let $C_{i,j}$ be the least such club in the well-ordering of $H(\kappa^+)$ 
which is a predicate of $\mathcal A$. 
Since $S_j \subseteq B_j$, 
$S_j \cap C_{i,j} \subseteq B_{i+1}$. 
But by definition, $B_{i+1}$ is disjoint from $S_i$. 
Thus, $S_i \cap S_j \cap C_{i,j} = \emptyset$. 
This completes the verification of the properties described 
in Notation 10.2 and Assumption 10.3.

Finally, for each $i < \kappa^+$, let $\p_i$ denote the forcing 
poset defined in Definition 4.2 for adding a partial square sequence on 
$S_i$, and let $\q$ be the product forcing defined in Definition 10.6.

\bigskip

This completes the choice of all of the background 
objects and the verification of all of 
the assumptions which we made about them. 
By Corollary 11.22, the forcing poset $\q$ preserves $\omega_1$, 
is $\kappa$-c.c., forces that $\kappa$ is equal to $\omega_2$, and 
forces that for all $i < \kappa^+$, $S_i \in I[\omega_2]$. 

\bigskip

It remains to show that $\q$ forces that $I[\omega_2]$ does not 
have a maximal set modulo clubs.
We will prove a technical lemma about names and then finish 
the proof of the consistency result.

\begin{lemma}
Suppose that $P \in \mathcal Y$ is simple and 
$P \prec (H(\kappa^+),\in,\q)$. 
Let $\dot a \in P$ be a nice $\q$-name for a 
countable subset of $\kappa$. 
Then for any generic filter $G$ on $\q$, 
$\dot a^G \in V[G \cap P]$.
\end{lemma}

Recall that by Corollary 11.20, $P \cap \q$ is a regular suborder of $\q$. 
Therefore, $G \cap P$ is a generic filter on $P \cap \q$ and 
$V[G \cap P] \subseteq V[G]$.

\begin{proof}
Let $\alpha := P \cap \kappa$. 
Since $\q$ is $\kappa$-c.c., by elementarity we can fix a set $b \in P$ 
such that $b$ is a bounded subset of $\kappa$ and 
$\q$ forces that $\dot a \subseteq b$. 
Note that $b \subseteq \alpha$.

Since $\dot a$ is a nice name, 
for each $\gamma < \kappa$ there is a unique antichain 
$A_\gamma$ such that 
$$
(p,\check \gamma) \in \dot a \iff p \in A_\gamma.
$$
Moreover, as $\q$ forces that $\dot a \subseteq \alpha$, 
$A_\gamma = \emptyset$ for all $\gamma \in \kappa \setminus \alpha$. 
Let $\gamma < \alpha$. 
Since $\dot a \in P$, by elementarity $A_\gamma$ is in $P$. 
As $\q$ is $\kappa$-c.c., $|A_\gamma| < \kappa$.  
Since $P \cap \kappa \in \kappa$, we have that 
$A_\gamma \subseteq P \cap \q$. 
It follows that $\dot a$ is actually a $(P \cap \q)$-name. 
Since $P \cap \q$ is a regular suborder of $\q$, 
$\dot a^G = \dot a^{G \cap P}$. 
Thus, $\dot a^G \in V[G \cap P]$.
\end{proof}

\begin{thm}
The forcing poset $\q$ forces that the approachability ideal 
$I[\omega_2]$ does not have a maximal set modulo clubs.
\end{thm}

\begin{proof}
Suppose for a contradiction that there is a condition $p$ and a 
sequence $\vec{\dot a} = \langle \dot a_i : i < \kappa \rangle$ of 
$\q$-names for countable subsets of $\kappa$ such that $p$ forces that 
$S_{\vec{\dot a}}$ is a maximal set in $I[\omega_2]$ modulo clubs. 
This means that $p$ forces that whenever $S \in I[\omega_2]$, 
then there is a club $C \subseteq \omega_2$ such that 
$S \cap C \subseteq S_{\vec{\dot a}}$.

Without loss of generality, we may assume that each $\dot a_i$ is a 
nice $\q$-name for a countable subset of $\kappa$. 
Since $\q$ is $\kappa$-c.c., it follows that each name $\dot a_i$, and therefore 
the entire sequence of names $\vec{\dot a}$, is a member of $H(\kappa^+)$.

In the ground model $V$, fix a set $X$ satisfying:
\begin{enumerate}
\item $X \prec \mathcal A$ and 
$X \prec (H(\kappa^+),\in,\q,\vec{\dot a})$;

\item $|X| = \kappa$;

\item $\tau := X \cap \kappa^+$ is an ordinal in $\kappa^+$;

\item $X^{<\kappa} \subseteq X$.
\end{enumerate}
This is possible since $\kappa$ is strongly inaccessible. 
Note that by (3) and (4), $\cf(\tau) = \kappa$.

Fix a membership increasing and continuous sequence 
$\langle P_i : i < \kappa \rangle$ of sets of size 
less than $\kappa$, whose union is equal to $X$, such that 
each $P_i$ is an elementary substructure of $\mathcal A$ and 
$(H(\kappa^+),\in,\q,\vec{\dot a})$. 
This is possible since $X^{< \kappa} \subseteq X$ by (4). 
Note that by elementarity, $\sup(P_i \cap \kappa^+) \in P_{i+1}$ 
for all $i < \kappa^+$.

Using the properties of the sequence 
$\langle B_i : i < \kappa^+ \rangle$ described 
at the beginning of this section, 
we can fix, for each $i < \tau$, a club set $C_i \subseteq \kappa$ such that 
$$
B_\tau \cap C_i \subseteq B_i.
$$
Define a function $F : \tau \times \kappa \to \kappa$ by letting 
$$
F(i,\gamma) := \min(C_i \setminus \gamma),
$$
for all $(i,\gamma) \in \tau \times \kappa$. 
Note that $F \subseteq X$.
Also, fix a club $C^\tau \subseteq \kappa$ 
such that $B_{\tau+1} \cap C^\tau \subseteq B_\tau$.

Fix a club 
$D \subseteq \kappa$ such that for all $\beta \in D$, 
$P_\beta \cap \kappa = \beta$, $\beta \in C^\tau$, and 
$P_\beta$ is closed under the function $F$.

\bigskip

\noindent \emph{Claim 1:} Suppose that $\beta \in D$ and 
$\beta$ is strongly inaccessible. Then 
$P_\beta \in \mathcal Y$, $P_\beta$ is simple, 
$P_\beta \prec (H(\kappa^+),\in,\q)$, 
$\beta \in C^\tau$, and for all 
$i \in P_\beta \cap \kappa^+$, $\beta \in C_i$.

\bigskip

The fact that $P_\beta$ is closed under $F$ easily implies that 
for all $i \in P_\beta \cap \kappa^+ = P_\beta \cap \tau$, 
$P_\beta \cap \kappa = \beta$ is a limit point of $C_i$, and therefore 
is in $C_i$. 
And $\beta \in C^\tau$ by the definition of $D$. 
The set $P_\beta$ is an elementary substructure of $\mathcal A$ and 
$(H(\kappa^+),\in,\q)$ by the choice of the sequence 
$\langle P_i : i < \kappa \rangle$. 
Since $P_\beta$ is the union of the sequence 
$\langle P_i : i < \beta \rangle$, and 
$\sup(P_i \cap \kappa^+) \in P_{i+1}$ for all 
$i < \beta$, it follows that 
$$
\cf(\sup(P_\beta \cap \kappa^+)) = \cf(\beta) = \beta = P_\beta \cap \kappa.
$$
By Lemma 13.3, it follows that $P_\beta \in \mathcal Y$ and $P_\beta$ is simple. 
This completes the proof of Claim 1.

\bigskip

Let $G$ be a generic filter on $\q$ with $p \in G$. 
For each $i < \kappa$, let $a_i := \dot a_i^G$, and 
let $\vec a  := \langle a_i : i < \kappa \rangle$. 
Then by assumption, in $V[G]$ the set 
$S_{\vec a}$ is maximal in $I[\omega_2]$ modulo clubs. 
In $V[G]$ the set $S_{\tau+1}$ is in $I[\omega_2]$. 
Since $S_{\vec a}$ is maximal modulo clubs, 
fix a club $C \subseteq \kappa$ in $V[G]$ such that 
$$
S_{\tau+1} \cap C \subseteq S_{\vec a}.
$$

As the set $S_{\tau+1}$ is stationary in $V[G]$, we can fix 
$$
\beta \in S_{\tau+1} \cap \lim(D) \cap C.
$$
Then $\beta \in S_{\tau+1} \cap C \subseteq S_{\vec a}$. 
Let $P := P_\beta$. 

Note that since $\beta \in S_{\tau+1}$, 
$\beta$ is strongly inaccessible in $V$. 
Also, $\beta \in D$. 
So by Claim 1, $P \in \mathcal Y$, $P$ is simple, 
$P \prec (H(\kappa^+),\in,\q)$, 
$\beta \in C^\tau$, 
and for all 
$i \in P \cap \kappa^+$, $\beta \in C_i$.

Since $\beta \in C^\tau$, $S_{\tau+1} \subseteq B_{\tau+1}$, 
and $B_{\tau+1} \cap C^\tau \subseteq B_\tau$, it follows that 
$\beta \in B_\tau$. 
Also, for all $i \in P \cap \kappa^+$, 
$\beta \in B_\tau \cap C_i \subseteq B_i$, so $\beta \in B_i$. 
Therefore, for all $i \in P \cap \kappa^+$, 
since $i+1 \in P \cap \kappa^+$ by elementarity, $\beta \in B_{i+1}$. 
But $S_i \subseteq B_{i} \setminus B_{i+1}$. 
So for all $i \in P \cap \kappa^+$, $P \cap \kappa = \beta \notin S_i$.

By Corollary 11.20 and Theorem 12.7, it follows that 
$P \cap \q$ is a regular suborder of $\q$ and 
$P \cap \q$ forces that $\q / \dot G_{P \cap \q}$ has the 
$\omega_1$-approximation property.

\bigskip

\noindent \emph{Claim 2:} The forcing poset $P \cap \q$ 
is $\beta$-c.c.

\bigskip

Let $A$ be an antichain of $P \cap \q$, and we will prove that 
$|A| < \beta$. 
Without loss of generality, assume that $A$ is maximal. 
Since $\beta$ is strongly inaccessible and is a limit point of $D$, 
we have that $D \cap \beta$ is a club subset of $\beta$. 
As $P$ is the union of the $\subseteq$-increasing and continuous 
sequence $\langle P_i : i < \beta \rangle$, there is a club 
$E \subseteq \beta$ such that for all 
$\gamma \in E$:
\begin{enumerate}
\item $\gamma \in D \cap \beta$;
\item $P_\gamma \prec (P,\in,P \cap \q,A)$.
\end{enumerate}

Now $\beta \in B_{\tau+1}$, which implies that $B_\tau \cap \beta$ 
is stationary in $\beta$. 
Since $E$ is a club subset of $\beta$, we can fix 
$\gamma \in E \cap B_\tau$. 
Then in particular, $\gamma \in D$ and $\gamma$ is strongly 
inaccessible.

By Claim 1, it follows that 
$P_\gamma \in \mathcal Y$, $P_\gamma$ is simple, and 
$P_\gamma \prec (H(\kappa^+),\in,\q)$. 
Therefore, by Corollary 11.20, 
$P_\gamma \cap \q$ is a regular suborder of $\q$. 
Since $P_\gamma \cap \q \subseteq P \cap \q$, it follows 
from Lemma 1.1 that $P_\gamma \cap \q$ is a regular 
suborder of $P \cap \q$.

We claim that $P_\gamma \cap A$ is a maximal antichain of 
$P_\gamma \cap \q$. 
It is obviously an antichain. 
Let $v \in P_\gamma \cap \q$. 
Then since $A$ is a maximal antichain of $P \cap \q$, 
there is $s \in A$ such that $s$ is compatible in $P \cap \q$ with $v$. 
But $P_\gamma$ 
is an elementary substructure of $(P,\in,P \cap \q,A)$. 
So by elementarity, there is $s \in A \cap P_\gamma$ which is 
compatible in $P \cap \q$ with $v$. 
Again by elementarity, $s$ and $v$ are compatible in $P_\gamma \cap \q$. 
This completes the proof that $P_\gamma \cap A$ is a maximal 
antichain of $P_\gamma \cap \q$.

As $P_\gamma \cap A$ is a maximal antichain of $P_\gamma \cap \q$ 
and $P_\gamma \cap \q$ is a regular suborder of $P \cap \q$, 
it follows that $P_\gamma \cap A$ is predense in $P \cap \q$. 
But $A$ is a maximal antichain of $P \cap \q$; therefore, 
it must be the case that $A = P_\gamma \cap A$. 
So $A \subseteq P_{\gamma}$. 
But as $P_\gamma \in P$, we have that 
$|P_\gamma| \in P \cap \kappa = \beta$ by elementarity. 
So $|A| < \beta$. 
This completes the proof of Claim 2.

\bigskip

Recall that $\beta \in S_{\vec a}$ in $V[G]$. 
Therefore, in $V[G]$ there is a set $c$ which is cofinal in 
$\beta$ with order type $\omega_1$, and for all $\xi < \beta$, 
$c \cap \xi \in \{ a_i : i < \beta \}$. 
Now for all $i < \beta$, the name $\dot a_i$ is in $P$ by elementarity. 
By Lemma 13.5, it follows that $\dot a_i^G = a_i$ is in $V[G \cap P]$.

So $c$ is a cofinal subset of $\beta$ with order type 
$\omega_1$, and every proper initial segment of $c$ is in 
$V[G \cap P]$. 
It easily follows that whenever 
$x$ is a countable set in $V[G \cap P]$, then 
$c \cap x \in V[G \cap P]$. 
Since $\q / (G \cap P)$ has the $\omega_1$-approximation property 
in $V[G \cap P]$, we have that $c \in V[G \cap P]$. 
As $c$ has order type $\omega_1$ and $\beta$ is a strongly 
inaccessible cardinal in $V$, it follows that $\beta$ is no longer 
regular in $V[G \cap P]$. 
But by Claim 2, $P \cap \q$ is $\beta$-c.c., and so $P \cap \q$ 
preserves the regularity of $\beta$. 
Since $V[G \cap P]$ is a generic extension of $V$ 
by the forcing poset 
$P \cap \q$, we have a contradiction.
\end{proof}

\bibliographystyle{plain}
\bibliography{paper29}

\end{document}